\documentclass[12pt, leqno]{article}
\usepackage{amsfonts,cite}
\usepackage{amsmath, amssymb}
\usepackage{euler}
\usepackage{ marvosym }
\usepackage{enumitem}
\usepackage[amsmath,thmmarks]{ntheorem}
\usepackage{chngcntr,color}

\counterwithin*{equation}{section}
\counterwithin*{equation}{subsection}

\newtheorem{theorem}{Theorem}

\numberwithin{section}{part}
\newtheorem{corollary}{Corollary}
\newtheorem{lemma}{Lemma}
\newtheorem*{remark}{Remark}
\newtheorem*{definition}{Definition}
\newtheorem*{remarks}{Remarks}

\newenvironment{proof}[1][Proof]{\noindent\textbf{#1.} }{\ \rule{0.5em}{0.5em}}

\newcommand{\A}{\mathcal{A}}
\newcommand{\M}{\mathcal{M}}

\input{tcilatex}

\begin{document}

\title{Finiteness Principles for Smooth Selection}
\date{\today }
\author{Charles Fefferman, Arie Israel, Garving K. Luli \thanks{%
The first author is supported in part by NSF grant DMS-1265524, AFOSR
grant FA9550-12-1-0425, and Grant No 2014055 from the United States-Israel Binational Science Foundation (BSF). The third author is supported in part by NSF grant
DMS-1355968 and a start-up fund from UC Davis.}}
\maketitle

\section*{Introduction}

In this paper, we extend a basic finiteness principle \cite{bs-1994,f-2005},
used in \cite{fb1,fb2} to fit smooth functions $F$ to data. Our results
raise the hope that one can start to understand constrained interpolation
problems in which e.g. the interpolating function $F$ is required to be
nonnegative everywhere.

Let us set up notation. We fix positive integers $m$, $n$, $D$. We will work
with the function spaces $C^{m}(\mathbb{R}^{n},\mathbb{R}^{D})$ and $%
C^{m-1,1}(\mathbb{R}^{n},\mathbb{R}^{D})$ and their norms $\Vert F\Vert
_{C^{m}\left( \mathbb{R}^{n},\mathbb{R}^{D}\right) }$ and $\left\Vert
F\right\Vert _{C^{m-1,1}\left( \mathbb{R}^{n},\mathbb{R}^{D}\right) }$.
Here, $C^{m}\left( \mathbb{R}^{n},\mathbb{R}^{D}\right) $ denotes the space
of all functions $F:\mathbb{R}^{n}\rightarrow \mathbb{R}^{D}$ whose
derivatives $\partial ^{\beta }F$ (for all $\left\vert \beta \right\vert
\leq m$) are continuous and bounded on $\mathbb{R}^{n}$, and $%
C^{m-1,1}\left( \mathbb{R}^{n},\mathbb{R}^{D}\right) $ denotes the space of
all $F:\mathbb{R}^{n}\rightarrow \mathbb{R}^{D}$ whose derivatives $\partial
^{\beta }F$ (for all $\left\vert \beta \right\vert \leq m-1$) are bounded
and Lipschitz on $\mathbb{R}^{n}$. When $D=1$, we write $C^{m}\left( \mathbb{%
R}^{n}\right) $ and $C^{m-1,1}\left( \mathbb{R}^{n}\right) $ in place of $%
C^{m}\left( \mathbb{R}^{n},\mathbb{R}^{D}\right) $ and $C^{m-1,1}\left( 
\mathbb{R}^{n},\mathbb{R}^{D}\right) $.

Expressions $c\left( m,n\right) $, $C\left( m,n\right) $, $k\left(
m,n\right) $, etc. denote constants depending only on $m$, $n$; these
expressions may denote different constants in different occurrences. Similar
conventions apply to constants denoted by $C\left( m,n,D\right)$, $k\left(
D\right)$, etc.

If $X$ is any finite set, then $\#\left( X\right) $ denotes the number of
elements in $X$.

We recall the basic finiteness principle of \cite{f-2005}.

\begin{theorem}
\label{Th1}For large enough $k^{\#}=k\left( m,n\right) $ and $C^{\#}=C\left(
m,n\right) $ the following hold:

\begin{description}
\item[(A) $C^{m}\text{ FLAVOR}$] Let $f:E\rightarrow \mathbb{R}$ with $%
E\subset \mathbb{R}^{n}$ finite. Suppose that for each $S\subset E$ with $%
\#\left( S\right) \leq k^{\#}$ there exists $F^{S}\in C^{m}\left( \mathbb{R}%
^{n}\right) $ with norm $\left\Vert F^{S}\right\Vert _{C^{m}\left( \mathbb{R}%
^{n}\right) }\leq 1$, such that $F^{S}=f$ on $S$. Then there exists $F\in
C^{m}\left( \mathbb{R}^{n}\right) $ with norm $\left\Vert F\right\Vert
_{C^{m}\left( \mathbb{R}^{n}\right) }\leq C^{\#}$, such that $F=f$ on $E$.

\item[(B) $C^{m-1,1}\text{ FLAVOR} $] Let $f:E\rightarrow \mathbb{R}$ with $%
E\subset \mathbb{R}^{n}$ arbitrary. Suppose that for each $S\subset E$ with $%
\#\left( S\right) \leq k^{\#}$, there exists $F^{S}\in C^{m-1,1}\left( 
\mathbb{R}^{n}\right) $ with norm $\left\Vert F^{S}\right\Vert
_{C^{m-1,1}\left( \mathbb{R}^{n}\right) }\leq 1$, such that $F^{S}=f$ on $S$%
. Then there exists $F\in C^{m-1,1}\left( \mathbb{R}^{n}\right) $ with norm $%
\left\Vert F\right\Vert _{C^{m-1,1}\left( \mathbb{R}^{n}\right) }\leq C^{\#}$%
, such that $F=f$ on $E. $
\end{description}
\end{theorem}

Theorem \ref{Th1} and several related results were conjectured by Y. Brudnyi
and P. Shvartsman in \cite{bs-1994-a} and \cite{bs-1994} (see also \cite{pavel-1984,pavel-1986,pavel-1987}). The first nontrivial case $m=2$
with the sharp \textquotedblleft finiteness constant\textquotedblright\ $%
k^{\#}=3\cdot 2^{n-1}$ was proven by P. Shvartsman \cite{pavel-1982,pavel-1987}; see also \cite{pavel-1984, bs-2001}. The proof of Theorem \ref{Th1}
for general $m$, $n$ appears in \cite{f-2005}. For general $m$, $n$, the optimal $%
k^{\#}$ is unknown, but see \cite{bm-2007, s-2008}.

The proof \cite{pavel-1984,pavel-1987} of Theorem \ref{Th1} for $m=2$ was based on a generalization of the
following \textquotedblleft finiteness principle for Lipschitz selection" \cite{pavel-1986} for maps of metric spaces.

\begin{theorem}
\label{Th2}For large enough $k^{\#}=k^{\#}\left( D\right) $ and $%
C^{\#}=C^{\#}\left( D\right) $, the following holds. \newline
Let $X$ be a metric space. For each $x\in X$, let $K\left( x\right) \subset 
\mathbb{R}^{D}$ be an affine subspace in $\mathbb{R}^D$ of dimension $\leq d$. Suppose that for each $S\subset X$
with $\#\left( S\right) \leq k^{\#}$ there exists a map $F^{S}:S\rightarrow 
\mathbb{R}^{D}$ with Lipschitz constant $\leq 1$, such that $F^{S}\left(
x\right) \in K\left( x\right) $ for all $x\in S$.\newline
Then there exists a map $F:X\rightarrow \mathbb{R}^{D}$ with Lipschitz
constant $\leq C^{\#}$, such that $F\left( x\right) \in K\left( x\right) $
for all $x\in X$.
\end{theorem}

In fact, P. Shvartsman in \cite{pavel-1986} showed that one can take $k^\# = 2^{d+1}$ in Theorem \ref{Th2} and that the constant $k^\# = 2^{d+1}$ is sharp, see \cite{pavel-1992}.

P. Shvartsman also showed that Theorem \ref{Th2} remains valid when $\mathbb{R}^D$ is replaced by a Hilbert space (see \cite{s-2001}) or a Banach space (see \cite{pavel-2004}). 

It is conjectured in \cite{bs-1994} that Theorem \ref{Th2} should hold for any compact convex subsets $K(x) \subset \mathbb{R}^D$. In \cite{s-2002}, P. Shvartsman provided evidence for this conjecture: He showed that the conjecture holds in the case when $D = 2$ and in the case when $X$ is a finite metric space and the constant $C^\#$ is allowed to depend on the cardinality of $X$.

In this paper we prove finiteness principles for $C^{m}\left( \mathbb{R}^{n},%
\mathbb{R}^{D}\right) $-selection, and for $C^{m-1,1}\left( \mathbb{R}^{n},%
\mathbb{R}^{D}\right) $-selection, in particular providing a proof for the conjecture in \cite{bs-1994} for the case $X = \mathbb{R}^n$. 

\begin{theorem}
\label{Th3}For large enough $k^{\#}=k\left( m,n,D\right) $ and $%
C^{\#}=C\left( m,n,D\right) $, the following hold.

\begin{description}
\item[(A) $C^{m}$ FLAVOR] Let $E\subset \mathbb{R}^{n}$ be finite. For each $%
x\in E$, let $K\left( x\right) \subset \mathbb{R}^{D}$ be convex. Suppose
that for each $S\subset E$ with $\#\left( S\right) \leq k^{\#}$, there
exists $F^{S}\in C^{m}\left( \mathbb{R}^{n},\mathbb{R}^{D}\right) $ with
norm $\left\Vert F^{S}\right\Vert _{C^{m}\left( \mathbb{R}^{n},\mathbb{R}%
^{D}\right) }\leq 1$, such that $F^{S}\left( x\right) \in K\left( x\right) $
for all $x\in S$. \newline
Then there exists $F\in C^{m}\left( \mathbb{R}^{n},\mathbb{R}^{D}\right) $
with norm $\left\Vert F\right\Vert _{C^{m}\left( \mathbb{R}^{n},\mathbb{R}%
^{D}\right) }\leq C^{\#}$, such that $F\left( x\right) \in K\left( x\right) $
for all $x\in E$.

\item[(B) $C^{m-1,1}$ FLAVOR] For each $x\in \mathbb{R}^n$, let $K\left(
x\right) \subset \mathbb{R}^{D}$ be a closed convex set. Suppose that for
each $S\subset \mathbb{R}^{n}$ with $\#\left( S\right) \leq k^{\#}$, there
exists $F^{S}\in C^{m-1,1}\left( \mathbb{R}^{n},\mathbb{R}^{D}\right) $ with
norm $\left\Vert F^{S}\right\Vert _{C^{m-1,1}\left( \mathbb{R}^{n},\mathbb{R}%
^{D}\right) }\leq 1$, such that $F^{S}\left( x\right) \in K\left( x\right) $
for all $x\in S$.\newline
Then there exists $F\in C^{m-1,1}\left( \mathbb{R}^{n},\mathbb{R}^{D}\right) 
$ with norm $\left\Vert F\right\Vert _{C^{m-1,1}\left( \mathbb{R}^{n},%
\mathbb{R}^{D}\right) }\leq C^{\#}$, such that $F\left( x\right) \in K\left(
x\right) $ for all $x\in \mathbb{R}^{n}$.
\end{description}
\end{theorem}

We also prove a closely related result on interpolation by nonnegative functions. 

\begin{theorem}
\label{Th4}For large enough $k^{\#}=k\left( m,n\right) $ and $C^{\#}=C\left(
m,n\right) $ the following hold.

\begin{description}
\item[(A) $C^{m}$ FLAVOR] Let $f:E\rightarrow [0,\infty)$ with $%
E\subset \mathbb{R}^{n}$ finite. Suppose that for each $S\subset E$ with $%
\#\left( S\right) \leq k^{\#}$, there exists $F^{S}\in C^{m}\left( \mathbb{R}%
^{n}\right) $ with norm $\left\Vert F^{S}\right\Vert _{C^{m}\left( \mathbb{R}%
^{n}\right) }\leq 1$, such that $F^{S}=f\ $on $S$ and $F^{S}\geq 0$ on $%
\mathbb{R}^{n}$. \newline
Then there exists $F\in C^{m}\left( \mathbb{R}^{n}\right) $ with norm $%
\left\Vert F\right\Vert _{C^{m}\left( \mathbb{R}^{n}\right) }\leq C^{\#}$,
such that $F=f$ on $E$ and $F\geq 0$ on $\mathbb{R}^{n}$.

\item[(B) $C^{m-1,1}$ FLAVOR] Let $f:E\rightarrow [0,\infty)$ with $%
E\subset \mathbb{R}^{n}$ arbitrary. Suppose that for each $S\subset E$ with $%
\#\left( S\right) \leq k^{\#}$, there exists $F^{S}\in C^{m-1,1}\left( 
\mathbb{R}^{n}\right) $ with norm $\left\Vert F^{S}\right\Vert
_{C^{m-1,1}\left( \mathbb{R}^{n}\right) }\leq 1$, such that $F^{S}=f\ $on $S$
and $F^{S}\geq 0$ on $\mathbb{R}^{n}$. \newline
Then there exists $F\in C^{m-1,1}\left( \mathbb{R}^{n}\right) $ with norm $%
\left\Vert F\right\Vert _{C^{m-1,1}\left( \mathbb{R}^{n}\right) }\leq C^{\#}$%
, such that $F=f$ on $E$ and $F\geq 0$ on $\mathbb{R}^{n}$.
\end{description}
\end{theorem}

One might wonder how to decide whether the relevant $F^{S}$ exist in the
above results. This issue is addressed in \cite{f-2005} for Theorem \ref{Th1}; and in Sections \ref{ccs}, \ref{afc} below for Theorems \ref{Th3} and \ref%
{Th4}.

A weaker version of the case $D=1$ of Theorem \ref{Th3} appears in \cite%
{f-2005}. There, each $K\left( x\right) $ is an interval $\left[ f\left(
x\right) -\varepsilon \left( x\right) ,f\left( x\right) +\varepsilon \left(
x\right) \right] $. In place of the conclusion $F\left( x\right) \in K\left(
x\right) $ in Theorem \ref{Th3}, \cite{f-2005} obtains the weaker conclusion 
$F\left( x\right) \in \left[ f\left( x\right) -C\varepsilon \left( x\right)
,f\left( x\right) +C\varepsilon \left( x\right) \right] $ for a constant $C$
determined by $m$, $n$.

Our interest in Theorems \ref{Th3} and \ref{Th4} arises in part from their
possible connection to the interpolation algorithms of Fefferman-Klartag 
\cite{fb1,fb2}. Given a function $f:E\rightarrow \mathbb{R}$ with $E\subset 
\mathbb{R}^{n}$ finite, the goal of \cite{fb1,fb2} is to compute a function $%
F\in C^{m}\left( \mathbb{R}^{n}\right) $ such that $F=f$ on $E$, with norm $%
\left\Vert F\right\Vert _{C^{m}\left( \mathbb{R}^{n}\right) }$ as small as
possible up to a factor $C\left( m,n\right) $. Roughly speaking, the
algorithm in \cite{fb1,fb2} computes such an $F$ using $O\left( N\log
N\right) $ computer operations, where $N=\#\left( E\right) $. The algorithm
is based on ideas from the proof \cite{f-2005} of Theorem \ref{Th1}.
Accordingly, Theorems \ref{Th3} and \ref{Th4} raise the hope that we can
start to understand constrained interpolation problems, in which e.g. the
interpolant $F$ is constrained to be nonnegative everywhere on $\mathbb{R}%
^{n}$.

Theorems \ref{Th3} and \ref{Th4} follow from a more general result --
Theorem \ref{Th6} below -- which is the real content of this paper. To
arrive at the statement of Theorem \ref{Th6}, we first recall some ideas
from the proof of Theorem \ref{Th1} in \cite{f-2005}. Next, we formulate the
main definitions, namely that of a \textquotedblleft shape field" and that of
\textquotedblleft $\left( C_{w},\delta _{\max }\right) $-convexity".
Finally, we state Theorem \ref{Th6}, say a few words about its proof, and
explain its relationship with Theorems \ref{Th3} and \ref{Th4}.

We will need a bit more notation. For $F\in C_{loc}^{m-1}\left( \mathbb{R}%
^{n}\right) $ and $x\in \mathbb{R}^{n}$, we write $J_{x}\left( F\right) $
(the \textquotedblleft jet" of $F$ at $x$) to denote the $\left( m-1\right)
^{rst}$ degree Taylor polynomial of $F$ at $x$, i.e., 
\begin{equation*}
\left[ J_{x}\left( F\right) \right] \left( y\right) =\sum_{\left\vert \alpha
\right\vert \leq m-1}\frac{1}{\alpha !}\left( \partial ^{\alpha }F\left(
x\right) \right) \cdot \left( y-x\right) ^{\alpha }\text{ for }y\in \mathbb{R%
}^{n}\text{.}
\end{equation*}%
Thus, $J_{x}\left( F\right) $ belongs to $\mathcal{P} $, the vector space of
all real-valued polynomials of degree at most $\left( m-1\right) $ on $%
\mathbb{R}^{n}$. Given $x\in \mathbb{R}^{n}$, there is a natural
multiplication $\odot _{x}$ on $\mathcal{P} $ (\textquotedblleft
multiplication of jets at $x$"), uniquely specified by demanding that 
\begin{equation*}
J_{x}\left( FG\right) =J_{x}\left( F\right) \odot _{x}J_{x}\left( G\right)
\end{equation*}%
for $F,G\in C_{loc}^{m-1}\left( \mathbb{R}^{n}\right) $. More explicitly, 
\begin{equation*}
P\odot _{x}Q=J_{x}\left( PQ\right)
\end{equation*}%
for $P,Q\in \mathcal{P} $.

We now present the ideas from the proof of Theorem \ref{Th1} that will form
the background for the statement of Theorem \ref{Th6}. We pose a key
question:

{\em Let $f:E\rightarrow \mathbb{R}$ with $E\subset \mathbb{R}^{n}$ finite, and
let $M>0$. Suppose that $F=f$ on $E$ and $||F||_{C^m(\mathbb{R}^n)} \leq M$.
What can we say about $J_{x}\left( F\right) $ for a given $x\in \mathbb{R}%
^{n}$?}

Although the question makes sense for any $x\in \mathbb{R}^{n}$, we restrict
ourselves here to $x\in E$. We define the (possibly empty) convex set 
\begin{eqnarray}
\Gamma \left( x,M\right) &=&\left\{ J_{x}\left( F\right) :F\in C^{m}\left( 
\mathbb{R}^{n}\right) ,\left\Vert F\right\Vert _{C^{m}\left( \mathbb{R}%
^{n}\right) }\leq M,F=f\text{ on }E\right\}  \label{Eq1} \\
&\subset &\mathcal{P}  \notag
\end{eqnarray}
and undertake to compute its approximate size and shape.

The set $\Gamma \left( x,M\right) $ carries a lot of information. Already
the mere assertion that $\Gamma \left( x_{0},M_{0}\right) $ is non-empty for a
particular $x_{0}$, $M_{0}$ tells us that there exists an interpolant with $%
C^{m}$-norm $\leq M_{0}$. Once we compute the approximate size and shape of
the $\Gamma \left( x,M\right) $, Theorem \ref{Th1} will follow easily.

To understand $\Gamma \left( x,M\right) $, we first define the (possibly
empty) convex set 
\begin{equation}
\Gamma _{0}\left( x,M\right) =\left\{ P\in \mathcal{P} :\left\vert \partial
^{\beta }P\left( x\right) \right\vert \leq M\text{ for }\left\vert \beta
\right\vert \leq m-1,P\left( x\right) =f\left( x\right) \right\} \text{,}
\label{Eq2}
\end{equation}%
for $x\in E$, $M>0$.

Trivially, $\Gamma _{0}\left( x,M\right) \supset \Gamma \left( x,M\right) $
for all $x\in E$, $M>0$; most likely, $\Gamma _{0}\left( x,M\right) $ is
much bigger than $\Gamma \left( x,M\right) $. To remedy this, we pass from $%
\Gamma _{0}\left( x,M\right) $ to smaller sets $\Gamma _{l}\left( x,M\right) 
$, by induction on $l$. Each $\Gamma _{l}\left( x,M\right) $ $\left( x\in
E,M>0,l\geq 0\right) $ will be a (possibly empty) convex subset of $\mathcal{%
P} $. Our $\Gamma _{l}\left( x,M\right) $ will decrease with $l$, but we
will still have 
\begin{equation}
\Gamma _{l}\left( x,M\right) \supset \Gamma \left( x,M\right) \text{ for all 
}x\in E\text{, }M>0.  \label{Eq3}
\end{equation}

We pass from $\Gamma _{l}$ to $\Gamma _{l+1}$ as follows.

Let $x\in E$, $M>0$, $l\geq 0$. Then $\Gamma _{l+1}\left( x,M\right) $
consists of all $P\in \Gamma _{l}\left( x,M\right) $ such that for each $%
y\in E$ there exists $P^{\prime }\in \Gamma _{l}\left( y,M\right) $ such
that 
\begin{equation}
\left\vert \partial ^{\beta }\left( P-P^{\prime }\right) \left( x\right)
\right\vert \leq C\left( m,n\right) M\left\vert x-y\right\vert
^{m-\left\vert \beta \right\vert }\text{ for }\left\vert \beta \right\vert
\leq m-1.  \label{Eq4}
\end{equation}%
(Compare with the \textquotedblleft pedagogical algorithm" in \cite{fb1}.)

Evidently, $\Gamma _{l+1}\left( x,M\right) \subset \Gamma _{l}\left(
x,M\right)$; the $\Gamma _{l}\left( x,M\right) $ decrease with $l$, as
promised.

Moreover, if (\ref{Eq3}) holds for a given $l$, then 
\begin{equation}
\Gamma _{l+1}\left( x,M\right) \supset \Gamma \left( x,M\right) \text{ for
all }x\in E,M>0.  \label{Eq5}
\end{equation}

Indeed, suppose $P\in \Gamma \left( x,M\right) $. By definition, there
exists $F\in C^{m}\left( \mathbb{R}^{n}\right) $ with norm $\left\Vert
F\right\Vert _{C^{m}\left( \mathbb{R}^{n}\right) }\leq M$, such that $F=f$
on $E$ and $J_{x}\left( F\right) =P$.

Let $y\in E$, and set $P^{\prime }=J_{y}\left( F\right) $. By definition, $%
P^{\prime }\in \Gamma \left( y,M\right) $; hence, $P^{\prime }\in \Gamma
_{l}\left( y,M\right) $; by (\ref{Eq3}). Moreover, (\ref{Eq4}) holds, thanks
to Taylor's theorem. Thus, for every $y\in E$ we have found a $P^{\prime
}\in \Gamma_l\left( y,M\right) $ that satisfies (\ref{Eq4}). This means that $%
P\in \Gamma _{l+1}\left( x,M\right) $, completing the proof of (\ref{Eq5}).

We now know by induction that (\ref{Eq3}) holds for all $l\geq 0$.

The following result computes the approximate size and shape of the $\Gamma
\left( x,M\right) $.

\begin{theorem}
\label{Th5}For a large enough $l_{\ast }=l_{\ast }\left( m,n\right) $ and $%
C_{\ast }=C_{\ast }\left( m,n\right) $, we have 
\begin{equation*}
\Gamma \left( x,M\right) \subset \Gamma _{l_{\ast }}\left( x,M\right)
\subset \Gamma \left( x,C_{\ast }M\right) .
\end{equation*}
\end{theorem}

Once we know Theorem \ref{Th5}, we bring in Helly's theorem \cite%
{rock-convex} to complete the proof of Theorem \ref{Th1} (in the $C^{m}$
FLAVOR). We recall

\begin{description}
\item[Helly's Theorem] Let $K_{1},\cdots ,K_{N}$ be convex subsets of $%
\mathbb{R}^{D}$. If any $\left( D+1\right) $ of the sets $K_{1},\cdots
,K_{N} $ have a point in common, then $K_{1},\cdots ,K_{N}$ have a point in
common.
\end{description}

This elementary result is clearly relevant to Theorem \ref{Th1}.

To exploit Helly's theorem, we define 
\begin{equation*}
\tilde{\Gamma}\left( x,M,S\right) =\left\{ J_{x}\left( F\right) :F\in
C^{m}\left( \mathbb{R}^{n}\right) ,\left\Vert F\right\Vert _{C^{m}\left( 
\mathbb{R}^{n}\right) }\leq M,F=f\text{ on }S\right\} \text{, for }S\subset
E,
\end{equation*}%
and then set 
\begin{equation*}
\tilde{\Gamma}_{l}\left( x,M\right) =\bigcap_{S\subset E,\#\left( S\right)
\leq \left( 2+\dim \mathcal{P} \right) ^{l}}\tilde{\Gamma}\left(
x,M,S\right) .
\end{equation*}

Suppose that $f:E\rightarrow \mathbb{R}$ satisfies

\begin{description}
\item[$\left( \star \right) $] For each $S\subset E$ with $\#\left( S\right)
\leq k^{\#}$ there exists $F^{S}\in C^{m}\left( \mathbb{R}^{n}\right) $ with
norm $\left\Vert F^{S}\right\Vert _{C^{m}\left( \mathbb{R}^{n}\right) }\leq
M $ such that $F^{S}=f$ on $S$.
\end{description}

If $k^{\#}\geq \left( 2+\dim \mathcal{P} \right) ^{l+1}$, then an easy
argument using Helly's theorem shows that $\tilde{\Gamma}_{l}\left(
x,M\right) $ is nonempty. Also, an easy induction on $l$, again using
Helly's theorem, shows that $\tilde{\Gamma}_{l}\left( x,M\right) \subset
\Gamma _{l}\left( x,M\right) $ for all $l\geq 0$.

Now let $l_{\ast }$ be as in Theorem \ref{Th5}, and suppose $\left( \star
\right) $ holds with $k^{\#}\geq \left( 2+\dim \mathcal{P} \right) ^{l_{\ast
}+1}$. Then $\tilde{\Gamma}_{l_{\ast }}\left( x,M\right) $ is nonempty,
hence $\Gamma _{l_{\ast }}\left( x,M\right) $ is nonempty. Theorem \ref{Th5}
now tells us that $\Gamma \left( x,C_{\ast }M\right) $ is non-empty, which
means that

\begin{description}
\item[$\left( \star \star \right) $] There exists $F\in C^{m}\left( \mathbb{R%
}^{n}\right) $ with norm $\left\Vert F\right\Vert _{C^{m}\left( \mathbb{R}%
^{n}\right) }\leq C_{\ast }M$ such that $F=f$ on $E$.
\end{description}

Thus, $\left( \star \right) $ implies $\left( \star \star \right) $, which
is the content of the $C^{m}$ FLAVOR\ of Theorem \ref{Th1}.

The $C^{m-1,1}$ FLAVOR of Theorem \ref{Th1} follows easily. This completes
our discussion of Theorem \ref{Th1}.

Motivated by the above discussion, we define a key notion, preparing for the
statement of Theorem \ref{Th6}. Let $E\subset \mathbb{R}^{n}$ be finite. A 
\underline{shape field} on $E$ is a family $\vec{\Gamma}=\left( \Gamma
\left( x,M\right) \right) _{x\in E,M>0}$ of (possibly empty) convex sets $%
\Gamma \left( x,M\right) \subset \mathcal{P} $, indexed by $x\in E$ and $M>0$%
, with the property that $M^{\prime }<M$ implies $\Gamma \left( x,M^{\prime
}\right) \subset \Gamma \left( x,M\right) $ for any $x$, $M$, $M^{\prime }$.

Note that finiteness of $E$ is part of the definition of a shape field.

Our goal is to give an analogue of Theorem \ref{Th5} starting from a shape
field $\vec{\Gamma}_{0}=\left( \Gamma _{0}\left( x,M\right) \right) _{x\in
E,M>0}$ much more general than that given by (\ref{Eq2}).

Instead of seeking functions $F\in C^{m}\left( \mathbb{R}^{n}\right) $ such
that $F=f$ on $E$, we will now be seeking functions $F\in C^{m}\left( 
\mathbb{R}^{n}\right) $ such that $J_{x}\left( F\right) \in \Gamma
_{0}\left( x,CM\right) $ for all $x\in E$. In the context of Theorem \ref%
{Th5} with $\vec{\Gamma}_{0}$ given by (\ref{Eq2}), this is consistent with
our previous discussion.

We can pass from $\vec{\Gamma}_{0}$ to its ``$l^{th}$-refinement" $\vec{\Gamma}%
_{l}=\left( \Gamma _{l}\left( x,M\right) \right) _{x\in E,M>0}$ for $%
l=1,2,3,\cdots $ by the same rule as before:

$\Gamma _{l+1}\left( x,M\right) $ consists of all $P\in \Gamma _{l}\left(
x,M\right) $ such that for each $y\in E$ there exists $P^{\prime }\in \Gamma
_{l}\left( y,M\right) $ such that $\left\vert \partial ^{\beta }\left(
P-P^{\prime }\right) \left( x\right) \right\vert \leq M\left\vert
x-y\right\vert ^{m-\left\vert \beta \right\vert }$ for $\left\vert \beta
\right\vert \leq m-1$.

We will need a hypothesis on $\vec{\Gamma}_{0}$ called $\left( C_{w},\delta
_{\max }\right) $-convexity. This notion is a bit technical.\footnote{%
The reader may prefer to skip over the definition of $\left( C_{w},\delta
_{\max }\right) $-convexity on first reading.} To motivate it, we anticipate
a situation that arises in the proof of Theorem \ref{Th6}.

We will be trying to produce a function $F\in C^{m}\left( \mathbb{R}%
^{n}\right) $ (with control on its derivatives up to $m^{th}$ order), such
that

\begin{description}
\item[(P1)] $J_{x}\left( F\right) \in \Gamma _{0}\left( x,CM\right) $ for
all $x\in E\cap Q_{\max }$, where $Q_{\max }$ is a cube of sidelength $%
\delta _{\max }$.
\end{description}

The cube $Q_{\max }$ will be partitioned into Calder\'on-Zygmund cubes $%
Q_{\nu }$ with sidelengths $\delta _{\nu }\leq \delta _{\max }$. We will
introduce a partition of unity

\begin{description}
\item[(P2)] $1=\sum_{\nu }\theta _{\nu }^{2}$ on $Q_{\max }$,
\end{description}

adapted to the Calder\'on-Zygmund cubes $Q_{\nu }$. More precisely, for each 
$\nu $,

\begin{description}
\item[(P3)] $\theta _{\nu }$ is supported in the double of $Q_{\nu }$, with
estimates $\left\vert \partial ^{\beta }\theta _{\nu }\right\vert \leq
C\delta _{\nu }^{-\left\vert \beta \right\vert }$ for $\left\vert \beta
\right\vert \leq m$.
\end{description}

For each $Q_{\nu }$, we will produce a local solution $F_{\nu }$ such that

\begin{description}
\item[(P4)] $J_{x}\left( F_{\nu }\right) \in \Gamma _{0}\left( x,M\right) $
for all $x\in E\cap Q_{\nu }$, with control on the derivatives of $F_{\nu }$
up to order $m$.
\end{description}

Moreover, the $F_{\nu }$ agree with one another, in the sense that

\begin{description}
\item[(P5)] $\left\vert \partial ^{\beta }\left( F_{\nu }-F_{\nu ^{\prime
}}\right) \right\vert \leq CM\delta _{\nu }^{m-\left\vert \beta \right\vert
} $ on $Q_{\nu }\cap Q_{\nu ^{\prime }}$, for $\left\vert \beta \right\vert
\leq m$.
\end{description}

Conditions (P2), (P3), (P5) are familiar from the classical proof \cite%
{whitney-1934} of the Whitney extension theorem. As in that proof, we patch
together our local solutions $F_{\nu }$ by setting

\begin{description}
\item[(P6)] $F=\sum_{\nu }\theta _{\nu }^{2}F_{\nu }$.
\end{description}

We hope that conditions (P2)$\cdots $(P5) will guarantee that the function 
$F$ given by (P6) will satisfy (P1).

The above discussion motivates the following definitions, where $P_{\nu }$, $%
Q_{\nu }$ play the r\^{o}les of $J_{x}\left( F_{\nu }\right) $, $%
J_{x}\left( \theta _{\nu }\right) $ respectively, for a given $x\in E\cap
Q_{\max }$.

Let $\vec{\Gamma}_{0}=\left( \Gamma _{0}\left( x,M\right) \right) _{x\in
E,M>0}$ be a shape field. Let $C_{w}$, $\delta _{\max }>0$ be given.

\begin{definition}
We say that $\vec{\Gamma}_{0}$ is $\left( C_{w},\delta _{\max }\right) $%
-convex if the following holds: Suppose $0<\delta \leq \delta _{\max }$; $%
x\in E$; $M>0$; $P_{1},P_{2}\in \Gamma _{0}\left( x,M\right) $; $%
Q_{1},Q_{2}\in \mathcal{P} $. Assume the estimates 
\begin{eqnarray*}
\left\vert \partial ^{\beta }\left( P_{1}-P_{2}\right) \left( x\right)
\right\vert &\leq &M\delta ^{m-\left\vert \beta \right\vert }, \\
\left\vert \partial ^{\beta }Q_{1}\left( x\right) \right\vert &\leq &\delta
^{-\left\vert \beta \right\vert }, \\
\left\vert \partial ^{\beta }Q_{2}\left( x\right) \right\vert &\leq &\delta
^{-\left\vert \beta \right\vert }\text{, for }\left\vert \beta \right\vert
\leq m-1.
\end{eqnarray*}%
Assume also that $Q_{1}\odot _{x}Q_{1}+Q_{2}\odot _{x}Q_{2}=1$. 

Then $%
Q_{1}\odot _{x}Q_{1}\odot _{x}P_{1}+Q_{2}\odot _{x}Q_{2}\odot _{x}P_{2}\in
\Gamma _{0}\left( x,C_{w}M\right) $.
\end{definition}

This notion is a close relative of \textquotedblleft Whitney convexity" in 
\cite{f-2005-a}, and our Theorem \ref{Th6} below generalizes the main
results of \cite{f-2005-a}.

\begin{theorem}
\label{Th6}For a large enough $l_{\ast }=l_{\ast }\left( m,n\right) $, the
following holds. Let $\Gamma _{0}=\left( \Gamma _{0}\left( x,M\right)
\right) _{x\in E,M>0}$ be a $\left( C_{w},\delta _{\max }\right) $-convex
shape field, and for $l\geq 1$, let $\Gamma _{l}=\left( \Gamma _{l}\left(
x,M\right) \right) _{x\in E,M>0}$ be its $l^{th}$ refinement. Suppose we are
given a cube $Q_{\max }$ of sidelength $\delta _{\max }$, a point $x_{0}\in
E\cap Q_{\max }$, a number $M_{0}>0$, and a polynomial $P_{0}\in \Gamma
_{l_{\ast }}\left( x_{0},M_{0}\right) $. Then there exists $F\in C^{m}\left( 
\mathbb{R}^{n}\right) $ such that

\begin{itemize}
\item $J_{x}(F)\in \Gamma _{0}(x,C_{\ast }M_{0})$ for all $x\in Q_{\max
}\cap E$, and

\item $|\partial ^{\beta }(F-P_{0})(x)|\leq C_{\ast }M_{0}\delta _{\max
}^{m-\left\vert \beta \right\vert }$ for all $x\in Q_{\max }$, $|\beta| \leq
m$.
\end{itemize}

Here, $C_{\ast }$ depends only on $m$, $n$, $C_{w}$.
\end{theorem}

Just as Theorem \ref{Th5} implies Theorem \ref{Th1}, our Theorem \ref{Th6}
implies the following finiteness principle for shape fields.

\begin{theorem}
\label{Th7}For a large enough $k^{\#}=k\left( m,n\right) $ the following
holds. Let $\vec{\Gamma}_{0}=\left( \Gamma _{0}\left( x,M\right) \right)
_{x\in E,M>0}$ be a $\left( C_{w},1\right) $-convex shape field. Let $%
M_{0}>0 $. Suppose that for each $S\subset E$ with $\#\left( S\right) \leq
k^{\#}$, there exists $F^{S}\in C^{m}\left( \mathbb{R}^{n}\right) $ with
norm $\left\Vert F^{S}\right\Vert _{C^{m}\left( \mathbb{R}^{n}\right) }\leq
M_{0}$, such that $J_{x}\left( F^{S}\right) \in \Gamma _{0}\left(
x,M_{0}\right) $ for all $x\in S$. Then there exists $F\in C^{m}\left( 
\mathbb{R}^{n}\right) $ with norm $\left\Vert F\right\Vert _{C^{m}\left( 
\mathbb{R}^{n}\right) }\leq C_{\ast }M_{0}$, such that $J_{x}\left( F\right)
\in \Gamma _{0}\left( x,C_{\ast }M_{0}\right) $ for all $x\in E$. Here, $%
C_{\ast }$ depends only on $m$, $n$, $C_{w}$.
\end{theorem}

Once we know Theorem \ref{Th7}, it is relatively straightforward to deduce
Theorems \ref{Th3} and \ref{Th4}. To deduce Theorem \ref{Th3} (in the $C^{m}$
FLAVOR), we introduce a new variable $\xi =\left( \xi _{1},\cdots ,\xi
_{D}\right) \in \mathbb{R}^{D}$. Given $\left( K\left( x\right) \right)
_{x\in E}$ as in Theorem \ref{Th3} (A), we set $E^{+}=E\times \left\{
0\right\} \subset \mathbb{R}^{n}\times \mathbb{R}^{D}$. To find functions $F:%
\mathbb{R}^{n}\rightarrow \mathbb{R}^{D}$ such that $F\left( x\right) \in
K\left( x\right) $ for $x\in E$, we look for functions $\mathcal{F}\left(
x,\xi \right) $ mapping $\mathbb{R}^{n+D}$ to $\mathbb{R}$ such that $%
\mathcal{F}\left( x,0\right) =0$ and $\nabla _{\xi }\mathcal{F}\left(
x,0\right) \in K\left( x\right) $ for all $\left( x,0\right) \in E^{+} $;
compare with Fefferman-Luli \cite{fl-2014}. Theorem \ref{Th7} then allows us
to find such an $\mathcal{F}\in C^{m+1}\left( \mathbb{R}^{n+D}\right) $.
Setting $F\left( x\right) =\nabla _{\xi }\mathcal{F}\left( x,0\right) $, we
obtain Theorem \ref{Th3} in the $C^{m}$ FLAVOR. The $C^{m-1,1}$ FLAVOR then
follows easily.

Theorem \ref{Th4} (B) is an immediate consequence of Theorem \ref{Th3} (B);
in the latter theorem, we take $D=1$ and set $K\left( x\right) =\left\{
f\left( x\right) \right\} $ for $x\in E$, $K\left( x\right) =[0,\infty )$
for $x\in \mathbb{R}^{n}\setminus E$.

We omit from this introduction the derivation of Theorem \ref{Th4} (A) from
Theorem \ref{Th6}.

We give a brief sketch of the proof of Theorem \ref{Th6}, sacrificing
accuracy for simplicity; see the later sections of this paper for correct
details.

We adapt from \cite{f-2007} the key idea of an \textquotedblleft $\mathcal{A}
$-basis". To formulate this notion, we introduce notation and definitions.

We write $\mathcal{M}$ to denote the set of all multiindices $\alpha =\left(
\alpha _{1},\cdots ,\alpha _{n}\right) $ of order $\left\vert \alpha
\right\vert =\alpha _{1}+\cdots +\alpha _{n}\leq m-1$. We denote subsets of $%
\mathcal{M}$ by $\mathcal{A}$, $\mathcal{B}$, $\mathcal{A}^{\prime }$, etc.
A subset $\mathcal{A}\subset \mathcal{M}$ will be called ``monotonic" if,
for any $\alpha \in \mathcal{A}$ and $\gamma \in \mathcal{M}$, $\alpha
+\gamma \in \mathcal{M}$ implies $\alpha +\gamma \in \mathcal{A}$. We use
the total order relation $<$ on subsets of $\mathcal{M}$ defined in \cite%
{f-2005}. With respect to this order relation $\mathcal{M}$ itself is the
minimal subset of $\mathcal{M}$, and the empty set $\emptyset $ is the
maximal subset.

Now suppose $\vec{\Gamma}=\left( \Gamma \left( x,M\right) \right) _{x\in
E,M>0}$ is a shape field. Let $x_{0}\in E$, $M_{0}>0$, $P_{0}\in \Gamma
\left( x_{0},M_{0}\right) $ be given. Let $\mathcal{A}\subset \mathcal{M}$
be monotonic, and let $C_{B},\delta >0$ be real numbers. Then we say that 
\underline{$\vec{\Gamma}$ has an $\left( \mathcal{A},\delta ,C_{B}\right) $%
-basis at $\left( x_{0},M_{0},P_{0}\right) $} if there exist polynomials $%
P_{\alpha }\in \mathcal{P} $ for $\alpha \in \mathcal{A}$ with the following
properties:

\begin{itemize}
\item $\partial ^{\beta }P_{\alpha }\left( x_{0}\right) =\delta _{\beta
\alpha }$ (Kronecker delta) for all $\beta ,\alpha \in \mathcal{A}$.

\item $\left\vert \partial ^{\beta }P_{\alpha }\left( x_{0}\right)
\right\vert \leq C_{B}\delta ^{\left\vert \alpha \right\vert -\left\vert
\beta \right\vert }$ for all $\alpha \in \mathcal{A}$, $\beta \in \mathcal{M}
$.

\item $P_{0}+\tau M_{0}\delta ^{m-\left\vert \alpha \right\vert }P_{\alpha
}\in \Gamma \left( x_{0},C_{B}M_{0}\right) $ for all $\alpha \in \mathcal{A}$%
, $\left\vert \tau \right\vert \leq \frac{1}{C_{B}}$.
\end{itemize}

Note that for $\mathcal{A}=\emptyset $, $\vec{\Gamma}$ always has an $\left( 
\mathcal{A},\delta ,C_{B}\right) $-basis at $\left( x_{0},M_{0},P_{0}\right) 
$, since the desired list of polynomials $\left( P_{\alpha }\right) _{\alpha
\in \mathcal{A}}$ is then empty.

By induction on the monotonic set $\mathcal{A}$ with respect to the order $<$%
, we will prove the following result. We write $\delta _{Q}$ to denote the
sidelength of a cube $Q\subset \mathbb{R}^{n}$.

\begin{lemma}[Main Lemma for $\mathcal{A}$]
Let $\vec{\Gamma}_{0}=\left( \Gamma _{0}\left( x,M\right) \right) _{x\in
E,M>0}$ be a $\left( C_{w},\delta _{\max }\right) $-convex shape field, with 
$l$-th refinement $\vec{\Gamma}_{l}=\left( \Gamma _{l}\left( x,M\right)
\right) _{x\in E,M>0}$ (each $l\geq 0$). Let $Q_{0}\subset \mathbb{R}^{n}$
be a cube with sidelength $\delta _{Q_{0}}\leq \delta _{\max }$, and let $%
x_{0}\in E\cap Q_{0}$, $M_{0}>0$, $P_{0}\in \Gamma _{l\left( \mathcal{A}%
\right) }\left( x_{0},M_{0}\right) $ for a suitable $l\left( \mathcal{A}%
\right) $ determined by $\mathcal{A}$, $m$, $n$. Assume $\vec{\Gamma}%
_{l\left( \mathcal{A}\right) }$ has an $\left( \mathcal{A},\delta
_{Q_{0}},C_{B}\right) $-basis at $\left( x_{0},M_{0},P_{0}\right) $. Then
there exists $F\in C^{m}\left( \mathbb{R}^{n}\right) $ such that 
\begin{equation}
\left\vert \partial ^{\beta }\left( F-P_{0}\right) \left( x\right)
\right\vert \leq C_{\ast }M_{0}\delta _{Q_{0}}^{m-\left\vert \beta
\right\vert }\text{ for }x\in Q_{0},\left\vert \beta \right\vert \leq m;
\label{intro14}
\end{equation}%
and 
\begin{equation}
J_{x}\left( F\right) \in \Gamma _{0}\left( x,C_{\ast }M_{0}\right) \text{
for all }x\in E\cap Q_{0}\text{.}  \label{intro15}
\end{equation}%
Here, $C_{\ast }$ depends only on $C_{B}$, $C_{w}$, $m$, $n$.
\end{lemma}

Theorem \ref{Th6} follows at once from the Main Lemma for $\mathcal{A}%
=\emptyset $ (the empty set), since there is always a $\left( \emptyset
,\delta _{Q_{0}},C_{B}\right) $-basis at $\left( x_{0},M_{0},P_{0}\right) $,
where (say) $C_{B}=1$. Thus, our task is to establish the above Main Lemma
by induction on $\mathcal{A}$.

The base case of our induction is the Main Lemma for $\mathcal{A}=\mathcal{M}
$. In this case, one can just set $F=P_{0}$. Then (\ref{intro14}) holds
trivially, and we can show that (\ref{intro15}) holds as well.

For the induction step, we fix a monotonic $\mathcal{A}\subset \mathcal{M}$ $%
\left( \mathcal{A\not=M}\right) $, and suppose that the Main Lemma for $%
\mathcal{A}^{\prime }$ is valid for all $\mathcal{A}^{\prime }<\mathcal{A}$.
We will then prove the Main Lemma for $\mathcal{A}$.

Suppose $\vec{\Gamma}_{l}$, $x_{0}$, $M_{0}$, $P_{0}$, $Q_{0}$ etc. are as
in the hypotheses of the Main Lemma for $\mathcal{A}$. We must find a
function $F\in C^{m}\left( \mathbb{R}^{n}\right) $ satisfying (\ref{intro14}%
) and (\ref{intro15}).

We will make a Calder\'on-Zygmund decomposition of $Q_{0}$ into subcubes $%
\left\{ Q_{\nu }\right\} $. Some of the $Q_{\nu }$ may have empty
intersection with $E$, but in this expository discussion we pretend that we
can pick $x_{\nu }\in E\cap Q_{\nu }$ for each $\nu $. We write $\frac{65}{64%
}Q_{\nu }$ to denote the cube obtained by dilating $Q_{\nu }$ about its
center by a factor $\frac{65}{64}$. For each $Q_{\nu }$, we will carefully
pick $\mathcal{A}_{\nu }^{\prime }<\mathcal{A}$ monotonic, and $P_{\nu }\in
\Gamma _{l\left( \mathcal{A}_{\nu }^{\prime }\right) }\left( x_{\nu
},CM_{0}\right) $, such that 
\begin{equation}
\vec{\Gamma}_{l\left( \mathcal{A}_{\nu }^{\prime }\right) }\text{ has an }%
\left( \mathcal{A}_{\nu }^{\prime },\delta _{\frac{65}{64}Q_{\nu
}},C^{\prime }\right) \text{-basis at }\left( x_{\nu },M_{0},P_{\nu }\right) 
\text{.}  \label{intro16}
\end{equation}%
Since $\mathcal{A}_{\nu }^{\prime }<\mathcal{A}$ and $\mathcal{A}_{\nu
}^{\prime }$ is monotonic, our induction hypothesis applies; thus, the Main
Lemma for $\mathcal{A}_{\nu }^{\prime }$ is valid. Therefore, by (\ref%
{intro16}), there exists $F_{\nu }\in C^{m}\left( \mathbb{R}^{n}\right) $
such that 
\begin{equation}
\left\vert \partial ^{\beta }\left( F_{\nu }-P_{\nu }\right) \left( x\right)
\right\vert \leq C_{\ast }M_{0}\delta _{Q_{\nu }}^{m-\left\vert \beta
\right\vert }\text{ for }x\in \frac{65}{64}Q_{\nu }\text{, }\left\vert \beta
\right\vert \leq m,  \label{intro17}
\end{equation}%
and 
\begin{equation}
J_{x}\left( F_{\nu }\right) \in \Gamma _{0}\left( x,CM\right) \text{ for }%
x\in E\cap \frac{65}{64}Q_{\nu }\text{.}  \label{intro18}
\end{equation}

We patch together the local solutions $F_{\nu }$ into an $F\in C^{m}\left( 
\mathbb{R}^{n}\right) $ as in (P1)$\cdots $(P5) above. By picking the
above $P_{\nu }$ with great care, we can arrange that 
\begin{equation*}
\left\vert \partial ^{\beta }\left( P_{\nu }-P_{\nu ^{\prime }}\right)
\left( x\right) \right\vert \leq CM_{0}\delta _{Q_{\nu }}^{m-\left\vert
\beta \right\vert }\text{ for }x\in \frac{65}{64}Q_{\nu }\cap \frac{65}{64}%
Q_{\nu ^{\prime }},\left\vert \beta \right\vert \leq m.
\end{equation*}%
Consequently, estimate (\ref{intro17}) implies the crucial consistency
condition (P5). Thus, our $\theta _{\nu }$ and $F_{\nu }$ satisfy (P2)$\cdots
$(P5). The $\left( C_{w},\delta _{\max }\right) $-convexity of $\vec{\Gamma}%
_{0}$ now implies that the function $F$ given by (P5) satisfies 
\begin{equation}
J_{x}\left( F\right) \in \Gamma _{0}\left( x,CM_{0}\right) \text{ for all }%
x\in E\cap Q_{0}.  \label{intro19}
\end{equation}%
Moreover, Whitney's classic argument \cite{whitney-1934} yields the estimate 
\begin{equation}
\left\vert \partial ^{\beta }\left( F-P_{0}\right) \left( x\right)
\right\vert \leq CM_{0}\delta _{Q_{0}}^{m-\left\vert \beta \right\vert }%
\text{ for }x\in Q_{0}\text{, }\left\vert \beta \right\vert \leq m\text{.}
\label{intro20}
\end{equation}

Our results (\ref{intro19}), (\ref{intro20}) are precisely the conclusions
of the Main Lemma for $\mathcal{A}$.

This completes our induction on $\mathcal{A}$, proving the Main Lemma for
all monotonic $\mathcal{A}$, and consequently establishing Theorems \ref{Th3}%
, \ref{Th4}, \ref{Th6}, \ref{Th7} as explained above.

Our Main Lemma and its proof are analogous to the main ideas in \cite%
{f-2005, f-2006, f-2007}. However, at one crucial point (Case 2 in the proof
of Lemma \ref{lemma-gn1} below), the analogy breaks down, and we are saved
from disaster by a lucky accident.

This concludes our summary of the proof of Theorem \ref{Th6}. We again warn the reader that it is oversimplified. In the following sections, we start from scratch and provide correct proofs.

This paper is part of a literature on extension, interpolation, and
selection of functions, going back to H. Whitney's seminal work \cite%
{whitney-1934}, and including fundamental contributions by G. Glaeser \cite%
{gl-1958}, Y, Brudnyi and P. Shvartsman \cite%
{bs-1985, bs-1994,bs-1997,bs-1998,bs-2001, pavel-1982,pavel-1984,pavel-1986, pavel-1987, pavel-1992, s-2001,s-2002,pavel-2004, s-2008}, J. Wells 
\cite{jw-1973}, E. Le Gruyer \cite{lgw-2009}, and E. Bierstone, P. Milman,
and W. Paw{\l }ucki \cite{bmp-2003,bmp-2006,bm-2007}, as well as our own
papers \cite{f-2005, f-2006, f-2007, f-2009-b, fb1, fb2, fl-2014}. See e.g. \cite{f-2009-b}
for the history of the problem, as well as Zobin \cite{zobin-1998,zobin-1999} for a related problem. 

We are grateful to the American Institute of Mathematics, the Banff
International Research Station, the Fields Institute, and the College of
William and Mary for hosting workshops on interpolation and extension. We
are grateful also to the Air Force Office of Scientific Research, the
National Science Foundation, and the Office of Naval Research for financial
support.

We are also grateful to Pavel Shvartsman and Alex Brudnyi for their comments on an earlier version of our manuscript, and to all the participants of the Eighth Whitney Problems Workshop for their interest in our work. 

\part{Shape Fields and Their Refinements}

\section{Notation and Preliminaries\label{notation-and-preliminaries}}

Fix $m$, $n\geq 1$. We will work with cubes in $\mathbb{R}^{n}$; all our
cubes have sides parallel to the coordinate axes. If $Q$ is a cube, then $%
\delta _{Q}$ denotes the sidelength of $Q$. For real numbers $A>0$, $AQ$
denotes the cube whose center is that of $Q$, and whose sidelength is $%
A\delta _{Q}$.

A \underline{dyadic} cube is a cube of the form $I_{1}\times I_{2}\times
\cdots \times I_{n}\subset \mathbb{R}^{n}$, where each $I_{\nu }$ has the
form $[2^{k}\cdot i_{\nu },2^{k}\cdot \left( i_{\nu }+1\right) )$ for
integers $i_{1},\cdots ,i_{n}$, $k$. Each dyadic cube $Q$ is contained in
one and only one dyadic cube with sidelength $2\delta _{Q}$; that cube is
denoted by $Q^{+}$.

We write $\mathcal{P}$ to denote the vector space of all real-valued
polynomials of degree at most $\left( m-1\right) $ on $\mathbb{R}^{n}$. If $%
x\in \mathbb{R}^{n}$ and $F$ is a real-valued $C^{m-1}$ function on a
neighborhood of $x$, then $J_{x}\left( F\right) $ (the \textquotedblleft
jet" of $F$ at $x$) denotes the $\left( m-1\right) ^{rst}$ order Taylor
polynomial of $F$ at $x$, i.e.,

\begin{equation*}
J_{x}\left( F\right) \left( y\right) =\sum_{\left\vert \alpha \right\vert
\leq m-1}\frac{1}{\alpha !}\partial ^{\alpha }F\left( x\right) \cdot \left(
y-x\right) ^{\alpha }.
\end{equation*}%
Thus, $J_{x}\left( F\right) \in \mathcal{P}$.

For each $x\in \mathbb{R}^{n}$, there is a natural multiplication $\odot
_{x} $ on $\mathcal{P}$ (\textquotedblleft multiplication of jets at $x$")
defined by setting%
\begin{equation*}
P\odot _{x}Q=J_{x}\left( PQ\right) \text{ for }P,Q\in \mathcal{P}\text{.}
\end{equation*}%
We write $C^{m}\left( \mathbb{R}^{n}\right) $ to denote the Banach space of
real-valued locally $C^{m}$ functions $F$ on $\mathbb{R}^{n}$ for which the
norm 
\begin{equation*}
\left\Vert F\right\Vert _{C^{m}\left( \mathbb{R}^{n}\right) }=\sup_{x\in 
\mathbb{R}^{n}}\max_{\left\vert \alpha \right\vert \leq m}\left\vert
\partial ^{\alpha }F\left( x\right) \right\vert
\end{equation*}%
is finite. Similarly, for $D\geq 1$, we write $C^{m}\left( \mathbb{R}^{n},%
\mathbb{R}^{D}\right) $ to denote the Banach space of all $\mathbb{R}^{D}$%
-valued locally $C^{m}$ functions $F$ on $\mathbb{R}^{n}$, for which the
norm 
\begin{equation*}
\left\Vert F\right\Vert _{C^{m}\left( \mathbb{R}^{n},\mathbb{R}^{D}\right)
}=\sup_{x\in \mathbb{R}^{n}}\max_{\left\vert \alpha \right\vert \leq
m}\left\Vert \partial ^{\alpha }F\left( x\right) \right\Vert
\end{equation*}%
is finite. Here, we use the Euclidean norm on $\mathbb{R}^{D}$.

If $F$ is a real-valued function on a cube $Q$, then we write $F\in
C^{m}\left( Q\right) $ to denote that $F$ and its derivatives up to $m$-th
order extend continuously to the closure of $Q$. For $F\in C^{m}\left(
Q\right) $, we define 
\begin{equation*}
\left\Vert F\right\Vert _{C^{m}\left( Q\right) }=\sup_{x\in
Q}\max_{\left\vert \alpha \right\vert \leq m}\left\vert \partial ^{\alpha
}F\left( x\right) \right\vert .
\end{equation*}%
Similarly, if $F$ is an $\mathbb{R}^{D}$-valued function on a cube $Q$,
then we write $F\in C^{m}\left( Q,\mathbb{R}^{D}\right) $ to denote that $F$
and its derivatives up to $m$-th order extend continuously to the closure of 
$Q$. For $F\in C^{m}\left( Q,\mathbb{R}^{D}\right) $, we define 
\begin{equation*}
\left\Vert F\right\Vert _{C^{m}\left( Q,\mathbb{R}^{D}\right) }=\sup_{x\in
Q}\max_{\left\vert \alpha \right\vert \leq m}\left\Vert \partial ^{\alpha
}F\left( x\right) \right\Vert \text{,}
\end{equation*}%
where again we use the Euclidean norm on $\mathbb{R}^{D}$.

If $F\in C^{m}\left( Q\right) $ and $x$ belongs to the boundary of $Q$, then
we still write $J_{x}\left( F\right) $ to denote the $\left( m-1\right)
^{rst}$ degree Taylor polynomial of $F$ at $x$, even though $F$ isn't
defined on a full neighborhood of $x\in \mathbb{R}^{n}$.

Let $S\subset \mathbb{R}^{n}$ be non-empty and finite. A \underline{Whitney
field} on $S$ is a family of polynomials 
\begin{equation*}
\vec{P}=\left( P^{y}\right) _{y\in S}\text{ (each }P^{y}\in \mathcal{P}\text{%
),}
\end{equation*}%
parametrized by the points of $S$.

We write $Wh\left( S\right) $ to denote the vector space of all Whitney
fields on $S$.

For $\vec{P}=\left( P^{y}\right) _{y\in S}\in Wh\left( S\right) $, we define
the seminorm 
\begin{equation*}
\left\Vert \vec{P}\right\Vert _{\dot{C}^{m}\left( S\right) }=\max_{x,y\in
S,\left( x\not=y\right), |\alpha| \leq m}\frac{\left\vert \partial ^{\alpha
}\left( P^{x}-P^{y}\right) \left( x\right) \right\vert }{\left\vert
x-y\right\vert ^{m-\left\vert \alpha \right\vert }}\text{.}
\end{equation*}

(If $S$ consists of a single point, then $\left\Vert \vec{P}\right\Vert _{%
\dot{C}^{m}\left( S\right) }=0$.)

We write $\mathcal{M}$ to denote the set of all multiindices $\alpha =\left(
\alpha _{1},\cdots ,\alpha _{n}\right) $ of order $\left\vert \alpha
\right\vert =\alpha _{1}+\cdots +\alpha _{n} \leq m-1$.

We define a (total) order relation $<$ on $\mathcal{M}$, as follows. Let $%
\alpha =\left( \alpha _{1},\cdots ,\alpha _{n}\right) $ and $\beta =\left(
\beta _{1},\cdots ,\beta _{n}\right) $ be distinct elements of $\mathcal{M}$%
. Pick the largest $k$ for which $\alpha _{1}+\cdots +\alpha
_{k}\not=\beta _{1}+\cdots +\beta _{k}$. (There must be at least one such $k$%
, since $\alpha $ and $\beta $ are distinct). Then we say that $\alpha
<\beta $ if $\alpha _{1}+\cdots +\alpha _{k}<\beta _{1}+\cdots +\beta _{k}$.

We also define a (total) order relation $<$ on subsets of $\mathcal{M}$, as
follows. Let $\mathcal{A},\mathcal{B}$ be distinct subsets of $\mathcal{M}$,
and let $\gamma $ be the least element of the symmetric difference $\left( 
\mathcal{A\setminus B}\right) \cup \left( \mathcal{B\setminus A}\right) $
(under the above order on the elements of $\mathcal{M}$). Then we say that $%
\mathcal{A}<\mathcal{B}$ if $\gamma \in \mathcal{A}$.

One checks easily that the above relations $<$ are indeed total order
relations. Note that $\mathcal{M}$ is minimal, and the empty set $\emptyset $
is maximal under $<$. A set $\mathcal{A}\subseteq \mathcal{M}$ is called 
\underline{monotonic} if, for all $\alpha \in \mathcal{A}$ and $\gamma \in 
\mathcal{M}$, $\alpha +\gamma \in \mathcal{M}$ implies $\alpha +\gamma \in 
\mathcal{A}$. We make repeated use of a simple observation:

Suppose $\mathcal{A}\subseteq \mathcal{M}$ is monotonic, $P\in \mathcal{P}$
and $x_{0}\in \mathbb{R}^{n}$. If $\partial ^{\alpha }P\left( x_{0}\right)
=0 $ for all $\alpha \in \mathcal{A}$, then $\partial ^{\alpha }P\equiv 0$
on $\mathbb{R}^{n}$.

This follows by writing $\partial ^{\alpha }P\left( y\right)
=\sum_{\left\vert \gamma \right\vert \leq m-1-\left\vert \alpha \right\vert }%
\frac{1}{\gamma !}\partial ^{\alpha +\gamma }P\left( x_{0}\right) \cdot
\left( y-x_{0}\right) ^{\gamma }$ and noting that all the relevant $\alpha
+\gamma $ belong to $\mathcal{A}$, hence $\partial ^{\alpha +\gamma }P\left(
x_{0}\right) =0$.

We need a few elementary facts about convex sets. We recall

\theoremstyle{plain} \newtheorem*{thm Helly}{Helly's Theorem}%
\begin{thm Helly}
Let $K_{1},\cdots ,K_{N}\subset \mathbb{R}^{D}$ be convex. Suppose that $K_{i_{1}}\cap \cdots \cap K_{i_{D+1}}$ is nonempty for any $i_{1}, \cdots, i_{D+1}\in
\{1,\cdots ,N\}$. Then $K_{1}\cap \cdots \cap K_{N}$ is nonempty.\end{thm Helly}

See \cite{rock-convex}.

We also use the following \theoremstyle{plain} 
\newtheorem*{thm Trivial
Remark on Convex Sets}{Trivial Remark on Convex Sets}%
\begin{thm Trivial Remark on Convex Sets}
Let $\Gamma $ be a convex set, and let $P_{0}$, $P_{0}~+~P_{\nu }$, $P_{0}~-~P_{\nu }\in \Gamma $ for $\nu =1,\cdots ,\nu _{\max }$. Then for any
real numbers $t_{1},\cdots ,t_{\nu _{\max }}$ with 
\begin{equation*}
\sum_{\nu =1}^{\nu _{\max }}\left\vert t_{\nu }\right\vert \leq 1,
\end{equation*}we have 
\begin{equation*}
P_{0}+\sum_{\nu =1}^{\nu _{\max }}t_{\nu }P_{\nu }\in \Gamma \text{.}
\end{equation*}\end{thm Trivial Remark on Convex Sets}

To see this, we write each $t_{\nu }=\sigma _{\nu }\left\vert t_{\nu
}\right\vert $ with $\sigma _{\nu } \in \left\{ -1,1\right\} $. Then 
\begin{equation*}
P_{0}+\sum_{\nu =1}^{\nu _{\max }}t_{\nu }P_{\nu }=\sum_{\nu =1}^{\nu _{\max
}}\left\vert t_{\nu }\right\vert \cdot \left( P_{0}+\sigma _{\nu }P_{\nu
}\right) +\left[ 1-\sum_{\nu =1}^{\nu _{\max }}\left\vert t_{\nu
}\right\vert \right] \cdot \left( P_{0}\right) \text{.}
\end{equation*}%
The right-hand side is a convex combination of vectors in $\Gamma $, proving
the trivial remark.

For finite sets $X$, we write $\#\left( X\right) $ to denote the numbers of
elements in $X$.

If $\lambda =\left( \lambda _{1},\cdots ,\lambda _{n}\right) $ is an $n$%
-tuple of positive real numbers, and if $\beta =\left( \beta _{1},\cdots
,\beta _{n}\right) \in \mathbb{Z}^{n}$, then we write $\lambda ^{\beta }$ to
denote 
\begin{equation*}
\lambda _{1}^{\beta _{1}}\cdots \lambda _{n}^{\beta _{n}}.
\end{equation*}%
We write $B_{n}\left( x,r\right) $ to denote the open ball in $\mathbb{R}%
^{n} $ with center $x$ and radius $r$, with respect to the Euclidean metric. 

\section{Shape Fields}

Let $E\subset \mathbb{R}^{n}$ be finite. For each $x\in E$, $M\in \left(
0,\infty \right) $, let $\Gamma \left( x,M\right) \subseteq \mathcal{P}$ be
a (possibly empty) convex set. We say that $\vec{\Gamma}=\left( \Gamma
\left( x,M\right) \right) _{x\in E,M>0}$ is a \underline{shape field} if for
all $x\in E$ and $0<M^{\prime }\leq M<\infty $, we have 
\begin{equation*}
\Gamma \left( x,M^{\prime }\right) \subseteq \Gamma \left( x,M\right) .
\end{equation*}

Let $\vec{\Gamma}=\left( \Gamma \left( x,M\right) \right) _{x\in E,M>0}$ be
a shape field and let $C_{w},\delta _{\max }$ be positive real numbers. We
say that $\vec{\Gamma}$ is \underline{$\left( C_{w},\delta _{\max }\right) $%
-convex} if the following condition holds:

Let $0<\delta \leq \delta _{\max }$, $x\in E$, $M\in \left( 0,\infty \right) 
$, $P_{1}$, $P_{2}$, $Q_{1}$, $Q_{2}\in \mathcal{P}$. Assume that

\begin{itemize}
\item[\refstepcounter{equation}\text{(\theequation)}\label{wsf1}] $P_1,P_2
\in\Gamma(x,M)$;

\item[\refstepcounter{equation}\text{(\theequation)}\label{wsf2}] $%
|\partial^\beta(P_1-P_2)(x)| \leq M\delta^{m-|\beta|}$ for $|\beta| \leq m-1$%
;

\item[\refstepcounter{equation}\text{(\theequation)}\label{wsf3}] $%
|\partial^\beta Q_i(x)|\leq \delta^{-|\beta|}$ for $|\beta| \leq m-1$ for $%
i=1,2$;

\item[\refstepcounter{equation}\text{(\theequation)}\label{wsf4}] $%
Q_1\odot_x Q_1 + Q_2 \odot_x Q_2 =1$.
\end{itemize}

Then

\begin{itemize}
\item[\refstepcounter{equation}\text{(\theequation)}\label{wsf5}] $%
P:=Q_1\odot_x Q_1\odot_x P_1 + Q_2 \odot_x Q_2 \odot_x P_2\in\Gamma(x,C_wM)$.
\end{itemize}

\begin{lemma}
\label{lemma-wsf1} Suppose $\vec{\Gamma}=\left( \Gamma \left( x,M\right)
\right) _{x\in E,M>0}$ is a $\left( C_{w},\delta _{\max }\right) $-convex
shape field. Let

\begin{itemize}
\item[\refstepcounter{equation}\text{(\theequation)}\label{6}] $0<\delta
\leq \delta_{\max}$, $x \in E$, $M>0$, $P_1,P_2,Q_1,Q_2 \in \mathcal{P}$ and 
$A^{\prime },A^{\prime \prime }>0$.
\end{itemize}

Assume that

\begin{itemize}
\item[\refstepcounter{equation}\text{(\theequation)}\label{7}] $P_1,P_2 \in
\Gamma(x,A^{\prime }M)$;

\item[\refstepcounter{equation}\text{(\theequation)}\label{8}] $\left\vert
\partial ^{\beta }\left( P_{1}-P_{2}\right) \left( x\right) \right\vert \leq
A^{\prime}M \delta^{  m-\left\vert \beta \right\vert }$ for $\left\vert \beta
\right\vert \leq m-1$;

\item[\refstepcounter{equation}\text{(\theequation)}\label{9}] $\left\vert
\partial ^{\beta }Q_{i}\left( x\right) \right\vert \leq A^{\prime \prime 
}\delta^{-\left\vert \beta \right\vert }$ for $\left\vert \beta \right\vert \leq m-1$
and $i=1,2$;

\item[\refstepcounter{equation}\text{(\theequation)}\label{10}] $Q_{1}\odot
_{x}Q_{1}+Q_{2}\odot _{x}Q_{2}=1$.
\end{itemize}

Then

\begin{itemize}
\item[\refstepcounter{equation}\text{(\theequation)}\label{11}] $%
P:=Q_{1}\odot _{x}Q_{1}\odot _{x}P_{1}+Q_{2}\odot _{x}Q_{2}\odot
_{x}P_{2}\in \Gamma \left( x,CM\right) $ with $C$ determined by $A^{\prime }$%
, $A^{\prime \prime }$, $C_{w}$, $m$, and $n$.
\end{itemize}
\end{lemma}

\begin{proof}
Without loss of generality, we may suppose that $A^{\prime },A^{\prime
\prime }\geq 1$. We simply apply the definition of $(C_w,\delta_{\max})$%
-convexity, with $M, \delta$ (in the definition) replaced by $[A^{\prime
}\cdot (A^{\prime \prime  })^m\cdot M]$ and $[\frac{\delta}{A^{\prime \prime }}%
] $, respectively. Note that $|Q_i(x)| \leq 1$ by \eqref{10}, and $0< \frac{\delta}{A^{\prime \prime }}\leq
\delta_{\max}$. We obtain the desired conclusion \eqref{11} with $C = C_w
A^{\prime }\cdot (A^{\prime \prime })^m$.
\end{proof}

\begin{lemma}
\label{lemma-wsf2} Suppose $\vec{\Gamma}=\left( \Gamma \left( x,M\right)
\right) _{x\in E,M>0}$ is a $\left( C_{w},\delta _{\max }\right) $-convex
shape field. Let

\begin{itemize}
\item[\refstepcounter{equation}\text{(\theequation)}\label{12}] $0<\delta
\leq \delta _{\max }$, $x\in E$, $M>0,A^{\prime },A^{\prime \prime }>0$, $%
P_{1},\cdots P_{k},Q_{1},\cdots ,Q_{k}\in \mathcal{P}$.
\end{itemize}

Assume that

\begin{itemize}
\item[\refstepcounter{equation}\text{(\theequation)}\label{13}] $P_{i}\in
\Gamma \left( x,A^{\prime }M\right) $ for $i=1,\cdots ,k$;

\item[\refstepcounter{equation}\text{(\theequation)}\label{14}] $\left\vert
\partial ^{\beta }\left( P_{i}-P_{j}\right) \left( x\right) \right\vert \leq
A^{\prime}M\delta^{  m-\left\vert \beta \right\vert }$ for $\left\vert \beta
\right\vert \leq m-1$, $i,j=1,\cdots ,k$;

\item[\refstepcounter{equation}\text{(\theequation)}\label{15}] $\left\vert
\partial ^{\beta }Q_{i}\left( x\right) \right\vert \leq A^{\prime \prime }\delta^{
-\left\vert \beta \right\vert }$ for $\left\vert \beta \right\vert \leq m-1$
and $i=1,\cdots ,k $;

\item[\refstepcounter{equation}\text{(\theequation)}\label{16}] $%
\sum_{i=1}^{k}Q_{i}\odot _{x}Q_{i}=1$.
\end{itemize}

Then

\begin{itemize}
\item[\refstepcounter{equation}\text{(\theequation)}\label{17}] $%
\sum_{i=1}^kQ_{i}\odot _{x}Q_{i}\odot _{x}P_{i}\in \Gamma \left( x,CM\right)
, $ with $C$ determined by $A^{\prime }$, $A^{\prime \prime }$, $C_{w}$, $m$%
, $n$, $k$.
\end{itemize}
\end{lemma}

\begin{proof}
We write $c$, $C$, $C^{\prime }$, etc., to denote constants determined by $%
A^{\prime }$, $A^{\prime \prime }$, $C_w$, $m$, $n$, $k$. These symbols may
denote different constants in different occurrences. We first check %
\eqref{17} in the (trivial) case $k=1$. From \eqref{16} we have $Q_1 \odot_x
Q_1 =1$, hence $\sum_{i=1}^k Q_i \odot_x Q_i \odot_x P_i =P_1$, and %
\eqref{17} follows from \eqref{13}. To prove \eqref{17} for $k \geq 2$, we
proceed by induction on $k$. In the base case $k=2$, \eqref{17} follows at
once from $\eqref{12} \cdots \eqref{16}$ and Lemma \ref{lemma-wsf1}. For the
induction step, we fix $k \geq 3$ and assume the inductive hypothesis

\begin{itemize}
\item[\refstepcounter{equation}\text{(\theequation)}\label{18}] Lemma \ref%
{lemma-wsf2} holds with $(k-1)$ in place of $k$.
\end{itemize}

Under the assumption \eqref{18}, we prove that $\eqref{12} \cdots \eqref{16}$
imply \eqref{17}. This will complete the proof of Lemma \ref{lemma-wsf2}.

Suppose $\eqref{12} \cdots \eqref{16}$ hold. Thanks to \eqref{16}, we have $%
(Q_i(x))^2 \geq \frac{1}{k}$ for some $i \in \{1,\cdots, k \}$. We assume
without loss of generality that

\begin{itemize}
\item[\refstepcounter{equation}\text{(\theequation)}\label{19}] $(Q_k(x))^2
\geq \frac{1}{k}$.
\end{itemize}

We then define

\begin{itemize}
\item[\refstepcounter{equation}\text{(\theequation)}\label{20}] $Q_i^\#=Q_i$
for $i=1,\cdots, k-2$;
\end{itemize}

\begin{itemize}
\item[\refstepcounter{equation}\text{(\theequation)}\label{21}] $%
Q_{k-1}^{\#}=J_{x}\left( \left[ \left( Q_{k-1}\right) ^{2}+\left(
Q_{k}\right) ^{2}\right] ^{1/2}\right) $;

\item[\refstepcounter{equation}\text{(\theequation)}\label{22}] $%
P_{i}^{\#}=P_{i}$ for $i=1,\cdots ,k-2$;

\item[\refstepcounter{equation}\text{(\theequation)\label{23}}] $%
P_{k-1}^{\#}=J_{x}\left( \frac{\left( Q_{k-1}\right) ^{2}P_{k-1}+\left(
Q_{k}\right) ^{2}P_{k}}{\left( Q_{k-1}\right) ^{2}+\left( Q_{k}\right) ^{2}}%
\right) $;

\item[\refstepcounter{equation}\text{(\theequation)\label{24}}] $\tilde{Q}%
_{1}=J_{x}\left( \frac{Q_{k-1}}{\left[ \left( Q_{k-1}\right) ^{2}+\left(
Q_{k}\right) ^{2}\right] ^{1/2}}\right) $;

\item[\refstepcounter{equation}\text{(\theequation)\label{25}}] $\tilde{Q}%
_{2}=J_{x}\left( \frac{Q_{k}}{\left[ \left( Q_{k-1}\right) ^{2}+\left(
Q_{k}\right) ^{2}\right] ^{1/2}}\right) $.
\end{itemize}

The above definitions make sense thanks to \eqref{19}. We have the
identities:

\begin{itemize}
\item[\refstepcounter{equation}\text{(\theequation)}\label{26}] $\tilde{Q}%
_{1}\odot _{x}\tilde{Q}_{1}+\tilde{Q}_{2}\odot _{x}\tilde{Q}_{2}=1$,
\end{itemize}

\begin{itemize}
\item[\refstepcounter{equation}\text{(\theequation)}\label{26a}] $%
\sum_{i=1}^{k-1}Q_{i}^{\#}\odot Q_{i}^{\#}=1$,
\end{itemize}

\begin{itemize}
\item[\refstepcounter{equation}\text{(\theequation)}\label{27}] $%
P_{k-1}^{\#}=\tilde{Q}_{1}\odot _{x}\tilde{Q}_{1}\odot _{x}P_{k-1}+\tilde{Q}%
_{2}\odot _{x}\tilde{Q}_{2}\odot _{x}P_{k}$, and

\item[\refstepcounter{equation}\text{(\theequation)\label{28}}] $%
\sum_{i=1}^{k}Q_{i}\odot _{x}Q_{i}\odot
_{x}P_{i}=\sum_{i=1}^{k-1}Q_{i}^{\#}\odot _{x}Q_{i}^{\#}\odot _{x}P_{i}^{\#}$%
.
\end{itemize}

From \eqref{20}, \eqref{21}, \eqref{24}, \eqref{25}, and \eqref{15}, %
\eqref{19}, we have the estimates

\begin{itemize}
\item[\refstepcounter{equation}\text{(\theequation)}\label{29}] $\left\vert
\partial ^{\beta }\tilde{Q}_{1}(x)\right\vert $, $\left\vert \partial
^{\beta }\tilde{Q}_{2}\left( x\right) \right\vert \leq C\delta ^{-\left\vert
\beta \right\vert }$ for $\left\vert \beta \right\vert \leq m-1$
\end{itemize}

and

\begin{itemize}
\item[\refstepcounter{equation}\text{(\theequation)}\label{30}] $\left\vert
\partial ^{\beta }Q_{i}^{\#}\left( x\right) \right\vert \leq C\delta
^{-\left\vert \beta \right\vert }$ for $\left\vert \beta \right\vert \leq
m-1 $, $i=1,\cdots ,k-1$.
\end{itemize}

Note also that

\begin{itemize}
\item[\refstepcounter{equation}\text{(\theequation)}\label{31}] $\left\vert
\partial ^{\beta }\left( P_{i}^{\#}-P_{j}^{\#}\right) \left( x\right)
\right\vert \leq CM\delta ^{m-\left\vert \beta \right\vert }$ for $%
\left\vert \beta \right\vert \leq m-1$, $i,j=1,\cdots ,k-1$.
\end{itemize}

Indeed \eqref{31} is immediate from \eqref{14}, \eqref{22} unless $i$ or $%
j=k-1$.

If (say) $i=k-1$ and $j<k-1$, then we write 
\begin{equation*}
P_{k-1}^{\#}-P_{j}^{\#}=\tilde{Q}_{1}\odot _{x}\tilde{Q}_{1}\odot _{x}\left(
P_{k-1}-P_{j}\right) +\tilde{Q}_{2}\odot _{x}\tilde{Q}_{2}\odot _{x}\left(
P_{k}-P_{j}\right)
\end{equation*}%
and apply our estimates \eqref{29} and \eqref{14} to complete the proof of %
\eqref{31}.

Thus, \eqref{31} holds in all cases.

Recall that $0<\delta \leq \delta _{\max }$. Thanks to \eqref{13}, \eqref{14}%
, \eqref{29}, and \eqref{26}, we may apply Lemma \ref{lemma-wsf1} to
conclude that%
\begin{equation*}
\tilde{Q}_{1}\odot _{x}\tilde{Q}_{1}\odot _{x}P_{k-1}+\tilde{Q}_{2}\odot _{x}%
\tilde{Q}_{2}\odot _{x}P_{k}\in \Gamma \left( x,CM\right) \text{.}
\end{equation*}%
Thus, 
\begin{equation*}
P_{k-1}^{\#}\in \Gamma \left( x,CM\right) \text{ (see (\ref{27})).}
\end{equation*}%
In view of (\ref{13}) and (\ref{22}), we have

\begin{itemize}
\item[\refstepcounter{equation}\text{(\theequation)}\label{32}] $%
P_{i}^{\#}\in \Gamma \left( x,CM\right) $ for all $i=1,\cdots ,k-1$.
\end{itemize}

Now, applying (\ref{26a}), (\ref{30}), (\ref{31}), (\ref{32}) and our
induction hypothesis (\ref{18}), we conclude that 
\begin{equation*}
\sum_{k=1}^{k-1}Q_{i}^{\#}\odot _{x}Q_{i}^{\#}\odot _{x}P_{i}^{\#}\in \Gamma
\left( x,CM\right) .
\end{equation*}%
By (\ref{28}), we have therefore 
\begin{equation*}
\sum_{i=1}^{k}Q_{i}\odot _{x}Q_{i}\odot _{x}P_{i}\in \Gamma \left(
x,CM\right) \text{,}
\end{equation*}%
which is our desired conclusion (\ref{17}).

This completes our induction on $k$, proving Lemma \ref{lemma-wsf2}.
\end{proof}

Next, we define the \underline{first refinement} of a shape field $\vec{%
\Gamma}=\left( \Gamma \left( x,M\right) \right) _{x\in E,M>0}$ to be $\vec{%
\Gamma}^{\#}=\left( \Gamma ^{\#}\left( x,M\right) \right) _{x\in E,M>0}$,
where $\Gamma ^{\#}\left( x,M\right) $ consists of those $P^{\#}\in \mathcal{%
P}$ such that for all $y\in E$ there exists $P\in \Gamma \left( y,M\right) $
for which

\begin{itemize}
\item[\refstepcounter{equation}\text{(\theequation)}\label{33}] $\left\vert
\partial ^{\beta }\left( P^{\#}-P\right) \left( x\right) \right\vert \leq
M\left\vert x-y\right\vert ^{m-\left\vert \beta \right\vert }$ for $%
\left\vert \beta \right\vert \leq m-1$.
\end{itemize}

Note that each $\Gamma ^{\#}\left( x,M\right) $ is a (possibly empty) convex
subset of $\mathcal{P}$, and that $M^{\prime }\leq M$ implies $\Gamma
^{\#}\left( x,M^{\prime }\right) \subseteq \Gamma ^{\#}\left( x,M\right) $.
Thus, $\vec{\Gamma}^{\#}$ is again a shape field. Taking $y=x$ in \eqref{33}%
, we see that $\Gamma^\#(x,M) \subset \Gamma(x,M)$.

\begin{lemma}
\label{lemma-wsf3} Let $\vec{\Gamma}=\left( \Gamma \left( x,M\right) \right)
_{x\in E,M>0}$ be a $\left( C_{w},\delta _{\max }\right) $-convex shape
field, and $\vec{\Gamma}^{\#}=\left( \Gamma ^{\#}\left( x,M\right) \right)
_{x\in E,M>0}$ be the first refinement of $\vec{\Gamma}^{\#}$. Then $\vec{%
\Gamma}^{\#}$ is $\left( C,\delta _{\max }\right) $-convex, where $C$ is
determined by $C_{w}$, $m$, $n$.
\end{lemma}

\begin{proof}
We write $c$, $C$, $C^{\prime }$, etc., to denote constants determined by $%
C_w$, $m$, $n$. These symbols may denote different constants in different
occurrences. Let $0 < \delta \leq \delta_{\max}$, $M > 0$, $x \in E$, $%
P_1^\#, P_2^\#, Q_1^\#, Q_2^\# \in \mathcal{P}$, and assume that

\begin{itemize}
\item[\refstepcounter{equation}\text{(\theequation)}\label{34}] $P_i^\# \in
\Gamma^\#(x,M)$ for $i=1,2$;
\end{itemize}

\begin{itemize}
\item[\refstepcounter{equation}\text{(\theequation)}\label{35}] $\left\vert
\partial ^{\beta }\left( P_{1}^{\#}-P_{2}^{\#}\right) \left( x\right)
\right\vert \leq M\delta ^{m-\left\vert \beta \right\vert }$ for $\left\vert
\beta \right\vert \leq m-1$;

\item[\refstepcounter{equation}\text{(\theequation)}\label{36}] $\left\vert
\partial ^{\beta }Q_{i}^{\#}\left( x\right) \right\vert \leq \delta
^{-\left\vert \beta \right\vert }$ for $\left\vert \beta \right\vert \leq
m-1 $, $i=1,2$; and

\item[\refstepcounter{equation}\text{(\theequation)}\label{37}] $%
Q_{1}^{\#}\odot _{x}Q_{1}^{\#}+Q_{2}^{\#}\odot _{x}Q_{2}^{\#}=1$.
\end{itemize}

Under the above assumptions, we must show that%
\begin{equation*}
Q_{1}^{\#}\odot _{x}Q_{1}^{\#}\odot _{x}P_{1}^{\#}+Q_{2}^{\#}\odot
_{x}Q_{2}^{\#}\odot _{x}P_{2}^{\#}\in \Gamma ^{\#}\left( x,CM\right) \text{.}
\end{equation*}%
By definition of $\Gamma ^{\#}\left( \cdot ,\cdot \right) $, this means that
given any $y\in E$ there exists

\begin{itemize}
\item[\refstepcounter{equation}\text{(\theequation)}\label{38}] $P \in
\Gamma(y,CM)$ such that
\end{itemize}

\begin{itemize}
\item[\refstepcounter{equation}\text{(\theequation)}\label{39}] $%
|\partial^\beta(P^\#-P)(x)| \leq CM|x-y|^{m-|\beta|}$ for $|\beta| \leq m-1$%
, where
\end{itemize}

\begin{itemize}
\item[\refstepcounter{equation}\text{(\theequation)}\label{40}] $%
P^\#=Q_{1}^{\#}\odot _{x}Q_{1}^{\#}\odot _{x}P_{1}^{\#}+Q_{2}^{\#}\odot
_{x}Q_{2}^{\#}\odot _{x}P_{2}^{\#}$.
\end{itemize}

Thus, to prove Lemma \ref{lemma-wsf3}, we must prove that there exists $P$
satisfying \eqref{38}, \eqref{39}, under the assumptions \eqref{34}$\cdots$%
\eqref{37}. We begin the proof of \eqref{38}, \eqref{39} (for some $P$) by
defining the functions

\begin{itemize}
\item[\refstepcounter{equation}\text{(\theequation)}\label{41}] $\theta _{i}=%
\frac{Q_{i}^{\#}}{\left[ \left( Q_{1}^{\#}\right) ^{2}+\left(
Q_{2}^{\#}\right) ^{2}\right] ^{1/2}}$ on $B_{n}(x,c_{0}\delta )$ ($i=1,2$).
\end{itemize}

We pick $c_0 <1$ small enough so that \eqref{36}, \eqref{37} guarantee that $%
\theta_i$ is well-defined on $B_n(x,c_0\delta)$ and satisfies

\begin{itemize}
\item[\refstepcounter{equation}\text{(\theequation)}\label{42}] $|\partial
^{\beta }\theta _{i}|\leq C\delta ^{-\left\vert \beta \right\vert }$ on $%
B_{n}\left( x,c_{0}\delta \right) $ for $\left\vert \beta \right\vert \leq m$%
, $i=1,2$,
\end{itemize}

and

\begin{itemize}
\item[\refstepcounter{equation}\text{(\theequation)}\label{43}] $%
\theta_1^2+\theta_2^2=1$ on $B_n(x,c_0\delta)$.
\end{itemize}

Also

\begin{itemize}
\item[\refstepcounter{equation}\text{(\theequation)}\label{44}] $%
J_x(\theta_i)=Q_i^\#$ for $i=1,2$,
\end{itemize}

thanks to \eqref{37}.

We now divide the proof of \eqref{38}, \eqref{39} (for some $P$) into two
cases.

\underline{CASE 1: Suppose $y\in B_n(x,c_0\delta)$.}

By \eqref{34} and the definition of $\Gamma^\#(\cdot, \cdot)$, there exist

\begin{itemize}
\item[\refstepcounter{equation}\text{(\theequation)}\label{45}] $P_i \in
\Gamma(y,M)$ ($i=1,2$)
\end{itemize}

satisfying

\begin{itemize}
\item[\refstepcounter{equation}\text{(\theequation)}\label{46}] $%
|\partial^\beta(P_i^\# - P_i)(x)| \leq M|x-y|^{m- |\beta|}$ for $|\beta|
\leq m-1$, $i=1,2$.
\end{itemize}

Since we are in CASE 1, estimates \eqref{35} and \eqref{46} together imply
that

\begin{itemize}
\item[\refstepcounter{equation}\text{(\theequation)}\label{47}] $%
|\partial^\beta(P_1-P_2)(x)| \leq CM\delta^{m-|\beta|}$ for $|\beta| \leq
m-1 $.
\end{itemize}

Consequently,

\begin{itemize}
\item[\refstepcounter{equation}\text{(\theequation)}\label{48}] $%
|\partial^\beta(P_1-P_2)|\leq CM \delta^{m-|\beta|}$ on $B_n(x,c_0 \delta)$
for $|\beta| \leq m$.
\end{itemize}

(Recall that $P_1,P_2$ are polynomials of degree at most $m-1$.) In
particular,

\begin{itemize}
\item[\refstepcounter{equation}\text{(\theequation)}\label{49}] $%
|\partial^\beta(P_1-P_2)(y)| \leq CM \delta^{m-|\beta|}$ for $|\beta| \leq
m-1$.
\end{itemize}

Thanks to \eqref{45}, \eqref{49}, and \eqref{42}, \eqref{43}, and the $%
\left( C_{w},\delta _{\max }\right) $-convexity of $\vec{\Gamma}$, we may
apply Lemma \ref{lemma-wsf1} to the polynomials $P_{1}$, $P_{2}$, $Q_{1}$, $%
Q_{2}$, where $Q_{i}=J_{y}\left( \theta _{i}\right) $ for $i=1,2$. This
tells us that

\begin{itemize}
\item[\refstepcounter{equation}\text{(\theequation)}\label{50}] $%
P:=J_{y}\left( \theta _{1}^{2}P_{1}+\theta _{2}^{2}P_{2}\right) \in \Gamma
\left( y,CM\right) $.
\end{itemize}

That is, the $P$ in \eqref{50} satisfies \eqref{38}. We will show that it
also satisfies \eqref{39}.

Thanks to \eqref{40}, \eqref{44}, we have

\begin{itemize}
\item[\refstepcounter{equation}\text{(\theequation)}\label{51}] $P^\# =
J_x(\theta_1^2 P_1^\#+\theta_2^2 P_2^\#)$.
\end{itemize}

In view of \eqref{50}, \eqref{51}, our desired estimate \eqref{39} is
equivalent to the following:

\begin{itemize}
\item[\refstepcounter{equation}\text{(\theequation)}\label{52}] $\left\vert
\partial ^{\beta }\left( \theta _{1}^{2}P_{1}^{\#}+\theta
_{2}^{2}P_{2}^{\#}-J_{y}\left( \theta _{1}^{2}P_{1}+\theta
_{2}^{2}P_{2}\right) \right) \left( x\right) \right\vert \leq CM\left\vert
x-y\right\vert ^{m-\left\vert \beta \right\vert }$ 
\end{itemize}
for $\left\vert \beta
\right\vert \leq m-1$.

Thus, we have proven the existence of $P$ satisfying \eqref{38}, \eqref{39}
in CASE 1, provided we can prove \eqref{52}.

Since $\theta _{1}^{2}+\theta _{2}^{2}=1$ (see \eqref{43}) and $%
J_{y}P_{1}=P_{1}$, the following holds on $B_n(x,c_0\delta)$: 
\begin{eqnarray*}
&&\left( \theta _{1}^{2}P_{1}^{\#}+\theta _{2}^{2}P_{2}^{\#}-J_{y}\left(
\theta _{1}^{2}P_{1}+\theta _{2}^{2}P_{2}\right) \right) \\
&=&\theta _{1}^{2}\left( P_{1}^{\#}-P_{1}\right) +\theta _{2}^{2}\left(
P_{2}^{\#}-P_{2}\right) \\
&&+\left[ \theta _{1}^{2}P_{1}+\theta _{2}^{2}P_{2}-J_{y}\left( \theta
_{1}^{2}P_{1}+\theta _{2}^{2}P_{2}\right) \right] \\
&=&\theta _{1}^{2}\left( P_{1}^{\#}-P_{1}\right) +\theta _{2}^{2}\left(
P_{2}^{\#}-P_{2}\right) \\
&&+\left[ P_{1}+\theta _{2}^{2}\left( P_{2}-P_{1}\right) -J_{y}\left(
P_{1}+\theta _{2}^{2}\left( P_{2}-P_{1}\right) \right) \right] \\
&=&\theta _{1}^{2}\left( P_{1}^{\#}-P_{1}\right) +\theta _{2}^{2}\left(
P_{2}^{\#}-P_{2}\right) \\
&&+\left[ \theta _{2}^{2}\left( P_{2}-P_{1}\right) -J_{y}\left( \theta
_{2}^{2}\left( P_{2}-P_{1}\right) \right) \right] .
\end{eqnarray*}%
Consequently, the desired estimate (\ref{52}) will follow if we can show
that 
\begin{equation}
\left\vert \partial ^{\beta }\left[ \theta _{i}\left(
P_{i}^{\#}-P_{i}\right) \left( x\right) \right] \right\vert \leq
CM\left\vert x-y\right\vert ^{m-\left\vert \beta \right\vert }\text{ for }%
\left\vert \beta \right\vert \leq m-1\text{, }i=1\text{,}2\text{,}
\label{53}
\end{equation}%
and 
\begin{equation}
\left\vert \partial ^{\beta }\left[ \theta _{2}^{2}\left( P_{1}-P_{2}\right)
-J_{y}\left( \theta _{2}^{2}\left( P_{1}-P_{2}\right) \right) \right] \left(
x\right) \right\vert \leq CM\left\vert x-y\right\vert ^{m-\left\vert \beta
\right\vert }\text{ for }\left\vert \beta \right\vert \leq m-1.  \label{54}
\end{equation}

Moreover, (\ref{53}) follows at once from (\ref{46}) and (\ref{42}), since $%
\delta ^{-\left\vert \beta \right\vert }\leq C\left\vert x-y\right\vert
^{-\left\vert \beta \right\vert }$ in CASE\ 1.

To check (\ref{54}), we apply (\ref{42}) and (\ref{48}) to deduce that 
\begin{equation*}
\left\vert \partial ^{\beta }\left[ \theta _{2}^{2}\left( P_{1}-P_{2}\right) %
\right] \left( y\right) \right\vert \leq CM\delta ^{m-\left\vert \beta
\right\vert }\text{ for all }y\in B_{n}\left( x,c_{0}\delta \right) \text{
for }\left\vert \beta \right\vert \leq m\text{.}
\end{equation*}%
In particular, $\left\vert \partial ^{\beta }\left[ \theta _{2}^{2}\left(
P_{1}-P_{2}\right) \right] \right\vert \leq CM$ on $B\left( x,c_{0}\delta
\right) $ for $\left\vert \beta \right\vert =m$.

Therefore, (\ref{54}) follows from Taylor's theorem.

This proves the existence of a $P$ satisfying (\ref{38}), (\ref{39}) in\
CASE 1.

\underline{CASE 2: Suppose that $y\not\in B_{n}(x,c_{0}\delta )$.}

Since $P_{1}^{\#}\in \Gamma \left( x, M\right) $ (see (\ref{34})), there
exists

\begin{itemize}
\item[\refstepcounter{equation}\text{(\theequation)}\label{55}] $P_{1}\in
\Gamma \left( y,M\right) $
\end{itemize}

such that

\begin{itemize}
\item[\refstepcounter{equation}\text{(\theequation)}\label{56}] $\left\vert
\partial ^{\beta }\left( P_{1}^{\#}-P_{1}\right) \left( x\right) \right\vert
\leq M\left\vert x-y\right\vert ^{m-\left\vert \beta \right\vert }$ for $%
\left\vert \beta \right\vert \leq m-1$.
\end{itemize}

Thanks to (\ref{37}), we may rewrite (\ref{40}) in the form $$%
P^{\#}=P_{1}^{\#}+Q_{2}^{\#}\odot _{x}Q_{2}^{\#}\odot _{x}\left(
P_{2}^{\#}-P_{1}^{\#}\right).$$ Our assumptions (\ref{35}), (\ref{36})
therefore yield the estimates $$\left\vert \partial ^{\beta }\left(
P^{\#}-P_{1}^{\#}\right) \left( x\right) \right\vert \leq CM\delta
^{m-\left\vert \beta \right\vert }$$ for $\left\vert \beta \right\vert \leq
m-1$.

Since we are in CASE 2, it follows that

\begin{itemize}
\item[\refstepcounter{equation}\text{(\theequation)}\label{57}] $\left\vert
\partial ^{\beta }\left( P^{\#}-P_{1}^{\#}\right) \left( x\right)
\right\vert \leq CM\left\vert x-y\right\vert ^{m-\left\vert \beta
\right\vert }$ for $\left\vert \beta \right\vert \leq m-1$.
\end{itemize}

From \eqref{56} and \eqref{57}, we learn that

\begin{itemize}
\item[\refstepcounter{equation}\text{(\theequation)}\label{58}] $\left\vert
\partial ^{\beta }\left( P^{\#}-P_{1}\right) \left( x\right) \right\vert
\leq CM\left\vert x-y\right\vert ^{m-\beta }$ for $\left\vert \beta
\right\vert \leq m-1$.
\end{itemize}

We now know from \eqref{55} and \eqref{58} that $P:=P_1$ satisfies \eqref{38}
and \eqref{39}. Thus, in CASE 2 we again have a polynomial $P$ satisfying %
\eqref{38} and \eqref{39}. We have seen in all cases that there exists $P
\in \mathcal{P}$ satisfying \eqref{38}, \eqref{39}.

The proof of Lemma \ref{lemma-wsf3} is complete.
\end{proof}

Next we define the higher refinements of a given shape field $\vec{\Gamma}_0
= \left( \Gamma_0(x,M) \right)_{x\in E,M>0}$. By induction on $l \geq 0$, we
define $\vec{\Gamma}_l = \left(\Gamma_l(x,M) \right)_{x \in E, M>0}$; to do
so, we start with our given $\vec{\Gamma}_0$, and define $\vec{\Gamma}_{l+1}$
to be the first refinement of $\vec{\Gamma}_l$, for each $l \geq 0$. Thus,
each $\vec{\Gamma}_l$ is a shape field.

\begin{lemma}
\label{lemma-wsf4}

\begin{itemize}
\item[(A)] Let $x,y\in E$, $l\geq 1$, $M>0$, and $P\in \Gamma _{l}(x,M)$.
Then there exists $P^{\prime }\in \Gamma _{l-1}(y,M)$ such that 
\begin{equation*}
\left\vert \partial ^{\beta }\left( P-P^{\prime }\right) \left( x\right)
\right\vert \leq M\left\vert x-y\right\vert ^{m-\left\vert \beta \right\vert
}\text{ for }\left\vert \beta \right\vert \leq m-1\text{.}
\end{equation*}

\item[(B)] If $\vec{\Gamma}_{0}$ is $\left( C_{w},\delta _{\max }\right) $%
-convex, then for each $l\geq 0$, $\vec{\Gamma}_{l}$ is $\left( C_{l},\delta
_{\max }\right) $-convex, where $C_{l}$ is determined by $C_{w}$, $l$, $m$, $%
n$.
\end{itemize}
\end{lemma}

\begin{proof}
(A) is immediate from the definition of the first refinement, since $\vec{%
\Gamma}_{l}$ is the first refinement of $\vec{\Gamma}_{l-1}$.

(B) follows trivially from Lemma \ref{lemma-wsf3} and induction on $l$.
\end{proof}

We call $\vec{\Gamma}_{l}$ the $l$-th refinement of $\vec{\Gamma}_{0}$.
(This is consistent with our previous definition of the first refinement.)

\section{Polynomial Bases}

\label{polynomial-bases}

Let $\vec{\Gamma}=\left( \Gamma \left( x,M\right) \right) _{x\in E,M>0}$ be
a shape field. Let $x_{0}\in E$, $M_{0}>0$, $P^{0}\in \mathcal{P}$, $%
\mathcal{A}\subseteq \mathcal{M}$, $P_{\alpha }\in \mathcal{P}$ for $\alpha
\in \mathcal{A}$, $C_{B}>0$, $\delta >0$ be given. Then we say that $\left(
P_{\alpha }\right) _{\alpha \in \mathcal{A}}$ forms \underline{an $\left( 
\mathcal{A},\delta ,C_{B}\right) $-basis for $\vec{\Gamma}$ at $\left(
x_{0},M_{0},P^{0}\right) $} if the following conditions are satisfied:

\begin{itemize}
\item[\refstepcounter{equation}\text{(\theequation)}\label{pb1}] $P^{0}\in
\Gamma \left( x_{0},C_{B}M_{0}\right) $.
\end{itemize}

\begin{itemize}
\item[\refstepcounter{equation}\text{(\theequation)}\label{pb2}] $P^{0}+%
\frac{M_{0}\delta ^{m-\left\vert \alpha \right\vert }}{C_{B}}P_{\alpha }$, $%
P^{0}-\frac{M_{0}\delta ^{m-\left\vert \alpha \right\vert }}{C_{B}}P_{\alpha
}\in \Gamma \left( x_{0},C_{B}M_{0}\right) $ for all $\alpha \in \mathcal{A}$%
.
\end{itemize}

\begin{itemize}
\item[\refstepcounter{equation}\text{(\theequation)}\label{pb3}] $\partial
^{\beta }P_{\alpha }\left( x_{0}\right) =\delta _{\alpha \beta }$ (Kronecker
delta) for $\beta ,\alpha \in \mathcal{A}$.
\end{itemize}

\begin{itemize}
\item[\refstepcounter{equation}\text{(\theequation)}\label{pb4}] $\left\vert
\partial ^{\beta }P_{\alpha }\left( x_{0}\right) \right\vert \leq
C_{B}\delta ^{\left\vert \alpha \right\vert -\left\vert \beta \right\vert }$
for all $\alpha \in \mathcal{A}$, $\beta \in \mathcal{M}$.
\end{itemize}

We say that $(P_{\alpha })_{\alpha \in \mathcal{A}}$ forms a \underline{weak 
$(\mathcal{A},\delta ,C_{B})$-basis for $\vec{\Gamma}$ at $%
(x_{0},M_{0},P^{0})$} if conditions \eqref{pb1}, \eqref{pb2}, \eqref{pb3}
hold as stated, and condition \eqref{pb4} holds for $\alpha \in \mathcal{A}%
,\beta \in \mathcal{M},\beta \geq \alpha $.

We make a few obvious remarks.

\begin{itemize}
\item[\refstepcounter{equation}\text{(\theequation)}\label{pb5}] Any $(%
\mathcal{A}, \delta, C_B)$-basis for $\vec{\Gamma}$ at $(x_0,M_0,P^0)$ is
also an $(\mathcal{A}, \delta, C_B^{\prime })$-basis for $\vec{\Gamma}$ at $%
(x_0,M_0,P^0)$, whenever $C^{\prime }_B \geq C_B$.
\end{itemize}

\begin{itemize}
\item[\refstepcounter{equation}\text{(\theequation)}\label{pb6}] Any $(%
\mathcal{A}, \delta, C_B)$-basis for $\vec{\Gamma}$ at $(x_0,M_0,P^0)$ is
also an $(\mathcal{A}, \delta^{\prime }, C_B\cdot[\max\{\frac{\delta^{\prime
}}{\delta},\frac{\delta}{\delta^{\prime }} \}]^m)$-basis for $\vec{\Gamma}$
at $(x_0,M_0,P^0)$, for any $\delta^{\prime }>0$.
\end{itemize}

\begin{itemize}
\item[\refstepcounter{equation}\text{(\theequation)}\label{pb7}] {Any weak $(%
\mathcal{A}, \delta, C_B)$-basis for $\vec{\Gamma}$ at $(x_0,M_0,P^0)$ is
also a weak $(\mathcal{A}, \delta^{\prime }, C_B^{\prime })$-basis for $\vec{%
\Gamma}$ at $(x_0,M_0,P^0)$, whenever $0<\delta^{\prime }\leq \delta$ and $%
C^{\prime }_B \geq C_B$. }
\end{itemize}

Note that \eqref{pb1} need not follow from \eqref{pb2}, since $\mathcal{A}$
may be empty.

\begin{itemize}
\item[\refstepcounter{equation}\text{(\theequation)}\label{pb7a}] If $%
\mathcal{A}=\emptyset$, then the existence of an $(\mathcal{A},\delta,C_B)$%
-basis for $\vec{\Gamma}$ at $(x_0,M_0,P^0)$ is equivalent to the assertion
that $P^0 \in \Gamma(x_0, C_BM_0)$.
\end{itemize}

The main result of this section is Lemma \ref{lemma-pb2} below. To prove it,
we first establish the following result.

\begin{lemma}
\label{lemma-pb1}Let $\vec{\Gamma}=\left( \Gamma \left( x,M\right) \right)
_{x\in E,M>0}$be a $\left( C_{w},\delta _{\max }\right) $-convex shape
field. Fix $x_{0}\in E$, $M_{0}>0$, $0<\delta \leq \delta _{\max }$, $C_{1}>0
$, and let $P^{0}$, $\hat{P}$, $\hat{S}\in \mathcal{P}$.

Assume that

\begin{itemize}
\item[\refstepcounter{equation}\text{(\theequation)}\label{pb8}] $P^{0}+%
\frac{1}{C_{1}}\hat{P}$, $P^{0}-\frac{1}{C_{1}}\hat{P}\in \Gamma \left(
x_{0},C_{1}M\right) $;
\end{itemize}

\begin{itemize}
\item[\refstepcounter{equation}\text{(\theequation)}\label{pb9}] $\left\vert
\partial ^{\beta }\hat{P}\left( x_{0}\right) \right\vert \leq
C_{1}M_{0}\delta ^{m-\left\vert \beta \right\vert }$ for $\left\vert \beta
\right\vert \leq m-1$; and
\end{itemize}

\begin{itemize}
\item[\refstepcounter{equation}\text{(\theequation)}\label{pb10}] $%
\left\vert \partial ^{\beta }\hat{S}\left( x_{0}\right) \right\vert \leq
C_{1}\delta ^{-\left\vert \beta \right\vert }$ for $\left\vert \beta
\right\vert \leq m-1$.
\end{itemize}

Then

\begin{itemize}
\item[\refstepcounter{equation}\text{(\theequation)}\label{pb11}] $P^{0}+%
\frac{1}{C_2}\hat{S}\odot _{x_{0}}\hat{P}$, $P^{0}-\frac{1}{C_2}\hat{S}\odot
_{x_{0}}\hat{P}\in \Gamma \left( x_{0},C_{2}M\right) $, with $C_{2}$
determined by $C_{1}$, $C_{w}$, $m$, $n$.
\end{itemize}
\end{lemma}

\begin{proof}
We write $c$, $C$, $C^{\prime }$, etc., to denote constants determined by $%
C_1$, $C_w$, $m$, $n$. These symbols may denote distinct constants in
different occurrences. For a small enough $c_0$, \eqref{pb10} guarantees
that the polynomials $Q_1:=J_{x_0}([\frac{1}{2}+c_0\hat{S}]^{1/2})$ and $%
Q_2:=J_{x_0}([\frac{1}{2}-c_0\hat{S}]^{1/2})$ are well-defined and satisfy

\begin{itemize}
\item[\refstepcounter{equation}\text{(\theequation)}\label{pb12}] $%
|\partial^\beta Q_i(x_0)|\leq C\delta^{-|\beta|}$ for $|\beta|\leq m-1$, $%
i=1,2$;
\end{itemize}

\begin{itemize}
\item[\refstepcounter{equation}\text{(\theequation)}\label{pb13}] $%
Q_{1}\odot _{x_{0}}Q_{1}=\frac{1}{2}+c_{0}\hat{S}$, $Q_{2}\odot
_{x_{0}}Q_{2}=\frac{1}{2}-c_{0}\hat{S}$, $Q_{1}\odot
_{x_{0}}Q_{1}+Q_{2}\odot _{x_{0}}Q_{2}=1$.
\end{itemize}

Define $P^{1}=P^{0}+\frac{1}{C_1}\hat{P}$, $P^{2}=P^{0}-\frac{1}{C_1}\hat{P}$%
. Then

\begin{itemize}
\item[\refstepcounter{equation}\text{(\theequation)}\label{pb14}] $%
P^{1},P^{2}\in \Gamma \left( x_{0},C_{1}M_{0}\right) $ and
\end{itemize}

\begin{itemize}
\item[\refstepcounter{equation}\text{(\theequation)}\label{pb15}] $%
\left\vert \partial ^{\beta }\left( P^{1}-P^{2}\right) \left( x_{0}\right)
\right\vert \leq CM_{0}\delta ^{m-\left\vert \beta \right\vert }$ for $%
\left\vert \beta \right\vert \leq m-1$.
\end{itemize}

Recall that $0<\delta \leq \delta _{\max }$. Hence, \eqref{pb12}$\cdots$%
\eqref{pb15} and Lemma \ref{lemma-wsf1} together imply that 
\begin{equation*}
P:=Q_{1}\odot _{x_{0}}Q_{1}\odot _{x_{0}}P^{1}+Q_{2}\odot _{x_{0}}Q_{2}\odot
_{x_{0}}P^{2}\in \Gamma \left( x_{0},CM_{0}\right) \text{.}
\end{equation*}%
However, \eqref{pb13} and the definitions of $P^{1}$, $P^{2}$, $P$ yield 
\begin{eqnarray*}
P &=&\left[ \frac{1}{2}+c_{0}\hat{S}\right] \odot _{x_{0}}\left[ P^{0}+\frac{%
\hat{P}}{C_{1}}\right] +\left[ \frac{1}{2}-c_{0}\hat{S}\right] \odot _{x_{0}}%
\left[ P^{0}-\frac{\hat{P}}{C_{1}}\right] \\
&=&P^0+\frac{2c_{0}}{C_{1}}\hat{S}\odot _{x_{0}}\hat{P}\text{.}
\end{eqnarray*}%
Thus, 
\begin{equation}
P^{0}+\frac{2c_{0}}{C_{1}}\hat{S}\odot _{x_{0}}\hat{P}\in \Gamma \left(
x_{0},CM_{0}\right) \text{.}  \label{pb16}
\end{equation}%
Replacing $\hat{S}$ by $-\hat{S}$ in the proof of \eqref{pb16}, we learn
that also 
\begin{equation}
P^{0}-\frac{2c_{0}}{C_{1}}\hat{S}\odot _{x_{0}}\hat{P}\in \Gamma \left(
x_{0},CM_{0}\right) \text{.}  \label{pb17}
\end{equation}%
Taking $C_{2}>\max \left\{ C,\frac{C_{1}}{2c_{0}}\right\} $ with $C$ as in %
\eqref{pb16} and \eqref{pb17}, we obtain from \eqref{pb16}, \eqref{pb17} and
convexity of $\Gamma \left( x_{0},CM_{0}\right) $ that $P^{0}\pm \frac{\hat{S%
}\odot _{x_{0}}\hat{P}}{C_{2}}\in \Gamma \left( x_{0},CM\right) \subseteq
\Gamma \left( x_{0},C_{2}M\right) $, where the last inclusion holds because $%
\vec{\Gamma}$ is a shape field.

The proof of Lemma \ref{lemma-pb1} is complete.
\end{proof}

We also need the following result, which is immediate\footnote{%
Lemma 16.1 in \cite{f-2005} involves also real numbers $F_{\alpha ,\beta }$
with $\left\vert \beta \right\vert =m$. Setting those $F_{\alpha ,\beta }=0$%
, we recover the Rescaling Lemma stated here.} from Lemma 16.1 in \cite%
{f-2005}.

\begin{lemma}[Rescaling Lemma]
\label{rescaling-lemma} Let $\mathcal{A}\subseteq \mathcal{M}$, and let $C,a$
be positive real numbers. Suppose we are given real numbers $F_{\alpha
,\beta }$, indexed by $\alpha \in \mathcal{A}$, $\beta \in \mathcal{M}$.
Assume that the following conditions are satisfied.

\begin{itemize}
\item[\refstepcounter{equation}\text{(\theequation)}\label{pb18}] $%
F_{\alpha,\alpha}\not=0$ for all $\alpha \in \mathcal{A}$.
\end{itemize}

\begin{itemize}
\item[\refstepcounter{equation}\text{(\theequation)}\label{pb19}] $%
|F_{\alpha,\beta}|\leq C \cdot |F_{\alpha,\alpha}|$ for all $\alpha \in 
\mathcal{A}$, $\beta\in \mathcal{M}$ with $\beta \geq \alpha$.
\end{itemize}

\begin{itemize}
\item[\refstepcounter{equation}\text{(\theequation)}\label{pb20}] $%
F_{\alpha,\beta}=0$ for all $\alpha,\beta\in \mathcal{A}$ with $\alpha
\not=\beta$.
\end{itemize}

Then there exist positive numbers $\lambda_1, \cdots, \lambda_n$ and a map $%
\phi: \mathcal{A} \rightarrow \mathcal{M}$, with the following properties:

\begin{itemize}
\item[\refstepcounter{equation}\text{(\theequation)}\label{pb21}] $c(a) \leq
\lambda_i \leq 1$ for each $i$, where $c(a)$ is determined by $C$, $a$, $m$, 
$n$;
\end{itemize}

\begin{itemize}
\item[\refstepcounter{equation}\text{(\theequation)}\label{pb22}] $%
\phi(\alpha) \leq \alpha$ for each $\alpha \in \mathcal{A}$;
\end{itemize}

\begin{itemize}
\item[\refstepcounter{equation}\text{(\theequation)}\label{pb23}] For each $%
\alpha \in \mathcal{A}$, either $\phi(\alpha) = \alpha$ or $\phi(\alpha)
\not\in \mathcal{A}$.
\end{itemize}

Suppose we define $\hat{F}_{\alpha,\beta}=\lambda^\beta F_{\alpha,\beta}$
for $\alpha \in \mathcal{A}$, $\beta \in \mathcal{M}$, where we recall that $%
\lambda^\beta$ denotes $\lambda^{\beta_1}_1\cdots \lambda^{\beta_n}_n$ for $%
\beta=(\beta_1,\cdots, \beta_n)$. Then

\begin{itemize}
\item[\refstepcounter{equation}\text{(\theequation)}\label{pb24}] $|\hat{F}%
_{\alpha,\beta}|\leq a \cdot |\hat{F}_{\alpha,\phi(\alpha)}|$ for $\alpha
\in \mathcal{A}$, $\beta \in \mathcal{M}\setminus \{ \phi(\alpha)\}$.
\end{itemize}
\end{lemma}

Lemma 3.3 in \cite{f-2005} tells us that any map $\phi: \mathcal{A}
\rightarrow \mathcal{M}$ satisfying \eqref{pb22}, \eqref{pb23} satisfies also

\begin{itemize}
\item[\refstepcounter{equation}\text{(\theequation)}\label{pb25}] $\phi(%
\mathcal{A}) \leq \mathcal{A}$, with equality only if $\phi$ is the identity
map.
\end{itemize}

We are ready to state the main result of this section.

\begin{lemma}[Relabeling Lemma]
\label{lemma-pb2} Let $\vec{\Gamma}=\left( \Gamma \left( x,M\right) \right)
_{x\in E,M>0}$ be a $(C_w,\delta_{\max})$-convex shape field. Let $x_{0}\in E
$, $M_{0}>0$, $0<\delta \leq \delta _{\max }$, $C_{B}>0$, $P^{0}\in \Gamma
\left( x_{0},M_{0}\right) $, $\mathcal{A}\subseteq \mathcal{M}$. Suppose $%
\left( P_{\alpha }^{00}\right) _{\alpha \in \mathcal{A}}$ is a weak $\left( 
\mathcal{A},\delta ,C_{B}\right) $-basis for $\vec{\Gamma}$ at $\left(
x_{0},M_{0},P^{o}\right) $. Then, for some monotonic $\hat{\mathcal{A}}\leq 
\mathcal{A}$, $\vec{\Gamma}$ has an $(\hat{\mathcal{A}},\delta
,C_{B}^{\prime })$-basis at $(x_{0},M_{0},P^{0})$, with $C_{B}^{\prime }$
determined by $C_{B}$, $C_{w}$, $m$, $n$. Moreover, if $\max_{\alpha \in 
\mathcal{A},\beta \in \mathcal{M}}\delta ^{|\beta |-|\alpha |}|\partial
^{\beta }P_{\alpha }^{00}(x_{0})|$ exceeds a large enough constant
determined by $C_{B}$, $C_{w}$, $m$, $n$, then we can take $\hat{\mathcal{A}}%
<\mathcal{A}$ (strict inequality).
\end{lemma}

\begin{proof} If $\A$ is empty, then we can take $\hat{\A}$ empty; note that $$\max_{\alpha \in \A,\beta \in \M} \delta ^{|\beta|-|\alpha|} |\partial^{\beta}P^{00}_\alpha(x_0)|$$
is defined to be zero for $\A$ empty. Thus, Lemma \ref{lemma-pb2} holds trivially for $\A = \emptyset$. We suppose that $\A \not= \emptyset$.

Without loss of generality, we may take $x_0 =0$. We introduce a constant $%
a>0$ to be picked later, satisfying the

\emph{Small $a$ condition: $a$ is less than a small enough constant
determined by $C_B$, $C_w$, $m$, $n$.}

We write $c$, $C$, $C^{\prime }$, etc., to denote constants determined by $%
C_B$, $C_w$, $m$, $n$; and we write $c(a)$, $C(a)$, $C^{\prime }(a)$, etc., to
denote constants determined by $a$, $C_B$, $C_w$, $m$, $n$. These symbols
may denote different constants in different occurrences.

Since $(P_\alpha^{00})_{\alpha \in \mathcal{A}}$ is a weak $(\mathcal{A},
\delta, C_B)$-basis for $\vec{\Gamma}$ at $(0, M_0, P^0)$, we have the
following.

\begin{itemize}
\item[\refstepcounter{equation}\text{(\theequation)}\label{pb26}] $%
P^{0},P^{0}\pm cM_{0}\delta ^{m-\left\vert \alpha \right\vert }P_{\alpha
}^{00}\in \Gamma \left( 0,CM_{0}\right) $ for $\alpha \in \mathcal{A}$.
\end{itemize}

\begin{itemize}
\item[\refstepcounter{equation}\text{(\theequation)}\label{pb27}] $\partial
^{\beta }P_{\alpha }^{00}\left( 0\right) =\delta _{\beta \alpha }$ for $%
\beta ,\alpha \in \mathcal{A}$.
\end{itemize}

\begin{itemize}
\item[\refstepcounter{equation}\text{(\theequation)}\label{pb28}] $%
\left\vert \partial ^{\beta }P_{\alpha }^{00}\left( 0\right) \right\vert
\leq C\delta ^{\left\vert \alpha \right\vert -\left\vert \beta \right\vert }$
for $\alpha \in \mathcal{A}$, $\beta \in \mathcal{M}$, $\beta \geq \alpha $.
\end{itemize}

Thanks to \eqref{pb27}, \eqref{pb28}, the numbers $F_{\alpha,\beta} =
\delta^{|\beta|-|\alpha|}\partial^\beta P_\alpha^{00}(0)$ satisfy %
\eqref{pb18}, \eqref{pb19}, \eqref{pb20}. Applying Lemma \ref%
{rescaling-lemma}, we obtain real numbers $\lambda_1, \cdots, \lambda_n$ and
a map $\phi: \mathcal{A} \rightarrow \mathcal{M}$ satisfying \eqref{pb21}, $%
\cdots$, \eqref{pb24}. We define a linear map $T: \mathbb{R}^n \rightarrow 
\mathbb{R}^n$ by setting

\begin{itemize}
\item[\refstepcounter{equation}\text{(\theequation)}\label{pb29}] $T\left(
x_{1},\cdots ,x_{n}\right) =\left( \lambda _{1}x_{1},\cdots ,\lambda
_{n}x_{n}\right) $ for $\left( x_{1},\cdots ,x_{n}\right) \in \mathbb{R}^{n}$%
.
\end{itemize}

From \eqref{pb21} $\cdots $\eqref{pb24} for our $F_{\alpha ,\beta }$, we
obtain the following.

\begin{itemize}
\item[\refstepcounter{equation}\text{(\theequation)}\label{pb30}] $c\left(
a\right) \leq \lambda _{i}\leq 1$ for $i=1,\cdots ,n$.
\end{itemize}

\begin{itemize}
\item[\refstepcounter{equation}\text{(\theequation)}\label{pb31}] $\phi
\left( \alpha \right) \leq \alpha $ for all $\alpha \in \mathcal{A}$.
\end{itemize}

\begin{itemize}
\item[\refstepcounter{equation}\text{(\theequation)}\label{pb32}] For each $%
\alpha \in \mathcal{A}$, either $\phi \left( \alpha \right) =\alpha $ or $%
\phi \left( \alpha \right) \not\in \mathcal{A}$.
\end{itemize}

\begin{itemize}
\item[\refstepcounter{equation}\text{(\theequation)}\label{pb33}] $\delta
^{\left\vert \beta \right\vert -\left\vert \alpha \right\vert }\left\vert
\partial ^{\beta }\left( P_{\alpha }^{00}\circ T\right) \left( 0\right)
\right\vert \leq a\cdot \delta ^{\left\vert \phi \left( \alpha \right)
\right\vert -\left\vert \alpha \right\vert }\left\vert \partial ^{\phi
\left( \alpha \right) }\left( P_{\alpha }^{00}\circ T\right) \left( 0\right)
\right\vert $ for all $\alpha \in \mathcal{A}$, $\beta \in \mathcal{%
M\setminus }\phi \left( \alpha \right) $.
\end{itemize}

Since the left-hand side of \eqref{pb33} is equal to $\lambda^\beta \geq
c(a) $ for $\beta=\alpha$ (by \eqref{pb27} and \eqref{pb30}), it follows
from \eqref{pb33} that

\begin{itemize}
\item[\refstepcounter{equation}\text{(\theequation)}\label{pb34}] $\delta
^{\left\vert \phi \left( \alpha \right) \right\vert -\left\vert \alpha
\right\vert }\left\vert \partial ^{\phi \left( \alpha \right) }\left(
P_{\alpha }^{00}\circ T\right) \left( 0\right) \right\vert \geq c\left(
a\right) $ for $\alpha \in \mathcal{A}$.
\end{itemize}

We define

\begin{itemize}
\item[\refstepcounter{equation}\text{(\theequation)}\label{pb35}] $\bar{%
\mathcal{A}}=\phi(\mathcal{A})$
\end{itemize}
and introduce a map
\begin{itemize}
\item[\refstepcounter{equation}\text{(\theequation)}\label{pb36}] $\psi: 
\bar{\mathcal{A}} \rightarrow \mathcal{A}$
\end{itemize}
such that
\begin{itemize}
\item[\refstepcounter{equation}\text{(\theequation)}\label{pb37}] $\phi
(\psi (\bar{\alpha}))=\bar{\alpha}$ for all $\bar{\alpha}\in \bar{\mathcal{A}%
}$.
\end{itemize}

Thanks to \eqref{pb31}, \eqref{pb32}, and Lemma 3.3 in \cite{f-2005}
(mentioned above), we have 
\begin{equation*}
\bar{\mathcal{A}}\leq \mathcal{A},
\end{equation*}%
with equality only when $\phi =$identity. Moreover, suppose $\phi =$
identity. Then \eqref{pb27} and \eqref{pb33} show that%
\begin{eqnarray*}
\left\vert \lambda ^{\beta }\delta ^{\left\vert \beta \right\vert
-\left\vert \alpha \right\vert }\partial ^{\beta }P_{\alpha }^{00}\left(
0\right) \right\vert  &=&\delta ^{\left\vert \beta \right\vert -\left\vert
\alpha \right\vert }\left\vert \partial ^{\beta }\left( P_{\alpha
}^{00}\circ T\right) \left( 0\right) \right\vert  \\
&\leq &\left\vert \partial ^{\alpha }\left( P_{\alpha }^{00}\circ T\right)
\left( 0\right) \right\vert  \\
&=&\lambda ^{\alpha }
\end{eqnarray*}%
for $\alpha \in \mathcal{A}$ and $\beta \in \mathcal{M}$. Hence, \eqref{pb30}
yields

\begin{itemize}
\item[\refstepcounter{equation}\text{(\theequation)}\label{pb38}] $\delta
^{\left\vert \beta \right\vert -\left\vert \alpha \right\vert }\left\vert
\partial ^{\beta }P_{\alpha }^{00}\left( 0\right) \right\vert \leq C\left(
a\right) $ for all $\alpha \in \mathcal{A}$ and $\beta \in \mathcal{M}$;
\end{itemize}
this holds provided $\phi =$identity.

If $\max_{\alpha \in \mathcal{A},\beta \in \mathcal{M}}\delta ^{\left\vert
\beta \right\vert -\left\vert \alpha \right\vert }\left\vert \partial
^{\beta }P_{\alpha }^{00}\left( 0\right) \right\vert >C(a) $ with $C\left(
a\right) $ as in \eqref{pb38}, then $\phi $ cannot be the identity map, and
therefore $\bar{\mathcal{A}} < \mathcal{A} $ (strict inequality).

Thus,

\begin{itemize}
\item[\refstepcounter{equation}\text{(\theequation)}\label{pb39}] $\bar{%
\mathcal{A}} \leq \mathcal{A}$, with strict inequality if $\max_{\alpha \in 
\mathcal{A},\beta \in \mathcal{M}}\delta ^{\left\vert \beta \right\vert
-\left\vert \alpha \right\vert }\left\vert \partial ^{\beta }P_{\alpha
}^{00}\left( 0\right) \right\vert $ exceeds a large enough constant
determined by $a$, $C_{w}$, $C_{B}$, $m$, $n$.
\end{itemize}

For $\bar{\alpha} \in \bar{\mathcal{A}}$, we define

\begin{itemize}
\item[\refstepcounter{equation}\text{(\theequation)}\label{pb40}] $b_{\bar{%
\alpha}}=\left[ \delta ^{\left\vert \bar{\alpha}\right\vert -\left\vert \psi
\left( \bar{\alpha}\right) \right\vert }\partial ^{\bar{\alpha}}\left(
P_{\psi \left( \bar{\alpha}\right) }^{00}\circ T\right) \left( 0\right) %
\right] ^{-1}$;
\end{itemize}

estimate \eqref{pb34} with $\alpha =\psi \left( \bar{\alpha}\right) $ gives

\begin{itemize}
\item[\refstepcounter{equation}\text{(\theequation)}\label{pb41}] $%
\left\vert b_{\bar{\alpha}}\right\vert \leq C\left( a\right) $ for all $\bar{%
\alpha}\in \bar{\mathcal{A}}$.
\end{itemize}

For $\bar{\alpha}\in \bar{\mathcal{A}}$, we also define

\begin{itemize}
\item[\refstepcounter{equation}\text{(\theequation)}\label{pb42}] $\bar{P}_{%
\bar{\alpha}}=b_{\bar{\alpha}}\delta ^{\left\vert \bar{\alpha}\right\vert
-\left\vert \psi \left( \bar{\alpha}\right) \right\vert }\cdot P_{\psi
\left( \bar{\alpha}\right) }^{00}$.
\end{itemize}

From \eqref{pb33}, \eqref{pb40}, \eqref{pb42}, we find that

\begin{itemize}
\item[\refstepcounter{equation}\text{(\theequation)}\label{pb43}] $%
\left\vert \delta ^{\left\vert \beta \right\vert -\left\vert \bar{\alpha}%
\right\vert }\partial ^{\beta }\left( \bar{P}_{\bar{\alpha}}\circ T\right)
\left( 0\right) -\delta _{\beta \bar{\alpha }}\right\vert \leq a$ for $\bar{\alpha}%
\in \bar{\mathcal{A}}$, $\beta \in \mathcal{M}$.
\end{itemize}

Note that $\delta ^{m-\left\vert \bar{\alpha}\right\vert }\bar{P}_{\bar{%
\alpha}}=b_{\bar{\alpha}}\delta ^{m-\left\vert \psi \left( \bar{\alpha}%
\right) \right\vert }P_{\psi \left( \bar{\alpha}\right) }^{00}$. Hence, %
\eqref{pb41} and \eqref{pb26} (with $\alpha =\psi \left( \bar{\alpha}\right) 
$) tell us that

\begin{itemize}
\item[\refstepcounter{equation}\text{(\theequation)}\label{pb44}] $P^{0}\pm
c\left( a\right) M_{0}\delta ^{m-\left\vert \bar{\alpha}\right\vert }\bar{P}%
_{\bar{\alpha}}\in \Gamma \left( 0,CM_{0}\right) $,
\end{itemize}

since $\Gamma(0,CM_0)$ is convex.

Next, define

\begin{itemize}
\item[\refstepcounter{equation}\text{(\theequation)}\label{pb45}] $\hat{%
\mathcal{A}}=\left\{ \gamma \in \mathcal{M}:\gamma =\bar{\alpha}+\bar{\gamma}%
\text{ for some }\bar{\alpha}\in \bar{\mathcal{A}}\text{, }\bar{\gamma}\in 
\mathcal{M}\right\} $,
\end{itemize}

and introduce maps $\chi :\hat{\mathcal{A}}\rightarrow \bar{\mathcal{A}}$, $%
\omega :\hat{\mathcal{A}}\rightarrow \mathcal{M}$, such that

\begin{itemize}
\item[\refstepcounter{equation}\text{(\theequation)}\label{pb46}] $\hat{%
\alpha}=\chi \left( \hat{\alpha}\right) +\omega \left( \hat{\alpha}\right) $
for all $\hat{\alpha}\in \hat{\mathcal{A}}$.
\end{itemize}

By definition \eqref{pb45},

\begin{itemize}
\item[\refstepcounter{equation}\text{(\theequation)}\label{pb47}] $\hat{%
\mathcal{A}}$ is monotonic,
\end{itemize}

and $\bar{\mathcal{A}} \subseteq \hat{\mathcal{A}}$, hence

\begin{itemize}
\item[\refstepcounter{equation}\text{(\theequation)}\label{pb48}] $\hat{%
\mathcal{A}}\leq \bar{\mathcal{A}}$.
\end{itemize}

From \eqref{pb39} and \eqref{pb48}, we have

\begin{itemize}
\item[\refstepcounter{equation}\text{(\theequation)}\label{pb49}] $\hat{%
\mathcal{A}} \leq \mathcal{A}$, with strict inequality if $\max_{\alpha \in 
\mathcal{A},\beta \in \mathcal{M}}\delta ^{\left\vert \beta \right\vert
-\left\vert \alpha \right\vert }\left\vert \partial ^{\beta }P_{\alpha
}^{00}\left( 0\right) \right\vert $ exceeds a large enough constant
determined by $a$, $C_{w}$, $C_{B}$, $m$, $n$.
\end{itemize}

For $\hat{\alpha} \in \hat{\mathcal{A}}$, we introduce the monomial

\begin{itemize}
\item[\refstepcounter{equation}\text{(\theequation)}\label{pb50}] $S_{\hat{%
\alpha}}\left( x\right) =\frac{\chi \left( \hat{\alpha}\right) !}{\hat{\alpha%
}!}\lambda ^{-\omega \left( \hat{\alpha}\right) }x^{\omega \left( \hat{\alpha%
}\right) }$ $\left( x\in \mathbb{R}^{n}\right) $.
\end{itemize}

We have

\begin{itemize}
\item[\refstepcounter{equation}\text{(\theequation)}\label{pb51}] $S_{\hat{%
\alpha}}\circ T\left( x\right) =\frac{\chi \left( \hat{\alpha}\right) !}{%
\hat{\alpha}!}x^{\omega \left( \hat{\alpha}\right) }$,
\end{itemize}
hence
\begin{itemize}
\item[\refstepcounter{equation}\text{(\theequation)}\label{pb52}] $\partial
^{\beta }\left( S_{\hat{\alpha}}\circ T\right) \left( 0\right) =\frac{\chi
\left( \hat{\alpha}\right) !\omega \left( \hat{\alpha}\right) !}{\hat{\alpha}%
!}\delta _{\beta \omega \left( \hat{\alpha}\right) }$ for $\hat{\alpha}\in 
\hat{\mathcal{A}},\beta \in \mathcal{M}$.
\end{itemize}

We study the derivatives $\partial ^{\beta }\left( \left[ \left( S_{\hat{%
\alpha}}\bar{P}_{\chi \left( \hat{\alpha}\right) }\right) \circ T\right]
\left( 0\right) \right) $ for $\hat{\alpha}\in \hat{\mathcal{A}},\beta \in 
\mathcal{M}$.

Case 1: If $\beta$ is not of the form $\beta = \omega(\hat{\alpha})+\tilde{%
\beta}$ for some $\tilde{\beta}\in \mathcal{M}$, then \eqref{pb52} gives

\begin{itemize}
\item[\refstepcounter{equation}\text{(\theequation)}\label{pb53}] $\partial
^{\beta }\left[ \left( S_{\hat{\alpha}}\bar{P}_{\chi \left( \hat{\alpha}%
\right) }\right) \circ T\right] \left( 0\right) =0$.
\end{itemize}

Case 2: Suppose $\beta =\omega (\hat{\alpha})+\tilde{\beta}$ for some $%
\tilde{\beta}\in \mathcal{M}$. Then \eqref{pb52} gives%
\begin{eqnarray*}
&&\partial ^{\beta }\left[ \left( S_{\hat{\alpha}}\bar{P}_{\chi \left( \hat{%
\alpha}\right) }\right) \circ T\right] \left( 0\right)  \\
&=&\frac{\beta !}{\omega \left( \hat{\alpha}\right) !\tilde{\beta}!}\left[
\partial ^{\omega \left( \hat{\alpha}\right) }\left( S_{\hat{\alpha}}\circ
T\right) \left( 0\right) \right] \cdot \left[ \partial ^{\bar{\beta}}\left( 
\bar{P}_{\chi \left( \hat{\alpha}\right) }\circ T\right) \left( 0\right) %
\right]  \\
&=&\frac{\beta !}{\omega \left( \hat{\alpha}\right) !\tilde{\beta}!}\frac{%
\chi \left( \hat{\alpha}\right) !\omega \left( \hat{\alpha}\right) !}{\hat{%
\alpha}!}\cdot \left[ \partial ^{\bar{\beta}}\left( \bar{P}_{\chi \left( 
\hat{\alpha}\right) }\circ T\right) \left( 0\right) \right]  \\
&=&\frac{\beta !\chi \left( \hat{\alpha}\right) !}{\hat{\alpha}!\tilde{\beta}%
!}\left[ \partial ^{\bar{\beta}}\left( \bar{P}_{\chi \left( \hat{\alpha}%
\right) }\circ T\right) \left( 0\right) \right] \text{.}
\end{eqnarray*}%
Hence, by \eqref{pb43}, we have 
\begin{equation}
\left\vert \frac{\hat{\alpha}!\tilde{\beta}!}{\beta !\chi \left( \hat{\alpha}%
\right) !}\delta ^{\left\vert \tilde{\beta}\right\vert -\left\vert \chi
\left( \hat{\alpha}\right) \right\vert }\partial ^{\beta }\left[ \left( S_{%
\hat{\alpha}}\cdot \bar{P}_{\chi \left( \hat{\alpha}\right) }\right) \circ T%
\right] \left( 0\right) -\delta _{\tilde{\beta}\chi \left( \hat{\alpha}%
\right) }\right\vert \leq a\text{.}  \label{pb54}
\end{equation}%
Since in this case $\beta =\omega \left( \hat{\alpha}\right) +\tilde{\beta}$
and $\hat{\alpha}=\omega \left( \hat{\alpha}\right) +\chi \left( \hat{\alpha}%
\right) $ (see \eqref{pb46}), we have $\delta _{\tilde{\beta}\chi \left( 
\hat{\alpha}\right) }=\delta _{\beta \hat{\alpha}}$, $|\tilde{\beta}%
|-\left\vert \chi \left( \hat{\alpha}\right) \right\vert =\left\vert \beta
\right\vert -\left\vert \hat{\alpha}\right\vert $, and $\frac{\hat{\alpha}!%
\tilde{\beta}!}{\beta !\chi \left( \hat{\alpha}\right) !}=1$ if $\beta =\hat{%
\alpha}$.

Hence, \eqref{pb54} implies that 
\begin{equation}
\left\vert \delta ^{\left\vert \beta \right\vert -\left\vert \hat{\alpha}%
\right\vert }\partial ^{\beta }\left[ \left( S_{\hat{\alpha}}\cdot \bar{P}%
_{\chi \left( \hat{\alpha}\right) }\right) \circ T\right] \left( 0\right)
-\delta _{\beta \hat{\alpha}}\right\vert \leq Ca  \label{pb55}
\end{equation}%
in Case 2.

Thanks to \eqref{pb46} and \eqref{pb53}, estimate \eqref{pb55} holds also in
Case 1.

Thus, \eqref{pb55} holds for all $\hat{\alpha}\in \hat{\mathcal{A}}$, $\beta
\in \mathcal{M}$. Consequently, 
\begin{equation}
\left\vert \delta ^{\left\vert \beta \right\vert -\left\vert \hat{\alpha}%
\right\vert }\partial ^{\beta }\left( \left[ S_{\hat{\alpha}}\odot _{0}\bar{P%
}_{\chi \left( \hat{\alpha}\right) }\right] \circ T\left( 0\right) \right)
-\delta _{\beta \hat{\alpha}}\right\vert \leq Ca  \label{pb56}
\end{equation}%
for all $\hat{\alpha}\in \hat{\mathcal{A}}$, $\beta \in \mathcal{M}$.

We prepare to apply Lemma \ref{lemma-pb1}, with $\hat{S}=\delta
^{-\left\vert \omega \left( \hat{\alpha}\right) \right\vert }S_{\hat{\alpha}%
} $ and $\hat{P}=M_{0}\delta ^{m-\left\vert \chi \left( \hat{\alpha}\right)
\right\vert }\bar{P}_{\chi \left( \hat{\alpha}\right) }$. From \eqref{pb30}
and \eqref{pb50}, we have 
\begin{equation}
\left\vert \partial ^{\beta }\left( \delta ^{-\left\vert \omega \left( \hat{%
\alpha}\right) \right\vert }S_{\hat{\alpha}}\right) \left( 0\right)
\right\vert \leq C\left( a\right) \delta ^{-\left\vert \beta \right\vert }%
\text{ for }\hat{\alpha}\in \hat{\mathcal{A}}\text{, }\beta \in \mathcal{M}.
\label{pb57}
\end{equation}

From \eqref{pb43} we have 
\begin{equation*}
\left\vert \partial ^{\beta }\left( \bar{P}_{\chi \left( \hat{\alpha}\right)
}\circ T\right) \left( 0\right) \right\vert \leq C\delta ^{\left\vert \chi
\left( \hat{\alpha}\right) \right\vert -\left\vert \beta \right\vert }\text{
for }\hat{\alpha}\in \hat{\mathcal{A}}\text{, }\beta \in \mathcal{M}\text{;}
\end{equation*}%
hence, by \eqref{pb30}, 
\begin{equation}
\left\vert \partial ^{\beta }\left( M_{0}\delta ^{m-\left\vert \chi \left( 
\hat{\alpha}\right) \right\vert }\bar{P}_{\chi \left( \hat{\alpha}\right)
}\right) \left( 0\right) \right\vert \leq C\left( a\right) M_{0}\delta
^{m-\left\vert \beta \right\vert }\text{ for }\hat{\alpha}\in \hat{\mathcal{A%
}}\text{, }\beta \in \mathcal{M}\text{.}  \label{pb58}
\end{equation}%
Also, \eqref{pb44} gives 
\begin{equation}
P^{0}\pm c\left( a\right) M_{0}\delta ^{m-\left\vert \chi \left( \hat{\alpha}%
\right) \right\vert }\bar{P}_{\chi \left( \hat{\alpha}\right) }\in \Gamma
\left( 0,CM_{0}\right) \text{ for }\hat{\alpha}\in \hat{\mathcal{A}}\text{.}
\label{pb59}
\end{equation}

Our results \eqref{pb57}, \eqref{pb58}, \eqref{pb59} are the hypotheses of
Lemma \ref{lemma-pb1} for $\hat{S}$, $\hat{P}$ as given above. Applying that
lemma, we learn that 
\begin{equation*}
P^{0}\pm c\left( a\right) \left( \delta ^{-\left\vert \omega \left( \hat{%
\alpha}\right) \right\vert }S_{\hat{\alpha}}\right) \odot _{0}\left(
M_{0}\delta ^{m-\left\vert \chi \left( \hat{\alpha}\right) \right\vert }\bar{%
P}_{\chi \left( \hat{\alpha}\right) }\right) \in \Gamma \left( 0,C\left(
a\right) M_{0}\right) \text{ for }\hat{\alpha}\in \hat{\mathcal{A}}\text{.}
\end{equation*}%
Recalling \eqref{pb46}, we conclude that 
\begin{equation}
P^{0}\pm c\left( a\right) M_{0}\delta ^{m-\left\vert \hat{\alpha}\right\vert
}S_{\hat{\alpha}}\odot _{0}\bar{P}_{\chi \left( \hat{\alpha}\right) }\in
\Gamma \left( 0,C\left( a\right) M_{0}\right) \text{ for }\hat{\alpha}\in 
\hat{\mathcal{A}}\text{.}  \label{pb60}
\end{equation}

Next, \eqref{pb56} and the \emph{small }$a$ \emph{condition} tell us that
there exists a matrix of real numbers $\left( b_{\gamma \hat{\alpha}}\right)
_{\gamma ,\hat{\alpha}\in \hat{\mathcal{A}}}$ satisfying 
\begin{equation}
\sum_{\hat{\alpha}\in \hat{\mathcal{A}}}b_{\gamma \hat{\alpha}}\partial
^{\beta }\left( \left[ S_{\hat{\alpha}}\odot _{0}\bar{P}_{\chi \left( \hat{%
\alpha}\right) }\right] \circ T\right) \left( 0\right) \cdot \delta
^{\left\vert \beta \right\vert -\left\vert \hat{\alpha}\right\vert }=\delta
_{\gamma \beta }\text{ for }\gamma ,\beta \in \hat{\mathcal{A}}  \label{pb61}
\end{equation}%
and 
\begin{equation}
\left\vert b_{\gamma \hat{\alpha}}\right\vert \leq 2\text{, for all }\gamma ,%
\hat{\alpha}\in \hat{\mathcal{A}}\text{.}  \label{pb62}
\end{equation}%
From \eqref{pb61} we have 
\begin{equation}
\partial ^{\beta }\left\{ \sum_{\hat{\alpha}\in \hat{\mathcal{A}}}b_{\gamma 
\hat{\alpha}}\delta ^{\left\vert \gamma \right\vert -\left\vert \hat{\alpha}%
\right\vert }\left( \left[ S_{\hat{\alpha}}\odot _{0}\bar{P}_{\chi \left( 
\hat{\alpha}\right) }\right] \circ T\right) \right\} \left( 0\right) =\delta
_{\gamma \beta }\text{ for }\gamma ,\beta \in \hat{\mathcal{A}}\text{.}
\label{pb63}
\end{equation}%
Also, \eqref{pb56} and \eqref{pb62} imply that 
\begin{equation}
\left\vert \partial ^{\beta }\left\{ \sum_{\hat{\alpha}\in \hat{\mathcal{A}}%
}b_{\gamma \hat{\alpha}}\delta ^{\left\vert \gamma \right\vert -\left\vert 
\hat{\alpha}\right\vert }\left( \left[ S_{\hat{\alpha}}\odot _{0}\bar{P}%
_{\chi \left( \hat{\alpha}\right) }\right] \circ T\right) \right\} \left(
0\right) \right\vert \leq C\delta ^{\left\vert \gamma \right\vert
-\left\vert \beta \right\vert }\text{ for }\gamma \in \hat{\mathcal{A}}%
,\beta \in \mathcal{M}\text{.}  \label{pb64}
\end{equation}%
Since 
\begin{eqnarray*}
&&M_{0}\delta ^{m-\left\vert \gamma \right\vert }\left\{ \sum_{\hat{\alpha}%
\in \hat{\mathcal{A}}}b_{\gamma \hat{\alpha}}\delta ^{\left\vert \gamma
\right\vert -\left\vert \hat{\alpha}\right\vert }\left( \left[ S_{\hat{\alpha%
}}\odot _{0}\bar{P}_{\chi \left( \hat{\alpha}\right) }\right] \right)
\right\}  \\
&=&\sum_{\hat{\alpha}\in \hat{\mathcal{A}}}b_{\gamma \hat{\alpha}}\cdot
\left\{ M_{0}\delta ^{m-\left\vert \hat{\alpha}\right\vert }\cdot \left[ S_{%
\hat{\alpha}}\odot _{0}\bar{P}_{\chi \left( \hat{\alpha}\right) }\right]
\right\} \text{,}
\end{eqnarray*}%
we learn from \eqref{pb60}, \eqref{pb62}, and the Trivial Remark on Convex
Sets in Section \ref{notation-and-preliminaries}, that 
\begin{equation}
P^{0}\pm c\left( a\right) M_{0}\delta ^{m-\left\vert \gamma \right\vert
}\left\{ \sum_{\hat{\alpha}\in \hat{\mathcal{A}}}b_{\gamma \hat{\alpha}%
}\delta ^{\left\vert \gamma \right\vert -\left\vert \hat{\alpha}\right\vert
}\left( \left[ S_{\hat{\alpha}}\odot _{0}\bar{P}_{\chi \left( \hat{\alpha}%
\right) }\right] \right) \right\} \in \Gamma \left( 0,C\left( a\right)
M_{0}\right)   \label{pb65}
\end{equation}%
for $\gamma \in \hat{\mathcal{A}}$.

We define%
\begin{equation}
P_{\gamma }=\lambda ^{\gamma }\sum_{\hat{\alpha}\in \hat{\mathcal{A}}%
}b_{\gamma \hat{\alpha}}\delta ^{\left\vert \gamma \right\vert -\left\vert 
\hat{\alpha}\right\vert }\left[ S_{\hat{\alpha}}\odot _{0}\bar{P}_{\chi
\left( \hat{\alpha}\right) }\right] \text{ for }\gamma \in \hat{\mathcal{A}}%
\text{.}  \label{pb66}
\end{equation}

From \eqref{pb30}, we have $\left\vert \lambda ^{\gamma }\right\vert \leq
C\left( a\right) $. Recall from \eqref{pb26} that $P^{0}\in \Gamma \left(
0,CM_{0}\right) $. Hence, we deduce from \eqref{pb65} (and from the
convexity of $\Gamma \left( 0,C\left( a\right) M_{0}\right) $) that

\begin{itemize}
\item[\refstepcounter{equation}\text{(\theequation)}\label{pb67}] $%
P^{0},P^{0}+ c\left( a\right) M_{0}\delta ^{m-\left\vert \gamma
\right\vert }P_{\gamma },P^{0}-c\left( a\right) M_{0}\delta ^{m-\left\vert
\gamma \right\vert }P_{\gamma }\in \Gamma \left( 0,C\left( a\right)
M_{0}\right) $ for $\gamma \in \hat{\mathcal{A}}$.
\end{itemize}

Also, \eqref{pb63} and \eqref{pb66} give 
\begin{equation}
\partial ^{\beta }P_{\gamma }\left( 0\right) =\delta _{\beta \gamma }\text{
for }\beta ,\gamma \in \hat{\mathcal{A}}\text{.}  \label{pb68}
\end{equation}

From \eqref{pb30}, \eqref{pb64}, \eqref{pb66}, we have 
\begin{equation}
\left\vert \partial ^{\beta }P_{\gamma }\left( 0\right) \right\vert \leq
C\left( a\right) \delta ^{\left\vert \gamma \right\vert -\left\vert \beta
\right\vert }\text{ for }\gamma \in \hat{\mathcal{A}}\text{, }\beta \in 
\mathcal{M}\text{.}  \label{pb69}
\end{equation}

Our results \eqref{pb67}, \eqref{pb68}, \eqref{pb69} show that

\begin{itemize}
\item[\refstepcounter{equation}\text{(\theequation)}\label{pb70}] $\left(
P_{\gamma }\right) _{\gamma \in \hat{\mathcal{A}}}$ is an $\left( \hat{%
\mathcal{A}},\delta ,C\left( a\right) \right) $-basis for $\vec{\Gamma}$ at $%
\left( 0,M_{0},P^{0}\right) $.
\end{itemize}

We now pick $a$ to to be a constant determined by $C_{B}$, $C_{w}$, $m$, $n$%
, small enough to satisfy our \emph{small }$a$ \emph{condition}. Then %
\eqref{pb47}, \eqref{pb49}, and \eqref{pb70} immediately imply the
conclusions of Lemma \ref{lemma-pb2}.

The proof of that lemma is complete.
\end{proof}

The next result is a consequence of the Relabeling Lemma (Lemma \ref%
{lemma-pb2}).

\begin{lemma}[Control $\Gamma $ Using Basis]
\label{lemma-pb3} Let $\vec{\Gamma}=\left( \Gamma \left( x,M\right) \right)
_{x\in E,M>0}$ be a $\left( C_{w},\delta _{\max }\right) $-convex shape
field. Let $x_{0}\in E$, $M_{0}>0$, $0<\delta \leq \delta _{\max }$, $C_{B}>0
$, $\mathcal{A\subseteq M}$, and let $P$, $P^{0}\in \mathcal{P}$. Suppose $\vec{%
\Gamma}$ has an $\left( \mathcal{A},\delta ,C_{B}\right) $-basis at $\left(
x_{0},M_{0},P^{0}\right) $. Suppose also that

\begin{itemize}
\item[\refstepcounter{equation}\text{(\theequation)}\label{pb71}] $P\in
\Gamma \left( x_{0},C_{B}M_{0}\right) $,
\end{itemize}

\begin{itemize}
\item[\refstepcounter{equation}\text{(\theequation)}\label{pb72}] $\partial
^{\beta }\left( P-P^{0}\right) \left( x_{0}\right) =0$ for all $\beta \in 
\mathcal{A}$, and
\end{itemize}

\begin{itemize}
\item[\refstepcounter{equation}\text{(\theequation)}\label{pb73}] $%
\max_{\beta \in \mathcal{M}}\delta ^{\left\vert \beta \right\vert
}\left\vert \partial ^{\beta }\left( P-P^{0}\right) \left( x_{0}\right)
\right\vert \geq M_{0}\delta ^{m}$.
\end{itemize}

Then there exist $\hat{\mathcal{A}} \subseteq \mathcal{M}$ and $\hat{P}^0
\in \mathcal{P}$ with the following properties.

\begin{itemize}
\item[\refstepcounter{equation}\text{(\theequation)}\label{pb74}] $\hat{%
\mathcal{A}}$ is monotonic.
\end{itemize}

\begin{itemize}
\item[\refstepcounter{equation}\text{(\theequation)}\label{pb75}] $\hat{%
\mathcal{A}} < \mathcal{A}$ (strict inequality).
\end{itemize}

\begin{itemize}
\item[\refstepcounter{equation}\text{(\theequation)}\label{pb76}] $\vec{%
\Gamma}$ has an $(\hat{\mathcal{A}},\delta,C_B^{\prime })$-basis at $%
(x_0,M_0,\hat{P}^0)$, with $C_B^{\prime }$ determined by $C_B$, $C_w$, $m$, $%
n$.
\end{itemize}

\begin{itemize}
\item[\refstepcounter{equation}\text{(\theequation)}\label{pb77}] $\partial
^{\beta }\left( \hat{P}^{0}-P^{0}\right) \left( x_{0}\right) =0$ for all $%
\beta \in \mathcal{A}$.
\end{itemize}

\begin{itemize}
\item[\refstepcounter{equation}\text{(\theequation)}\label{pb78}] $%
\left\vert \partial ^{\beta }\left( \hat{P}^{0}-P^{0}\right) \left(
x_{0}\right) \right\vert \leq M_{0}\delta ^{m-\left\vert \beta \right\vert }$
for all $\beta \in \mathcal{M}$.
\end{itemize}
\end{lemma}

\begin{proof}
We write $c$, $C$, $C^{\prime }$, etc., to denote constants determined by $%
C_B$, $C_w$, $m$, $n$. These symbols may denote different constants in
different occurrences.

Let $(P_\alpha)_{\alpha \in \mathcal{A}}$ be an $(\mathcal{A},\delta,C_B)$%
-basis for $\vec{\Gamma}$ at $(x_0,M_0,P^0)$. By definition,

\begin{itemize}
\item[\refstepcounter{equation}\text{(\theequation)}\label{pb79}] $P^{0}\in
\Gamma (x_{0},C_{B}M_{0})$,
\end{itemize}

\begin{itemize}
\item[\refstepcounter{equation}\text{(\theequation)}\label{pb80}] $P^{0}\pm
cM_{0}\delta ^{m-\left\vert \alpha \right\vert }P_{\alpha }\in \Gamma \left(
x_{0},C_{B}M\right) $ for all $\alpha \in \mathcal{A}$,
\end{itemize}

\begin{itemize}
\item[\refstepcounter{equation}\text{(\theequation)}\label{pb81}] $\partial
^{\beta }P_{\alpha }\left( x_{0}\right) =\delta _{\beta \alpha }$ for $\beta
,\alpha \in \mathcal{A}$,
\end{itemize}

\begin{itemize}
\item[\refstepcounter{equation}\text{(\theequation)}\label{pb82}] $%
\left\vert \partial ^{\beta }P_{\alpha }\left( x_{0}\right) \right\vert \leq
C\delta ^{\left\vert \alpha \right\vert -\left\vert \beta \right\vert }$ for
all $\alpha \in A,\beta \in \mathcal{M}$.
\end{itemize}

Thanks to \eqref{pb79}, we may replace $P$ by a convex combination of $P^0$
and $P$ in \eqref{pb71}, \eqref{pb72}, \eqref{pb73} to retain \eqref{pb71}, %
\eqref{pb72} but replace \eqref{pb73} by the stronger assertion

\begin{itemize}
\item[\refstepcounter{equation}\text{(\theequation)}\label{pb83}] $%
\max_{\beta \in \mathcal{M}}\delta ^{\left\vert \beta \right\vert
}\left\vert \partial ^{\beta }\left( P-P^{0}\right) \left( x_{0}\right)
\right\vert =M_{0}\delta ^{m}$.
\end{itemize}

We pick $\gamma \in \mathcal{M}$ to achieve the above $\max$. Thanks to %
\eqref{pb72}, we have

\begin{itemize}
\item[\refstepcounter{equation}\text{(\theequation)}\label{pb84}] $\gamma
\not\in \mathcal{A}$,
\end{itemize}

hence

\begin{itemize}
\item[\refstepcounter{equation}\text{(\theequation)}\label{pb85}] $\mathcal{A%
} \cup \{ \gamma \} < \mathcal{A}$ (strict inequality).
\end{itemize}

We set

\begin{itemize}
\item[\refstepcounter{equation}\text{(\theequation)}\label{pb86}] $\hat{P}^0=%
\frac{1}{2}(P^0+P)$.
\end{itemize}

From \eqref{pb71}, \eqref{pb79}, \eqref{pb80}, we have

\begin{itemize}
\item[\refstepcounter{equation}\text{(\theequation)}\label{pb87}] $\hat{P}%
^{0},\hat{P}^{0}\pm c^{\prime }M_{0}\delta ^{m-\left\vert \alpha \right\vert
}P_{\alpha }\in \Gamma \left( x_{0},CM_{0}\right) $ for $\alpha \in \mathcal{%
A}$.
\end{itemize}

Also,

\begin{itemize}
\item[\refstepcounter{equation}\text{(\theequation)}\label{pb88}] $\hat{P}%
^{0}\pm \frac{1}{2}\left( P-P^{0}\right) \in \Gamma \left( x_{0},CM\right) $
\end{itemize}

since $\hat{P}^{0}+\frac{1}{2}\left( P-P^{0}\right) =P$ and $\hat{P}^{0}-%
\frac{1}{2}\left( P-P^{0}\right) =P^{0}$.

From \eqref{pb72} and \eqref{pb83}, we have

\begin{itemize}
\item[\refstepcounter{equation}\text{(\theequation)}\label{pb89}] $\partial
^{\beta }\left( \hat{P}^{0}-P^{0}\right) \left( x_{0}\right) =0$ for all $%
\beta \in \mathcal{A}$,
\end{itemize}

and

\begin{itemize}
\item[\refstepcounter{equation}\text{(\theequation)}\label{pb90}] $%
\left\vert \partial ^{\beta }\left( \hat{P}^{0}-P^{0}\right) \left(
x_{0}\right) \right\vert \leq M_{0}\delta ^{m-\left\vert \beta \right\vert }$
for all $\beta \in \mathcal{M}$.
\end{itemize}

We define

\begin{itemize}
\item[\refstepcounter{equation}\text{(\theequation)}\label{pb91}] $P_{\gamma
}^{\#}=\left[ \partial ^{\gamma }\left( P-P^{0}\right) \left( x_{0}\right) %
\right] ^{-1}\cdot \left( P-P^{0}\right) $.
\end{itemize}

We are not dividing by zero here; by \eqref{pb83} and the definition of $%
\gamma$, we have

\begin{itemize}
\item[\refstepcounter{equation}\text{(\theequation)}\label{pb92}] $%
\left\vert \partial ^{\gamma }\left( P-P^{0}\right) \left( x_{0}\right)
\right\vert ^{-1}=M_{0}^{-1}\delta ^{\left\vert \gamma \right\vert -m}$.
\end{itemize}

From \eqref{pb72}, \eqref{pb84}, \eqref{pb91}, we have

\begin{itemize}
\item[\refstepcounter{equation}\text{(\theequation)}\label{pb93}] $\partial
^{\beta }P_{\gamma }^{\#}\left( x_{0}\right) =\delta _{\beta \alpha }$ for
all $\beta \in \mathcal{A}\cup \{\gamma \}$.
\end{itemize}

Also, \eqref{pb83}, \eqref{pb91}, \eqref{pb92} give

\begin{itemize}
\item[\refstepcounter{equation}\text{(\theequation)}\label{pb94}] $%
\left\vert \partial ^{\beta }P_{\gamma }^{\#}\left( x_{0}\right) \right\vert
\leq M_{0}^{-1}\delta ^{\left\vert \gamma \right\vert -m}\cdot M_{0}\delta
^{m-\left\vert \beta \right\vert }=\delta ^{\left\vert \gamma \right\vert
-\left\vert \beta \right\vert }$ for all $\beta \in \mathcal{M}$.
\end{itemize}

From \eqref{pb91}, \eqref{pb92}, we have $P-P^0 = \sigma M_0
\delta^{m-|\gamma|}P_\gamma^\#$ for $\sigma=1$ or $\sigma =-1$. Therefore, %
\eqref{pb88} implies that

\begin{itemize}
\item[\refstepcounter{equation}\text{(\theequation)}\label{pb95}] $\hat{P}%
^{0}\pm cM_{0}\delta ^{m-\left\vert \gamma \right\vert }P_{\gamma }^{\#}\in
\Gamma \left( x_{0},CM_{0}\right) $.
\end{itemize}

From \eqref{pb87}, \eqref{pb95} and the Trivial Remark on Convex Sets in
Section \ref{notation-and-preliminaries}, we conclude that

\begin{itemize}
\item[\refstepcounter{equation}\text{(\theequation)}\label{pb96}] $\hat{P}%
^{0}+sM_{0}\delta ^{m-\left\vert \gamma \right\vert }P_{\gamma
}^{\#}+\sum_{\alpha \in \mathcal{A}}t_{\alpha }\cdot M_{0}\delta
^{m-\left\vert \alpha \right\vert }P_{\alpha }\in \Gamma \left(
x_{0},CM_{0}\right) $, 
\end{itemize}
whenever $\left\vert s\right\vert $, $\left\vert
t_{\alpha }\right\vert \leq c$ (all $\alpha \in \mathcal{A}$) for a small
enough $c$.

For $\alpha \in \mathcal{A}$, we define

\begin{itemize}
\item[\refstepcounter{equation}\text{(\theequation)}\label{pb97}] $P_{\alpha
}^{\#}=P_{\alpha }-\left[ \partial ^{\gamma }P_{\alpha }\left( x_{0}\right) %
\right] \cdot P_{\gamma }^{\#}$.
\end{itemize}

Fix $\alpha \in \mathcal{A}$. If $\beta \in \mathcal{A}$, then \eqref{pb81}, %
\eqref{pb84}, \eqref{pb93} imply%
\begin{equation*}
\partial ^{\beta }P_{\alpha }^{\#}\left( x_{0}\right) =\partial ^{\beta
}P_{\alpha }\left( x_{0}\right) -\left[ \partial ^{\gamma }P_{\alpha }\left(
x_{0}\right) \right] \cdot \partial ^{\beta }P_{\gamma }^{\#}\left(
x_{0}\right) =\delta _{\beta \alpha }\text{.}
\end{equation*}

On the other hand, \eqref{pb84} and \eqref{pb93} yield%
\begin{equation*}
\partial ^{\gamma }P_{\alpha }^{\#}\left( x_{0}\right) =\partial ^{\gamma
}P_{\alpha }\left( x_{0}\right) -\left[ \partial ^{\gamma }P_{\alpha }\left(
x_{0}\right) \right] \cdot \partial ^{\gamma }P_{\gamma }^{\#}\left(
x_{0}\right) =0=\delta _{\gamma \alpha }\text{.}
\end{equation*}
Thus, 
\begin{equation*}
\partial ^{\beta }P_{\alpha }^{\#}\left( x_{0}\right) =\delta _{\beta \alpha
}\text{ for }\beta \in \mathcal{A\cup }\left\{ \gamma \right\} \text{, }%
\alpha \in \mathcal{A}\text{.}
\end{equation*}%
Together with \eqref{pb93}, this tells us that

\begin{itemize}
\item[\refstepcounter{equation}\text{(\theequation)}\label{pb98}] $\partial
^{\beta }P_{\alpha }^{\#}\left( x_{0}\right) =\delta _{\beta \alpha }$ for $%
\beta ,\alpha \in \mathcal{A}\cup \{\gamma \}$.
\end{itemize}

Next, we learn from \eqref{pb82}, \eqref{pb94}, \eqref{pb97} that 
\begin{eqnarray*}
\left\vert \partial ^{\beta }P_{\alpha }^{\#}\left( x_{0}\right) \right\vert
&\leq &\left\vert \partial ^{\beta }P_{\alpha }\left( x_{0}\right)
\right\vert +\left\vert \partial ^{\gamma }P_{\alpha }\left( x_{0}\right)
\right\vert \cdot \left\vert \partial ^{\beta }P_{\gamma }^{\#}\left(
x_{0}\right) \right\vert  \\
&\leq &C\delta ^{\left\vert \alpha \right\vert -\left\vert \beta \right\vert
}+C\delta ^{\left\vert \alpha \right\vert -\left\vert \gamma \right\vert
}\cdot \delta ^{\left\vert \gamma \right\vert -\left\vert \beta \right\vert }
\\
&\leq &C^{\prime }\delta ^{\left\vert \alpha \right\vert -\left\vert \beta
\right\vert }\text{ for }\alpha \in \mathcal{A},\beta \in \mathcal{M}\text{.}
\end{eqnarray*}

Together with \eqref{pb94}, this tells us that

\begin{itemize}
\item[\refstepcounter{equation}\text{(\theequation)}\label{pb99}] $%
\left\vert \partial ^{\beta }P_{\alpha }^{\#}\left( x_{0}\right) \right\vert
\leq C\delta ^{\left\vert \alpha \right\vert -\left\vert \beta \right\vert }$
for all $\alpha \in \mathcal{A}\cup \{\gamma \},\beta \in \mathcal{M}$.
\end{itemize}

Next, note that for $\alpha \in \mathcal{A}$, we have%
\begin{equation*}
M_{0}\delta ^{m-\left\vert \alpha \right\vert }P_{\alpha }^{\#}=M_{0}\delta
^{m-\left\vert \alpha \right\vert }P_{\alpha }-\left[ \delta ^{\left\vert
\gamma \right\vert -\left\vert \alpha \right\vert }\partial ^{\gamma
}P_{\alpha }\left( x_{0}\right) \right] \cdot M_{0}\delta ^{m-\left\vert
\gamma \right\vert }P_{\gamma }^{\#}\text{,}
\end{equation*}%
with $\left\vert \left[ \delta ^{\left\vert \gamma \right\vert -\left\vert
\alpha \right\vert }\partial ^{\gamma }P_{\alpha }\left( x_{0}\right) \right]
\right\vert \leq C$ by \eqref{pb82}.

Therefore, \eqref{pb96} shows that 
\begin{equation*}
\hat{P}^{0}\pm cM_{0}\delta ^{m-\left\vert \alpha \right\vert }P_{\alpha
}^{\#}\in \Gamma \left( x_{0},CM_{0}\right) \text{ for }\alpha \in \mathcal{A%
}\text{,}
\end{equation*}%
provided we take $c$ small enough. Together with \eqref{pb95}, this yields

\begin{itemize}
\item[\refstepcounter{equation}\text{(\theequation)}\label{pb100}] $\hat{P}%
^{0}\pm cM_{0}\delta ^{m-\left\vert \alpha \right\vert }P_{\alpha }^{\#}\in
\Gamma \left( x_{0},CM_{0}\right) $ for all $\alpha \in \mathcal{A}\cup
\{\gamma \}$.
\end{itemize}

Our results \eqref{pb98}, \eqref{pb99}, \eqref{pb100} tell us
that $\left( P_{\alpha }^{\#}\right) _{\alpha \in \mathcal{A\cup }\left\{
\gamma \right\} }$ is an $\left( \mathcal{A\cup }\left\{ \gamma \right\}
,\delta ,C\right) $-basis \ for $\vec{\Gamma}$ at $\left( x_{0},M_{0},\hat{P}%
^{0}\right) $.

Consequently, the Relabeling Lemma (Lemma \ref{lemma-pb2}) produces a set $%
\hat{\mathcal{A}}\subseteq \mathcal{M}$ with the following properties.

\begin{itemize}
\item[\refstepcounter{equation}\text{(\theequation)}\label{pb101}] $\hat{%
\mathcal{A}}$ is monotonic.
\end{itemize}

\begin{itemize}
\item[\refstepcounter{equation}\text{(\theequation)}\label{pb102}] $\hat{%
\mathcal{A}} \leq \mathcal{A} \cup \{ \gamma\}< \mathcal{A}$, see %
\eqref{pb85}.
\end{itemize}

\begin{itemize}
\item[\refstepcounter{equation}\text{(\theequation)}\label{pb103}] $\vec{%
\Gamma}$ has an $(\hat{\mathcal{A}},\delta,C^{\prime })$-basis at $(x_0,M_0,%
\hat{P}^0)$.
\end{itemize}

Our results \eqref{pb101}, \eqref{pb102}, \eqref{pb103}, \eqref{pb89}, %
\eqref{pb90} are the conclusions \eqref{pb74}$\cdots $\eqref{pb78} of Lemma %
\ref{lemma-pb3}.

The proof of that lemma is complete.
\end{proof}

\section{The Transport Lemma}

\label{transport-lemma}

In this section, we prove the following result.

\begin{lemma}[Transport Lemma]
\label{lemma-transport}Let $\vec{\Gamma}_{0}=\left( \Gamma _{0}\left(
x,M\right) \right) _{x\in E,M>0}$ be a shape field. For $l\geq 1$, let $\vec{%
\Gamma}_{l}=\left( \Gamma _{l}\left( x,M\right) \right) _{x\in E,M>0}$ be
the $l$-th refinement of $\vec{\Gamma}_{0}$.

\begin{itemize}
\item[\refstepcounter{equation}\text{(\theequation)}\label{t1}] Suppose $%
\mathcal{A\subseteq M}$ is monotonic and $\hat{\mathcal{A}} \subseteq 
\mathcal{M}$ (not necessarily monotonic).
\end{itemize}

Let $x_{0}\in E$, $M_{0}>0$, $l_{0}\geq 1$, $\delta >0$, $C_{B}$, $\hat{C}%
_{B}$, $C_{DIFF}>0$. Let $P^{0}$, $\hat{P}^{0}\in \mathcal{P}$.
Assume that the following hold.

\begin{itemize}
\item[\refstepcounter{equation}\text{(\theequation)}\label{t2}] $\vec{\Gamma}%
_{l_{0}}$ has an $\left( \mathcal{A},\delta ,C_{B}\right) $-basis at $\left(
x_{0},M_{0},P^{0}\right) $, and an $\left(\hat{\mathcal{A}} ,\delta ,\hat{C}%
_{B}\right) $-basis at $\left(x_{0}, M_{0},\hat{P}^{0}\right) $.
\end{itemize}

\begin{itemize}
\item[\refstepcounter{equation}\text{(\theequation)}\label{t3}] $%
\partial^\beta(P^0-\hat{P}^0)\equiv 0$ for $\beta \in \mathcal{A}$.
\end{itemize}

\begin{itemize}
\item[\refstepcounter{equation}\text{(\theequation)}\label{t4}] $%
|\partial^\beta (P^0 - \hat{P}^0)(x_0)|\leq C_{DIFF}M_0\delta^{m-|\beta|}$
for $\beta \in \mathcal{M}$.
\end{itemize}

Let $y_0 \in E$, and suppose that

\begin{itemize}
\item[\refstepcounter{equation}\text{(\theequation)}\label{t5}] $%
|x_0-y_0|\leq \epsilon_0\delta $,
\end{itemize}
where $\epsilon_0$ is a a small enough constant determined by $C_B$, $\hat{C}%
_B$, $C_{DIFF}$, $m$, $n$. Then there exists $\hat{P}^\# \in \mathcal{P}$
with the following properties.

\begin{itemize}
\item[\refstepcounter{equation}\text{(\theequation)}\label{t6}] $\vec{\Gamma}%
_{l_0-1}$ has both an $(\mathcal{A},\delta,C_B^{\prime })$-basis and an $(%
\hat{\mathcal{A}},\delta,C_B^{\prime })$-basis at $(y_0,M_0,\hat{P}^\#)$,
with $C_B^{\prime }$ determined by $C_B$, $\hat{C}_B$, $C_{DIFF}$, $m$, $n$.
\end{itemize}

\begin{itemize}
\item[\refstepcounter{equation}\text{(\theequation)}\label{t7}] $%
\partial^\beta(\hat{P}^\#-P^0)\equiv0$ for $\beta \in \mathcal{A}$.
\end{itemize}

\begin{itemize}
\item[\refstepcounter{equation}\text{(\theequation)}\label{t8}] $%
|\partial^\beta(\hat{P}^\#-P^0)(x_0)|\leq C^{\prime }M_0\delta^{m-|\beta|}$
for $\beta \in \mathcal{M}$, with $C^{\prime }$ determined by $C_B$, $\hat{C}%
_B$, $C_{DIFF}$, $m$, $n$.
\end{itemize}
\end{lemma}

\begin{remark}
Note that $\mathcal{A}$ and $\hat{\mathcal{A}}$ play different r\^oles here;
see \eqref{t1}, \eqref{t3}, and \eqref{t7}.
\end{remark}

\begin{proof}[Proof of the Transport Lemma]
In the trivial case $\mathcal{A}=\hat{\mathcal{A}}=\emptyset $, the
Transport Lemma holds simply because (by definition of the $l$-th
refinement) there exists $\hat{P}^{\#}\in \Gamma _{l_{0}-1}\left(
y_{0},C_{B}M_{0}\right) $ such that 
\begin{eqnarray*}
\left\vert \partial ^{\beta }\left( \hat{P}^{\#}-P^{0}\right) \left(
x_{0}\right) \right\vert &\leq &C_{B}M_{0}\left\vert x_{0}-y_{0}\right\vert
^{m-\left\vert \beta \right\vert } \\
&\leq &C_{B}M_{0}\delta ^{m-\left\vert \beta \right\vert }\text{ for }\beta
\in \mathcal{M}\text{.}
\end{eqnarray*}%
(Recall that $P^{0}\in \Gamma _{l_{0}}\left( x_{0},C_{B}M_{0}\right) $ since 
$\vec{\Gamma}_{l_{0}}$ has an $\left( \mathcal{A},\delta ,C_{B}\right) $%
-basis at $\left( x_{0},M_{0},P^{0}\right) $.)

From now on, we suppose that 
\begin{equation}
\#\left( \mathcal{A}\right) +\#\left( {\hat{\mathcal{A}}}\right) \not=0\text{%
.}  \label{t9}
\end{equation}

In proving the Transport Lemma, we do not yet take $\epsilon _{0}$ to be a
constant determined by $C_{B}$, $\hat{C}_{B}$, $C_{DIFF}$, $m$, $n$. Rather,
we make the following

\begin{itemize}
\item[\refstepcounter{equation}\text{(\theequation)}\label{t10}] \emph{Small 
$\epsilon_0$ assumption: $\epsilon_0$ is less than a small enough constant
determined by $C_{B}$, $\hat{C}_{B}$, $C_{DIFF}$, $m$, $n$. }
\end{itemize}

Assuming \eqref{t1}$\cdots $\eqref{t5} and \eqref{t9}, \eqref{t10}, we will
prove that there exists $\hat{P}^{\#}\in \mathcal{P}$ satisfying \eqref{t6}, %
\eqref{t7}, \eqref{t8}. Once we do so, we may then pick $\epsilon _{0}$ to
be a constant determined by $C_{B}$, $\hat{C}_{B}$, $C_{DIFF}$, $m$, $n$,
small enough to satisfy \eqref{t10}. That will complete the proof of the
Transport Lemma.

Thus, assume \eqref{t1}$\cdots$\eqref{t5} and \eqref{t9}, \eqref{t10}.

We write $c$, $C$, $C^{\prime }$, etc. to denote ``controlled constants'',
i.e., constants determined by $C_{B}$, $\hat{C}_{B}$, $C_{DIFF}$, $m$, $n$.
These symbols may denote different controlled constants in different
occurrences.

Let $(P_\alpha)_{\alpha \in \mathcal{A}}$ be an $(\mathcal{A},\delta,C_B)$%
-basis for $\vec{\Gamma}_{l_0}$ at $(x_0,M_0,P^0)$, and let $(\hat{P}%
_\alpha)_{\alpha \in \hat{\mathcal{A}}}$ be an $(\hat{\mathcal{A}},\delta,%
\hat{C}_B)$-basis for $\vec{\Gamma}_{l_0}$ at $(x_0,M_0,\hat{P}^0)$.

By definition, the following hold.

\begin{itemize}
\item[\refstepcounter{equation}\text{(\theequation)}\label{t11}] $%
P^{0}+c_{0}\sigma M_{0}\delta ^{m-\left\vert \alpha \right\vert }P_{\alpha
}\in \Gamma _{l_{0}}\left( x_{0},CM_{0}\right) $ for $\alpha \in \mathcal{A}$%
, $\sigma \in \left\{ 1,-1\right\} $.
\end{itemize}

\begin{itemize}
\item[\refstepcounter{equation}\text{(\theequation)}\label{t12}] $\hat{P}%
^{0}+\hat{c}_{0}\sigma M_{0}\delta ^{m-\left\vert \alpha \right\vert }\hat{P}%
_{\alpha }\in \Gamma _{l_{0}}\left( x_{0},CM_{0}\right) $ for $\alpha \in 
\hat{\mathcal{A}}$, $\sigma \in \left\{ 1,-1\right\} $.
\end{itemize}

\begin{itemize}
\item[\refstepcounter{equation}\text{(\theequation)}\label{t13}] $\partial
^{\beta }P_{\alpha }\left( x_{0}\right) =\delta _{\beta \alpha }$ for $%
\alpha ,\beta \in \mathcal{A}$.
\end{itemize}

\begin{itemize}
\item[\refstepcounter{equation}\text{(\theequation)}\label{t14}] $\partial
^{\beta }\hat{P}_{\alpha }\left( x_{0}\right) =\delta _{\beta \alpha }$ for $%
\alpha ,\beta \in \hat{\mathcal{A}}$.
\end{itemize}

\begin{itemize}
\item[\refstepcounter{equation}\text{(\theequation)}\label{t15}] $\left\vert
\partial ^{\beta }P_{\alpha }\left( x_{0}\right) \right\vert \leq C\delta
^{\left\vert \alpha \right\vert -\left\vert \beta \right\vert }$ for $\alpha
\in \mathcal{A}$, $\beta \in \mathcal{M}$.
\end{itemize}

\begin{itemize}
\item[\refstepcounter{equation}\text{(\theequation)}\label{t16}] $\left\vert
\partial ^{\beta }\hat{P}_{\alpha }\left( x_{0}\right) \right\vert \leq
C\delta ^{\left\vert \alpha \right\vert -\left\vert \beta \right\vert }$ for 
$\alpha \in \hat{\mathcal{A}}$, $\beta \in \mathcal{M}$.
\end{itemize}

We fix controlled constants $c_{0},\hat{c}_{0}$ as in \eqref{t11}, %
\eqref{t12}. Recall that $\vec{\Gamma}_{l_{0}}$ is the first refinement of $%
\vec{\Gamma}_{l_{0}-1}$. Therefore, by \eqref{t10}, \eqref{t11}, there
exists $\tilde{P}_{\alpha ,\sigma }\in \Gamma _{l_{0}-1}(y_{0},CM_{0})$ ($%
\alpha \in \mathcal{A},\sigma \in \{1,-1\}$) such that 
\begin{eqnarray*}
&&\left\vert \partial ^{\beta }\left( \tilde{P}_{\alpha ,\sigma }-\left[
P^{0}+c_{0}\sigma M_{0}\delta ^{m-\left\vert \alpha \right\vert }P_{\alpha }%
\right] \right) \left( x_{0}\right) \right\vert \\
&\leq &CM_{0}\left\vert x_{0}-y_{0}\right\vert ^{m-\left\vert \beta
\right\vert }\leq C\epsilon _{0}M_{0}\delta ^{m-\left\vert \beta \right\vert
}\text{, for }\beta \in \mathcal{M}\text{.}
\end{eqnarray*}

Writing 
\begin{equation*}
E_{\alpha ,\sigma }=\frac{\tilde{P}_{\alpha ,\sigma }-\left[
P^{0}+c_{0}\sigma M_{0}\delta ^{m-\left\vert \alpha \right\vert }P_{\alpha }%
\right] }{c_{0}\sigma M_{0}\delta ^{m-\left\vert \alpha \right\vert }},
\end{equation*}
we have

\begin{itemize}
\item[\refstepcounter{equation}\text{(\theequation)}\label{t17}] $%
P^{0}+c_{0}\sigma M_{0}\delta ^{m-\left\vert \alpha \right\vert }\left(
P_{\alpha }+E_{\alpha ,\sigma }\right) \in \Gamma _{l_{0}-1}\left(
y_{0},CM_{0}\right) $ for $\alpha \in \mathcal{A},\sigma \in \left\{
1,-1\right\} $,
\end{itemize}
and
\begin{itemize}
\item[\refstepcounter{equation}\text{(\theequation)}\label{t18}] $\left\vert
\partial ^{\beta }E_{\alpha ,\sigma }\left( x_{0}\right) \right\vert \leq
C\epsilon _{0}\delta ^{\left\vert \alpha \right\vert -\left\vert \beta
\right\vert }$ for $\alpha \in \mathcal{A},\beta \in \mathcal{M}$, $\sigma
\in \left\{ 1,-1\right\} $.
\end{itemize}

Similarly, we obtain $\hat{E}_{\alpha ,\sigma }\in \mathcal{P}$ $\left( \alpha \in 
\hat{\mathcal{A}},\sigma \in \left\{ 1,-1\right\} \right) $, satisfying

\begin{itemize}
\item[\refstepcounter{equation}\text{(\theequation)}\label{t19}] $\hat{P}%
^{0}+\hat{c}_{0}\sigma M_{0}\delta ^{m-\left\vert \alpha \right\vert }\left( 
\hat{P}_{\alpha }+\hat{E}_{\alpha ,\sigma }\right) \in \Gamma
_{l_{0}-1}\left( y_{0},CM_{0}\right) $ for $\alpha \in \hat{\mathcal{A}}$, $%
\sigma \in \left\{ 1,-1\right\} $,
\end{itemize}
and
\begin{itemize}
\item[\refstepcounter{equation}\text{(\theequation)}\label{t20}] $\left\vert
\partial ^{\beta }\hat{E}_{\alpha ,\sigma }\left( x_{0}\right) \right\vert
\leq C\epsilon _{0}\delta ^{\left\vert \alpha \right\vert -\left\vert \beta
\right\vert }$ for $\alpha \in \hat{\mathcal{A}},\beta \in \mathcal{M}$, $%
\sigma \in \left\{ 1,-1\right\} $.
\end{itemize}

We introduce the following polynomials:

\begin{itemize}
\item[\refstepcounter{equation}\text{(\theequation)}\label{t21}] 
\begin{equation*}
\hat{P}^{\prime }=\frac{1}{2}\left[ \#\left( \mathcal{A}\right) +\#\left( 
\hat{\mathcal{A}}\right) \right] ^{-1}\left( 
\begin{array}{c}
\sum_{\alpha \in \mathcal{A}\text{,}\sigma =\pm 1}\left\{ P^{0}+c_{0}\sigma
M_{0}\delta ^{m-\left\vert \alpha \right\vert }\left( P_{\alpha }+E_{\alpha
,\sigma }\right) \right\}  \\ 
+\sum_{\alpha \in \hat{\mathcal{A}}\text{,}\sigma =\pm 1}\left\{ \hat{P}^{0}+%
\hat{c}_{0}\sigma M_{0}\delta ^{m-\left\vert \alpha \right\vert }\left( \hat{%
P}_{\alpha }+\hat{E}_{\alpha ,\sigma }\right) \right\} 
\end{array}%
\right) 
\end{equation*}%
(see \eqref{t9});
\end{itemize}

\begin{itemize}
\item[\refstepcounter{equation}\text{(\theequation)}\label{t22}] 
\begin{eqnarray*}
P_{\alpha }^{\prime } &=&\frac{1}{2c_{0}M_{0}\delta ^{m-\left\vert \alpha
\right\vert }}\left( 
\begin{array}{c}
\left\{ P^{0}+c_{0}M_{0}\delta ^{m-\left\vert \alpha \right\vert }\left(
P_{\alpha }+E_{\alpha ,1}\right) \right\}  \\ 
-\left\{ P^{0}-c_{0}M_{0}\delta ^{m-\left\vert \alpha \right\vert }\left(
P_{\alpha }+E_{\alpha ,-1}\right) \right\} 
\end{array}%
\right)  \\
&=&P_{\alpha }+\frac{1}{2}\left( E_{\alpha ,1}+E_{\alpha ,-1}\right) \text{
for }\alpha \in \mathcal{A}\text{;}
\end{eqnarray*}
\end{itemize}

\begin{itemize}
\item[\refstepcounter{equation}\text{(\theequation)}\label{t23}] 
\begin{eqnarray*}
\hat{P}_{\alpha }^{\prime } &=&\frac{1}{2\hat{c}_{0}M_{0}\delta
^{m-\left\vert \alpha \right\vert }}\left( 
\begin{array}{c}
\left\{ \hat{P}^{0}+\hat{c}_{0}M_{0}\delta ^{m-\left\vert \alpha \right\vert
}\left( \hat{P}_{\alpha }+\hat{E}_{\alpha ,1}\right) \right\}  \\ 
-\left\{ \hat{P}^{0}-\hat{c}_{0}M_{0}\delta ^{m-\left\vert \alpha
\right\vert }\left( \hat{P}_{\alpha }+\hat{E}_{\alpha ,-1}\right) \right\} 
\end{array}%
\right)  \\
&=&\hat{P}_{\alpha }+\frac{1}{2}\left( \hat{E}_{\alpha ,1}+\hat{E}_{\alpha
,-1}\right) \text{ for }\alpha \in \hat{\mathcal{A}}.
\end{eqnarray*}
\end{itemize}

For a small enough controlled constant $c_1$, we have

\begin{itemize}
\item[\refstepcounter{equation}\text{(\theequation)}\label{t24}] $\hat{P}%
^{\prime }+c_{1}M_{0}\delta ^{m-\left\vert \alpha \right\vert }P_{\alpha
}^{\prime },\hat{P}^{\prime }-c_{1}M_{0}\delta ^{m-\left\vert \alpha
\right\vert }P_{\alpha }^{\prime }\in \Gamma _{l_{0}-1}\left(
y_{0},CM_{0}\right) $ for $\alpha \in \mathcal{A}$,
\end{itemize}

and

\begin{itemize}
\item[\refstepcounter{equation}\text{(\theequation)}\label{t25}] $\hat{P}%
^{\prime }+c_{1}M_{0}\delta ^{m-\left\vert \alpha \right\vert }\hat{P}%
_{\alpha }^{\prime },\hat{P}^{\prime }-c_{1}M_{0}\delta ^{m-\left\vert
\alpha \right\vert }\hat{P}_{\alpha }^{\prime }\in \Gamma _{l_{0}-1}\left(
y_{0},CM_{0}\right) $ for $\alpha \in \hat{\mathcal{A}}$,
\end{itemize}

because each of the polynomials in \eqref{t24}, \eqref{t25} is a convex
combination of the polynomials in \eqref{t17}, \eqref{t19}.

From \eqref{t24}, \eqref{25} and the Trivial Remark on Convex Sets in
Section \ref{notation-and-preliminaries}, we obtain the following, for a
small enough controlled constant $c_2$.

\begin{itemize}
\item[\refstepcounter{equation}\text{(\theequation)}\label{t26}] $\hat{P}%
^{\prime }+\sum_{\alpha \in \mathcal{A}}s_{\alpha }M_{0}\delta
^{m-\left\vert \alpha \right\vert }P_{\alpha }^{\prime }+\sum_{\alpha \in 
\hat{\mathcal{A}}}t_{\alpha }M_{0}\delta ^{m-\left\vert \alpha \right\vert }%
\hat{P}_{\alpha }^{\prime }\in \Gamma _{l_{0}-1}\left( y_{0},CM_{0}\right) $%
, whenever $\left\vert s_{\alpha }\right\vert \leq c_{2}$ for all $\alpha
\in \mathcal{A}$ and $\left\vert t_{\alpha }\right\vert \leq c_{2}$ for all $%
\alpha \in \hat{\mathcal{A}}$.
\end{itemize}

Note also that \eqref{t21} may be written in the equivalent form

\begin{itemize}
\item[\refstepcounter{equation}\text{(\theequation)}\label{t27}] 
\begin{eqnarray*}
&&\hat{P}^{\prime }=P^{0}+\left[ \frac{\#\left( \hat{\mathcal{A}}\right) }{%
\#\left( \mathcal{A}\right) +\#\hat{\left( \mathcal{A}\right) }}\right]
\left( \hat{P}^{0}-P^{0}\right)  \\
&&+\frac{1}{2\left[ \#\left( \mathcal{A}\right) +\#\left( \hat{\mathcal{A}}%
\right) \right] }\left\{ 
\begin{array}{c}
\sum_{\alpha \in \mathcal{A}\text{,}\sigma =\pm 1}c_{0}\sigma M_{0}\delta
^{m-\left\vert \alpha \right\vert }E_{\alpha ,\sigma } \\ 
+\sum_{\alpha \in \hat{\mathcal{A}}\text{,}\sigma =\pm 1}\hat{c}_{0}\sigma
M_{0}\delta ^{m-\left\vert \alpha \right\vert }\hat{E}_{\alpha ,\sigma }%
\end{array}%
\right\} \text{.}
\end{eqnarray*}
\end{itemize}

Consequently, \eqref{t3}, \eqref{t4}, \eqref{t18}, \eqref{t20} tell us that

\begin{itemize}
\item[\refstepcounter{equation}\text{(\theequation)}\label{t28}] $\left\vert
\partial ^{\beta }\left( \hat{P}^{\prime }-P^{0}\right) \left( x_{0}\right)
\right\vert \leq C\epsilon _{0}M_{0}\delta ^{m-\left\vert \beta \right\vert }
$ for $\beta \in \mathcal{A}$;
\end{itemize}

\begin{itemize}
\item[\refstepcounter{equation}\text{(\theequation)}\label{t29}] $\left\vert
\partial ^{\beta }\left( \hat{P}^{\prime}-P^{0}\right) \left( x_{0}\right)
\right\vert \leq CM_{0}\delta ^{m-\left\vert \beta \right\vert }$ for $\beta
\in \mathcal{M}$.
\end{itemize}

Similarly, \eqref{t13}$\cdots$\eqref{t16}, \eqref{t18}, \eqref{t20}, and %
\eqref{t22}, \eqref{t23} together imply the estimates

\begin{itemize}
\item[\refstepcounter{equation}\text{(\theequation)}\label{t30}] $\left\vert
\partial ^{\beta }P_{\alpha }^{\prime }\left( x_{0}\right) -\delta _{\beta
\alpha }\right\vert \leq C\epsilon _{0}\delta ^{\left\vert \alpha
\right\vert -\left\vert \beta \right\vert }$ for $\alpha ,\beta \in \mathcal{%
A}$;
\end{itemize}

\begin{itemize}
\item[\refstepcounter{equation}\text{(\theequation)}\label{t31}] $\left\vert
\partial ^{\beta }\hat{P}_{\alpha }^{\prime }\left( x_{0}\right) -\delta
_{\beta \alpha }\right\vert \leq C\epsilon _{0}\delta ^{\left\vert \alpha
\right\vert -\left\vert \beta \right\vert }$ for $\alpha ,\beta \in \hat{%
\mathcal{A}}$;
\end{itemize}

\begin{itemize}
\item[\refstepcounter{equation}\text{(\theequation)}\label{t32}] $\left\vert
\partial ^{\beta }P_{\alpha }^{\prime }\left( x_{0}\right) \right\vert \leq
C\delta ^{\left\vert \alpha \right\vert -\left\vert \beta \right\vert }$ for 
$\alpha \in \mathcal{A}$, $\beta \in \mathcal{M}$; and
\end{itemize}

\begin{itemize}
\item[\refstepcounter{equation}\text{(\theequation)}\label{t33}] $\left\vert
\partial ^{\beta }\hat{P}_{\alpha }^{\prime }\left( x_{0}\right) \right\vert
\leq C\delta ^{\left\vert \alpha \right\vert -\left\vert \beta \right\vert }$
for $\alpha \in \hat{\mathcal{A}}$, $\beta \in \mathcal{M}$.
\end{itemize}

From \eqref{t30}$\cdots$\eqref{t33} and \eqref{t5}, we have also

\begin{itemize}
\item[\refstepcounter{equation}\text{(\theequation)}\label{t34}] $\left\vert
\partial ^{\beta }P_{\alpha }^{\prime }\left( y_{0}\right) -\delta _{\beta
\alpha }\right\vert \leq C\epsilon _{0}\delta ^{\left\vert \alpha
\right\vert -\left\vert \beta \right\vert }$ for $\beta ,\alpha \in \mathcal{%
A}$;
\end{itemize}

\begin{itemize}
\item[\refstepcounter{equation}\text{(\theequation)}\label{t35}] $\left\vert
\partial ^{\beta }\hat{P}_{\alpha }^{\prime }\left( y_{0}\right) -\delta
_{\beta \alpha }\right\vert \leq C\epsilon _{0}\delta ^{\left\vert \alpha
\right\vert -\left\vert \beta \right\vert }$ for $\beta ,\alpha \in \hat{%
\mathcal{A}}$;
\end{itemize}

\begin{itemize}
\item[\refstepcounter{equation}\text{(\theequation)}\label{t36}] $\left\vert
\partial ^{\beta }P_{\alpha }^{\prime }\left( y_{0}\right) \right\vert \leq
C\delta ^{\left\vert \alpha \right\vert -\left\vert \beta \right\vert }$ for 
$\alpha \in \mathcal{A}$, $\beta \in \mathcal{M}$; and
\end{itemize}

\begin{itemize}
\item[\refstepcounter{equation}\text{(\theequation)}\label{t37}] $\left\vert
\partial ^{\beta }\hat{P}_{\alpha }^{\prime }\left( y_{0}\right) \right\vert
\leq C\delta ^{\left\vert \alpha \right\vert -\left\vert \beta \right\vert }$
for $\alpha \in \hat{\mathcal{A}},\beta \in \mathcal{M}$.
\end{itemize}

Next, we prove that there exists $\hat{P}^{\#}\in \mathcal{P}$ with the
following properties:

\begin{itemize}
\item[\refstepcounter{equation}\text{(\theequation)}\label{t38}] $\partial
^{\beta }\left( \hat{P}^{\#}-P^{0}\right) \left( x_{0}\right) =0$ for $\beta
\in \mathcal{A}$;
\end{itemize}

\begin{itemize}
\item[\refstepcounter{equation}\text{(\theequation)}\label{t39}] $\left\vert
\partial ^{\beta }\left( \hat{P}^{\#}-P^{0}\right) \left( x_{0}\right)
\right\vert \leq CM_{0}\delta ^{m-\left\vert \beta \right\vert }$ for $\beta
\in \mathcal{M}$;
\end{itemize}

for a small enough controlled constant $c_{3}$, we have

\begin{itemize}
\item[\refstepcounter{equation}\text{(\theequation)}\label{t40}] $\hat{P}%
^{\#}+\sum_{\alpha \in \mathcal{A}}s_{\alpha }M_{0}\delta ^{m-\left\vert
\alpha \right\vert }P_{\alpha }^{\prime }+\sum_{\alpha \in \hat{\mathcal{A}}%
}t_{\alpha }M_{0}\delta ^{m-\left\vert \alpha \right\vert }\hat{P}_{\alpha
}^{\prime }\in \Gamma _{l_{0}-1}\left( y_{0},CM_{0}\right) $, whenever all $%
\left\vert s_{\alpha }\right\vert $, $\left\vert t_{\alpha }\right\vert $
are less than $c_{3}$.
\end{itemize}

Indeed, if $\mathcal{A}=\emptyset $, then set $\hat{P}^{\#}=\hat{P}^{\prime }
$; then \eqref{t38} holds vacuously, and \eqref{t39}, \eqref{t40} simply
restate \eqref{t29}, \eqref{t26}. Suppose $\mathcal{A}\not=\emptyset $. We
will pick coefficients $s_{\alpha }^{\#}$ ($\alpha \in \mathcal{A}$) for
which

\begin{itemize}
\item[\refstepcounter{equation}\text{(\theequation)}\label{t41}] $\hat{P}%
^{\#}:=\hat{P}^{\prime }+\sum_{\alpha \in \mathcal{A}}s_{\alpha
}^{\#}M_{0}\delta ^{m-\left\vert \alpha \right\vert }P_{\alpha }^{\prime }$
satisfies \eqref{t38}, \eqref{t39}, \eqref{t40}.
\end{itemize}

In fact, with $\hat{P}^\#$ given by \eqref{t41}, equation \eqref{t38} is
equivalent to the system of the linear equations

\begin{itemize}
\item[\refstepcounter{equation}\text{(\theequation)}\label{t42}] $%
\sum_{\alpha \in \mathcal{A}}\left[ \delta ^{\left\vert \beta \right\vert
-\left\vert \alpha \right\vert }\partial ^{\beta }P_{\alpha }^{\prime
}\left( x_{0}\right) \right] s_{\alpha }^{\#}=-M_{0}^{-1}\delta ^{\left\vert
\beta \right\vert -m}\partial ^{\beta }\left( \hat{P}^{\prime }-P^{0}\right)
\left( x_{0}\right) $ $\left( \beta \in \mathcal{A}\right) $.
\end{itemize}

By \eqref{t28}, the right-hand side of \eqref{t42} has absolute
value at most $C\epsilon _{0}$. Hence, by \eqref{t30} and the small $%
\epsilon _{0}$ assumption \eqref{t10}, we can solve \eqref{t42} for the $%
s_{\alpha }^{\#}$, and we have

\begin{itemize}
\item[\refstepcounter{equation}\text{(\theequation)}\label{t43}] $\left\vert
s_{\alpha }^{\#}\right\vert \leq C\epsilon _{0}$ for all $\alpha \in 
\mathcal{A}$.
\end{itemize}

The resulting $\hat{P}^{\#}$ given by \eqref{t41} then satisfies \eqref{t38}%
. Moreover, for $\beta \in \mathcal{M}$, we have%
\begin{eqnarray*}
\left\vert \partial ^{\beta }\left( \hat{P}^{\#}-P^{0}\right) \left(
x_{0}\right) \right\vert  &\leq &\left\vert \partial ^{\beta }\left( \hat{P}%
^{\prime }-P^{0}\right) \left( x_{0}\right) \right\vert +\sum_{\alpha \in 
\mathcal{A}}\left\vert s_{\alpha }^{\#}\right\vert \cdot M_{0}\delta
^{m-\left\vert \alpha \right\vert }\left\vert \partial ^{\beta }P_{\alpha
}^{\prime }\left( x_{0}\right) \right\vert  \\
&\leq &\left\vert \partial ^{\beta }\left( \hat{P}^{\prime }-P^{0}\right)
\left( x_{0}\right) \right\vert +C\epsilon _{0}\delta ^{m-\left\vert \beta
\right\vert }M_{0}\text{,}
\end{eqnarray*}%
thanks to \eqref{t32}, \eqref{t41}, \eqref{t43}.

Therefore, \eqref{t29} gives%
\begin{equation*}
\left\vert \partial ^{\beta }\left( \hat{P}^{\#}-P^{0}\right) \left(
x_{0}\right) \right\vert \leq CM_{0}\delta ^{m-\left\vert \beta \right\vert }%
\text{ for }\beta \in \mathcal{M}\text{,}
\end{equation*}
proving \eqref{t39}.

Finally, \eqref{t40} follows at once from \eqref{t26}, \eqref{t41}, %
\eqref{t43} and the small $\epsilon_0$ assumption \eqref{t10}.

Thus, in all cases, there exists $\hat{P}^{\#}$ satisfying \eqref{t38}, %
\eqref{t39}, \eqref{t40}. We fix such a $\hat{P}^{\#}$.

Next, we produce an $(\mathcal{A},\delta ,C)$-basis for $\vec{\Gamma}%
_{l_{0}-1}$ at $(y_{0},M_{0},\hat{P}^{\#})$. To do so, we first suppose that 
$\mathcal{A}\not=\emptyset $, and set

\begin{itemize}
\item[\refstepcounter{equation}\text{(\theequation)}\label{t44}] $P_{\gamma
}^{\#}=\sum_{\alpha \in \mathcal{A}}b_{\gamma \alpha }\delta ^{\left\vert
\gamma \right\vert -\left\vert \alpha \right\vert }P_{\alpha }^{\prime }$
\end{itemize}

for real coefficients $(b_{\gamma \alpha })_{\gamma ,\alpha \in \mathcal{A}}$
to be picked below. For $\beta ,\gamma \in \mathcal{A}$, we have 
\begin{equation*}
\partial ^{\beta }P_{\gamma }^{\#}\left( y_{0}\right) =\delta ^{\left\vert
\gamma \right\vert -\left\vert \beta \right\vert }\cdot \sum_{\alpha \in 
\mathcal{A}}b_{\gamma \alpha }\left[ \delta ^{\left\vert \beta \right\vert
-\left\vert \alpha \right\vert }\partial ^{\beta }P_{\alpha }^{\prime
}\left( y_{0}\right) \right] \text{.}
\end{equation*}

Thanks to \eqref{t34} and the small $\epsilon_0$ assumption \eqref{t10}, we
may define $(b_{\gamma\alpha})_{\gamma,\alpha \in \mathcal{A}}$ as the
inverse matrix of $\left(\delta^{|\beta|-|\alpha|}\partial^\beta
P_\alpha^{\prime }(y_0)\right)_{\alpha,\beta \in \mathcal{A}}$, and we then
have

\begin{itemize}
\item[\refstepcounter{equation}\text{(\theequation)}\label{t45}] $\partial
^{\beta }P_{\gamma }^{\#}\left( y_{0}\right) =\delta _{\beta \gamma }$ $%
\left( \beta ,\gamma \in \mathcal{A}\right) $
\end{itemize}
and
\begin{itemize}
\item[\refstepcounter{equation}\text{(\theequation)}\label{t46}] $\left\vert
b_{\gamma \alpha }-\delta _{\gamma \alpha }\right\vert \leq C\epsilon _{0} $
for $\gamma ,\alpha \in \mathcal{A}$.
\end{itemize}

In particular,

\begin{itemize}
\item[\refstepcounter{equation}\text{(\theequation)}\label{t47}] $\left\vert
b_{\gamma \alpha }\right\vert \leq C$, 
\end{itemize}
and therefore for $\gamma \in 
\mathcal{A}$, $\beta \in \mathcal{M}$ we have
\begin{itemize}
\item[\refstepcounter{equation}\text{(\theequation)}\label{t48}] 
\begin{eqnarray*}
\left\vert \partial ^{\beta }P_{\gamma }^{\#}\left( y_{0}\right) \right\vert
&\leq &\sum_{\alpha \in \mathcal{A}}\left\vert b_{\gamma \alpha }\right\vert
\delta ^{\left\vert \gamma \right\vert -\left\vert \alpha \right\vert
}\left\vert \partial ^{\beta }P_{\alpha }^{\prime }\left( y_{0}\right)
\right\vert  \\
&\leq &C\sum_{\alpha \in \mathcal{A}}\delta ^{\left\vert \gamma \right\vert
-\left\vert \alpha \right\vert }\delta ^{\left\vert \alpha \right\vert
-\left\vert \beta \right\vert }\leq C^{\prime }\delta ^{\left\vert \gamma
\right\vert -\left\vert \beta \right\vert }\text{,}
\end{eqnarray*}
\end{itemize}
thanks to \eqref{t36}.

Also, for $\gamma \in \mathcal{A}$, we have%
\begin{equation*}
\left[ M_{0}\delta ^{m-\left\vert \gamma \right\vert }P_{\gamma }^{\#}\right]
=\sum_{\alpha \in \mathcal{A}}b_{\gamma \alpha }\cdot \left[ M_{0}\delta
^{m-\left\vert \alpha \right\vert }P_{\alpha }^{\prime }\right] \text{.}
\end{equation*}

Therefore, for a small enough controlled constant $c_4$, we have

\begin{itemize}
\item[\refstepcounter{equation}\text{(\theequation)}\label{t49}] $\hat{P}%
^{\#}+c_{4}M_{0}\delta ^{m-\left\vert \gamma \right\vert }P_{\gamma }^{\#}$, 
$\hat{P}^{\#}-c_{4}M_{0}\delta ^{m-\left\vert \gamma \right\vert }P_{\gamma
}^{\#}\in \Gamma _{l_{0}-1}\left( y_{0},CM_{0}\right) $ for $\gamma \in 
\mathcal{A}$,
\end{itemize}

thanks to \eqref{t40}. Since we are assuming that $\mathcal{A}
\not=\emptyset $, \eqref{t49} implies that also

\begin{itemize}
\item[\refstepcounter{equation}\text{(\theequation)}\label{t50}] $\hat{P}%
^{\#}\in \Gamma _{l_{0}-1}\left( y_{0},CM_{0}\right) $.
\end{itemize}

Our results \eqref{t45}, \eqref{t48}, \eqref{t49}, \eqref{t50} tell us that $%
\left( P_{\gamma }^{\#}\right) _{\gamma \in \mathcal{A}}$ is an $(\mathcal{A}%
,\delta ,C)$-basis for $\vec{\Gamma}_{l_{0}-1}$ at $\left( y_{0},M_{0},\hat{P%
}^{\#}\right) $.

Thus, we have produced the desired $(\mathcal{A},\delta ,C)$-basis, provided
that $\mathcal{A}\not=\emptyset $. On the other hand, if $\mathcal{A}%
=\emptyset $, then the existence of an $(\mathcal{A},\delta ,C)$-basis for $%
\vec{\Gamma}_{l_{0}-1}$ at $\left( y_{0},M_{0},\hat{P}^{\#}\right) $ is
equivalent to the assertion that

\begin{itemize}
\item[\refstepcounter{equation}\text{(\theequation)}\label{t51}] $\hat{P}%
^{\#}\in \Gamma _{l_{0}-1}\left( y_{0},CM_{0}\right) $,
\end{itemize}

and \eqref{t51} follows at once from \eqref{t40}. Thus, in all cases,

\begin{itemize}
\item[\refstepcounter{equation}\text{(\theequation)}\label{t52}] $\vec{\Gamma%
}_{l_{0}-1}$ has an $\left( \mathcal{A},\delta ,C\right) $-basis at $\left(
y_{0},M_{0},\hat{P}^{\#}\right) $.
\end{itemize}

Similarly, we can produce an $(\hat{\mathcal{A}},\delta ,C)$-basis for $\vec{%
\Gamma}_{l_{0}-1}$ at $\left( y_{0},M_{0},\hat{P}^{\#}\right) $. We suppose
first that $\hat{\mathcal{A}}\not=\emptyset $, and set

\begin{itemize}
\item[\refstepcounter{equation}\text{(\theequation)}\label{t53}] $\hat{P}%
_{\gamma }^{\#}=\sum_{\alpha \in \hat{\mathcal{A}}}\hat{b}_{\beta \alpha
}\delta ^{\left\vert \gamma \right\vert -\left\vert \alpha \right\vert }\hat{%
P}_{\alpha }^{\prime }$ for $\gamma \in \hat{\mathcal{A}}$, with
coefficients $\hat{b}_{\gamma \alpha }$ to be picked below.
\end{itemize}

Thanks to \eqref{t35} and the \emph{small $\epsilon_0$ assumption} %
\eqref{t10}, we can pick the coefficients $\hat{b}_{\gamma\alpha}$ so that

\begin{itemize}
\item[\refstepcounter{equation}\text{(\theequation)}\label{t54}] $\partial
^{\beta }\hat{P}_{\gamma }^{\#}\left( y_{0}\right) =\delta _{\beta ,\gamma }$
for $\beta ,\gamma \in \hat{\mathcal{A}}$
\end{itemize}

and

\begin{itemize}
\item[\refstepcounter{equation}\text{(\theequation)}\label{t55}] $\left\vert 
\hat{b}_{\gamma \alpha }-\delta _{\gamma \alpha }\right\vert \leq C\epsilon
_{0}$ for $\gamma ,\alpha \in \hat{\mathcal{A}}$,
\end{itemize}

hence

\begin{itemize}
\item[\refstepcounter{equation}\text{(\theequation)}\label{t56}] $\left\vert 
\hat{b}_{\gamma \alpha }\right\vert \leq C$ for $\gamma ,\alpha \in \hat{%
\mathcal{A}}$.
\end{itemize}

From \eqref{t37}, \eqref{t53}, \eqref{t56}, we obtain the estimate

\begin{itemize}
\item[\refstepcounter{equation}\text{(\theequation)}\label{t57}] $\left\vert
\partial ^{\beta }\hat{P}_{\gamma }^{\#}\left( y_{0}\right) \right\vert \leq
C\delta ^{\left\vert \gamma \right\vert -\left\vert \beta \right\vert }$ for 
$\gamma \in \hat{\mathcal{A}}$ and $\beta \in \mathcal{M}$,
\end{itemize}

in analogy with \eqref{t48}. Also for $\gamma \in \hat{\mathcal{A}}$, we have%
\begin{equation*}
M_{0}\delta ^{m-\left\vert \gamma \right\vert }\hat{P}_{\gamma
}^{\#}=\sum_{\alpha \in \hat{\mathcal{A}}}\hat{b}_{\gamma \alpha }\left[
M_{0}\delta ^{m-\left\vert \alpha \right\vert }\hat{P}_{\alpha }^{\prime }%
\right] \text{.}
\end{equation*}

Together with \eqref{t40} and \eqref{t56}, this tells us that

\begin{itemize}
\item[\refstepcounter{equation}\text{(\theequation)}\label{t58}] $\hat{P}%
^{\#}+c_{5}M_{0}\delta ^{m-\left\vert \gamma \right\vert }\hat{P}_{\gamma
}^{\#}$, $\hat{P}^{\#}-c_{5}M_{0}\delta ^{m-\left\vert \gamma \right\vert }%
\hat{P}_{\gamma }^{\#}\in \Gamma _{l_{0}-1}\left( y_{0},CM_{0}\right) $ for $%
\gamma \in \hat{\mathcal{A}}$
\end{itemize}

in analogy with \eqref{t49}. Since we are assuming that $\hat{\mathcal{A}}%
\not=\emptyset $, \eqref{t58} implies that

\begin{itemize}
\item[\refstepcounter{equation}\text{(\theequation)}\label{t59}] $\hat{P}%
^{\#}\in \Gamma _{l_{0}-1}\left( y_{0},CM_{0}\right) $.
\end{itemize}

Our results \eqref{t54},\eqref{t57}, \eqref{t58}, \eqref{t59} tell us that $%
\left( \hat{P}_{\gamma }^{\#}\right) _{\gamma \in \hat{\mathcal{A}}}$ is an $%
(\hat{\mathcal{A}},\delta ,C)$-basis for $\vec{\Gamma}_{l_{0}-1}$ at $\left(
y_{0},M_{0},\hat{P}^{\#}\right) $.

Thus, we have produced the desired $(\hat{\mathcal{A}},\delta ,C)$-basis,
provided $\hat{\mathcal{A}}\not=\emptyset $. On the other hand, if $\hat{%
\mathcal{A}}=\emptyset $, then the existence of an $(\hat{\mathcal{A}}%
,\delta ,C)$-basis for $\vec{\Gamma}_{l_{0}-1}$ at $\left( y_{0},M_{0},\hat{P%
}^{\#}\right) $ is equivalent to the assertion that

\begin{itemize}
\item[\refstepcounter{equation}\text{(\theequation)}\label{t60}] $\hat{P}%
^{\#}\in \Gamma _{l_{0}-1}\left( y_{0},CM_{0}\right) $,
\end{itemize}

and \eqref{t60} follows at once from \eqref{t40}. Thus, in all cases,

\begin{itemize}
\item[\refstepcounter{equation}\text{(\theequation)}\label{t61}] $\vec{\Gamma%
}_{l_{0}-1}$ has an $\left( \hat{\mathcal{A}},\delta ,C\right) $-basis at $%
\left( y_{0},M_{0},\hat{P}^{\#}\right) $.
\end{itemize}

Our results \eqref{t52} and \eqref{t61} together yield conclusion \eqref{t6}
of the Transport Lemma (Lemma \ref{lemma-transport}). Also, our results %
\eqref{t38} and \eqref{t39} imply conclusions \eqref{t7} and \eqref{t8},
since $\mathcal{A}$ is monotonic. (See \eqref{t1}.)

Thus, starting from assumptions \eqref{t1}$\cdots$\eqref{t5} and \eqref{t9}, %
\eqref{t10}, we have proven conclusions \eqref{t6}, \eqref{t7}, \eqref{t8}
for our $\hat{P}^\#$.

The proof of the Transport Lemma (Lemma \ref{lemma-transport}) is complete.
\end{proof}

For future reference, we state the special case of the Transport Lemma in
which we take $\hat{\mathcal{A}} = \mathcal{A}$, $\hat{P}^0=P^0$.

\begin{corollary}
\label{cor-to-transport} Let $\vec{\Gamma}_{0}=\left( \Gamma _{0}\left(
x,M\right) \right) _{x\in E,M>0}$ be a shape field. For $l\geq 1$, let $\vec{%
\Gamma}_{l}=\left( \Gamma _{l}\left( x,M\right) \right) _{x\in E,M>0}$ be
the $l$-th refinement of $\vec{\Gamma}_{0}$. Suppose

\begin{itemize}
\item[\refstepcounter{equation}\text{(\theequation)}\label{t62}] $\mathcal{A}
\subseteq \mathcal{M}$ is monotonic.
\end{itemize}

Let $x_{0}\in E$, $M_{0}>0$, $l_{0}\geq 1$, $\delta >0,C_{B}>0$; and let $%
P^{0}\in \mathcal{P}$. Assume that

\begin{itemize}
\item[\refstepcounter{equation}\text{(\theequation)}\label{t63}] $\vec{\Gamma%
}_{l_{0}}$ has an $\left( \mathcal{A},\delta ,C_{B}\right) $-basis at $%
\left( x_{0},M_{0},P^{0}\right) $.
\end{itemize}

Let $y_0 \in E$, and suppose that

\begin{itemize}
\item[\refstepcounter{equation}\text{(\theequation)}\label{t64}] $\left\vert
x_{0}-y_{0}\right\vert \leq \epsilon _{0}\delta $, where $\epsilon _{0} $ is
a small enough constant determined by $C_{B}$, $m$, $n$.
\end{itemize}

Then there exists $\hat{P}^{\#}\in \mathcal{P}$ with the following
properties.

\begin{itemize}
\item[\refstepcounter{equation}\text{(\theequation)}\label{t65}] $\vec{\Gamma%
}_{l_{0}-1}$ has an $\left( \mathcal{A},\delta ,C_{B}^{\prime }\right) $%
-basis at $\left( y_{0},M_{0},\hat{P}^{\#}\right) $, with $C_{B}^{\prime }$
determined by $C_{B}$, $m$, $n$.
\end{itemize}

\begin{itemize}
\item[\refstepcounter{equation}\text{(\theequation)}\label{t66}] $\partial
^{\beta }\left( \hat{P}^{\#}-P^{0}\right) \equiv 0$ for $\beta \in \mathcal{A%
}$.
\end{itemize}

\begin{itemize}
\item[\refstepcounter{equation}\text{(\theequation)}\label{t67}] $\left\vert
\partial ^{\beta }\left( \hat{P}^{\#}-P^{0}\right) \left( x_{0}\right)
\right\vert \leq C^{\prime }M_{0}\delta ^{m-\left\vert \beta \right\vert }$
for all $\beta \in \mathcal{M}$, with $C^{\prime }$ determined by $C_{B}$, $m
$, $n$.
\end{itemize}
\end{corollary}

\part{The Main Lemma}

\section{Statement of the Main Lemma}

\label{statement-of-the-main-lemma}

For $\mathcal{A}\subseteq \mathcal{M}$ monotonic, we define

\begin{itemize}
\item[\refstepcounter{equation}\text{(\theequation)}\label{m1}] $l\left( 
\mathcal{A}\right) =1+3\cdot \#\left\{ \mathcal{A}^{\prime }\subseteq 
\mathcal{M}:\mathcal{A}^{\prime }\text{ monotonic, }\mathcal{A}^{\prime }<%
\mathcal{A}\right\} $.
\end{itemize}

Thus,

\begin{itemize}
\item[\refstepcounter{equation}\text{(\theequation)}\label{m2}] $l\left( 
\mathcal{A}\right) -3\geq l\left( \mathcal{A}^{\prime }\right) $ for $%
\mathcal{A}^{\prime },\mathcal{A}\subseteq \mathcal{M}$ monotonic with $%
\mathcal{A}^{\prime }<\mathcal{A}$.
\end{itemize}

By induction on $\mathcal{A}$ (with respect to the order relation $<$), we
will prove the following result.

{\textbf{Main Lemma for $\mathcal{A}$}}\thinspace \thinspace {\textit{Let $%
\vec{\Gamma}_{0}=\left( \Gamma _{0}\left( x,M\right) \right) _{x\in E,M>0}$
be a $\left( C_{w},\delta _{\max }\right) $-convex shape field, and for $%
l\geq 1$, let $\vec{\Gamma}_{l}=\left( \Gamma _{l}\left( x,M\right) \right)
_{x\in E,M>0}$ be the $l$-th refinement of $\vec{\Gamma}_{0}$. Fix a dyadic
cube $Q_{0}\subset \mathbb{R}^{n}$, a point $x_{0}\in E\cap 5\left(
Q_{0}^{+}\right) $ and a polynomial $P^{0}\in \mathcal{P}$, as well as
positive real numbers $M_{0}$, $\epsilon $, $C_{B}$. We make the following
assumptions.}}

\begin{itemize}
\item[(A1)] \textit{$\vec{\Gamma}_{l\left( \mathcal{A}\right) }$ has an $%
\left( \mathcal{A},\epsilon ^{-1}\delta _{Q_{0}},C_{B}\right) $-basis at $%
\left( x_{0},M_{0},P^{0}\right) $.}

\item[(A2)] $\epsilon ^{-1}\delta _{Q_{0}}\leq \delta _{\max }$.

\item[(A3)] (\textquotedblleft Small $\epsilon $ Assumption")\textit{\ $%
\epsilon $ is less than a small enough constant determined by $C_{B}$, $C_{w}
$, $m$, $n$.}
\end{itemize}

\textit{Then there exists $F\in C^{m}\left( \frac{65}{64}Q_{0}\right) $
satisfying the following conditions.}

\begin{itemize}
\item[(C1)] \textit{$\left\vert \partial ^{\beta }\left( F-P^{0}\right)
\right\vert \leq C\left( \epsilon \right) M_{0}\delta _{Q_{0}}^{m-\left\vert
\beta \right\vert }$ on $\frac{65}{64}Q_{0}$ for $\left\vert \beta
\right\vert \leq m$, where $C\left( \epsilon \right) $ is determined by $%
\epsilon $, $C_{B}$, $C_{w}$, $m$, $n$.}

\item[(C2)] \textit{$J_{z}\left( F\right) \in \Gamma _{0}\left( z,C^{\prime
}\left( \epsilon \right) M_{0}\right) $ for all $z\in E\cap \frac{65}{64}%
Q_{0}$, where $C^{\prime }\left( \epsilon \right) $ is determined by $%
\epsilon $, $C_{B}$, $C_{w}$, $m$, $n$.}
\end{itemize}

\begin{remark}
We state the Main Lemma only for monotonic $\mathcal{A}$.
\end{remark}

Note that $x_0$ may fail to belong to $\frac{65}{64}Q_0$, hence the
assertion $J_{x_0}(F)=P^0$ may be meaningless. Even if $x_0 \in \frac{65}{64}%
Q_0$, we do not assert that $J_{x_0}(F)=P^0$.

\section{The Base Case}

\label{the-base-case} The base case of our induction on $\mathcal{A}$ is the
case $\mathcal{A}= \mathcal{M}$.

In this section, we prove the Main Lemma for $\mathcal{M}$. The hypotheses
of the lemma are as follows:

\begin{itemize}
\item[\refstepcounter{equation}\text{(\theequation)}\label{b1}] $\vec{\Gamma}%
_{0}=\left( \Gamma _{0}\left( x,M\right) \right) _{x\in E,M>0}$ is a $\left(
C_{w},\delta _{\max }\right) $-convex shape field.
\end{itemize}

\begin{itemize}
\item[\refstepcounter{equation}\text{(\theequation)}\label{b2}] $\vec{\Gamma}%
_{1}=\left( \Gamma _{1}\left( x,M\right) \right) _{x\in E,M>0}$ is the first
refinement of $\vec{\Gamma}_{0}$.
\end{itemize}

\begin{itemize}
\item[\refstepcounter{equation}\text{(\theequation)}\label{b3}] $\vec{\Gamma}%
_{1}$ has an $\left( \mathcal{M},\epsilon ^{-1}\delta _{Q_{0}},C_{B}\right) $%
-basis at $\left( x_{0},M_{0},P^{0}\right) $.
\end{itemize}

\begin{itemize}
\item[\refstepcounter{equation}\text{(\theequation)}\label{b4}] $\epsilon
^{-1}\delta _{Q_{0}}\leq \delta _{\max }$.
\end{itemize}

\begin{itemize}
\item[\refstepcounter{equation}\text{(\theequation)}\label{b5}] $\epsilon $
is less than a small enough constant determined by $C_{B}$, $C_{w}$, $m$, $n$%
.
\end{itemize}

\begin{itemize}
\item[\refstepcounter{equation}\text{(\theequation)}\label{b6}] $x_{0}\in
5\left( Q_{0}\right) ^{+}\cap E$.
\end{itemize}

We write $c$, $C$, $C^{\prime }$, etc., to denote constants determined by $%
C_{B}$, $C_{W}$, $m$, $n$. These symbols may denote different constants in
different occurrences.

\begin{itemize}
\item[\refstepcounter{equation}\text{(\theequation)}\label{b7}] Let $z\in
E\cap \frac{65}{64}Q_{0}$.
\end{itemize}

Then \eqref{b6}, \eqref{b7} imply that

\begin{itemize}
\item[\refstepcounter{equation}\text{(\theequation)}\label{b8}] $\left\vert
z-x_{0}\right\vert \leq C\delta _{Q_{0}}=C\epsilon \cdot \left( \epsilon
^{-1}\delta _{Q_{0}}\right) $.
\end{itemize}

From \eqref{b1}, \eqref{b2}, \eqref{b3}, \eqref{b5}, \eqref{b8}, and
Corollary \ref{cor-to-transport} in Section \ref{transport-lemma}, we obtain
a polynomial $\hat{P}^{\#}\in \mathcal{P}$ such that

\begin{itemize}
\item[\refstepcounter{equation}\text{(\theequation)}\label{b9}] $\vec{\Gamma}%
_{0}$ has an $\left( \mathcal{M},\epsilon ^{-1}\delta _{Q_{0}},C^{\prime
}\right) $-basis at $\left( z,M_{0},\hat{P}^{\#}\right) $, and
\end{itemize}

\begin{itemize}
\item[\refstepcounter{equation}\text{(\theequation)}\label{b10}] $\partial
^{\beta }\left( \hat{P}^{\#}-P^{0}\right) =0$ for $\beta \in \mathcal{M}$.
\end{itemize}

From \eqref{b9}, we have $\hat{P}^\# \in \Gamma_0(z,C^{\prime }M_0)$, while %
\eqref{b10} tells us that $\hat{P}^\# = P^0$. Thus,

\begin{itemize}
\item[\refstepcounter{equation}\text{(\theequation)}\label{b11}] $P^{0}\in
\Gamma _{0}\left( z,C^{\prime }M_{0}\right) $ for all $z\in \frac{65}{64}%
Q_{0}\cap E$.
\end{itemize}

Consequently, the function $F:=P^0$ on $\frac{65}{64}Q_0$ satisfies the
conclusions (C1), (C2) of the Main Lemma for $\mathcal{M}$.

This completes the proof of the Main Lemma for $\mathcal{M}$.   $ \blacksquare $

\section{Setup for the Induction Step}

\label{setup-for-the-induction-step}

Fix a monotonic set $\mathcal{A}$ strictly contained in $\mathcal{M}$, and
assume the following

\begin{itemize}
\item[\refstepcounter{equation}\text{(\theequation)}\label{s1}] \underline{%
Induction Hypothesis}: The Main Lemma for $\mathcal{A}^{\prime }$ holds for
all monotonic $\mathcal{A}^{\prime }< \mathcal{A}$.
\end{itemize}

Under this assumption, we will prove the Main Lemma for $\mathcal{A}$. Thus,
let $\vec{\Gamma}_{0}$, $\vec{\Gamma}_{l}$ $(l\geq 1)$, $C_{w}$, $\delta
_{\max }$, $Q_{0}$, $x_{0}$, $P^{0}$, $M_{0}$, $\epsilon $, $C_{B}$ be as in
the hypotheses of the Main Lemma for $\mathcal{A}$. Our goal is to prove the
existence of $F\in C^{m}(\frac{65}{64}Q_{0})$ satisfying conditions (C1) and
(C2). To do so, we introduce a constant $A\geq 1$, and make the following
additional assumptions.

\begin{itemize}
\item[\refstepcounter{equation}\text{(\theequation)}\label{s2}] \underline{%
Large $A$ assumption:} $A$ exceeds a large enough constant determined by $C_B
$, $C_w$, $m$, $n$.
\end{itemize}

\begin{itemize}
\item[\refstepcounter{equation}\text{(\theequation)}\label{s3}] \underline{%
Small $\epsilon$ assumption:} $\epsilon$ is less than a small enough
constant determined by $A$, $C_B$, $C_w$, $m$, $n$.
\end{itemize}

We write $c$, $C$, $C^{\prime }$, etc., to denote constants determined by $%
C_{B}$, $C_{w}$, $m$, $n$. Also we write $c(A)$, $C(A)$, $C^{\prime }(A)$,
etc., to denote constants determined by $A$, $C_{B}$, $C_{W}$, $m$, $n$.
Similarly, we write $C\left( \epsilon \right) $, $c\left( \epsilon \right) $%
, $C^{\prime }\left( \epsilon \right) $, etc., to denote constants
determined by $\epsilon $, $A$, $C_{B}$, $C_{w}$, $m$, $n$. These symbols
may denote different constants in different occurrences.

In place of (C1), (C2), we will prove the existence of a function $F\in
C^{m}\left( \frac{65}{64}Q_{0}\right) $ satisfying

\begin{itemize}
\item[(C*1)] $\left\vert \partial ^{\beta }\left( F-P^{0}\right) \right\vert
\leq C\left( \epsilon \right) M_{0}\delta _{Q_{0}}^{m-\left\vert \beta
\right\vert }$ on $\frac{65}{64}Q_{0}$ for $|\beta |\leq \mathcal{M}$; and

\item[(C*2)] $J_{z}\left( F\right) \in \Gamma _{0}\left( z,C\left( \epsilon
\right) M_{0}\right) $ for all $z\in E\cap \frac{65}{64}Q_{0}$.
\end{itemize}

Conditions (C*1), (C*2) differ from (C1), (C2) in that the constants in
(C*1), (C*2) may depend on $A$.

Once we establish (C*1) and (C*2), we may fix $A$ to be a constant
determined by $C_{B}$, $C_{w}$, $m$, $n$, large enough to satisfy the Large $%
A$ Assumption \eqref{s2}. The Small $\epsilon $ Assumption \eqref{s3} will
then follow from the Small $\epsilon $ Assumption (A3) in the Main Lemma for 
$\mathcal{A}$; and the desired conclusions (C1), (C2) will then follow from
(C*1), (C*2).

Thus, our goal is to prove the existence of $F\in C^{m}\left( \frac{65}{64}%
Q_{0}\right) $ satisfying (C*1) and (C*2), assuming \eqref{s1}, \eqref{s2}, %
\eqref{s3} above, along with hypotheses of the Main Lemma for $\mathcal{A}$.
This will complete our induction on $\mathcal{A}$ and establish the Main
Lemma for all monotonic subsets of $\mathcal{M}$.

\section{Calder\'on-Zygmund Decomposition}

\label{cz-decomposition}

We place ourselves in the setting of Section \ref%
{setup-for-the-induction-step}. Let $Q$ be a dyadic cube. We say that $Q$ is
``OK'' if \eqref{cz1} and \eqref{cz2} below are satisfied.

\begin{itemize}
\item[\refstepcounter{equation}\text{(\theequation)}\label{cz1}] $%
5Q\subseteq 5Q_{0}$.
\end{itemize}

\begin{itemize}
\item[\refstepcounter{equation}\text{(\theequation)}\label{cz2}] Either $%
\#(E\cap 5Q)\leq 1$ or there exists $\hat{A}<\mathcal{A}$ (strict
inequality) for which the following holds:
\end{itemize}

\begin{itemize}
\item[\refstepcounter{equation}\text{(\theequation)}\label{cz3}] For each $%
y\in E\cap 5Q$ there exists $\hat{P}^{y}\in \mathcal{P}$ satisfying

\begin{itemize}
\item[(3a)] $\vec{\Gamma}_{l\left( \mathcal{A}\right) -3}$ has a weak $%
\left( \hat{\mathcal{A}},\epsilon ^{-1}\delta _{Q},A\right) $-basis at $%
\left( y,M_{0},\hat{P}^{y}\right) $.

\item[(3b)] $\left\vert \partial ^{\beta }\left( \hat{P}^{y}-P^{0}\right)
\left( x_{0}\right) \right\vert \leq AM_{0}\left( \epsilon ^{-1}\delta
_{Q_{0}}\right) ^{m-\left\vert \beta \right\vert }$ for all $\beta \in 
\mathcal{M}$.

\item[(3c)] $\partial ^{\beta }\left( \hat{P}^{y}-P^{0}\right) \equiv 0$ for 
$\beta \in \mathcal{A}$.
\end{itemize}
\end{itemize}

\begin{remark}
The argument in this section and the next will depend sensitively on several
details of the above definition. Note that (3a) involves $\vec{\Gamma}_{l(%
\mathcal{A})-3}$ rather than $\vec{\Gamma}_{l(\hat{\mathcal{A}})}$, and that
(3b) involves $x_0$, $\delta_{Q_0}$ rather than $y$, $\delta_Q$. Note also
that the set $\hat{\mathcal{A}}$ in \eqref{cz2}, \eqref{cz3} needn't be
monotonic.
\end{remark}

A dyadic cube $Q$ will be called a \underline{Calder\'on-Zygmund cube} (or a 
\underline{CZ} cube) if it is OK, but no dyadic cube strictly containing $Q$
is OK.

Recall that given any two distinct dyadic cubes $Q$, $Q^{\prime }$, either $Q
$ is strictly contained in $Q^{\prime }$, or $Q^{\prime }$ is strictly
contained in $Q$, or $Q \cap Q^{\prime }=\emptyset$. The first two
alternatives here are ruled out if $Q$, $Q^{\prime }$ are CZ cubes. Hence,
the Calder\'on-Zygmund cubes are pairwise disjoint.

Any CZ cube $Q$ satisfies \eqref{cz1} and is therefore contained in the
interior of $5Q_0$. On the other hand, let $x$ be an interior point of $5Q_0$%
. Then any sufficiently small dyadic cube $Q$ containing $x$ satisfies $5Q
\subset 5Q_0$ and $\#(E \cap 5Q) \leq 1$; hence, $Q$ is OK. However, any
sufficiently large dyadic cube $Q$ containing $x$ will fail to satisfy $5Q
\subseteq 5Q_0$; hence $Q$ is not OK. It follows that $x$ is contained in a
maximal OK dyadic cube. Thus, we have proven

\begin{lemma}
\label{lemma-cz1} The CZ cubes form a partition of the interior of $5Q_0$.
\end{lemma}

Next, we establish

\begin{lemma}
\label{lemma-cz2} Let $Q$, $Q^{\prime }$ be CZ cubes. If $\frac{65}{64}Q
\cap \frac{65}{64}Q^{\prime }\not= \emptyset$, then $\frac{1}{2}\delta_Q
\leq \delta_{Q^{\prime }}\leq 2\delta_Q$.
\end{lemma}

\begin{proof}
Suppose not. Without loss of generality, we may suppose that $\delta_Q \leq 
\frac{1}{4}\delta_{Q^{\prime }}$. Then $\delta_{Q^+} \leq \frac{1}{2}%
\delta_{Q^{\prime }}$, and $\frac{65}{64}Q^+ \cap \frac{65}{64}Q^{\prime
}\not= \emptyset$; hence, $5Q^+ \subset 5Q^{\prime }$. The cube $Q^{\prime }$
is OK. Therefore,

\begin{itemize}
\item[\refstepcounter{equation}\text{(\theequation)}\label{cz4}] $%
5Q^{+}\subset 5Q^{\prime }\subseteq 5Q_{0}$.
\end{itemize}

If $\#\left( E\cap 5Q^{\prime }\right) \leq 1$, then also $\#\left( E\cap
5Q^{+}\right) \leq 1$. Otherwise, there exists $\hat{\mathcal{A}}<\mathcal{A}
$ such that for each $y\in E\cap 5Q^{\prime }$ there exists $\hat{P}^{y}\in 
\mathcal{P}$ satisfying

\begin{itemize}
\item[\refstepcounter{equation}\text{(\theequation)}\label{cz5}] $\vec{\Gamma%
}_{l\left( \mathcal{A}\right) -3}$ has a weak $\left( \hat{\mathcal{A}}%
,\epsilon ^{-1}\delta _{Q^{\prime }},A\right) $-basis at $\left( y,M_{0},%
\hat{P}^{y}\right) $,
\end{itemize}

\begin{itemize}
\item[\refstepcounter{equation}\text{(\theequation)}\label{cz6}] $\left\vert
\partial ^{\beta }\left( \hat{P}^{y}-P^{0}\right) \left( x_{0}\right)
\right\vert \leq AM_{0}\left( \epsilon ^{-1}\delta _{Q_{0}}\right)
^{m-\left\vert \beta \right\vert }$ for $\beta \in \mathcal{M}$, and
\end{itemize}

\begin{itemize}
\item[\refstepcounter{equation}\text{(\theequation)}\label{cz7}] $\partial
^{\beta }\left( \hat{P}^{y}-P^{0}\right) \equiv 0$ for $\beta \in \mathcal{A}
$.
\end{itemize}

For each $y \in E \cap 5Q^+\subseteq E \cap 5Q^{\prime }$, the above $\hat{P}%
^y$ satisfies \eqref{cz6}, \eqref{cz7}; and \eqref{cz5} implies

\begin{itemize}
\item[\refstepcounter{equation}\text{(\theequation)}\label{cz8}] $\vec{\Gamma%
}_{l\left( \mathcal{A}\right) -3}$ has a weak $\left( \hat{\mathcal{A}}%
,\epsilon ^{-1}\delta _{Q^{+}},A\right) $-basis at $\left( y,M_{0},\hat{P}%
^{y}\right) $
\end{itemize}

because $\epsilon ^{-1}\delta _{Q^{+}}<\epsilon ^{-1}\delta _{Q^{\prime }}$,
and because \eqref{cz5}, \eqref{cz8} deal with weak bases. Thus, \eqref{cz4}
holds, and either $\#\left( E\cap 5Q^{+}\right) \leq 1$ or else our $\hat{%
\mathcal{A}}<\mathcal{A}$ and $\hat{P}^{y}$ $\left( y\in E\cap 5Q^{+}\right) 
$ satisfy \eqref{6}, \eqref{7}, \eqref{8}. This tells us that $Q^{+}$ is OK.
However, $Q^{+}$ strictly contains the CZ cube $Q$; therefore, $Q^{+}$
cannot be OK. This contradiction completes the proof of Lemma \ref{lemma-cz2}%
.
\end{proof}

Note that the proof of Lemma \ref{lemma-cz2} made use of our decision to
involve $x_0$, $\delta_{Q_0}$ rather than $y$, $\delta_Q$ in (3b), as well
as our decision to use weak bases in (3a).

\begin{lemma}
\label{lemma-cz3} Only finitely many CZ cubes $Q$ satisfy the condition

\begin{itemize}
\item[\refstepcounter{equation}\text{(\theequation)}\label{cz9}] $\frac{65}{%
64}Q \cap \frac{65}{64} Q_0 \not= \emptyset$.
\end{itemize}
\end{lemma}

\begin{proof}
There exists some small positive number $\delta _{\ast }$ such that any
dyadic cube $Q$ satisfying \eqref{cz9} and $\delta _{Q}\leq \delta _{\ast }$
must satisfy also $5Q\subset 5Q_{0}$ and $\#(E\cap 5Q)\leq 1$. (Here we use
the finiteness of $E$.)

Consequently, any CZ cube $Q$ satisfying \eqref{cz9} must have sidelength $\delta
_{Q}\geq \delta _{\ast }$ (and also $\delta _{Q}\leq \delta
_{Q_{0}}$ since $5Q\subset 5Q_{0}$ because $Q$ is OK). There are only
finitely many dyadic cubes $Q$ satisfying both \eqref{cz9} and $\delta
_{\ast }\leq \delta _{Q}\leq \delta _{Q_{0}}$.

The proof of Lemma \ref{lemma-cz3} is complete.
\end{proof}

\section{Auxiliary Polynomials}

\label{auxiliary-polynomials}

We again place ourselves in the setting of Section \ref%
{setup-for-the-induction-step} and we make use of the Calder\'on-Zygmund
decomposition defined in Section \ref{cz-decomposition}.

Recall that $x_0 \in E \cap 5Q_0^+$, and that $\vec{\Gamma}_{l(\mathcal{A})}$
has an $(\mathcal{A}, \epsilon^{-1}\delta_{Q_0}, C_B)$-basis at $%
(x_0,M_0,P^0) $; moreover, $\mathcal{A} \subseteq \mathcal{M}$ is monotonic,
and $\epsilon$ is less than a small enough constant determined by $C_B$, $C_w
$, $m$, $n$.

Let $y\in E\cap 5Q_{0}$. Then $|x_{0}-y|\leq C\delta _{Q_{0}}=(C\epsilon
)(\epsilon ^{-1}\delta _{Q_{0}})$. Hence, by Corollary \ref{cor-to-transport}
in Section \ref{transport-lemma}, there exists $P^{y}\in \mathcal{P}$ with
the following properties.

\begin{itemize}
\item[\refstepcounter{equation}\text{(\theequation)}\label{ap1}] $\vec{\Gamma%
}_{l\left( \mathcal{A}\right) -1}$ has an $\left( \mathcal{A},\epsilon
^{-1}\delta _{Q_{0}},C\right) $-basis $\left( P_{\alpha }^{y}\right)
_{\alpha \in \mathcal{A}}$ at $\left( y,M_{0},P^{y}\right) $,
\end{itemize}

\begin{itemize}
\item[\refstepcounter{equation}\text{(\theequation)}\label{ap2}] $\partial
^{\beta }\left( P^{y}-P^{0}\right) \equiv 0$ for $\beta \in \mathcal{A}$,
\end{itemize}

\begin{itemize}
\item[\refstepcounter{equation}\text{(\theequation)}\label{ap3}] $\left\vert
\partial ^{\beta }\left( P^{y}-P^{0}\right) \left( x_{0}\right) \right\vert
\leq CM_{0}\left( \epsilon ^{-1}\delta _{Q_{0}}\right) ^{m-\left\vert \beta
\right\vert }$ for $\beta \in \mathcal{M}$.
\end{itemize}

We fix $P^{y},P_{\alpha }^{y}$ $\left( \alpha \in \mathcal{A}\right) $ as
above for each $y\in E\cap 5Q_{0}$. We study the relationship between the
polynomials $P^{y},P_{\alpha }^{y}$ $(\alpha \in \mathcal{A})$ and the Calder%
\'{o}n-Zygmund decomposition.

\begin{lemma}[``Controlled Auxiliary Polynomials'']
\label{lemma-ap1} Let $Q \in$ CZ, and suppose that

\begin{itemize}
\item[\refstepcounter{equation}\text{(\theequation)}\label{ap4}] $\frac{65}{%
64}Q\cap \frac{65}{64}Q_{0}\not=\emptyset $.
\end{itemize}

Let

\begin{itemize}
\item[\refstepcounter{equation}\text{(\theequation)}\label{ap5}] $y\in E\cap
5Q_{0}\cap 5Q^{+}$.
\end{itemize}

Then

\begin{itemize}
\item[\refstepcounter{equation}\text{(\theequation)}\label{ap6}] $\left\vert
\partial ^{\beta }P_{\alpha }^{y}\left( y\right) \right\vert \leq C\cdot
\left( \epsilon ^{-1}\delta _{Q}\right) ^{\left\vert \alpha \right\vert
-\left\vert \beta \right\vert }$ for $\alpha \in \mathcal{A}$, $\beta \in 
\mathcal{M}$.
\end{itemize}
\end{lemma}

\begin{proof}
Let $K\geq 1$ be a large enough constant to be picked below and assume that

\begin{itemize}
\item[\refstepcounter{equation}\text{(\theequation)}\label{ap7}] $%
\max_{\alpha \in \mathcal{A}\text{,}\beta \in \mathcal{M}}\left( \epsilon
^{-1}\delta _{Q}\right) ^{\left\vert \beta \right\vert -\left\vert \alpha
\right\vert }\left\vert \partial ^{\beta }P_{\alpha }^{y}\left( y\right)
\right\vert >K\text{.} $
\end{itemize}

We will derive a contradiction.

Thanks to \eqref{ap1}, we have

\begin{itemize}
\item[\refstepcounter{equation}\text{(\theequation)}\label{ap8}] $%
P^{y},P^{y}\pm CM_{0}\cdot \left( \epsilon ^{-1}\delta _{Q_{0}}\right)
^{m-\left\vert \alpha \right\vert }P_{\alpha }^{y}\in \Gamma _{l\left( 
\mathcal{A}\right) -1}\left( y,CM_{0}\right) $ for $\alpha \in \mathcal{A}$,
\end{itemize}

\begin{itemize}
\item[\refstepcounter{equation}\text{(\theequation)}\label{ap9}] $\partial
^{\beta }P_{\alpha }^{y}\left( y\right) =\delta _{\beta \alpha }$ for $\beta
,\alpha \in \mathcal{A}$,
\end{itemize}

and

\begin{itemize}
\item[\refstepcounter{equation}\text{(\theequation)}\label{ap10}] $%
\left\vert \partial ^{\beta }P_{\alpha }^{y}\left( y\right) \right\vert \leq
C\left( \epsilon ^{-1}\delta _{Q_{0}}\right) ^{\left\vert \alpha \right\vert
-\left\vert \beta \right\vert }$ for $\alpha \in \mathcal{A}$, $\beta \in 
\mathcal{M}$.
\end{itemize}

Also,

\begin{itemize}
\item[\refstepcounter{equation}\text{(\theequation)}\label{ap11}] $5Q\subset
5Q_{0}$ since $Q$ is OK.
\end{itemize}

If $\delta _{Q}\geq 2^{-12}\delta _{Q_{0}}$, then from \eqref{ap10}, %
\eqref{ap11}, we would have

\begin{itemize}
\item[\refstepcounter{equation}\text{(\theequation)}\label{ap14}] $%
\max_{\alpha \in \mathcal{A}\text{,}\beta \in \mathcal{M}}\left( \epsilon
^{-1}\delta _{Q}\right) ^{\left\vert \beta \right\vert -\left\vert \alpha
\right\vert }\left\vert \partial ^{\beta }P_{\alpha }^{y}\left( y\right)
\right\vert \leq C^{\prime }$.
\end{itemize}

We will pick

\begin{itemize}
\item[\refstepcounter{equation}\text{(\theequation)}\label{ap15}] $%
K>C^{\prime }$, with $C^{\prime }$ as in \eqref{ap14}.
\end{itemize}

Then \eqref{ap14} contradicts our assumption \eqref{ap7}.

Thus, we must have

\begin{itemize}
\item[\refstepcounter{equation}\text{(\theequation)}\label{ap16}] $\delta
_{Q}<2^{-12}\delta _{Q_{0}}$.
\end{itemize}

Let

\begin{itemize}
\item[\refstepcounter{equation}\text{(\theequation)}\label{ap17}] $Q=\hat{Q}%
_{0}\subset \hat{Q}_{1}\subset \cdots \subset \hat{Q}_{\nu _{\max }}$ be all
the dyadic cubes containing $Q$ and having sidelength at most $2^{-10}\delta
_{Q_{0}}$.
\end{itemize}

Then

\begin{itemize}
\item[\refstepcounter{equation}\text{(\theequation)}\label{ap18}] $\hat{Q}%
_{0}=Q$, $\delta _{\hat{Q}_{\nu _{\max }}}=2^{-10}\delta _{Q_{0}}$, $\hat{Q}%
_{\nu +1}=\left( \hat{Q}_{\nu }\right) ^{+}$ for $0\leq \nu \leq \nu _{\max
}-1$, and $\nu _{\max }\geq 2$.
\end{itemize}

For $1\leq \nu \leq \nu _{\max }$, we define

\begin{itemize}
\item[\refstepcounter{equation}\text{(\theequation)}\label{ap19}] $X_{\nu
}=\max_{\alpha \in \mathcal{A}\text{,}\beta \in \mathcal{M}}\left( \epsilon
^{-1}\delta _{\hat{Q}_{\nu }}\right) ^{\left\vert \beta \right\vert
-\left\vert \alpha \right\vert }\left\vert \partial ^{\beta }P_{\alpha
}^{y}\left( y\right) \right\vert $.
\end{itemize}

From \eqref{ap7} and \eqref{ap10}, we have

\begin{itemize}
\item[\refstepcounter{equation}\text{(\theequation)}\label{ap20}] $X_{0}>K$, 
$X_{\nu _{\max }}\leq C^{\prime }$,
\end{itemize}

and from \eqref{ap18}, \eqref{ap19}, we have

\begin{itemize}
\item[\refstepcounter{equation}\text{(\theequation)}\label{ap21}] $%
2^{-m}X_{\nu }\leq X_{\nu +1}\leq 2^{m}X_{\nu }$, for $0\leq \nu \leq \nu
_{\max }$.
\end{itemize}

We will pick

\begin{itemize}
\item[\refstepcounter{equation}\text{(\theequation)}\label{ap22}] $%
K>C^{\prime }$ with $C^{\prime }$ as in \eqref{ap20}.
\end{itemize}

Then $\tilde{\nu}:=\min \left\{ \nu :X_{\nu }\leq K\right\} $ and $\tilde{Q}=%
\hat{Q}_{\tilde{\nu}}$ satisfy the following, thanks to \eqref{ap20}, %
\eqref{ap21}, \eqref{ap22}: $\tilde{\nu}\not=0$, hence

\begin{itemize}
\item[\refstepcounter{equation}\text{(\theequation)}\label{ap23}] $\tilde{Q}$
is a dyadic cube strictly containing $Q$; also $2^{-m}K\leq X_{\tilde{\nu}%
}\leq K$,
\end{itemize}

hence

\begin{itemize}
\item[\refstepcounter{equation}\text{(\theequation)}\label{ap24}] $%
2^{-m}K\leq \max_{\alpha \in \mathcal{A}\text{,}\beta \in \mathcal{M}}\left(
\epsilon ^{-1}\delta _{\tilde{Q}}\right) ^{\left\vert \beta \right\vert
-\left\vert \alpha \right\vert }\left\vert \partial ^{\beta }P_{\alpha
}^{y}\left( y\right) \right\vert \leq K$.
\end{itemize}

Also, since $Q\subset \tilde{Q}$, we have $\frac{65}{64}\tilde{Q}\cap \frac{%
65}{64}Q_{0}\not=\emptyset $ by \eqref{ap4}; and since $\delta _{\tilde{Q}%
}\leq 2^{-10}\delta _{Q_{0}}$, we conclude that

\begin{itemize}
\item[\refstepcounter{equation}\text{(\theequation)}\label{ap25}] $5\tilde{Q}%
\subset 5Q_{0}$.
\end{itemize}

From \eqref{ap8}, \eqref{ap10}, and \eqref{ap25}, we have

\begin{itemize}
\item[\refstepcounter{equation}\text{(\theequation)}\label{ap26}] $%
P^{y},P^{y}\pm cM_{0}\left( \epsilon ^{-1}\delta _{\tilde{Q}}\right)
^{m-\left\vert \alpha \right\vert }P_{\alpha }^{y}\in \Gamma _{l\left( 
\mathcal{A}\right) -1}\left( y,CM_{0}\right) \subset \Gamma _{l\left( 
\mathcal{A}\right) -2}\left( y,CM_{0}\right) $ for $\alpha \in \mathcal{A}$;
\end{itemize}

and

\begin{itemize}
\item[\refstepcounter{equation}\text{(\theequation)}\label{ap27}] $%
\left\vert \partial ^{\beta }P_{\alpha }^{y}\left( y\right) \right\vert \leq
C\left( \epsilon ^{-1}\delta _{\tilde{Q}}\right) ^{\left\vert \alpha
\right\vert -\left\vert \beta \right\vert }$ for $\alpha \in \mathcal{A}$, $%
\beta \in \mathcal{M}$, $\beta \geq \alpha $.
\end{itemize}

Our results \eqref{ap9}, \eqref{ap26}, \eqref{ap27} tell us that

\begin{itemize}
\item[\refstepcounter{equation}\text{(\theequation)}\label{ap28}] $\left(
P_{\alpha }^{y}\right) _{\alpha \in \mathcal{A}}$ is a weak $\left( \mathcal{%
A},\epsilon ^{-1}\delta _{\tilde{Q}},C\right) $-basis for $\vec{\Gamma}%
_{l\left( \mathcal{A}\right) -2}$ at $\left( y,M_{0},P^{y}\right) $.
\end{itemize}

Note also that

\begin{itemize}
\item[\refstepcounter{equation}\text{(\theequation)}\label{ap29}] $\epsilon
^{-1}\delta _{\tilde{Q}}\leq \epsilon ^{-1}\delta _{Q_{0}}\leq \delta _{\max
}$, by \eqref{ap25} and hypothesis (A2) of the Main Lemma for $\mathcal{A}$.
\end{itemize}

Moreover,

\begin{itemize}
\item[\refstepcounter{equation}\text{(\theequation)}\label{ap30}] $\vec{%
\Gamma}_{l\left( \mathcal{A}\right) -2}$ is $\left( C,\delta _{\max }\right) 
$-convex, thanks to Lemma \ref{lemma-wsf4} (B).
\end{itemize}

If we take

\begin{itemize}
\item[\refstepcounter{equation}\text{(\theequation)}\label{ap31}] $K\geq
C^{\ast }$ for a large enough $C^{\ast }$,
\end{itemize}

then \eqref{ap24}, \eqref{ap28}$\cdots $\eqref{ap31} and the Relabeling
Lemma (Lemma \ref{lemma-pb2}) produce a monotonic set $\hat{\mathcal{A}}%
\subset \mathcal{M}$, such that

\begin{itemize}
\item[\refstepcounter{equation}\text{(\theequation)}\label{ap32}] $\hat{%
\mathcal{A}}<\mathcal{A}$ (strict inequality)
\end{itemize}

and

\begin{itemize}
\item[\refstepcounter{equation}\text{(\theequation)}\label{ap33}] $\vec{%
\Gamma}_{l\left( \mathcal{A}\right) -2}$ has an $\left( \hat{\mathcal{A}}%
,\epsilon ^{-1}\delta _{\tilde{Q}},C\right) $-basis at $\left(
y,M_{0},P^{y}\right) $.
\end{itemize}

Also, from \eqref{ap9}, \eqref{ap24}, \eqref{ap26}, we see that

\begin{itemize}
\item[\refstepcounter{equation}\text{(\theequation)}\label{ap34}] $\left(
P_{\alpha }^{y}\right) _{\alpha \in \mathcal{A}}$ is an $\left( \mathcal{A}%
,\epsilon ^{-1}\delta _{\tilde{Q}},CK\right) $-basis for $\vec{\Gamma}%
_{l\left( \mathcal{A}\right) -2}$ at $\left( y,M_{0},P^{y}\right) $.
\end{itemize}

We now pick

\begin{itemize}
\item[\refstepcounter{equation}\text{(\theequation)}\label{ap35}] $K=\hat{C}$
(a constant determined by $C_{B}$, $C_{w}$, $m$, $n$), with $\hat{C}\geq 1 $
large enough to satisfy \eqref{ap15}, \eqref{ap22}, \eqref{ap31}.
\end{itemize}

Then \eqref{ap33} and \eqref{ap34} tell us that

\begin{itemize}
\item[\refstepcounter{equation}\text{(\theequation)}\label{ap36}] $\vec{%
\Gamma}_{l\left( \mathcal{A}\right) -2}$ has both an $\left( \hat{\mathcal{A}%
},\epsilon ^{-1}\delta _{\tilde{Q}},C\right) $-basis and an $\left( \mathcal{%
A},\epsilon ^{-1}\delta _{\tilde{Q}},C\right) $-basis at $\left(
y,M_{0},P^{y}\right) $.
\end{itemize}

Let $z\in E\cap 5\tilde{Q}$. Then $z,y\in 5\tilde{Q}^{+}$, hence

\begin{itemize}
\item[\refstepcounter{equation}\text{(\theequation)}\label{ap37}] $%
\left\vert z-y\right\vert \leq C\delta _{\tilde{Q}}=C\epsilon \cdot \left(
\epsilon ^{-1}\delta _{\tilde{Q}}\right) $.
\end{itemize}

From \eqref{ap36}, \eqref{ap37}, the Small $\epsilon $ Assumption and Lemma %
\ref{lemma-transport} (and our hypothesis that $\mathcal{A}$ is monotonic;
see Section \ref{setup-for-the-induction-step}), we obtain a polynomial $%
\check{P}^{z}\in \mathcal{P}$, such that

\begin{itemize}
\item[\refstepcounter{equation}\text{(\theequation)}\label{ap38}] $\vec{\Gamma}
_{l\left( \mathcal{A}\right) -3}$ has an $\left( \hat{\mathcal{A}},\epsilon
^{-1}\delta _{\tilde{Q}},C\right) $-basis at $\left( z,M_{0},\check{P}%
^{z}\right) $,
\end{itemize}

\begin{itemize}
\item[\refstepcounter{equation}\text{(\theequation)}\label{ap39}] $\partial
^{\beta }\left( \check{P}^{z}-P^{y}\right) \equiv 0$ for $\beta \in \mathcal{%
A}$,
\end{itemize}

and

\begin{itemize}
\item[\refstepcounter{equation}\text{(\theequation)}\label{ap40}] $%
\left\vert \partial ^{\beta }\left( \check{P}^{z}-P^{y}\right) \left(
y\right) \right\vert \leq CM_{0}\left( \epsilon ^{-1}\delta _{\tilde{Q}%
}\right) ^{m-\left\vert \beta \right\vert }$ for $\beta \in \mathcal{M}$.
\end{itemize}

From \eqref{ap25} and \eqref{ap40}, we have

\begin{itemize}
\item[\refstepcounter{equation}\text{(\theequation)}\label{ap41}] $%
\left\vert \partial ^{\beta }\left( \check{P}^{z}-P^{y}\right) \left(
y\right) \right\vert \leq CM_{0}\left( \epsilon ^{-1}\delta _{Q_{0}}\right)
^{m-\left\vert \beta \right\vert }$ for $\beta \in \mathcal{M}$.
\end{itemize}

Since $y\in 5Q_{0}$ by hypothesis of Lemma \ref{lemma-ap1}, while $x_{0}\in
5Q_{0}^{+}$ by hypothesis of the Main Lemma for $\mathcal{A}$, we have $%
\left\vert x_{0}-y\right\vert \leq C\delta _{Q_{0}}$, and therefore %
\eqref{ap41} implies that

\begin{itemize}
\item[\refstepcounter{equation}\text{(\theequation)}\label{ap42}] $%
\left\vert \partial ^{\beta }\left( \check{P}^{z}-P^{y}\right) \left(
x_{0}\right) \right\vert \leq CM_{0}\left( \epsilon ^{-1}\delta
_{Q_{0}}\right) ^{m-\left\vert \beta \right\vert }$ for $\beta \in \mathcal{M%
}$.
\end{itemize}

From \eqref{ap2}, \eqref{ap3}, \eqref{ap39}, \eqref{ap42}, we now have

\begin{itemize}
\item[\refstepcounter{equation}\text{(\theequation)}\label{ap43}] $\partial
^{\beta }\left( \check{P}^{z}-P^{0}\right) \equiv 0$ for $\beta \in \mathcal{%
A}$
\end{itemize}

and

\begin{itemize}
\item[\refstepcounter{equation}\text{(\theequation)}\label{ap44}] $%
\left\vert \partial^{\beta }\left( \check{P}^{z}-P^{0}\right) \left(
x_{0}\right) \right\vert \leq CM_{0}\left( \epsilon ^{-1}\delta
_{Q_{0}}\right) ^{m-\left\vert \beta \right\vert }$ for $\beta \in \mathcal{M%
}$.
\end{itemize}

Our results \eqref{ap38}, \eqref{ap43}, \eqref{ap44} hold for every $z\in
E\cap 5\tilde{Q}$.

We recall the Large $A$ Assumption in the Section \ref%
{setup-for-the-induction-step}. Then \eqref{ap25}, \eqref{ap32}, \eqref{ap38}%
, \eqref{ap43}, \eqref{ap44} yield the following results: $5\tilde{Q}\subset
5Q_{0}$, $\hat{\mathcal{A}}<\mathcal{A}$ (strict inequality).

For every $z\in E\cap 5\tilde{Q}$, there exists $\check{P}^{z}\in \mathcal{P}
$ such that

\begin{itemize}
\item $\vec{\Gamma}_{l\left( \mathcal{A}\right) -3}$ has an $\left( \hat{%
\mathcal{A}},\epsilon ^{-1}\delta _{\tilde{Q}},A\right) $-basis at $\left(
z,M_{0},\check{P}^{z}\right) $.

\item $\partial ^{\beta }\left( \check{P}^{z}-P^{0}\right) \equiv 0$ for $%
\beta \in \mathcal{A}$.

\item $\left\vert \partial ^{\beta }\left( \check{P}^{z}-P^{0}\right) \left(
x_{0}\right) \right\vert \leq AM_{0}\left( \epsilon ^{-1}\delta
_{Q_{0}}\right) ^{m-\left\vert \beta \right\vert }$ for $\beta \in \mathcal{M%
}$.
\end{itemize}

Comparing the above results with the definition of an OK cube, we conclude
that $\tilde{Q}$ is OK.

However, since $\tilde{Q}$ properly contains the CZ cube $Q$, (see %
\eqref{ap23}), $\tilde{Q}$ cannot be OK.

This contradiction proves that our assumption \eqref{ap7} must be false.

Thus, $\left\vert \partial ^{\beta }P_{\alpha }^{y}\left( y\right)
\right\vert \leq K\left( \epsilon ^{-1}\delta _{Q}\right) ^{\left\vert
\alpha \right\vert -\left\vert \beta \right\vert }$ for $\alpha \in \mathcal{%
A}$, $\beta \in \mathcal{M}$.

Since we picked $K=\hat{C}$ in \eqref{ap35}, this implies the estimate %
\eqref{ap6}, completing the proof of Lemma \ref{lemma-ap1}.
\end{proof}

\begin{corollary}
\label{cor-to-lemma-ap1}

Let $Q\in $ CZ, and suppose $\frac{65}{64}Q\cap \frac{65}{64}%
Q_{0}\not=\emptyset $. Let $y\in E\cap 5Q_{0}\cap 5Q^{+}$. Then $\left(
P_{\alpha }^{y}\right) _{\alpha \in \mathcal{A}}$ is an $\left( \mathcal{A}%
,\epsilon ^{-1}\delta _{Q},C\right) $-basis for $\vec{\Gamma}_{l\left( 
\mathcal{A}\right) -1}$ at $\left( y,M_{0},P^{y}\right) $.
\end{corollary}

\begin{proof}
From \eqref{ap1} we have

\begin{itemize}
\item[\refstepcounter{equation}\text{(\theequation)}\label{ap45}] $%
P^{y},P^{y}\pm cM_{0}\left( \epsilon ^{-1}\delta _{Q_{0}}\right)
^{m-\left\vert \alpha \right\vert }P_{\alpha }\in \Gamma _{l\left( \mathcal{A%
}\right) -1}\left( y,CM_{0}\right) $ for $\alpha \in \mathcal{A}$;
\end{itemize}

and

\begin{itemize}
\item[\refstepcounter{equation}\text{(\theequation)}\label{ap46}] $\partial
^{\beta }P_{\alpha }^{y}\left( y\right) =\delta _{\beta \alpha }$ for $\beta
,\alpha \in \mathcal{A}$.
\end{itemize}

Since $5Q \subseteq 5Q_0$ (because $Q$ is OK), we have $\delta_Q \leq
\delta_{Q_0}$, and \eqref{ap45} implies

\begin{itemize}
\item[\refstepcounter{equation}\text{(\theequation)}\label{ap47}] $%
P^{y},P^{y}\pm cM_{0}\left( \epsilon ^{-1}\delta _{Q}\right) ^{m-\left\vert
\alpha \right\vert }P_{\alpha }\in \Gamma _{l\left( \mathcal{A}\right)
-1}\left( y,CM_{0}\right) $ for $\alpha \in \mathcal{A}$.
\end{itemize}

Lemma \ref{lemma-ap1} tells us that

\begin{itemize}
\item[\refstepcounter{equation}\text{(\theequation)}\label{ap48}] $%
\left\vert \partial ^{\beta }P_{\alpha }^{y}\left( y\right) \right\vert \leq
C\left( \epsilon ^{-1}\delta _{Q}\right) ^{\left\vert \alpha \right\vert
-\left\vert \beta \right\vert }$ for $\alpha \in \mathcal{A}$, $\beta \in 
\mathcal{M}$.
\end{itemize}

From \eqref{ap46}, \eqref{ap47}, \eqref{ap48}, we conclude that $%
(P^y_\alpha)_{\alpha \in \mathcal{A}}$ is an $(\mathcal{A}%
,\epsilon^{-1}\delta_Q,C)$-basis for $\vec{\Gamma}_{l(\mathcal{A})-1}$ at $%
(y,M_0,P^y)$, completing the proof of Corollary \ref{cor-to-lemma-ap1}.
\end{proof}

\begin{lemma}[\textquotedblleft Consistency of Auxiliary
Polynomials\textquotedblright ]
\label{lemma-ap2} Let $Q,Q^{\prime }\in $ CZ, with

\begin{itemize}
\item[\refstepcounter{equation}\text{(\theequation)}\label{ap49}] $\frac{65}{%
64}Q\cap \frac{65}{64}Q_{0}\not=\emptyset $, $\frac{65}{64}Q^{\prime }\cap 
\frac{65}{64}Q_{0}\not=\emptyset $
\end{itemize}

and

\begin{itemize}
\item[\refstepcounter{equation}\text{(\theequation)}\label{ap50}] $\frac{65}{%
64}Q\cap \frac{65}{64}Q^{\prime }\not=\emptyset $.
\end{itemize}

Let

\begin{itemize}
\item[\refstepcounter{equation}\text{(\theequation)}\label{ap51}] $y\in
E\cap 5Q_{0}\cap 5Q^{+}$, $y^{\prime }\in E\cap 5Q_{0}\cap 5\left( Q^{\prime
}\right) ^{+}$.
\end{itemize}

Then

\begin{itemize}
\item[\refstepcounter{equation}\text{(\theequation)}\label{ap52}] $%
\left\vert \partial ^{\beta }\left( P^{y}-P^{y^{\prime }}\right) \left(
y^{\prime }\right) \right\vert \leq CM_{0}\left( \epsilon ^{-1}\delta
_{Q}\right) ^{m-\left\vert \beta \right\vert }$ for $\beta \in \mathcal{M}$.
\end{itemize}
\end{lemma}

\begin{proof}
Suppose first that $\delta _{Q}\geq 2^{-20}\delta _{Q_{0}}$. Then \eqref{ap3}
(applied to $y$ and to $y^{\prime }$) tells us that 
\begin{equation*}
\left\vert \partial ^{\beta }\left( P^{y}-P^{y^{\prime }}\right) \left(
x_{0}\right) \right\vert \leq CM_{0}\left( \epsilon ^{-1}\delta
_{Q_{0}}\right) ^{m-\left\vert \beta \right\vert }\text{ for }\beta \in 
\mathcal{M}\text{.}
\end{equation*}%
Hence, $\left\vert \partial ^{\beta }\left( P^{y}-P^{y^{\prime }}\right)
\left( y^{\prime }\right) \right\vert \leq C^{\prime }M_{0}\left( \epsilon
^{-1}\delta _{Q_{0}}\right) ^{m-\left\vert \beta \right\vert }\leq C^{\prime
\prime }M_{0}\left( \epsilon ^{-1}\delta _{Q}\right) ^{m-\left\vert \beta
\right\vert }$ for $\beta \in \mathcal{M}$, since $x_{0}$, $y^{\prime }\in
5Q_{0}^{+}$. Thus, \eqref{ap52} holds if $\delta _{Q}\geq 2^{-20}\delta
_{Q_{0}}$. Suppose

\begin{itemize}
\item[\refstepcounter{equation}\text{(\theequation)}\label{ap53}] $\delta
_{Q}<2^{-20}\delta _{Q_{0}}$.
\end{itemize}

By \eqref{ap50} and Lemma \ref{lemma-cz2}, we have

\begin{itemize}
\item[\refstepcounter{equation}\text{(\theequation)}\label{ap54}] $\delta
_{Q},\delta _{Q^{\prime }}\leq 2^{-20}\delta _{Q_{0}}$ and $\frac{1}{2}%
\delta _{Q}\leq \delta _{Q^{\prime }}\leq 2\delta _{Q}$.
\end{itemize}

Together with \eqref{ap49}, this implies that

\begin{itemize}
\item[\refstepcounter{equation}\text{(\theequation)}\label{ap55}] $%
5Q^{+},5\left( Q^{\prime }\right) ^{+}\subseteq 5Q_{0}$.
\end{itemize}

From Corollary \ref{cor-to-lemma-ap1}, we have

\begin{itemize}
\item[\refstepcounter{equation}\text{(\theequation)}\label{ap56}] $\vec{%
\Gamma}_{l\left( \mathcal{A}\right) -1}$ has an $\left( \mathcal{A},\epsilon
^{-1}\delta _{Q^{\prime }},C\right) $-basis at $\left( y^{\prime
},M_{0},P^{y^{\prime }}\right) $.
\end{itemize}

From \eqref{ap50}, \eqref{ap51}, \eqref{ap54}, we have

\begin{itemize}
\item[\refstepcounter{equation}\text{(\theequation)}\label{ap57}] $%
\left\vert y-y^{\prime }\right\vert \leq C\delta _{Q^{\prime }}=C\epsilon
\left( \epsilon ^{-1}\delta _{Q^{\prime }}\right) $.
\end{itemize}

We recall from \eqref{ap54} and the hypotheses of the Main Lemma for $%
\mathcal{A}$ that

\begin{itemize}
\item[\refstepcounter{equation}\text{(\theequation)}\label{ap58}] $\epsilon
^{-1}\delta _{Q^{\prime }}\leq \epsilon ^{-1}\delta _{Q_{0}}\leq \delta
_{\max }$,
\end{itemize}

and we recall from Section \ref{setup-for-the-induction-step} that

\begin{itemize}
\item[\refstepcounter{equation}\text{(\theequation)}\label{ap59}] $\mathcal{A%
}$ is monotonic.
\end{itemize}

Thanks to \eqref{ap56}$\cdots $\eqref{ap59}, Corollary \ref{cor-to-transport}
in Section \ref{transport-lemma} produces a polynomial $P^{\prime }\in 
\mathcal{P}$ such that

\begin{itemize}
\item[\refstepcounter{equation}\text{(\theequation)}\label{ap60}] $\vec{%
\Gamma}_{l\left( \mathcal{A}\right) -2}$ has an $\left( \mathcal{A},\epsilon
^{-1}\delta _{Q^{\prime }},C\right) $-basis at $\left( y,M_{0},P^{\prime
}\right) $;
\end{itemize}

\begin{itemize}
\item[\refstepcounter{equation}\text{(\theequation)}\label{ap61}] $\partial
^{\beta }\left( P^{\prime }-P^{y^{\prime }}\right) \equiv 0$ for $\beta \in 
\mathcal{A}$;
\end{itemize}

and

\begin{itemize}
\item[\refstepcounter{equation}\text{(\theequation)}\label{ap62}] $%
\left\vert \partial ^{\beta }\left( P^{\prime }-P^{y^{\prime }}\right)
\left( y^{\prime }\right) \right\vert \leq CM_{0}\left( \epsilon ^{-1}\delta
_{Q^{\prime }}\right) ^{m-\left\vert \beta \right\vert }$ for $\beta \in 
\mathcal{M}$.
\end{itemize}

From \eqref{ap60} we have in particular that

\begin{itemize}
\item[\refstepcounter{equation}\text{(\theequation)}\label{ap63}] $P^{\prime
}\in \Gamma _{l\left( \mathcal{A}\right) -2}\left( y,CM_{0}\right) $,
\end{itemize}

and from \eqref{ap62} and \eqref{ap54} we obtain

\begin{itemize}
\item[\refstepcounter{equation}\text{(\theequation)}\label{ap64}] $%
\left\vert \partial ^{\beta }\left( P^{y^{\prime }}-P^{\prime }\right)
\left( y^{\prime }\right) \right\vert \leq CM_{0}\left( \epsilon ^{-1}\delta
_{Q}\right) ^{m-\left\vert \beta \right\vert }$ for $\beta \in \mathcal{M}$.
\end{itemize}

If we knew that

\begin{itemize}
\item[\refstepcounter{equation}\text{(\theequation)}\label{ap65}] $%
\left\vert \partial ^{\beta }\left( P^{y}-P^{\prime }\right) \left( y\right)
\right\vert \leq M_{0}\left( \epsilon ^{-1}\delta _{Q}\right) ^{m-\left\vert
\beta \right\vert }$ for $\beta \in \mathcal{M}$,
\end{itemize}

then also $\left\vert \partial ^{\beta }\left( P^{y}-P^{\prime }\right)
\left( y^{\prime }\right) \right\vert \leq C^{\prime }M_{0}\left( \epsilon
^{-1}\delta _{Q}\right) ^{m-\left\vert \beta \right\vert }$ for $\beta \in 
\mathcal{M}$ since $\left\vert y-y^{\prime }\right\vert \leq C\delta _{Q}$
thanks to \eqref{ap50}, \eqref{ap51}, \eqref{ap54}. Consequently, by %
\eqref{ap64}, we would have $\left\vert \partial ^{\beta }\left(
P^{y^{\prime }}-P^{y}\right) \left( y^{\prime }\right) \right\vert \leq
CM_{0}\left( \epsilon ^{-1}\delta _{Q}\right) ^{m-\left\vert \beta
\right\vert }$ for $\beta \in \mathcal{M}$, which is our desired inequality %
\eqref{ap52}. Thus, Lemma \ref{lemma-ap2} will follow if we can prove %
\eqref{ap65}.

Suppose \eqref{ap65} fails.

Corollary \ref{cor-to-lemma-ap1} shows that $\vec{\Gamma}_{l\left( \mathcal{A%
}\right) -1}$ has an $\left( \mathcal{A},\epsilon ^{-1}\delta _{Q},C\right) $%
-basis at $\left( y,M_{0},P^{y}\right) $. Since $\Gamma _{l\left( \mathcal{A}%
\right) -1}\left( x,M\right) \subset \Gamma _{l\left( \mathcal{A}\right)
-2}\left( x,M\right) $ for all $x\in E$, $M>0$, it follows that

\begin{itemize}
\item[\refstepcounter{equation}\text{(\theequation)}\label{ap66}] $\vec{%
\Gamma}_{l\left( \mathcal{A}\right) -2}$ has an $\left( \mathcal{A},\epsilon
^{-1}\delta _{Q},C\right) $-basis at $\left( y,M_{0},P^{y}\right) $.
\end{itemize}

From \eqref{ap61} and \eqref{ap2} (applied to $y$ and $y^{\prime }$), we see
that

\begin{itemize}
\item[\refstepcounter{equation}\text{(\theequation)}\label{ap67}] $\partial
^{\beta }\left( P^{y}-P^{\prime }\right) \equiv 0$ for $\beta \in \mathcal{A}
$.
\end{itemize}

Since we are assuming that \eqref{ap65} fails, we have

\begin{itemize}
\item[\refstepcounter{equation}\text{(\theequation)}\label{ap68}] $%
\max_{\beta \in \mathcal{M}}\left( \epsilon ^{-1}\delta _{Q}\right)
^{\left\vert \beta \right\vert }\left\vert \partial ^{\beta }\left(
P^{y}-P^{\prime }\right) \left( y\right) \right\vert \geq M_{0}\left(
\epsilon ^{-1}\delta _{Q}\right) ^{m}$.
\end{itemize}

Also, from \eqref{ap54} and the hypotheses of the Main Lemma for $\mathcal{A}
$, we have

\begin{itemize}
\item[\refstepcounter{equation}\text{(\theequation)}\label{ap69}] $\epsilon
^{-1}\delta _{Q}<\epsilon ^{-1}\delta _{Q_{0}}\leq \delta _{\max }$.
\end{itemize}

From Lemma \ref{lemma-wsf4} (B), we know that

\begin{itemize}
\item[\refstepcounter{equation}\text{(\theequation)}\label{ap70}] $\vec{%
\Gamma}_{l\left( \mathcal{A}\right) -2}$ is $\left( C,\delta _{\max }\right) 
$-convex.
\end{itemize}

Our results \eqref{ap63}, \eqref{ap66}$\cdots$\eqref{ap70} and Lemma \ref%
{lemma-pb3} produce a set $\hat{\mathcal{A}}\subseteq \mathcal{M}$ and a
polynomial $\hat{P}\in \mathcal{P}$, with the following properties:

\begin{itemize}
\item[\refstepcounter{equation}\text{(\theequation)}\label{ap71}] $\hat{%
\mathcal{A}}$ is monotonic;
\end{itemize}

\begin{itemize}
\item[\refstepcounter{equation}\text{(\theequation)}\label{ap72}] $\hat{%
\mathcal{A}}<\mathcal{A}$ (strict inequality);
\end{itemize}

\begin{itemize}
\item[\refstepcounter{equation}\text{(\theequation)}\label{ap73}] $\vec{%
\Gamma}_{l\left( \mathcal{A}\right) -2}$ has an $\left( \hat{\mathcal{A}}%
,\epsilon ^{-1}\delta _{Q},C\right) $-basis at $\left( y,M_{0},\hat{P}%
\right) $;
\end{itemize}

\begin{itemize}
\item[\refstepcounter{equation}\text{(\theequation)}\label{ap74}] $\partial
^{\beta }\left( \hat{P}-P^{y}\right) \equiv 0$ for $\beta \in \mathcal{A}$
(recall, $\mathcal{A}$ is monotonic);
\end{itemize}

and

\begin{itemize}
\item[\refstepcounter{equation}\text{(\theequation)}\label{ap75}] $%
\left\vert \partial ^{\beta }\left( \hat{P}-P^{y}\right) \left( y\right)
\right\vert \leq CM\left( \epsilon ^{-1}\delta _{Q}\right) ^{m-\left\vert
\beta \right\vert }$ for $\beta \in \mathcal{M}$.
\end{itemize}

Now let $z\in E\cap 5Q^{+}$. We recall that $\mathcal{A}$ is monotonic, and
that \eqref{ap66}, \eqref{ap67}, \eqref{ap73}, \eqref{ap74}, \eqref{ap75} hold. Moreover,
since $y,z\in 5Q^{+}$, we have $\left\vert y-z\right\vert \leq C\delta
_{Q}=C\epsilon \left( \epsilon ^{-1}\delta _{Q}\right) $. Thanks to the
above remarks and the Small $\epsilon $ Assumption, we may apply Lemma \ref%
{lemma-transport} to produce $\check{P}^{z}\in \mathcal{P}$ satisfying the
following conditions.

\begin{itemize}
\item[\refstepcounter{equation}\text{(\theequation)}\label{ap76}] $\vec{%
\Gamma}_{l\left( \mathcal{A}\right) -3}$ has an $\left( \hat{\mathcal{A}}%
,\epsilon ^{-1}\delta _{Q},C\right) $-basis at $\left( z,M_{0},\check{P}%
^{z}\right) $.
\end{itemize}

\begin{itemize}
\item[\refstepcounter{equation}\text{(\theequation)}\label{ap77}] $\partial
^{\beta }\left( \check{P}^{z}-P^{y}\right) \equiv 0$ for $\beta \in \mathcal{%
A}$.
\end{itemize}

\begin{itemize}
\item[\refstepcounter{equation}\text{(\theequation)}\label{ap78}] $%
\left\vert \partial ^{\beta }\left( \check{P}^{z}-P^{y}\right) \left(
y\right) \right\vert \leq CM_{0}\left( \epsilon ^{-1}\delta _{Q}\right)
^{m-\left\vert \beta \right\vert }$ for $\beta \in \mathcal{M}$.
\end{itemize}

By \eqref{ap76}, and the Large $A$ Assumption,

\begin{itemize}
\item[\refstepcounter{equation}\text{(\theequation)}\label{ap79}] $\vec{%
\Gamma}_{l\left( \mathcal{A}\right) -3}$ has an $\left( \hat{\mathcal{A}}%
,\epsilon ^{-1}\delta _{Q^{+}},A\right) $-basis at $\left( z,M_{0},\check{P}%
^{z}\right) $.
\end{itemize}

By \eqref{ap2} and \eqref{ap77}, we have

\begin{itemize}
\item[\refstepcounter{equation}\text{(\theequation)}\label{ap80}] $\partial
^{\beta }\left( \check{P}^{z}-P^{0}\right) \equiv 0$ for $\beta \in \mathcal{%
A}$.
\end{itemize}

By \eqref{ap54} and \eqref{ap78}, we have $\left\vert \partial ^{\beta
}\left( \check{P}^{z}-P^{y}\right) \left( y\right) \right\vert \leq
CM_{0}\left( \epsilon ^{-1}\delta _{Q_{0}}\right) ^{m-\left\vert \beta
\right\vert }$ for $\beta \in \mathcal{M}$, hence $\left\vert \partial
^{\beta }\left( \check{P}^{z}-P^{y}\right) \left( x_{0}\right) \right\vert
\leq CM_{0}\left( \epsilon ^{-1}\delta _{Q_{0}}\right) ^{m-\left\vert \beta
\right\vert }$ for $\beta \in \mathcal{M}$, since $x,y\in 5Q_{0}^{+}$.
Together with \eqref{ap3} and the Large $A$ Assumption, this yields the
estimate

\begin{itemize}
\item[\refstepcounter{equation}\text{(\theequation)}\label{ap81}] $%
\left\vert \partial ^{\beta }\left( \check{P}^{z}-P^{0}\right) \left(
x_{0}\right) \right\vert \leq AM_{0}\left( \epsilon ^{-1}\delta
_{Q_{0}}\right) ^{m-\left\vert \beta \right\vert }$ for $\beta \in \mathcal{M%
}$.
\end{itemize}

We have proven \eqref{ap79}, \eqref{ap80}, \eqref{ap81} for each $z\in E\cap
5Q^{+}$. Thus, $5Q^{+}\subset 5Q_{0}$ (see \eqref{ap55}), $\hat{\mathcal{A}}<%
\mathcal{A}$ (strict inequality; see \eqref{ap72}), and for each $z\in E\cap
5Q^{+}$ there exists $\check{P}^{z}\in \mathcal{P}$ such that

\begin{itemize}
\item $\vec{\Gamma}_{l\left( \mathcal{A}\right) -3}$ has an $\left( \hat{%
\mathcal{A}},\epsilon ^{-1}\delta _{Q^{+}},A\right) $-basis at $\left(
z,M_{0},\check{P}^{z}\right) $;

\item $\partial ^{\beta }\left( \check{P}^{z}-P^{0}\right) \equiv 0$ for $%
\beta \in \mathcal{A}$; and

\item $\left\vert \partial ^{\beta }\left( \check{P}^{z}-P^{0}\right) \left(
x_{0}\right) \right\vert \leq AM_{0}\left( \epsilon ^{-1}\delta
_{Q_{0}}\right) ^{m-\left\vert \beta \right\vert }$ for $\beta \in \mathcal{M%
}$. (See \eqref{ap79}, \eqref{ap80}, \eqref{ap81}.)
\end{itemize}

Comparing the above results with the definition of an OK cube, we see that $%
Q^{+}$ is OK. On the other hand $Q^{+}$ cannot be OK, since it properly
contains the CZ cube $Q$. Assuming that \eqref{ap65} fails, we have derived
a contradiction. Thus, \eqref{ap65} holds, completing the proof of Lemma \ref%
{lemma-ap2}.
\end{proof}

\section{Good News About CZ Cubes}

\label{gn}

In this section we again place ourselves in the setting of Section \ref%
{setup-for-the-induction-step}, and we make use of the auxiliary polynomials 
$P^{y}$ and the CZ cubes $Q$ defined above.

\begin{lemma}
\label{lemma-gn1} Let $Q \in$ CZ, with

\begin{itemize}
\item[\refstepcounter{equation}\text{(\theequation)}\label{gn1}] $\frac{65}{%
64}Q\cap \frac{65}{64}Q_{0}\not=\emptyset $
\end{itemize}

and

\begin{itemize}
\item[\refstepcounter{equation}\text{(\theequation)}\label{gn2}] $\#\left(
E\cap 5Q\right) \geq 2$.
\end{itemize}

Let

\begin{itemize}
\item[\refstepcounter{equation}\text{(\theequation)}\label{gn3}] $y\in E\cap
5Q$.
\end{itemize}

Then there exist a set $\mathcal{A}^\# \subseteq \mathcal{M}$ and a
polynomial $P^\# \in \mathcal{P}$ with the following properties.

\begin{itemize}
\item[\refstepcounter{equation}\text{(\theequation)}\label{gn4}] $\mathcal{A}%
^\#$ is monotonic.
\end{itemize}

\begin{itemize}
\item[\refstepcounter{equation}\text{(\theequation)}\label{gn5}] $\mathcal{A}%
^\# < \mathcal{A}$ (strict inequality).
\end{itemize}

\begin{itemize}
\item[\refstepcounter{equation}\text{(\theequation)}\label{gn6}] $\vec{\Gamma%
}_{l\left( \mathcal{A}\right) -3}$ has an $\left( \mathcal{A}^{\#},\epsilon
^{-1}\delta _{Q},C\left( A\right) \right) $-basis at $\left(
y,M_{0},P^{\#}\right) $.
\end{itemize}

\begin{itemize}
\item[\refstepcounter{equation}\text{(\theequation)}\label{gn7}] $\left\vert
\partial ^{\beta }\left( P^{\#}-P^{y}\right) \left( y\right) \right\vert
\leq C\left( A\right) M_{0}\left( \epsilon ^{-1}\delta _{Q}\right)
^{m-\left\vert \beta \right\vert }$ for $\beta \in \mathcal{M}$.
\end{itemize}
\end{lemma}

\begin{proof}
Recall that

\begin{itemize}
\item[\refstepcounter{equation}\text{(\theequation)}\label{gn8}] $\partial
^{\beta }\left( P^{y}-P^{0}\right) \equiv 0$ for $\beta \in \mathcal{A}$
(see \eqref{ap2} in Section \ref{auxiliary-polynomials})
\end{itemize}

and that

\begin{itemize}
\item[\refstepcounter{equation}\text{(\theequation)}\label{gn9}] $5Q
\subseteq 5Q_0$, since $Q$ is OK.
\end{itemize}

Thanks to \eqref{gn3} and \eqref{gn9}, Corollary \ref{cor-to-lemma-ap1} in
Section \ref{auxiliary-polynomials} applies, and it tells us that

\begin{itemize}
\item[\refstepcounter{equation}\text{(\theequation)}\label{gn10}] $\vec{%
\Gamma}_{l\left( \mathcal{A}\right) -1}$ has an $\left( \mathcal{A},\epsilon
^{-1}\delta _{Q},C\right) $-basis at $\left( y,M_{0},P^{y}\right) $.
\end{itemize}

On the other hand, $Q$ is OK and $\#(E\cap 5Q) \geq 2$; hence, there exist $%
\hat{\mathcal{A}} \subseteq \mathcal{M}$ and $\hat{P} \in \mathcal{P}$ with
the following properties

\begin{itemize}
\item[\refstepcounter{equation}\text{(\theequation)}\label{gn11}] $\vec{%
\Gamma}_{l\left( \mathcal{A}\right) -3}$ has a weak $\left( \hat{\mathcal{A}}%
,\epsilon ^{-1}\delta _{Q},A\right) $-basis at $\left( y,M_{0},\hat{P}%
\right) $.
\end{itemize}

\begin{itemize}
\item[\refstepcounter{equation}\text{(\theequation)}\label{gn12}] $%
\left\vert \partial ^{\beta }\left( \hat{P}-P^{0}\right) \left( x_{0}\right)
\right\vert \leq AM_{0}\left( \epsilon ^{-1}\delta _{Q_{0}}\right)
^{m-\left\vert \beta \right\vert }$ for $\beta \in \mathcal{M}$.
\end{itemize}

\begin{itemize}
\item[\refstepcounter{equation}\text{(\theequation)}\label{gn13}] $\partial
^{\beta }\left( \hat{P}-P^{0}\right) \equiv 0$ for $\beta \in \mathcal{A}$.
\end{itemize}

\begin{itemize}
\item[\refstepcounter{equation}\text{(\theequation)}\label{gn14}] $\hat{%
\mathcal{A}} < \mathcal{A}$ (strict inequality).
\end{itemize}

We consider separately two cases.

\underline{Case 1:} Suppose that

\begin{itemize}
\item[\refstepcounter{equation}\text{(\theequation)}\label{gn15}] $%
\left\vert \partial ^{\beta }\left( \hat{P}-P^{y}\right) \left( y\right)
\right\vert \leq M_{0}\left( \epsilon ^{-1}\delta _{Q}\right) ^{m-\left\vert
\beta \right\vert }$ for $\beta \in \mathcal{M}$.
\end{itemize}

By Lemma \ref{lemma-wsf4} (B),

\begin{itemize}
\item[\refstepcounter{equation}\text{(\theequation)}\label{gn16}] $\vec{%
\Gamma}_{l\left( \mathcal{A}\right) -3}$ is $\left( C,\delta _{\max }\right) 
$-convex.
\end{itemize}

Also, \eqref{gn9} and hypothesis (A2) of the Main Lemma for $\mathcal{A}$
give

\begin{itemize}
\item[\refstepcounter{equation}\text{(\theequation)}\label{gn17}] $\epsilon
^{-1}\delta _{Q}\leq \epsilon ^{-1}\delta _{Q_{0}}\leq \delta _{\max }$.
\end{itemize}

Applying \eqref{gn11}, \eqref{gn16}, \eqref{17}, and Lemma \ref{lemma-pb2},
we obtain a set $\mathcal{A}^\# \subseteq \mathcal{M}$ such that

\begin{itemize}
\item[\refstepcounter{equation}\text{(\theequation)}\label{gn18}] $\mathcal{A%
}^{\#}\leq \hat{\mathcal{A}}$,
\end{itemize}

\begin{itemize}
\item[\refstepcounter{equation}\text{(\theequation)}\label{gn19}] $
\mathcal{A}^\#$ is monotonic,
\end{itemize}

and

\begin{itemize}
\item[\refstepcounter{equation}\text{(\theequation)}\label{gn20}] $\vec{%
\Gamma}_{l\left( \mathcal{A}\right) -3}$ has an $\left( \mathcal{A}%
^{\#},\epsilon ^{-1}\delta _{Q},C\left( A\right) \right) $-basis at $\left(
y,M_{0},\hat{P}\right) $.
\end{itemize}

Setting $P^\# =\hat{P}$, we obtain the desired conclusions \eqref{gn4}$\cdots
$\eqref{gn7} at once from \eqref{gn14}, \eqref{gn15}, \eqref{gn18}, %
\eqref{gn19}, and \eqref{gn20}.

Thus, Lemma \ref{lemma-gn1} holds in Case 1.

\underline{Case 2:} Suppose that $\left\vert \partial ^{\beta }\left( \hat{P}%
-P^{y}\right) \left( y\right) \right\vert >M_{0}\left( \epsilon ^{-1}\delta
_{Q}\right) ^{m-\left\vert \beta \right\vert }$ for some $\beta \in \mathcal{%
M}$, i.e.,

\begin{itemize}
\item[\refstepcounter{equation}\text{(\theequation)}\label{gn21}] $%
\max_{\beta \in \mathcal{M}}\left( \epsilon ^{-1}\delta _{Q}\right)
^{\left\vert \beta \right\vert }\left\vert \partial ^{\beta }\left( \hat{P}%
-P^{y}\right) \left( y\right) \right\vert >M_{0}\left( \epsilon ^{-1}\delta
_{Q}\right) ^{m}$.
\end{itemize}

From \eqref{gn11} we have

\begin{itemize}
\item[\refstepcounter{equation}\text{(\theequation)}\label{gn22}] $\hat{P}%
\in \Gamma _{l\left( \mathcal{A}\right) -3}\left( y,AM_{0}\right) $
\end{itemize}

Since $\Gamma _{l(\mathcal{A})-1}(x,M)\subseteq \Gamma _{l\left( \mathcal{A}%
\right) -3}\left( x,M\right) $ for all $x\in E,M>0$, \eqref{gn10} implies
that

\begin{itemize}
\item[\refstepcounter{equation}\text{(\theequation)}\label{gn23}] $\vec{%
\Gamma}_{l\left( \mathcal{A}\right) -3}$ has an $\left( \mathcal{A},\epsilon
^{-1}\delta _{Q},C\right) $-basis at $\left( y,M_{0},P^{y}\right) $.
\end{itemize}

As in Case 1,

\begin{itemize}
\item[\refstepcounter{equation}\text{(\theequation)}\label{gn24}] $\vec{%
\Gamma}_{l\left( \mathcal{A}\right) -3}$ is $\left( C,\delta _{\max }\right) 
$-convex,
\end{itemize}

and

\begin{itemize}
\item[\refstepcounter{equation}\text{(\theequation)}\label{gn25}] $\epsilon
^{-1}\delta _{Q}\leq \epsilon ^{-1}\delta _{Q_{0}}\leq \delta _{\max }$.
\end{itemize}

From \eqref{gn8} and \eqref{gn13} we have

\begin{itemize}
\item[\refstepcounter{equation}\text{(\theequation)}\label{gn26}] $\partial
^{\beta }\left( \hat{P}-P^{y}\right) \equiv 0$ for $\beta \in \mathcal{A}$.
\end{itemize}

Thanks to \eqref{21}$\cdots$\eqref{26} and Lemma \ref{lemma-pb3} there exist 
$\mathcal{A}^\# \subseteq \mathcal{M}$ and $P^\# \in \mathcal{P}$ with the
following properties: $\mathcal{A}^\#$ is monotonic; $\mathcal{A}^\# <%
\mathcal{A}$ (strict inequality); $\vec{\Gamma}_{l(\mathcal{A})-3}$ has an $(%
\mathcal{A}^\#,\epsilon^{-1}\delta_{Q},C(A))$-basis at $(y,M_0,P^\#)$; $%
\partial^\beta(P^\#-P^y)\equiv 0$ for $\beta \in \mathcal{A}$; $%
|\partial^\beta (P^\# -P^y)(y)|\leq M_0 (\epsilon^{-1}\delta_Q)^{m-|\beta|}$
for $\beta \in \mathcal{M}$.

Thus, $\mathcal{A}^\#$ and $P^\#$ satisfy \eqref{gn4}$\cdots$\eqref{gn7},
proving Lemma \ref{lemma-gn1} in Case 2.

We have seen that Lemma \ref{lemma-gn1} holds in all cases.
\end{proof}

\begin{remarks}
\begin{itemize} \item The analysis of Case 2 in the proof of Lemma \ref{lemma-gn1} is a new
ingredient, with no analogue in our previous work on Whitney problems.

\item The proof of Lemma \ref{lemma-gn1} gives a $\hat{P}$ that satisfies
also $\partial^\beta(\hat{P}-P^0)\equiv 0$ for $\beta \in \mathcal{A}$, but
we make no use of that.

\item Note that $x_0$ and $\delta_{Q_0}$ appear in \eqref{gn12}, rather than
the desired $y, \delta_Q$. Consequently, \eqref{gn12} is of no help in the
proof of Lemma \ref{lemma-gn1}.
\end{itemize}
\end{remarks}

In the proof of our next result, we use our Induction Hypothesis that the
Main Lemma for $\mathcal{A}^{\prime }$ holds whenever $\mathcal{A}^{\prime }<%
\mathcal{A}$ and $\mathcal{A}^{\prime }$ is monotonic. (See Section \ref%
{setup-for-the-induction-step}.)

\begin{lemma}
\label{lemma-gn2}

Let $Q\in $ CZ. Suppose that

\begin{itemize}
\item[\refstepcounter{equation}\text{(\theequation)}\label{gn27}] $\frac{65}{%
64}Q\cap \frac{65}{64}Q_{0}\not=\emptyset $
\end{itemize}

and

\begin{itemize}
\item[\refstepcounter{equation}\text{(\theequation)}\label{gn28}] $\#\left(
E\cap 5Q\right) \geq 2$.
\end{itemize}

Let

\begin{itemize}
\item[\refstepcounter{equation}\text{(\theequation)}\label{gn29}] $y\in
E\cap 5Q$.
\end{itemize}

Then there exists $F^{y,Q} \in C^m(\frac{65}{64}Q)$ such that

\begin{itemize}
\item[(*1)] $\left\vert \partial ^{\beta }\left( F^{y,Q}-P^{y}\right)
\right\vert \leq C\left( \epsilon \right) M_{0}\delta _{Q}^{m-\left\vert
\beta \right\vert }$ on $\frac{65}{64}Q$, for $\left\vert \beta \right\vert
\leq m$; and

\item[(*2)] $J_{z}\left( F^{y,Q}\right) \in \Gamma _{0}\left( z,C\left(
\epsilon \right) M_{0}\right) $ for all $z\in E\cap \frac{65}{64}Q$.
\end{itemize}
\end{lemma}

\begin{proof}
Our hypotheses \eqref{gn27}, \eqref{gn28}, \eqref{gn29} are precisely the
hypotheses of Lemma \ref{lemma-gn1}. Let $\mathcal{A}^{\#}$, $P^{\#}$
satisfy the conclusions \eqref{gn4}$\cdots $\eqref{gn7} of that Lemma.
Recall the definition of $l(\mathcal{A})$; see \eqref{gn1}, \eqref{gn2} in
Section \ref{statement-of-the-main-lemma}. We have $l(\mathcal{A}^{\#})\leq
l(\mathcal{A})-3$ since $\mathcal{A}^{\#}<\mathcal{A}$; hence \eqref{gn6}
implies that

\begin{itemize}
\item[\refstepcounter{equation}\text{(\theequation)}\label{gn30}] $\vec{%
\Gamma}_{l\left( \mathcal{A}^{\#}\right) }$ has an $\left( \mathcal{A}%
^{\#},\epsilon ^{-1}\delta _{Q},C\left( A\right) \right) $-basis at $\left(
y,M_{0},P^{\#}\right) $.
\end{itemize}

Also, since $Q$ is OK, we have $5Q \subseteq 5Q_0$, hence $\delta_Q \leq
\delta_{Q_0}$. Hence, hypothesis (A2) of the Main Lemma for $\mathcal{A}$
implies that

\begin{itemize}
\item[\refstepcounter{equation}\text{(\theequation)}\label{gn31}] $\epsilon
^{-1}\delta _{Q}\leq \delta _{\max }$.
\end{itemize}

By \eqref{gn4}, \eqref{gn5}, and our Inductive Hypothesis, the Main Lemma
holds for $\mathcal{A}^{\#}$. Thanks to \eqref{gn29}, \eqref{gn30}, %
\eqref{gn31} and the Small $\epsilon $ Assumption in Section \ref%
{setup-for-the-induction-step}, the Main Lemma for $\mathcal{A}^{\#}$ now
yields a function $F\in C^{m}\left( \frac{65}{64}Q\right) $, such that

\begin{itemize}
\item[\refstepcounter{equation}\text{(\theequation)}\label{gn32}] $%
\left\vert \partial ^{\beta }\left( F-P^{\#}\right) \right\vert \leq C\left(
\epsilon \right) M_{0}\delta ^{m-\left\vert \beta \right\vert }$ on $\frac{65%
}{64}Q$, for $\left\vert \beta \right\vert \leq m$; and
\end{itemize}

\begin{itemize}
\item[\refstepcounter{equation}\text{(\theequation)}\label{gn33}] $%
J_{z}\left( F\right) \in \Gamma _{0}\left( z,C\left( \epsilon \right)
M_{0}\right) $ for all $z\in E\cap \frac{65}{64}Q$.
\end{itemize}

Thanks to conclusion \eqref{gn7} of Lemma \ref{lemma-gn1}; (together with %
\eqref{gn29}), we have also

\begin{itemize}
\item[\refstepcounter{equation}\text{(\theequation)}\label{gn34}] $%
\left\vert \partial ^{\beta }\left( P^{\#}-P^{y}\right) \right\vert \leq
C\left( \epsilon \right) M_{0}\delta _{Q}^{m-\left\vert \beta \right\vert }$
on $\frac{65}{64}Q$ for $\left\vert \beta \right\vert \leq m$.
\end{itemize}

(Recall that $P^\#-P^y$ is a polynomial of degree at most $m-1$.) Taking $%
F^{y,Q}=F$, we may read off the desired conclusions (*1) and (*2) from %
\eqref{gn32}, \eqref{gn33}, \eqref{gn34}.

The proof of Lemma \ref{lemma-gn2} is complete.
\end{proof}

\section{Local Interpolants}

In this section, we again place ourselves in the setting of Section \ref%
{setup-for-the-induction-step}. We make use of the Calder\'{o}n-Zygmund
cubes $Q$ and the auxiliary polynomials $P^{y}$ defined above. Let

\begin{itemize}
\item[\refstepcounter{equation}\text{(\theequation)}\label{li0}] $\mathcal{Q}%
=\left\{ Q\in CZ:\frac{65}{64}Q\cap \frac{65}{64}Q_{0}\not=\emptyset
\right\} $.
\end{itemize}

For each $Q\in \mathcal{Q}$, we define a function $F^{Q}\in C^{m}\left( 
\frac{65}{64}Q\right) $ and a polynomial $P^{Q}\in \mathcal{P}$. We proceed
by cases. We say that $Q \in \mathcal{Q}$ is

\begin{description}
\item[Type 1] if $\# (E \cap 5Q) \geq 2$,

\item[Type 2] if $\# (E \cap 5Q) = 1$,

\item[Type 3] if $\#(E\cap 5Q)=0$ and $\delta _{Q}\leq \frac{1}{1024}\delta
_{Q_{0}}$, and

\item[Type 4] if $\#(E\cap 5Q)=0$ and $\delta _{Q}>\frac{1}{1024}\delta
_{Q_{0}}$.
\end{description}

\underline{If $Q$ is of Type 1}, then we pick a point $y_{Q}\in E\cap 5Q$,
and set $P^{Q}=P^{y_{Q}}$. Applying Lemma \ref{lemma-gn2}, we obtain a
function $F^{Q}\in C^{m}\left( \frac{65}{64}Q\right) $ such that

\begin{itemize}
\item[\refstepcounter{equation}\text{(\theequation)}\label{li1}] $\left\vert
\partial ^{\beta }\left( F^{Q}-P^{Q}\right) \right\vert \leq C\left(
\epsilon \right) M_{0}\delta _{Q}^{m-\left\vert \beta \right\vert }$ on $%
\frac{65}{64}Q$, for $\left\vert \beta \right\vert \leq m$; and
\end{itemize}

\begin{itemize}
\item[\refstepcounter{equation}\text{(\theequation)}\label{li2}] $%
J_{z}\left( F^{Q}\right) \in \Gamma _{0}\left( z,C\left( \epsilon \right)
M_{0}\right) $ for all $z\in E\cap \frac{65}{64}Q$.
\end{itemize}

\underline{If $Q$ is of Type 2}, then we let $y_{Q}$ be the one and only
point of $E\cap 5Q$, and define $F^{Q}=P^{Q}=P^{y_{Q}}$. Then \eqref{li1}
holds trivially. If $y_{Q}\not\in \frac{65}{64}Q$ then \eqref{li2} holds
vacuously.

If $y_{Q}\in \frac{65}{64}Q$, then \eqref{li2} asserts that $P^{y_{Q}}\in
\Gamma _{0}\left( y_{Q},C\left( \epsilon \right) M_{0}\right) $. Thanks to %
\eqref{li1} in Section \ref{auxiliary-polynomials}, we know that $%
P^{y_{Q}}\in \Gamma _{l\left( \mathcal{A}\right) -1}\left(
y_{Q},CM_{0}\right) \subset \Gamma _{0}\left( y_{Q},C\left( \epsilon \right)
M_{0}\right) $. Thus, \eqref{li1} and \eqref{li2} hold also when $Q$ is of
Type 2.

\underline{If $Q$ is of Type 3}, then $5Q^{+}\subset 5Q_{0}$, since $\frac{65%
}{64}Q\cap \frac{65}{64}Q_{0}\not=\emptyset $ and $\delta _{Q}\leq \frac{1}{%
1024}\delta _{Q_{0}}$. However, $Q^{+}$ cannot be OK, since $Q$ is a CZ
cube. Therefore $\#\left( E\cap 5Q^{+}\right) \geq 2$. We pick $y_{Q}\in
E\cap 5Q^{+}$, and set $F^{Q}=P^{Q}=P^{y_{Q}}$. Then \eqref{li1} holds
trivially, and \eqref{li2} holds vacuously.

\underline{If $Q$ is of Type 4}, then we set $F^{Q}=P^{Q}=P^{0}$, and again %
\eqref{li1} holds trivially, and \eqref{li2} holds vacuously.

Note that if $Q$ is of Type 1, 2, or 3, then we have defined a point $y_{Q}$%
, and we have $P^Q=P^{y_Q}$ and 

\begin{itemize}
\item[\refstepcounter{equation}\text{(\theequation)}\label{li3}] $y_{Q}\in
E\cap 5Q^{+}\cap 5Q_{0}$.
\end{itemize}

(If $Q$ is of Type 1 or 2, then $y_{Q}\in E\cap 5Q$ and $5Q\subseteq 5Q_{0}$
since $Q$ is OK. If $Q$ is of Type 3, then $y_{Q}\in E\cap 5Q^{+}$ and $%
5Q^{+}\subset 5Q_{0}$). We have picked $F^{Q}$ and $P^{Q}$ for all $Q\in 
\mathcal{Q}$, and \eqref{li1}, \eqref{li2} hold in all cases.

\begin{lemma}[\textquotedblleft Consistency of the $P^{Q}$\textquotedblright 
]
\label{lemma-li1} Let $Q,Q^{\prime }\in \mathcal{Q}$, and suppose $\frac{65}{%
64}Q\cap \frac{65}{64}Q^{\prime }\not=\emptyset $. Then

\begin{itemize}
\item[\refstepcounter{equation}\text{(\theequation)}\label{li4}] $\left\vert
\partial ^{\beta }\left( P^{Q}-P^{Q^{\prime }}\right) \right\vert \leq
C\left( \epsilon \right) M_{0}\delta _{Q}^{m-\left\vert \beta \right\vert }$
on $\frac{65}{64}Q\cap \frac{65}{64}Q^{\prime }$, for $\left\vert \beta
\right\vert \leq m$.
\end{itemize}
\end{lemma}

\begin{proof}
Suppose first that neither $Q$ nor $Q^{\prime }$ is Type 4. Then $%
P^{Q}=P^{y_{Q}}$ and $P^{Q^{\prime }}=P^{y_{Q^{\prime }}}$ with $y_{Q}\in
E\cap 5Q^{+}\cap 5Q_{0}$, $y_{Q^{\prime }}\in E\cap 5\left( Q^{\prime
}\right) ^{+}\cap 5Q_{0}$. Thanks to Lemma \ref{lemma-ap2}, we have%
\begin{equation*}
\left\vert \partial ^{\beta }\left( P^{Q}-P^{Q^{\prime }}\right) \left(
y_{Q}\right) \right\vert \leq C\left( \epsilon \right) M_{0}\delta
_{Q}^{m-\left\vert \beta \right\vert }\text{ for }\beta \in \mathcal{M}\text{%
,}
\end{equation*}
which implies \eqref{li4}, since $y_{Q}\in 5Q^{+}$ and $P^{Q}-P^{Q^{\prime
}} $ is an $(m-1)^{rst}$ degree polynomial.

Next, suppose that $Q$ and $Q^{\prime }$ are both Type 4.

Then by definition $P^{Q}=P^{Q^{\prime }}=P^{0}$, and consequently %
\eqref{li4} holds trivially.

Finally, suppose that exactly one of $Q$, $Q^{\prime }$ is of Type 4.

Since $\delta _{Q}$ and $\delta _{Q^{\prime }}$, differ by at most a factor
of $2$, the cubes $Q$ and $Q^{\prime }$ may be interchanged without loss of
generality. Hence, we may assume that $Q^{\prime }$ is of Type 4 and $Q$ is
not. By definition of Type 4,

\begin{itemize}
\item[\refstepcounter{equation}\text{(\theequation)}\label{li5}] $\delta
_{Q^{\prime }}>\frac{1}{1024}\delta _{Q_{0}}$; hence also $\delta _{Q}\geq 
\frac{1}{1024}\delta _{Q_{0}}$,
\end{itemize}

since $\delta _{Q}$, $\delta _{Q^{\prime }}$, are powers of $2$ that differ
by at most a factor of $2$.

Since $Q^{\prime }$ is of Type 4 and $Q$ is not, we have $P^Q= P^{y_Q}$ and $%
P^{Q^{\prime }} = P^0$, with

\begin{itemize}
\item[\refstepcounter{equation}\text{(\theequation)}\label{li6}] $y_{Q}\in
E\cap 5Q^{+}\cap 5Q_{0}$.
\end{itemize}

Thus, in this case, \eqref{li4} asserts that

\begin{itemize}
\item[\refstepcounter{equation}\text{(\theequation)}\label{li7}] $\left\vert
\partial ^{\beta }\left( P^{y_{Q}}-P^{0}\right) \right\vert \leq C\left(
\epsilon \right) M_{0}\delta _{Q}^{m-\left\vert \beta \right\vert }$ on $%
\frac{65}{64}Q\cap \frac{65}{64}Q^{\prime }$, for $\left\vert \beta
\right\vert \leq m$.
\end{itemize}

However, by \eqref{li6} above, property \eqref{ap3} in Section \ref%
{auxiliary-polynomials} gives the estimate

\begin{itemize}
\item[\refstepcounter{equation}\text{(\theequation)}\label{li8}] $\left\vert
\partial ^{\beta }\left( P^{y_{Q}}-P^{0}\right) \left( x_{0}\right)
\right\vert \leq C\left( \epsilon \right) M_{0}\delta _{Q_{0}}^{m-\left\vert
\beta \right\vert }$ for $\left\vert \beta \right\vert \leq m-1$.
\end{itemize}

Recall from the hypotheses of the Main Lemma for $\mathcal{A}$ that $%
x_{0}\in 5\left( Q_{0}\right) ^{+}$. Since $P^{y_{Q}}-P^{0}$ is an $%
(m-1)^{rst}$ degree polynomial, we conclude from \eqref{li8} that

\begin{itemize}
\item[\refstepcounter{equation}\text{(\theequation)}\label{li9}] $\left\vert
\partial ^{\beta }\left( P^{y_{Q}}-P^{0}\right) \right\vert \leq C\left(
\epsilon \right) M_{0}\delta _{Q_{0}}^{m-\left\vert \beta \right\vert }$ on $%
5Q$, for $\left\vert \beta \right\vert \leq m$.
\end{itemize}

The desired inequality \eqref{li7} now follows from \eqref{li5} and %
\eqref{li9}. Thus, \eqref{li4} holds in all cases.

The proof of Lemma \ref{lemma-li1} is complete.
\end{proof}

From estimate \eqref{li1}, Lemma \ref{lemma-li1}, and Lemma \ref{lemma-cz2},
we immediately obtain the following.

\begin{corollary}
\label{cor-to-lemma-li1} Let $Q,Q^{\prime }\in \mathcal{Q}$ and suppose that 
$\frac{65}{64}Q\cap \frac{65}{64}Q^{\prime }\not=\emptyset $. Then

\begin{itemize}
\item[\refstepcounter{equation}\text{(\theequation)}\label{li10}] $%
\left\vert \partial ^{\beta }\left( F^{Q}-F^{Q^{\prime }}\right) \right\vert
\leq C\left( \epsilon \right) M_{0}\delta _{Q}^{m-\left\vert \beta
\right\vert }$ on $\frac{65}{64}Q\cap \frac{65}{64}Q^{\prime }$, for $%
\left\vert \beta \right\vert \leq m$.
\end{itemize}
\end{corollary}

Regarding the polynomials $P^{Q}$, we make the following simple observation.

\begin{lemma}
\label{lemma-li2} We have

\begin{itemize}
\item[\refstepcounter{equation}\text{(\theequation)}\label{li11}] $%
\left\vert \partial ^{\beta }\left( P^{Q}-P^{0}\right) \right\vert \leq
C\left( \epsilon \right) M_{0}\delta _{Q_{0}}^{m-\left\vert \beta
\right\vert }$ on $\frac{65}{64}Q$, for $\left\vert \beta \right\vert \leq m$
and $Q\in \mathcal{Q}$.
\end{itemize}
\end{lemma}

\begin{proof}
Recall that if $Q$ is of Type 1, 2, or 3, then $P^{Q}=P^{y_{Q}}$ for some $%
y_{Q}\in 5Q_{0}$. From estimate \eqref{ap3} in Section \ref%
{auxiliary-polynomials}, we know that

\begin{itemize}
\item[\refstepcounter{equation}\text{(\theequation)}\label{li12}] $%
\left\vert \partial ^{\beta }\left( P^{Q}-P^{0}\right) \left( x_{0}\right)
\right\vert \leq C\left( \epsilon \right) M_{0}\delta _{Q_{0}}^{m-\left\vert
\beta \right\vert }$ for $\left\vert \beta \right\vert \leq m-1$.
\end{itemize}

Since $x_{0}\in 5Q_{0}^{+}$ (see the hypotheses of the Main Lemma for $%
\mathcal{A}$) and $P^{Q}-P^{0}$ is a polynomial of degree at most $m-1$, and
since $\frac{65}{64}Q\subset 5Q\subset 5Q_{0}$ (because $Q$ is OK), estimate %
\eqref{li12} implies the desired estimate \eqref{li11}.

If instead, $Q$ is of Type 4, then by definition $P^{Q}=P^{0}$, hence
estimate \eqref{li11} holds trivially.

Thus, \eqref{li11} holds in all cases.
\end{proof}

\begin{corollary}
\label{cor-to-lemma-li2} For $Q\in \mathcal{Q}$ and $|\beta |\leq m$, we
have $\left\vert \partial ^{\beta }\left( F^{Q}-P^{0}\right) \right\vert
\leq C\left( \epsilon \right) M_{0}\delta _{Q_{0}}^{m-\left\vert \beta
\right\vert }$ on $\frac{65}{64}Q$.
\end{corollary}

\begin{proof}
Recall that, since $Q$ is OK, we have $5Q\subset 5Q_{0}$. The desired
estimate therefore follows from estimates \eqref{li1} and \eqref{li11}.
\end{proof}

\section{Completing the Induction}

\label{completing-the-induction}

We again place ourselves in the setting of Section \ref%
{setup-for-the-induction-step}. We use the CZ cubes $Q$ and the functions $%
F^Q$ defined above. We recall several basic results from earlier sections.

\begin{itemize}
\item[\refstepcounter{equation}\text{(\theequation)}\label{ci1}] $\vec{\Gamma%
}_{0}$ is a $\left( C,\delta _{\max }\right) $-convex shape field.
\end{itemize}

\begin{itemize}
\item[\refstepcounter{equation}\text{(\theequation)}\label{ci2}] $\epsilon
^{-1}\delta _{Q_{0}}\leq \delta _{\max }$, hence $\epsilon ^{-1}\delta
_{Q}\leq \delta _{\max }$ for $Q\in $ CZ.
\end{itemize}

\begin{itemize}
\item[\refstepcounter{equation}\text{(\theequation)}\label{ci3}] The cubes $%
Q\in $ CZ partition the interior of $5Q_{0}$.
\end{itemize}

\begin{itemize}
\item[\refstepcounter{equation}\text{(\theequation)}\label{ci4}] For $%
Q,Q^{\prime }\in $ CZ, if $\frac{65}{64}Q\cap \frac{65}{64}Q^{\prime
}\not=\emptyset $, then $\frac{1}{2}\delta _{Q}\leq \delta _{Q^{\prime
}}\leq 2\delta _{Q}$.
\end{itemize}

Let

\begin{itemize}
\item[\refstepcounter{equation}\text{(\theequation)}\label{ci5}] $\mathcal{Q}%
=\left\{ Q\in CZ:\frac{65}{64}Q\cap \frac{65}{64}Q_{0}\not=\emptyset
\right\} $.
\end{itemize}

Then

\begin{itemize}
\item[\refstepcounter{equation}\text{(\theequation)}\label{ci6}] $\mathcal{Q}
$ is finite.
\end{itemize}

For each $Q \in \mathcal{Q}$, we have

\begin{itemize}
\item[\refstepcounter{equation}\text{(\theequation)}\label{ci7}] $F^{Q}\in
C^{m}\left( \frac{65}{64}Q\right) $,
\end{itemize}

\begin{itemize}
\item[\refstepcounter{equation}\text{(\theequation)}\label{ci8}] $%
J_{z}\left( F^{Q}\right) \in \Gamma _{0}\left( z,C\left( \epsilon \right)
M_{0}\right) $ for $z\in E\cap \frac{65}{64}Q$, and
\end{itemize}

\begin{itemize}
\item[\refstepcounter{equation}\text{(\theequation)}\label{ci9}] $\left\vert
\partial ^{\beta }\left( F^{Q}-P^{0}\right) \right\vert \leq C\left(
\epsilon \right) M_{0}\delta _{Q_{0}}^{m-\left\vert \beta \right\vert }$ on $%
\frac{65}{64}Q$, for $\left\vert \beta \right\vert \leq m$.
\end{itemize}

\begin{itemize}
\item[\refstepcounter{equation}\text{(\theequation)}\label{ci10}] For each $%
Q,Q^{\prime }\in \mathcal{Q}$, if $\frac{65}{64}Q\cap \frac{65}{64}Q^{\prime
}\not=\emptyset $, then $\left\vert \partial ^{\beta }\left(
F^{Q}-F^{Q^{\prime }}\right) \right\vert \leq C\left( \epsilon \right)
M_{0}\delta _{Q}^{m-\left\vert \beta \right\vert }$ on $\frac{65}{64}Q\cap 
\frac{65}{64}Q^{\prime }$, for $\left\vert \beta \right\vert \leq m$.
\end{itemize}

We introduce a Whitney partition of unity adapted to the cubes $Q\in $ CZ.
For each $Q\in $ CZ, let $\tilde{\theta}_{Q}\in C^{m}\left( \mathbb{R}%
^{n}\right) $ satisfy 
\begin{equation*}
\tilde{\theta}_{Q}=1\text{ on }Q\text{, support }\left( \tilde{\theta}%
_{Q}\right) \subset \frac{65}{64}Q\text{, }\left\vert \partial ^{\beta }%
\tilde{\theta}_{Q}\right\vert \leq C\delta _{Q}^{-\left\vert \beta
\right\vert }\text{ for }\left\vert \beta \right\vert \leq m\text{.}
\end{equation*}%
Setting $\theta _{Q}=\tilde{\theta}_{Q}\cdot \left( \sum_{Q^{\prime }\in
CZ}\left( \tilde{\theta}_{Q^{\prime }}\right) ^{2}\right) ^{-1/2}$, we see
that

\begin{itemize}
\item[\refstepcounter{equation}\text{(\theequation)}\label{ci11}] $\theta
_{Q}\in C^{m}\left( \mathbb{R}^{n}\right) $ for $Q\in $ CZ;
\end{itemize}

\begin{itemize}
\item[\refstepcounter{equation}\text{(\theequation)}\label{ci12}] support $%
\left( \theta _{Q}\right) \subset \frac{65}{64}Q$ for $Q\in $ CZ.
\end{itemize}

\begin{itemize}
\item[\refstepcounter{equation}\text{(\theequation)}\label{ci13}] $%
\left\vert \partial ^{\beta }\theta _{Q}\right\vert \leq C\delta
_{Q}^{-\left\vert \beta \right\vert }$ for $\left\vert \beta \right\vert
\leq m,Q\in $ CZ;
\end{itemize}

and $\sum_{Q\in CZ}\theta _{Q}^{2}=1$ on the interior of $5Q_{0}$, hence

\begin{itemize}
\item[\refstepcounter{equation}\text{(\theequation)}\label{ci14}] $%
\sum_{Q\in \mathcal{Q}}\theta _{Q}^{2}=1$ on $\frac{65}{64}Q_{0}$.
\end{itemize}

We define

\begin{itemize}
\item[\refstepcounter{equation}\text{(\theequation)}\label{ci15}] $%
F=\sum_{Q\in \mathcal{Q}}\theta _{Q}^{2}F^{Q}$.
\end{itemize}

For each $Q\in \mathcal{Q}$, \eqref{ci7}, \eqref{ci11}, \eqref{ci12} show
that $\theta _{Q}^{2}F^{Q}\in C^{m}\left( \mathbb{R}^{n}\right) $. Since
also $\mathcal{Q}$ is finite (see \eqref{ci6}), it follows that

\begin{itemize}
\item[\refstepcounter{equation}\text{(\theequation)}\label{ci16}] $F\in
C^{m}\left( \mathbb{R}^{n}\right) $.
\end{itemize}

Moreover, for any $x\in \frac{65}{64}Q_{0}$ and any $\beta $ of order $%
|\beta |\leq m$, we have

\begin{itemize}
\item[\refstepcounter{equation}\text{(\theequation)}\label{ci17}] $\partial
^{\beta }F\left( x\right) =\sum_{Q\in \mathcal{Q}\left( x\right) }\partial
^{\beta }\left\{ \theta _{Q}^{2}F^{Q}\right\} $, where
\end{itemize}

\begin{itemize}
\item[\refstepcounter{equation}\text{(\theequation)}\label{ci18}] $\mathcal{Q%
}\left( x\right) =\left\{ Q\in \mathcal{Q}:x\in \frac{65}{64}Q\right\} $.
\end{itemize}

Note that

\begin{itemize}
\item[\refstepcounter{equation}\text{(\theequation)}\label{ci19}] $\#\left( 
\mathcal{Q}\left( x\right) \right) \leq C$, by \eqref{ci4}.
\end{itemize}

Let $\hat{Q}$ be the CZ cube containing $x$. (There is one and only one such
cube, thanks to \eqref{ci3}; recall that we suppose that $x\in \frac{65}{64}%
Q_{0}$.) Then $\hat{Q}\in \mathcal{Q}(x)$, and \eqref{ci17} may be written
in the form

\begin{itemize}
\item[\refstepcounter{equation}\text{(\theequation)}\label{ci20}] $\partial
^{\beta }\left( F-P^{0}\right) \left( x\right) =\partial ^{\beta }\left( F^{%
\hat{Q}}-P^{0}\right) \left( x\right) +\sum_{Q\in \mathcal{Q}\left( x\right)
}\partial ^{\beta }\left\{ \theta _{Q}^{2}\cdot \left( F^{Q}-F^{\hat{Q}%
}\right) \right\} \left( x\right) $.
\end{itemize}

(Here we use \eqref{ci14}.) The first term on the right in \eqref{ci20} has
absolute value at most $C\left( \epsilon \right) M_{0}\delta
_{Q_{0}}^{m-\left\vert \beta \right\vert }$; see \eqref{ci9}. At most $C$
distinct cubes $Q$ enter into the second term on the right in \eqref{ci20};
see \eqref{ci19}. For each $Q\in \mathcal{Q}(x)$, we have 
\begin{equation*}
\left\vert \partial ^{\beta }\left\{ \theta _{Q}^{2}\cdot \left( F^{Q}-F^{%
\hat{Q}}\right) \right\} \left( x\right) \right\vert \leq C\left( \epsilon
\right) M_{0}\delta _{Q}^{m-\left\vert \beta \right\vert }\text{,}
\end{equation*}%
by \eqref{ci10} and \eqref{ci13}. Hence, for each $Q\in \mathcal{Q}(x)$, we
have 
\begin{equation*}
\left\vert \partial ^{\beta }\left\{ \theta _{Q}^{2}\cdot \left( F^{Q}-F^{%
\hat{Q}}\right) \right\} \left( x\right) \right\vert \leq C\left( \epsilon
\right) M_{0}\delta _{Q_{0}}^{m-\left\vert \beta \right\vert };
\end{equation*}%
see \eqref{ci3}.

The above remarks and \eqref{ci19}, \eqref{20} together yield the estimate

\begin{itemize}
\item[\refstepcounter{equation}\text{(\theequation)}\label{ci21}] $%
\left\vert \partial ^{\beta }\left( F-P^{0}\right) \right\vert \leq C\left(
\epsilon \right) M_{0}\delta _{Q_{0}}^{m-\left\vert \beta \right\vert }$ on $%
\frac{65}{64}Q_{0}$, for $\left\vert \beta \right\vert \leq m$.
\end{itemize}

Moreover, let $z\in E\cap \frac{65}{64}Q_{0}$. Then 
\begin{equation*}
J_{z}\left( F\right) =\sum_{Q\in \mathcal{Q}\left( z\right) }J_{z}\left(
\theta _{Q}\right) \odot _{z}J_{z}\left( \theta _{Q}\right) \odot
_{z}J_{z}\left( F^{Q}\right) \text{ (see \eqref{ci17});}
\end{equation*}%
\begin{equation*}
\left\vert \partial ^{\beta }\left[ J_{z}\left( \theta _{Q}\right) \right]
\left( z\right) \right\vert \leq C\delta _{Q}^{-\left\vert \beta \right\vert
}\text{ for }\left\vert \beta \right\vert \leq m-1\text{, }Q\in \mathcal{Q}%
\left( z\right) \text{ (see \eqref{ci13});}
\end{equation*}%
\begin{equation*}
\sum_{Q\in \mathcal{Q}\left( z\right) }\left[ J_{z}\left( \theta _{Q}\right) %
\right] \odot _{z}\left[ J_{z}\left( \theta _{Q}\right) \right] =1
\end{equation*}%
(see \eqref{ci14} and note that $J_{z}(\theta _{Q})=0$ for $Q\not\in 
\mathcal{Q}(z)$ by \eqref{ci12} and \eqref{ci18}); 
\begin{equation*}
J_{z}\left( F^{Q}\right) \in \Gamma _{0}\left( z,C\left( \epsilon \right)
M_{0}\right) \text{ for }Q\in \mathcal{Q}\left( z\right) \text{ (see %
\eqref{ci8});}
\end{equation*}%
\begin{equation*}
\left\vert \partial ^{\beta }\left\{ J_{z}\left( F^{Q}\right) -J_{z}\left(
F^{Q^{\prime }}\right) \right\} \left( z\right) \right\vert \leq C\left(
\epsilon \right) M_{0}\delta _{Q}^{m-\left\vert \beta \right\vert }
\end{equation*}%
for $\left\vert \beta \right\vert \leq m-1$, $Q,Q^{\prime }\in \mathcal{Q}%
\left( z\right) $ (see\eqref{ci10}); 
\begin{equation*}
\#\left( \mathcal{Q}\left( z\right) \right) \leq C\text{ (see \eqref{ci19});}
\end{equation*}%
\begin{equation*}
\delta _{Q}\leq \delta _{\max }\text{ (see \eqref{ci2});}
\end{equation*}%
$\vec{\Gamma}_{0}$ is a $\left( C,\delta _{\max }\right) $-convex shape
field (see \eqref{ci1}). The above results, together with Lemma \ref%
{lemma-wsf2}, tell us that

\begin{itemize}
\item[\refstepcounter{equation}\text{(\theequation)}\label{ci22}] $%
J_{z}\left( F\right) \in \Gamma _{0}\left( z,C\left( \epsilon \right)
M_{0}\right) $ for all $z\in E\cap \frac{65}{64}Q_{0}$.
\end{itemize}

From \eqref{ci16}, \eqref{ci21}, \eqref{ci22}, we see at once that the
restriction of $F$ to $\frac{65}{64}Q_{0}$ belongs to $C^{m}\left( \frac{65}{%
64}Q_{0}\right) $ and satisfies conditions (C*1) and (C*2) in Section \ref%
{setup-for-the-induction-step}. As we explained in that section, once we
have found a function in $C^{m}\left( \frac{65}{64}Q_{0}\right) $ satisfying
(C*1) and (C*2), our induction on $\mathcal{A}$ is complete. Thus, we have
proven the Main Lemma for all monotonic $\mathcal{A}\subseteq \mathcal{M}$.

\section{Restatement of the Main Lemma}

\label{rml}

In this section, we reformulate the Main Lemma for $\mathcal{A}$ in the case
in which $\mathcal{A}$ is the empty set $\emptyset$. Let us examine
hypotheses (A1), (A2), (A3) for the Main Lemma for $\mathcal{A}$, taking $%
\mathcal{A} = \emptyset$.

Hypothesis (A1) says that $\vec{\Gamma}_{l\left( \emptyset \right) }$ has an 
$\left( \emptyset ,\epsilon ^{-1}\delta _{Q_{0}},C_{B}\right) $-basis at $%
\left( x_{0},M_{0},P^{0}\right) $. This means simply that $P^{0}\in \Gamma
_{l\left( \emptyset \right) }\left( x_{0},C_{B}M_{0}\right) $.

Hypothesis (A2) says that $\delta _{Q_{0}}\leq \epsilon \delta _{\max }$,
and hypothesis (A3) says that $\epsilon $ is less than a small enough
constant determined by $C_{B}$, $C_{w}$, $m$, $n$.

We take $\epsilon $ to be a small enough constant (determined by $C_{B}$, $%
C_{w}$, $m$, $n$) such that (A3) is satisfied. We take $C_{B}=1$. Thus, we
arrive at the following equivalent version of the Main Lemma for $\emptyset $%
.

\theoremstyle{plain} 
\newtheorem*{thm Restated Main Lemma}{Restated Main
Lemma}%
\begin{thm Restated Main Lemma}Let $\vec{\Gamma}_{0}=\left( \Gamma _{0}\left( x,M\right) \right) _{x\in
E,M>0}$ be a $\left( C_{w},\delta _{\max }\right) $-convex shape
field. For $l\geq 1$, let $\vec{\Gamma}_{l}=\left( \Gamma _{l}\left(
x,M\right) \right) _{x\in E,M>0}$ be the $l^{th}$-refinement of $\vec{\Gamma}_{0}$. Fix a dyadic cube $Q_0$ of sidelength $\delta _{Q_{0}}\leq \epsilon
\delta _{\max }$, where $\epsilon >0$ is a small enough constant determined
by $m$, $n$, $C_{W}$. Let $x_{0}\in E\cap 5Q_{0}^{+}$, and let $P_{0}\in \Gamma
_{l\left( \emptyset \right) }\left( x_{0},M_{0}\right) $.

Then there exists a function $F\in C^{m}\left( \frac{65}{64}Q_{0}\right) $,
satisfying 

\begin{itemize}
\item $\left\vert \partial ^{\beta }\left( F-P_{0}\right) \left( x\right)
\right\vert \leq C_{\ast }M_{0}\delta _{Q_{0}}^{m-\left\vert \beta
\right\vert }$ for $x\in \frac{65}{64}Q_{0}$, $\left\vert \beta \right\vert
\leq m$; and 
\item $J_{z}\left( F\right) \in \Gamma _{0}\left( z,C_{\ast }M_{0}\right) $
for all $z\in E\cap \frac{65}{64}Q_{0}$;
\end{itemize}
where $C_{\ast }$ is determined by $C_{w}$, $m$, $n$.\end{thm Restated Main Lemma}

\section{Tidying Up}

\label{tu}

In this section, we remove from the Restated Main Lemma the small constant $%
\epsilon$ and the assumption that $Q_0$ is dyadic.

\begin{theorem}
\label{theorem-tu1} Let $\vec{\Gamma}_{0}=\left( \Gamma _{0}\left(
x,M\right) \right) _{x\in E,M>0}$ be a $\left( C_{w},\delta _{\max }\right) $%
-convex shape field. For $l\geq 1$, let $\vec{\Gamma}_{l}=\left( \Gamma
_{l}\left( x,M\right) \right) _{x\in E,M>0}$ be the $l^{th}$-refinement of $%
\vec{\Gamma}_{0}$. Fix a cube $Q_0$ of sidelength $\delta _{Q_{0}}\leq
\delta _{\max }$, a point $x_0 \in E \cap 5Q_0$, and a real number $M_0>0$.
Let $P_{0}\in \Gamma _{l\left( \emptyset \right) +1}\left(
x_{0},M_{0}\right) $.

Then there exists a function $F\in C^{m}\left(Q_{0}\right) $ satisfying the
following, with $C_{\ast }$ determined by $C_{w}$, $m$, $n$.

\begin{itemize}
\item $\left\vert \partial ^{\beta }\left( F-P_{0}\right) \left( x\right)
\right\vert \leq C_{\ast }M_{0}\delta _{Q_{0}}^{m-\left\vert \beta
\right\vert }$ for $x\in Q_{0}$, $\left\vert \beta \right\vert \leq m$; and

\item $J_{z}\left( F\right) \in \Gamma _{0}\left( z,C_{\ast }M_{0}\right) $
for all $z\in E\cap Q_{0}$.
\end{itemize}
\end{theorem}

\begin{proof}
Let $\epsilon >0$ be the small constant in the statement of the Restated
Main Lemma in Section \ref{rml}. In particular, $\epsilon $ is determined by 
$C_{w}$, $m$, $n$. We write $c$, $C$, $C^{\prime }$, etc., to denote
constants determined by $C_{w}$, $m$, $n$. These symbols may denote
different constants in different occurrences.

We cover $CQ_{0}$ by a grid of dyadic cubes $\{Q_{\nu }\}$, all having same
sidelength $\delta _{Q_{\nu }}$, with $\frac{\epsilon }{20}\delta
_{Q_{0}}\leq \delta _{Q_{\nu }}\leq \epsilon \delta _{Q_{0}}$, and all
contained in $C^{\prime }Q_{0}$. (We use at most $C$ distinct $Q_{\nu }$ to
do so.)

For each $Q_{\nu }$ with $E\cap \frac{65}{64}Q_{\nu }\not=\emptyset $, we
pick a point $x_{\nu }\in E\cap \frac{65}{64}Q_{\nu }$; by definition of the 
$l^{th}$-refinement, there exists $P_{\nu }\in \Gamma _{l(\emptyset
)}(x_{\nu },M_{0})$ such that $\left\vert \partial ^{\beta }\left( P_{\nu
}-P_{0}\right) \left( x_{0}\right) \right\vert \leq CM_{0}\delta
_{Q_{0}}^{m-\left\vert \beta \right\vert }$ for $\beta \in \mathcal{M}$, and
therefore

\begin{itemize}
\item[\refstepcounter{equation}\text{(\theequation)}\label{tu1}] $\left\vert
\partial ^{\beta }\left( P_{\nu }-P_{0}\right) \left( x\right) \right\vert
\leq C^{\prime }M_{0}\delta _{Q_{0}}^{m-\left\vert \beta \right\vert }$ for $%
x\in \frac{65}{64}Q_{0}$ and $\left\vert \beta \right\vert \leq m$.
\end{itemize}

Since $x_{\nu }\in E\cap \frac{65}{64}Q_{\nu }$, $P_{\nu }\in \Gamma
_{l\left( \emptyset \right) }$, and $\delta _{Q_{\nu }}\leq \epsilon \delta
_{Q_{0}}\leq \epsilon \delta _{\max }$, the Restated Main Lemma applies to $%
x_{\nu },P_{\nu },Q_{\nu }$ to produce $F_{\nu }\in C^{m}\left( \frac{65}{64}%
Q_{\nu }\right) $ satisfying

\begin{itemize}
\item[\refstepcounter{equation}\text{(\theequation)}\label{tu2}] $\left\vert
\partial ^{\beta }\left( F_{\nu }-P_{\nu }\right) \left( x\right)
\right\vert \leq CM_{0}\delta _{Q_{\nu }}^{m-\left\vert \beta \right\vert
}\leq CM_{0}\delta _{Q_{0}}^{m-\left\vert \beta \right\vert }$ for $x\in 
\frac{65}{64}Q_{\nu }$, $\left\vert \beta \right\vert \leq m$;
\end{itemize}

and

\begin{itemize}
\item[\refstepcounter{equation}\text{(\theequation)}\label{tu3}] $%
J_{z}\left( F_{\nu }\right) \in \Gamma _{0}\left( z,CM_{0}\right) $ for all $%
z\in E\cap \frac{65}{64}Q_{\nu }$.
\end{itemize}

From \eqref{tu1} and \eqref{tu2}, we have

\begin{itemize}
\item[\refstepcounter{equation}\text{(\theequation)}\label{tu4}] $\left\vert
\partial ^{\beta }\left( F_{\nu }-P_{0}\right) \left( x\right) \right\vert
\leq CM_{0}\delta _{Q_{0}}^{m-\left\vert \beta \right\vert }$ for $x\in 
\frac{65}{64}Q_{\nu }$, $\left\vert \beta \right\vert \leq m$.
\end{itemize}

We have produced such $F_{\nu }$ for those $\nu $ satisfying $E\cap \frac{65%
}{64}Q_{\nu }\not=\emptyset $. If instead $E\cap \frac{65}{64}Q_{\nu
}=\emptyset $, then we set $F_{\nu }=P_{0}$. Then \eqref{tu3} holds
vacuously and \eqref{tu4} holds trivially. Thus, our $F_{\nu }$ satisfy %
\eqref{tu3}, \eqref{tu4} for all $\nu $. From \eqref{tu4} we obtain

\begin{itemize}
\item[\refstepcounter{equation}\text{(\theequation)}\label{tu5}] $\left\vert
\partial ^{\beta }\left( F_{\nu }-F_{\nu ^{\prime }}\right) \left( x\right)
\right\vert \leq CM_{0}\delta _{Q_{0}}^{m-\left\vert \beta \right\vert }$
for $x\in \frac{65}{64}Q_{\nu }\cap \frac{65}{64}Q_{\nu ^{\prime }}$, $%
\left\vert \beta \right\vert \leq m$.
\end{itemize}

Next, we introduce a partition of unity. We fix cutoff functions $\theta
_{\nu }\in C^{m}\left( \mathbb{R}^{n}\right) $ satisfying

\begin{itemize}
\item[\refstepcounter{equation}\text{(\theequation)}\label{tu6}] support $%
\theta _{\nu }\subset \frac{65}{64}Q_{\nu }$, $\left\vert \partial ^{\beta
}\theta _{\nu }\right\vert \leq C\delta _{Q_{0}}^{-\left\vert \beta
\right\vert }$ for $\left\vert \beta \right\vert \leq m$, $\sum_{\nu }\theta
_{\nu }^{2}=1$ on $Q_{0}$.
\end{itemize}

We then define

\begin{itemize}
\item[\refstepcounter{equation}\text{(\theequation)}\label{tu7}] $%
F=\sum_{\nu }\theta _{\nu }^{2}F_{\nu }$ on $Q_{0}$.
\end{itemize}

We have then

\begin{itemize}
\item[\refstepcounter{equation}\text{(\theequation)}\label{tu8}] $%
F-P_{0}=\sum_{\nu }\theta _{\nu }^{2}\left( F_{\nu }-P_{0}\right) $ on $%
Q_{0} $.
\end{itemize}

Thanks to \eqref{tu4} and \eqref{tu6}, we have $\theta _{\nu }^{2}\left(
F_{\nu }-P_{0}\right) \in C^{m}\left( Q_{0}\right) $ and $\left\vert
\partial ^{\beta }\left( \theta _{\nu }^{2}\cdot \left( F_{\nu
}-P_{0}\right) \right) \left( x\right) \right\vert \leq CM_{0}\delta
_{Q_{0}}^{m-\left\vert \beta \right\vert }$ for $x\in Q_{0},|\beta |\leq m$,
all $\nu $. Moreover, there are at most $C$ distinct $\nu $ appearing in %
\eqref{tu8}. Hence,

\begin{itemize}
\item[\refstepcounter{equation}\text{(\theequation)}\label{tu9}] $F\in
C^{m}\left( Q_{0}\right) $
\end{itemize}

and

\begin{itemize}
\item[\refstepcounter{equation}\text{(\theequation)}\label{tu10}] $%
\left\vert \partial ^{\beta }\left( F-P_{0}\right) \left( x\right)
\right\vert \leq CM_{0}\delta _{Q_{0}}^{m-\left\vert \beta \right\vert }$
for $x\in Q_{0}$, $\left\vert \beta \right\vert \leq m$.
\end{itemize}

Next, let $z\in E\cap Q_{0}$, and let $Y$ be the set of all $\nu $ such that 
$z\in \frac{65}{64}Q_{\nu }$. Then \eqref{tu6}, \eqref{tu7} give $%
J_{z}(F)=\sum_{\nu \in Y}J_{z}\left( \theta _{\nu }\right) \odot
_{z}J_{z}\left( \theta _{\nu }\right) \odot _{z}J_{z}\left( F_{\nu }\right) $
and we know that $J_{z}(F_{\nu })\in \Gamma _{0}\left( z,CM_{0}\right) $ for 
$\nu \in Y$ (by \eqref{tu3}); $|\partial ^{\beta }\left[ J_{z}\left( F_{\nu
}\right) -J_{z}\left( F_{\nu ^{\prime }}\right) \right] \left( z\right)
|\leq CM_{0}\delta _{Q_{0}}^{m-\left\vert \beta \right\vert }$ for $%
\left\vert \beta \right\vert \leq m-1$, $\nu ,\nu ^{\prime }\in Y$ (by %
\eqref{tu5}); $\left\vert \partial ^{\beta }\left[ J_{z}\left( \theta _{\nu
}\right) \right] \left( z\right) \right\vert \leq C\delta
_{Q_{0}}^{-\left\vert \beta \right\vert }$ for $\left\vert \beta \right\vert
\leq m-1$, $\nu \in Y$ (by \eqref{tu6}); $\sum_{\nu \in Y}J_{z}\left( \theta
_{\nu }\right) \odot _{z}J_{z}\left( \theta _{\nu }\right) =1$ (again thanks
to \eqref{tu6}); $\#(Y)\leq C$ (since there are at most $C$ distinct $%
Q_{\nu }$ in our grid); and $\delta _{Q_{0}}\leq \delta _{\max }$ (by
hypothesis of the Theorem we are proving). Since $\vec{\Gamma}_{0}$ is $%
(C,\delta _{\max })$-convex, the above remarks and Lemma \ref{lemma-wsf2}
tell us that $J_{z}(F)\in \Gamma _{0}(z,CM_{0})$. Thus,

\begin{itemize}
\item[\refstepcounter{equation}\text{(\theequation)}\label{tu11}] $%
J_{z}\left( F\right) \in \Gamma _{0}\left( z,CM_{0}\right) $ for all $z\in
E\cap Q_{0}$.
\end{itemize}

Our results \eqref{tu9}, \eqref{tu10}, \eqref{tu11} are the conclusions of
Theorem \ref{theorem-tu1}.

The proof of that Theorem is complete.
\end{proof}

\part{Applications}

\section{Finiteness Principle I}

\label{fp-i}

In this section we prove a finiteness principle for shape fields.

Let $\vec{\Gamma}_{0}=\left( \Gamma _{0}\left( x,M\right) \right) _{x\in
E,M>0}$ be a shape field. For $l\geq 1$, let $\vec{\Gamma}_{l}=\left( \Gamma
_{l}\left( x,M\right) \right) _{x\in E,M>0}$ be the $l^{th}$-refinement of $%
\vec{\Gamma}_{0}$. Fix $M_{0}>0$. For $x\in E$, $S\subset E$, define

\begin{itemize}
\item[\refstepcounter{equation}\text{(\theequation)}\label{fp1}] $\Gamma
\left( x,S\right) =\left\{ 
\begin{array}{c}
P^{x}:\vec{P}=\left( P^{y}\right) _{y\in S\cup \left\{ x\right\} }\in
Wh\left( S\cup \left\{ x\right\} \right) \text{, }\left\Vert \vec{P}%
\right\Vert _{\dot{C}^{m}\left( S\cup \left\{ x\right\} \right) }\leq M_{0},
\\ 
P^{y}\in \Gamma _{0}\left( y,M_{0}\right) \text{ for all }y\in S\cup \left\{
x\right\} \text{.}%
\end{array}%
\right\} $
\end{itemize}

(See Section \ref{notation-and-preliminaries} for the definition of $%
Wh(\cdot )$ and $||\cdot ||_{\dot{C}^{m}\left( \cdot \right)}$.) Note that

\begin{itemize}
\item[\refstepcounter{equation}\text{(\theequation)}\label{fp2}] $\Gamma
\left( x,\emptyset \right) =\Gamma _{0}\left( x,M_{0}\right) $.
\end{itemize}

Define

\begin{itemize}
\item[\refstepcounter{equation}\text{(\theequation)}\label{fp3}] $\Gamma
_{l}^{fp}\left( x,M_{0}\right) =\bigcap_{S\subset E,\#\left( S\right) \leq
\left( D+2\right) ^{l}}\Gamma \left( x,S\right) $ for $l\geq 0$, where
\end{itemize}

\begin{itemize}
\item[\refstepcounter{equation}\text{(\theequation)}\label{fp4}] $D=\dim 
\mathcal{P}$.
\end{itemize}

Note that

\begin{itemize}
\item[\refstepcounter{equation}\text{(\theequation)}\label{fp5}] $\Gamma
_{0}^{fp}\left( x,M_{0}\right) \subseteq \Gamma _{0}\left( x,M_{0}\right) $,
thanks to \eqref{fp2}.
\end{itemize}

Each $\Gamma (x,S)$, $\Gamma _{l}^{fp}\left( x,M_{0}\right) $ is a (possibly
empty) convex subset of $\mathcal{P}$.

\begin{lemma}
\label{lemma-fp1} Let $x\in E$, $l\geq 0$. Suppose $\Gamma
(x,S)\not=\emptyset $ for all $S\subset E$ with $\#\left( S\right) \leq
\left( D+2\right) ^{l+1}$. Then $\Gamma _{l}^{fp}\left( x,M_{0}\right)
\not=\emptyset $.
\end{lemma}

\begin{proof}
By Helly's theorem, it is enough to show that $\Gamma \left( x,S_{1}\right)
\cap \cdots \cap \Gamma \left( x,S_{D+1}\right) \not=\emptyset $ for any $%
S_{1},\cdots ,S_{D+1}\subset E$ with $\#\left( S_{i}\right) \leq \left(
D+2\right) ^{l}$ for each $i$. However, $\Gamma \left( x,S_{1}\right) \cap
\cdots \cap \Gamma \left( x,S_{D+1}\right) \supset \Gamma \left( x,S_{1}\cup
\cdots \cup S_{D+1}\right) \not=\emptyset $, since $\#\left( S_{1}\cup
\cdots \cup S_{D+1}\right) \leq \left( D+1\right) \cdot \left( D+2\right)
^{l}\leq \left( D+2\right) ^{l+1}$.
\end{proof}

\begin{lemma}
\label{lemma-fp2} For $x \in E, l \geq 0$, we have

\begin{itemize}
\item[\refstepcounter{equation}\text{(\theequation)}\label{fp6}] $\Gamma
_{l}^{fp}\left( x,M_{0}\right) \subseteq \Gamma _{l}\left( x,M_{0}\right) $.
\end{itemize}
\end{lemma}

\begin{proof}
We use induction on $l$. The base case $l=0$ is our observation \eqref{fp5}.
For the induction step, fix $l\geq 1$. We will prove \eqref{6} under the
inductive assumption

\begin{itemize}
\item[\refstepcounter{equation}\text{(\theequation)}\label{fp7}] $\Gamma
_{l-1}^{fp}\left( y,M_{0}\right) \subseteq \Gamma _{l-1}\left(
y,M_{0}\right) $ for all $y\in E$.
\end{itemize}

Thus, let $P\in \Gamma _{l}^{fp}\left( x,M_{0}\right) $ be given. We must
prove that $P\in \Gamma _{l}\left( x,M_{0}\right) $, which means that given $%
y\in E$ there exists

\begin{itemize}
\item[\refstepcounter{equation}\text{(\theequation)}\label{fp8}] $P^{\prime
}\in \Gamma _{l-1}\left( y,M_{0}\right) $ such that $\left\vert \partial
^{\beta }\left( P-P^{\prime }\right) \left( x\right) \right\vert \leq
M_{0}\left\vert x-y\right\vert ^{m-\left\vert \beta \right\vert }$ for $%
\left\vert \beta \right\vert \leq m-1$.
\end{itemize}

We will prove that there exists

\begin{itemize}
\item[\refstepcounter{equation}\text{(\theequation)}\label{fp9}] $P^{\prime
}\in \Gamma _{l-1}^{fp}\left( y,M_{0}\right) $ such that $\left\vert
\partial ^{\beta }\left( P-P^{\prime }\right) \left( x\right) \right\vert
\leq M_{0}\left\vert x-y\right\vert ^{m-\left\vert \beta \right\vert }$ for $%
\left\vert \beta \right\vert \leq m-1$.
\end{itemize}

Thanks to our inductive hypothesis \eqref{fp7}, we see that \eqref{fp9}
implies \eqref{fp8}. Therefore, to complete the proof of the Lemma, it is
enough to prove the existence of a $P^{\prime }$ satisfying \eqref{fp9}. For 
$S\subset E$, define 
\begin{equation*}
\hat{\Gamma}\left( S\right) =\left\{ 
\begin{array}{c}
P^{y}:\vec{P}=\left( P^{z}\right) _{z\in S\cup \left\{ x,y\right\} }\in
Wh\left( S\cup \left\{ x,y\right\} \right) \text{, }P^{x}=P\text{,}%
\left\Vert \vec{P}\right\Vert _{\dot{C}^{m}\left( S\cup \left\{ x,y\right\}
\right) }\leq M_{0}, \\ 
P^{z}\in \Gamma _{0}\left( z,M_{0}\right) \text{ for all }z\in S\cup \left\{
x,y\right\} \text{.}%
\end{array}%
\right\}
\end{equation*}%
By definition,

\begin{itemize}
\item[\refstepcounter{equation}\text{(\theequation)}\label{fp10}] $\hat{%
\Gamma}\left( S\right) \subset \Gamma \left( y,S\right) $ for $S\subset E$.
\end{itemize}

Let $S_{1},\cdots ,S_{D+1}\subset E$ with $\#\left( S_{i}\right) \leq \left(
D+2\right) ^{l-1}$ for each $i$.

Then $\hat{\Gamma}\left( S_{1}\right) \cap \cdots \cap \hat{\Gamma}\left(
S_{D+1}\right) \supset \hat{\Gamma}\left( S_{1}\cup \cdots \cup
S_{D+1}\right) $, and $\#\left( S_{1}\cup \cdots \cup S_{D+1}\cup \left\{
y\right\} \right) \leq \left( D+1\right) \left( D+2\right) ^{l-1}+1\leq
\left( D+2\right) ^{l}$. Since $P\in \Gamma _{l}^{fp}\left( x,M_{0}\right) $%
, it follows that there exists $\vec{P}=\left( P^{z}\right) _{z\in S_{1}\cup
\cdots \cup S_{D+1}\cup \left\{ x,y\right\} }\in Wh\left( S_{1}\cup \cdots
\cup S_{D+1}\cup \left\{ x,y\right\} \right) $ such that $P^{x}=P$, $%
\left\Vert \vec{P}\right\Vert _{\dot{C}^{m}\left( S_{1}\cup \cdots \cup
S_{D+1}\cup \left\{ x,y\right\} \right) }\leq M_{0}$, $P^{z}\in \Gamma
_{0}\left( z,M_{0}\right) $ for all $z\in S_{1}\cup \cdots \cup S_{D+1}\cup
\left\{ x,y\right\} $. We then have $P^{y}\in \hat{\Gamma}\left( S_{1}\cup
\cdots \cup S_{D+1}\right) $, hence $\hat{\Gamma}\left( S_{1}\right) \cap
\cdots \hat{\Gamma}\left( S_{D+1}\right) \supset \hat{\Gamma}\left(
S_{1}\cup \cdots \cup S_{D+1}\right) \not=\emptyset $.

By Helly's theorem, there exists

\begin{itemize}
\item[\refstepcounter{equation}\text{(\theequation)}\label{fp11}] $P^{\prime
}\in \bigcap_{S\subset E,\#\left( S\right) \leq \left( D+2\right) ^{l-1}}%
\hat{\Gamma}\left( S\right) $.
\end{itemize}

In particular, $P^{\prime }\in \hat{\Gamma}\left( \emptyset \right) $, which
implies that 
\begin{equation*}
\left\vert \partial ^{\beta }\left( P-P^{\prime }\right) \left( x\right)
\right\vert \leq M_{0}\left\vert x-y\right\vert ^{m-\left\vert \beta
\right\vert }\text{ for }\left\vert \beta \right\vert \leq m-1\text{.}
\end{equation*}%
Also, \eqref{fp10}, \eqref{fp11} imply that 
\begin{equation*}
P^{\prime }\in \bigcap_{S\subset E,\#\left( S\right) \leq \left( D+2\right)
^{l-1}}\Gamma \left( y,S\right) =\Gamma _{l-1}^{fp}\left( y,M_{0}\right) 
\text{.}
\end{equation*}%
Thus, $P^{\prime }$ satisfies \eqref{fp9}, completing the proof of Lemma \ref%
{lemma-fp2}.
\end{proof}

\begin{theorem}[Finiteness Principle for Shape Fields]
\label{theorem-fp-for-wsf} For a large enough $k^{\#}$ determined by $m$, $n$%
, the following holds. Let $\vec{\Gamma}_{0}=\left( \Gamma _{0}\left(
x,M\right) \right) _{x\in E,M>0}$ be a $\left( C_{w},\delta _{\max }\right) $%
-convex shape field and let $Q_{0}\subset \mathbb{R}^{n}$ be a cube of
sidelength $\delta _{Q_{0}}\leq \delta _{\max }$. Also, let $x_{0}\in E\cap
5Q_{0}$ and $M_{0}>0$ be given. Assume that for each $S\subset E$ with $%
\#\left( S\right) \leq k^{\#}$ there exists a Whitney field $\vec{P}%
^{S}=\left( P^{z}\right) _{z\in S}$ such that 
\begin{equation*}
\left\Vert \vec{P}^{S}\right\Vert _{\dot{C}^{m}\left( S\right) }\leq M_{0}%
\text{,}
\end{equation*}%
and 
\begin{equation*}
P^{z}\in \Gamma _{0}\left( z,M_{0}\right) \text{ for all }z\in S\text{.}
\end{equation*}%
Then there exist $P^{0}\in \Gamma _{0}\left( x_{0},M_{0}\right) $ and $F\in
C^{m}\left( Q_{0}\right) $ such that the following hold, with a constant $%
C_{\ast }$ determined by $C_{w}$, $m$, $n$:

\begin{itemize}
\item $J_{z}(F)\in \Gamma _{0}\left( z,C_{\ast }M_{0}\right) $ for all $z\in
E\cap Q_{0}$.

\item $|\partial ^{\beta }\left( F-P^{0}\right) \left( x\right) |\leq
C_{\ast }M_{0}\delta _{Q_{0}}^{m-\left\vert \beta \right\vert }$ for all $%
x\in Q_{0}$, $\left\vert \beta \right\vert \leq m$.

\item In particular, $\left\vert \partial ^{\beta }F\left( x\right)
\right\vert \leq C_{\ast }M_{0}$ for all $x\in Q_{0}$, $\left\vert \beta
\right\vert =m$.
\end{itemize}
\end{theorem}

\begin{proof}
For $l\geq 1$, define $\vec{\Gamma}_{l}=\left( \Gamma _{l}\left( x,M\right)
\right) _{x\in E,M>0}$ and $\vec{\Gamma}_{l}^{fp}=\left( \Gamma
_{l}^{fp}\left( x,M\right) \right) _{x\in E,M>0}$ as in Lemmas \ref%
{lemma-fp1} and \ref{lemma-fp2}. We take $l_{\ast }=100+l\left( \emptyset
\right) $ and $k^{\#}=100+\left( D+2\right) ^{l_{\ast }+100}$. (For the
definition of $l\left( \emptyset \right) $, see Section \ref%
{statement-of-the-main-lemma}.)

Lemmas \ref{lemma-fp1} and \ref{lemma-fp2} show that $\Gamma _{l_{\ast
}}^{fp}\left( x_{0},M_{0}\right) $ is nonempty, hence $\Gamma _{l\left(
\emptyset \right) +1}\left( x_{0},M_{0}\right) $ is nonempty. Pick any $%
P^{0}\in \Gamma _{l\left( \emptyset \right) +1}\left( x_{0},M_{0}\right)
\subset \Gamma _{0}\left( x_{0},M_{0}\right) $. Then Theorem \ref%
{theorem-tu1} in Section \ref{tu} produces a function $F\in C^{m}(Q_{0})$
with the desired properties.
\end{proof}

\section{Finiteness Principle II}

\label{fpii}

In this section, we prove the following result.

\begin{theorem}[Finiteness Principle for Vector-Valued Functions]
\label{theorem-basic-fininite-principle} Fix $m$, $n$, $D\geq 1$. Then there
exist $k^{\#}$, $C$ (determined by $m$, $n$, $D$), such that the following
holds.

Let $E \subset \mathbb{R}^n$ be finite. For each $x\in E$, let $K(x) \subset 
\mathbb{R}^D$ be convex. Assume that for any $S \subseteq E$ with $\#(S)
\leq k^\#$, there exists $F^S \in C^m(\mathbb{R}^n, \mathbb{R}^D)$ such that

\begin{itemize}
\item[\refstepcounter{equation}\text{(\theequation)}\label{fpii1}] $%
||F^S||_{C^m(\mathbb{R}^n, \mathbb{R}^D)} \leq 1$ and $F^S(z) \in K(z)$ for
all $z \in S$.
\end{itemize}

Then there exists $F\in C^{m}(\mathbb{R}^{n},\mathbb{R}^{D})$ such that

\begin{itemize}
\item[\refstepcounter{equation}\text{(\theequation)}\label{fpii2}] $%
||F||_{C^m(\mathbb{R}^n, \mathbb{R}^D)} \leq C$ and $F(z) \in K(z)$ for all $%
z \in E$.
\end{itemize}
\end{theorem}

\begin{proof}
Let us first set up notation. We write $c$, $C$, $C^{\prime }$, etc., to
denote constants determined by $m$, $n$, $D$; these symbols may denote
different constants in different occurrences. We will work with $C^{m}$
vector and scalar-valued functions on $\mathbb{R}^{n}$, and also with $%
C^{m+1}$ scalar-valued functions on $\mathbb{R}^{n+D}$. We use Roman letters
($x$, $y$, $z$$,\cdots $) to denote points of $\mathbb{R}^{n}$, and Greek
letters $(\xi ,\eta ,\zeta ,\cdots )$ to denote points of $\mathbb{R}^{D}$.
We denote points of the $\mathbb{R}^{n+D}$ by $(x,\xi )$, $(y,\eta )$, etc.
As usual, $\mathcal{P}$ denotes the vector space of real-valued polynomials of degree at
most $m-1$ on $\mathbb{R}^{n}$. We write $\mathcal{P}^{D}$ to denote the
direct sum of $D$ copies of $\mathcal{P}$. If $F\in C^{m-1}(\mathbb{R}^{n},%
\mathbb{R}^{D})$ with $F(x)=\left( F_{1}\left( x\right) ,\cdots ,F_{D}\left(
x\right) \right) $ for $x\in \mathbb{R}^{n}$, then $J_{x}(F):=(J_{x}\left(
F_{1}\right) ,\cdots ,J_{x}\left( F_{D}\right) )\in \mathcal{P}^{D}$.

We write $\mathcal{P}^{+}$ to denote the vector space of real-valued polynomials of
degree at most $m$ on $\mathbb{R}^{n+D}$. If $F\in C^{m+1}\left( \mathbb{R}%
^{n+D}\right) $, then we write $J_{\left( x,\xi \right) }^{+}F\in \mathcal{P}%
^{+}$ to denote the $m^{th}$-degree Taylor polynomial of $F$ at the point $%
\left( x,\xi \right) \in \mathbb{R}^{n+D}$.

When we work with $\mathcal{P}^{+}$, we write $\odot _{\left( x,\xi \right)
} $ to denote the multiplication 
\begin{equation*}
P\odot _{\left( x,\xi \right) }Q:=J_{\left( x,\xi \right) }^{+}\left(
PQ\right) \in \mathcal{P}^{+}\text{ for }P,Q\in \mathcal{P}^{+}\text{.}
\end{equation*}

We will use Theorem \ref{theorem-fp-for-wsf} for $C^{m+1}$-functions on $%
\mathbb{R}^{n+D}$. Thus, $m+1$ and $n+D$ will play the r\^oles of $m$, $n$,
respectively, when we apply that theorem.

We take $k^{\#}$ as in Theorem \ref{theorem-fp-for-wsf}, where we use $%
m+1,n+D$ in place of $m,n$, respectively.

We now introduce the relevant shape field.

Let $E^{+}=\left\{ \left( x,0\right) \in \mathbb{R}^{n+D}:x\in E\right\} $.
For $\left( x_{0},0\right) \in E^{+}$ and $M>0$, let

\begin{itemize}
\item[\refstepcounter{equation}\text{(\theequation)}\label{fpii3}] $\Gamma
\left( \left( x_0,0\right) ,M\right) =\left\{ 
\begin{array}{c}
P\in \mathcal{P}^{+}:P\left( x_{0},0\right) =0,\nabla _{\xi }P\left(
x_{0},0\right) \in K\left( x_{0}\right) , \\ 
\left\vert \partial _{x}^{\alpha }\partial _{\xi }^{\beta }P\left(
x_{0},0\right) \right\vert \leq M\text{ for }\left\vert \alpha \right\vert
+\left\vert \beta \right\vert \leq m%
\end{array}%
\right\} \mathcal{\subset \mathcal{P}}^{+}$.
\end{itemize}

Let $\vec{\Gamma}=\left( \Gamma \left( x_{0},0\right) ,M\right) $ $_{\left(
x_0,0\right) \in E^{+},M>0}$.

\begin{lemma}
\label{lemma-bfp1} $\vec{\Gamma}$ is a $(C,1)$-convex shape field.
\end{lemma}

\begin{proof}[Proof of Lemma \protect\ref{lemma-bfp1}]
Clearly, each $\Gamma ((x_0,0),M)$ is a (possibly empty) convex subset of $%
\mathcal{P}^{+}$, and $M^{\prime }\leq M$ implies $\Gamma \left( \left(
x_{0},0\right) ,M^{\prime }\right) \subseteq \Gamma \left( \left(
x_{0},0\right) ,M\right) $. Thus, $\vec{\Gamma}$ is a shape field (with $m+1$%
, $n+D$ playing the r\^oles of $m$, $n$, respectively). To prove $(C,1)$%
-convexity, let $x_{0}\in E,0<\delta \leq 1$, let

\begin{itemize}
\item[\refstepcounter{equation}\text{(\theequation)}\label{fpii4}] $%
P_{1},P_{2}\in \Gamma \left( \left( x_{0},0\right) ,M\right) $ with
\end{itemize}

\begin{itemize}
\item[\refstepcounter{equation}\text{(\theequation)}\label{fpii5}] $%
\left\vert \partial _{x}^{\alpha }\partial _{\xi }^{\beta }\left(
P_{1}-P_{2}\right) \left( x_{0},0\right) \right\vert \leq M\delta ^{\left(
m+1\right) -\left\vert \alpha \right\vert -\left\vert \beta \right\vert }$
for $\left\vert \alpha \right\vert +\left\vert \beta \right\vert \leq m$;
and let
\end{itemize}

\begin{itemize}
\item[\refstepcounter{equation}\text{(\theequation)}\label{fpii6}] $%
Q_{1},Q_{2}\in \mathcal{P}^{+}$, with
\end{itemize}

\begin{itemize}
\item[\refstepcounter{equation}\text{(\theequation)}\label{fpii7}] $%
\left\vert \partial _{x}^{\alpha }\partial _{\xi }^{\beta }Q_{i}\left(
x_{0},0\right) \right\vert \leq \delta ^{-\left\vert \alpha \right\vert
-\left\vert \beta \right\vert }$ for $i=1,2$, $\left\vert \alpha \right\vert
+\left\vert \beta \right\vert \leq m$, and with
\end{itemize}

\begin{itemize}
\item[\refstepcounter{equation}\text{(\theequation)}\label{fpii8}] $%
Q_{1}\odot _{\left( x_{0},0\right) }Q_{1}+Q_{2}\odot _{\left( x_{0},0\right)
}Q_{2}=1$.
\end{itemize}

We must show that the polynomial

\begin{itemize}
\item[\refstepcounter{equation}\text{(\theequation)}\label{fpii9}] $%
P:=Q_{1}\odot _{\left( x_{0},0\right) }Q_{1}\odot _{\left( x_{0},0\right)
}P_{1}+Q_{2}\odot _{\left( x_{0},0\right) }Q_{2}\odot _{\left(
x_{0},0\right) }P_{2}$ \end{itemize}
belongs to $\Gamma \left( \left( x_{0},0\right)
,CM\right) $.

From \eqref{fpii3}, \eqref{fpii4}, we have

\begin{itemize}
\item[\refstepcounter{equation}\text{(\theequation)}\label{fpii10}] $%
\left\vert \partial _{x}^{\alpha }\partial _{\xi }^{\beta }P_{1}\left(
x_{0},0\right) \right\vert \leq M$ for $\left\vert \alpha \right\vert
+\left\vert \beta \right\vert \leq m$,
\end{itemize}

\begin{itemize}
\item[\refstepcounter{equation}\text{(\theequation)}\label{fpii11}] $%
P_{1}\left( x_{0},0\right) =P_{2}\left( x_{0},0\right) =0$, and
\end{itemize}

\begin{itemize}
\item[\refstepcounter{equation}\text{(\theequation)}\label{fpii12}] $\nabla
_{\xi }P_{1}\left( x_{0},0\right) $, $\nabla _{\xi }P_{2}\left(
x_{0},0\right) \in K\left( x_{0}\right) $.
\end{itemize}

Then \eqref{fpii9}, \eqref{fpii11} give 
\begin{equation*}
P\left( x_{0},0\right) =0
\end{equation*}%
and 
\begin{equation*}
\nabla _{\xi }P\left( x_{0},0\right) =\left( Q_{1}\left( x_{0},0\right)
\right) ^{2}\nabla _{\xi }P_{1}\left( x_{0},0\right) +\left( Q_{2}\left(
x_{0},0\right) \right) ^{2}\nabla _{\xi }P_{2}\left( x_{0},0\right)
\end{equation*}%
while \eqref{fpii8} yields%
\begin{equation*}
\left( Q_{1}\left( x_{0},0\right) \right) ^{2}+\left( Q_{2}\left(
x_{0},0\right) \right) ^{2}=1\text{.}
\end{equation*}%
Together with \eqref{fpii12} and convexity of $K(x_{0})$, the above remarks
imply that

\begin{itemize}
\item[\refstepcounter{equation}\text{(\theequation)}\label{fpii13}] $P\left(
x_{0},0\right) =0$ and $\nabla _{\xi }P\left( x_{0},0\right) \in K\left(
x_{0}\right) $.
\end{itemize}

Also, \eqref{fpii8}, \eqref{fpii9} imply the formula

\begin{itemize}
\item[\refstepcounter{equation}\text{(\theequation)}\label{fpii14}] $%
P=P_{1}+Q_{2}\odot _{\left( x_{0},0\right) }Q_{2}\odot _{\left(
x_{0},0\right) }\left( P_{2}-P_{1}\right) $.
\end{itemize}

From \eqref{fpii5}, \eqref{fpii7}, and $\delta \leq 1$, we have 
\begin{eqnarray*}
\left\vert \partial _{x}^{\alpha }\partial _{\xi }^{\beta }\left[ Q_{2}\odot
_{\left( x_{0},0\right) }Q_{2}\odot _{\left( x_{0},0\right) }\left(
P_{2}-P_{1}\right) \right] \left( x_{0},0\right) \right\vert &\leq &CM\delta
^{\left( m+1\right) -\left\vert \alpha \right\vert -\left\vert \beta
\right\vert } \\
&\leq &CM
\end{eqnarray*}%
for $\left\vert \alpha \right\vert +\left\vert \beta \right\vert \leq m$.
Together with \eqref{fpii10} and \eqref{fpii14}, this tells us that

\begin{itemize}
\item[\refstepcounter{equation}\text{(\theequation)}\label{fpii15}] $%
\left\vert \partial _{x}^{\alpha }\partial _{\xi }^{\beta }P\left(
x_{0},0\right) \right\vert \leq CM$ for $\left\vert \alpha \right\vert
+\left\vert \beta \right\vert \leq m$.
\end{itemize}

From \eqref{fpii13}, \eqref{fpii15} and the definition \eqref{fpii3}, we see
that $P\in \Gamma \left( \left( x_{0},0\right) ,CM\right) $, completing the
proof of Lemma \ref{lemma-bfp1}.
\end{proof}

\begin{lemma}
\label{lemma-bfp2} Let $S^{+}\subset E^{+}$ with $\#\left( S^{+}\right) \leq
k^{\#}$. Then there exists $\vec{P}=\left( P^{z}\right) _{z\in S^{+}}$, with
each $P^{z}\in \mathcal{P}^{+}$, such that

\begin{itemize}
\item[\refstepcounter{equation}\text{(\theequation)}\label{fpii16}] $%
P^{z}\in \Gamma \left( z,C\right) $ for each $z\in S^{+}$, and
\end{itemize}

\begin{itemize}
\item[\refstepcounter{equation}\text{(\theequation)}\label{fpii17}] $%
\left\vert \partial _{x}^{\alpha }\partial _{\xi }^{\beta }\left(
P^{z}-P^{z^{\prime }}\right) \left( z\right) \right\vert \leq C\left\vert
z-z^{\prime }\right\vert ^{\left( m+1\right) -\left\vert \alpha \right\vert
-\left\vert \beta \right\vert }$ for $z$, $z^{\prime }\in S^{+}$ and $%
\left\vert \alpha \right\vert +\left\vert \beta \right\vert \leq m$.
\end{itemize}
\end{lemma}

\begin{proof}[Proof of Lemma \protect\ref{lemma-bfp2}]
Since $E^{+}=E\times \left\{ 0\right\} $, we have $S^{+}=S\times \left\{
0\right\} $ for an $S\subset E$ with $\#\left( S\right) \leq k^{\#}$. By
hypothesis of Theorem \ref{theorem-basic-fininite-principle}, there exists $%
F^{S}\in C^{m}\left( \mathbb{R}^{n},\mathbb{R}^{D}\right) $ such that

\begin{itemize}
\item[\refstepcounter{equation}\text{(\theequation)}\label{fpii18}] $%
\left\Vert F^{S}\right\Vert _{C^{m}\left( \mathbb{R}^{n},\mathbb{R}%
^{D}\right) }\leq 1$ and $F^{S}\left( x_{0}\right) \in K\left( x_{0}\right) $
for all $x_{0}\in S$.
\end{itemize}

Let $F^{S}\left( x\right) =\left( F_{1}^{S}\left( x\right) ,\cdots
,F_{D}^{S}\left( x\right) \right) $ for $x\in \mathbb{R}^{n}$, and let $\vec{P%
}=\left( P^{\left( x_{0},0\right) }\right) _{\left( x_{0},0\right) \in
S\times \left\{ 0\right\} }$ with

\begin{itemize}
\item[\refstepcounter{equation}\text{(\theequation)}\label{fpii19}] $%
P^{\left( x_{0},0\right) }\left( x,\xi \right) =\sum_{i=1}^{D}\xi _{i}\left[
J_{x_{0}}\left( F_{i}^{S}\right) \left( x\right) \right] $ for $x\in \mathbb{%
R}^{n}$, $\xi =\left( \xi _{1},\cdots ,\xi _{D}\right) \in \mathbb{R}^{D}$.
\end{itemize}

Each $P^{(x_0,0)}$ belongs to $\mathcal{P}^+$ and satisfies

\begin{itemize}
\item[\refstepcounter{equation}\text{(\theequation)}\label{fpii20}] $%
P^{\left( x_{0},0\right) }\left( x_{0},0\right) =0$, $\nabla _{\xi
}P^{\left( x_{0},0\right) }\left( x_{0},0\right) \in K\left( x_{0}\right) $,
and
\end{itemize}

and

\begin{itemize}
\item[\refstepcounter{equation}\text{(\theequation)}\label{fpii21}] $%
\left\vert \partial _{x}^{\alpha }\partial _{\xi }^{\beta }P^{\left(
x_{0},0\right) }\left( x_{0},0\right) \right\vert \leq C$ for $\left\vert
\alpha \right\vert +\left\vert \beta \right\vert \leq m$,
\end{itemize}

thanks to \eqref{fpii18}, \eqref{fpii19}. Our results \eqref{fpii20}, %
\eqref{fpii21} and definition \eqref{fpii3} together imply \eqref{fpii16}.
We pass to \eqref{fpii17}. Let $\left( x_{0},0\right) ,\left( y_{0},0\right)
\in S^{+}=S\times \left\{ 0\right\} $. From \eqref{fpii18}, \eqref{fpii19},
we have 
\begin{eqnarray*}
\left\vert \partial _{x}^{\alpha }\partial _{\xi _{j}}\left( P^{\left(
x_{0},0\right) }-P^{\left( y_{0},0\right) }\right) \left( x_{0},0\right)
\right\vert  &=&\left\vert \partial _{x}^{\alpha }\left( J_{x_{0}}\left(
F_{j}^{S}\right) -J_{y_{0}}\left( F_{j}^{S}\right) \right) \left(
x_{0}\right) \right\vert  \\
&\leq &C\left\vert x_{0}-y_{0}\right\vert ^{m-\left\vert \alpha \right\vert }
\\
&=&C\left\vert x_{0}-y_{0}\right\vert ^{\left( m+1\right) -\left( \left\vert
\alpha \right\vert +1\right) }
\end{eqnarray*}%
for $\left\vert \alpha \right\vert \leq m-1$, $j=1,\cdots ,D$. For $|\beta
|\not=1$, we have 
\begin{equation*}
\partial _{x}^{\alpha }\partial _{\xi }^{\beta }\left( P^{\left(
x_{0},0\right) }-P^{\left( y_{0},0\right) }\right) \left( x_{0},0\right) =0
\end{equation*}%
by \eqref{fpii19}. The above remarks imply \eqref{fpii17}, completing the
proof of Lemma \ref{lemma-bfp2}.
\end{proof}

\begin{lemma}
\label{lemma-bfp3} Given a cube $Q\subset \mathbb{R}^{n}$ of sidelength $%
\delta _{Q}=1$, there exists $F^{Q}\in C^{m}(Q,\mathbb{R}^{D})$ such that

\begin{itemize}
\item[\refstepcounter{equation}\text{(\theequation)}\label{fpii22}] $%
\left\vert \partial ^{\alpha }F^{Q}\left( x\right) \right\vert \leq C$ for $%
x\in Q$, $\left\vert \alpha \right\vert \leq m$; and
\end{itemize}

\begin{itemize}
\item[\refstepcounter{equation}\text{(\theequation)}\label{fpii23}] $%
F^{Q}\left( z\right) \in K\left( z\right) $ for all $z\in E\cap Q$.
\end{itemize}
\end{lemma}

\begin{proof}[Proof of Lemma \protect\ref{lemma-bfp3}]
If $E\cap Q=\emptyset $, then we can just take $F^{Q}\equiv 0$. Otherwise,
pick $x_{00}\in E\cap Q$, let $Q^{\prime }\in \mathbb{R}^{D}$ be a cube of
sidelength $\delta _{Q^{\prime }}=1$, containing the origin in its interior,
and apply Theorem \ref{theorem-fp-for-wsf} (with $m+1$, $n+D$ in place of $m$%
, $n$, respectively) to the shape field $\vec{\Gamma}=\left( \Gamma \left(
x_{0},0\right) ,M\right) _{\left( x_{0},0\right) \in E^{+},M>0}$ given by %
\eqref{fpii3}, the cube $Q_{0}:=Q\times Q^{\prime }\subset \mathbb{R}^{n+D}$%
, the point $(x_{00},0)$, and the number $M_{0}=C$.

Lemmas \ref{lemma-bfp1} and \ref{lemma-bfp2} tell us that the above data
satisfy the hypotheses of Theorem \ref{theorem-fp-for-wsf}. Applying Theorem %
\ref{theorem-fp-for-wsf}, we obtain

\begin{itemize}
\item[\refstepcounter{equation}\text{(\theequation)}\label{fpii24}] $%
P_{0}\in \Gamma \left( \left( x_{00},0\right) ,C\right) $ and $F\in
C^{m+1}\left( Q\times Q^{\prime }\right) $ such that
\end{itemize}

\begin{itemize}
\item[\refstepcounter{equation}\text{(\theequation)}\label{fpii25}] $%
\left\vert \partial _{x}^{\alpha }\partial _{\xi }^{\beta }\left(
F-P_{0}\right) \left( x,\xi \right) \right\vert \leq C$ for $\left\vert
\alpha \right\vert +\left\vert \beta \right\vert \leq m+1$ and $\left( x,\xi
\right) \in Q\times Q^{\prime }$; and
\end{itemize}

\begin{itemize}
\item[\refstepcounter{equation}\text{(\theequation)}\label{fpii26}] $%
J_{\left( z,0\right) }^{+}\left( F\right) \in \Gamma \left( \left(
z,0\right) ,C\right) $ for all $z\in E\cap Q$.
\end{itemize}

By \eqref{fpii24}, \eqref{fpii26} and definition \eqref{fpii3}, we have

\begin{itemize}
\item[\refstepcounter{equation}\text{(\theequation)}\label{fpii27}] $%
\left\vert \partial _{x}^{\alpha }\partial _{\xi }^{\beta }P_{0}\left(
x_{00},0\right) \right\vert \leq C$ for $\left\vert \alpha \right\vert
+\left\vert \beta \right\vert \leq m$
\end{itemize}

and

\begin{itemize}
\item[\refstepcounter{equation}\text{(\theequation)}\label{fpii28}] $\nabla
_{\xi }F\left( z,0\right) \in K\left( z\right) $ for all $z\in E\cap Q$.
\end{itemize}

Since $\left( x_{00},0\right) \in Q\times Q^{\prime }$ and $\delta _{Q\times
Q^{\prime }}=1$, \eqref{fpii27} implies that 
\begin{equation*}
\left\vert \partial _{x}^{\alpha }\partial _{\xi }^{\beta }P_{0}\left( x,\xi
\right) \right\vert \leq C
\end{equation*}%
for $\left( x,\xi \right) \in Q\times Q^{\prime }$, $\left\vert \alpha
\right\vert +\left\vert \beta \right\vert \leq m+1$. (Recall that $P_{0}$ is
a polynomial of degree at most $m$.) Together with \eqref{fpii25}, this
implies that

\begin{itemize}
\item[\refstepcounter{equation}\text{(\theequation)}\label{fpii29}] $%
\left\vert \partial _{x}^{\alpha }\partial _{\xi }^{\beta }F\left( x,\xi
\right) \right\vert \leq C$ for $\left( x,\xi \right) \in Q\times Q^{\prime }
$, $\left\vert \alpha \right\vert +\left\vert \beta \right\vert \leq m+1$.
\end{itemize}

Taking 
\begin{equation*}
F^{Q}\left( x\right) =\nabla _{\xi }F\left( x,0\right) \text{ for }x\in Q%
\text{,}
\end{equation*}%
we learn from \eqref{fpii28}, \eqref{fpii29} that $F^{Q}\in C^{m}\left( Q,%
\mathbb{R}^{D}\right) $; $\left\vert \partial ^{\alpha }F^{Q}\left( x\right)
\right\vert \leq C$ for $x\in Q$, $\left\vert \alpha \right\vert \leq m$;
and $F^{Q}\left( z\right) \in K\left( z\right) $ for all $z\in E\cap Q$.
Thus, $F^{Q}$ satisfies \eqref{fpii22} and \eqref{fpii23}, completing the
proof of Lemma \ref{lemma-bfp3}.
\end{proof}

Now, we can easily finish the proof of Theorem \ref%
{theorem-basic-fininite-principle}. We introduce a partition of unity%
\begin{equation*}
1=\sum_{\nu }\theta _{\nu }\text{ on }\mathbb{R}^{n}\text{,}
\end{equation*}%
where for each $\nu $: $\theta _{\nu }\in C^{m}\left( \mathbb{R}^{n}\right) $%
; $\theta _{\nu }\geq 0$, $\left\vert \partial ^{\alpha }\theta _{\nu
}\right\vert \leq C$ for $\left\vert \alpha \right\vert \leq m$, support $%
\theta _{\nu }\subset Q_{\nu }$ for a cube $Q_{\nu }$ of sidelength $\delta
_{Q_{\nu }}=1$; and any given point $x\in \mathbb{R}^{n}$ has a neighborhood
that intersects $Q_{\nu }$ for at most $C$ distinct $\nu $. For each $Q_{\nu
}$, we apply Lemma \ref{lemma-bfp3} to produce a function $F_{\nu }\in
C^{m}\left( Q_{\nu },\mathbb{R}^{D}\right) $ such that $\left\vert \partial
^{\alpha }F_{\nu }\left( x\right) \right\vert \leq C$ for $x\in Q_{\nu }$, $%
\left\vert \alpha \right\vert \leq m$; and $F_{\nu }\left( z\right) \in
K\left( z\right) $ for all $z\in E\cap Q_{\nu }$.

We then define

\begin{itemize}
\item[\refstepcounter{equation}\text{(\theequation)}\label{fpii30}] $%
F=\sum_{\nu }\theta _{\nu }F_{\nu }$ on $\mathbb{R}^{n}$.
\end{itemize}

In a small enough neighborhood of any given point of $\mathbb{R}^{n}$, this
sum contains at most $C$ nonzero terms, and for each $\nu $, we have 
\begin{equation*}
\left\vert \partial ^{\alpha }\left( \theta _{\nu }F_{\nu }\right) \left(
x\right) \right\vert \leq C\text{ for }x\in \mathbb{R}^{n}\text{, }%
\left\vert \alpha \right\vert \leq m\text{.}
\end{equation*}
Therefore, $F\in C^{m}\left( \mathbb{R}^{n},\mathbb{R}^{D}\right) $, and $%
\left\vert \partial ^{\alpha }F\left( x\right) \right\vert \leq C$ for $x\in 
\mathbb{R}^{n}$, $\left\vert \alpha \right\vert \leq m$; i.e.,

\begin{itemize}
\item[\refstepcounter{equation}\text{(\theequation)}\label{fpii31}] $%
\left\Vert F\right\Vert _{C^{m}\left( \mathbb{R}^{n},\mathbb{R}^{D}\right)
}\leq C$.
\end{itemize}

Moreover, let $z\in E$. Then $\theta _{\nu }(z)$ is nonzero for at most $C$
distinct $\nu $, each $\theta _{\nu }\left( z\right) $ is nonnegative; and $%
\sum_{\nu }\theta _{\nu }\left( z\right) =1$. Since $z\in E$, we have $%
F_{\nu }\left( z\right) \in K\left( z\right) $ whenever $z\in $ support $%
\theta _{\nu }$. Therefore, \eqref{fpii30} exhibits $F(z)$ as a convex
combination of vectors in $K(z)\subset \mathbb{R}^{D}$. Since $K(z)$ is
convex, we have

\begin{itemize}
\item[\refstepcounter{equation}\text{(\theequation)}\label{fpii32}] $F\left(
z\right) \in K\left( z\right) $ for all $z\in E$.
\end{itemize}

Thanks to \eqref{fpii31}, \eqref{fpii32}, our $F\in C^{m}\left( \mathbb{R}%
^{n},\mathbb{R}^{D}\right) $ satisfies \eqref{fpii2}.

This completes the proof of Theorem \ref{theorem-basic-fininite-principle}.
\end{proof}

\section{Interpolation by Nonnegative Functions}

\label{inf}

In this section, $c$, $C$, $C^{\prime }$, etc. denote constants determined
by $m$ and $n$. These symbols may denote different constants in different
occurrences. For $x \in \mathbb{R}^n$ and $M >0$, define

\begin{itemize}
\item[\refstepcounter{equation}\text{(\theequation)}\label{inf1}] $\Gamma
_{\ast }\left( x,M\right) =\left\{ 
\begin{array}{c}
P\in \mathcal{P}:\text{ There exists }F\in C^{m}\left( \mathbb{R}^{n}\right) 
\text{ with }\left\Vert F\right\Vert _{C^{m}\left( \mathbb{R}^{n}\right)
}\leq M\text{,} \\ 
F\geq 0\text{ on }\mathbb{R}^{n}\text{, }J_{x}\left( F\right) =P\text{.}%
\end{array}%
\right\} $
\end{itemize}

It is not immediately clear how to compute $\Gamma _{\ast }$; we will return
to this issue in a later section. Let $E\subset \mathbb{R}^{n}$ be finite,
and let $f:E\rightarrow \lbrack 0,\infty )$. Define $\vec{\Gamma}%
_{f}=(\Gamma _{f}(x,M))_{x\in E,M>0}$, where

\begin{itemize}
\item[\refstepcounter{equation}\text{(\theequation)}\label{inf2}] $\Gamma
_{f}\left( x,M\right) =\left\{ P\in \Gamma _{\ast }\left( x,M\right)
:P\left( x\right) =f\left( x\right) \right\} $.
\end{itemize}

\begin{lemma}
\label{lemma-inf1} $\vec{\Gamma}_f$ is a $(C,1)$-convex shape field.
\end{lemma}

\begin{proof}
It is clear that $\vec{\Gamma}_f$ is a shape field, i.e., each $\Gamma_f(x,M)
$ is convex, and $M^{\prime }\leq M$ implies $\Gamma_f(x,M^{\prime })
\subseteq \Gamma_f(x,M)$. To establish $(C,1)$-convexity, suppose we are
given the following:

\begin{itemize}
\item[\refstepcounter{equation}\text{(\theequation)}\label{inf3}] $0<\delta
\leq 1$, $x\in E$, $M>0$;
\end{itemize}

\begin{itemize}
\item[\refstepcounter{equation}\text{(\theequation)}\label{inf4}] $%
P_{1},P_{2}\in \Gamma _{f}\left( x,M\right) $ satisfying
\end{itemize}

\begin{itemize}
\item[\refstepcounter{equation}\text{(\theequation)}\label{inf5}] $%
\left\vert \partial ^{\beta }\left( P_{1}-P_{2}\right) \left( x\right)
\right\vert \leq M\delta ^{m-\left\vert \beta \right\vert }$ for $\left\vert
\beta \right\vert \leq m-1$;
\end{itemize}

\begin{itemize}
\item[\refstepcounter{equation}\text{(\theequation)}\label{inf6}] $%
Q_{1},Q_{2}\in \mathcal{P}$ satisfying
\end{itemize}

\begin{itemize}
\item[\refstepcounter{equation}\text{(\theequation)}\label{inf7}] $%
\left\vert \partial ^{\beta }Q_{i}\left( x\right) \right\vert \leq \delta
^{-\left\vert \beta \right\vert }$ for $\left\vert \beta \right\vert \leq
m-1 $, $i=1,2$, and
\end{itemize}

\begin{itemize}
\item[\refstepcounter{equation}\text{(\theequation)}\label{inf8}] $%
Q_{1}\odot _{x}Q_{1}+Q_{2}\odot _{x}Q_{2}=1$.
\end{itemize}

Set

\begin{itemize}
\item[\refstepcounter{equation}\text{(\theequation)}\label{inf9}] $%
P=Q_{1}\odot _{x}Q_{1}\odot _{x}P_{1}+Q_{2}\odot _{x}Q_{2}\odot _{x}P_{2}$.
\end{itemize}

We must prove that

\begin{itemize}
\item[\refstepcounter{equation}\text{(\theequation)}\label{inf10}] $P\in
\Gamma _{f}\left( x,CM\right) $.
\end{itemize}

Thanks to \eqref{inf4}, we have

\begin{itemize}
\item[\refstepcounter{equation}\text{(\theequation)}\label{inf11}] $%
P_{1}\left( x\right) =f\left( x\right) $ and $P_{2}\left( x\right) =f\left(
x\right) $,
\end{itemize}

and there exist functions $F_1, F_2 \in C^m(\mathbb{R}^n)$ such that

\begin{itemize}
\item[\refstepcounter{equation}\text{(\theequation)}\label{inf12}] $%
\left\Vert F_{i}\right\Vert _{C^{m}\left( \mathbb{R}^{n}\right) }\leq M$ $%
\left( i=1,2\right) $,
\end{itemize}

\begin{itemize}
\item[\refstepcounter{equation}\text{(\theequation)}\label{inf13}] $%
F_{i}\geq 0$ on $\mathbb{R}^{n}$ $\left( i=1,2\right) $, and
\end{itemize}

\begin{itemize}
\item[\refstepcounter{equation}\text{(\theequation)}\label{inf14}] $%
J_{x}\left( F_{i}\right) =P_{i}$ $\left( i=1,2\right) $.
\end{itemize}

We fix $F_1$, $F_2$ as above. By \eqref{inf8}, we have $|Q_i(x)| \geq \frac{1%
}{\sqrt{2}}$ for $i =1$ or for $i=2$. By possibly interchanging $Q_1$ and $Q_2$, and then possibly changing $Q_1$ to $-Q_1$,
we may suppose that

\begin{itemize}
\item[\refstepcounter{equation}\text{(\theequation)}\label{inf15}] $%
Q_{1}\left( x\right) \geq \frac{1}{\sqrt{2}}$.
\end{itemize}

For small enough $c_0$, \eqref{inf7} and \eqref{inf15} yield

\begin{itemize}
\item[\refstepcounter{equation}\text{(\theequation)}\label{inf16}] $%
Q_{1}\left( y\right) \geq \frac{1}{10}$ for $\left\vert y-x\right\vert \leq
c_{0}\delta $.
\end{itemize}

Fix $c_0$ as in \eqref{inf16}. We introduce a $C^m$ cutoff function $\chi$
on $\mathbb{R}^n$ with the following properties.

\begin{itemize}
\item[\refstepcounter{equation}\text{(\theequation)}\label{inf17}] $0\leq
\chi \leq 1$ on $\mathbb{R}^{n}$; $\chi =0$ outside $B_{n}\left(
x,c_{0}\delta \right) $; $\chi =1$ in a neighborhood of $x$;
\end{itemize}

\begin{itemize}
\item[\refstepcounter{equation}\text{(\theequation)}\label{inf18}] $%
\left\vert \partial ^{\beta }\chi \right\vert \leq C\delta ^{-\left\vert
\beta \right\vert }$ on $\mathbb{R}^{n}$, for $\left\vert \beta \right\vert
\leq m$.
\end{itemize}

We then define $\tilde{\theta}_{1}=\chi \cdot Q_{1}+\left( 1-\chi \right) $
and $\tilde{\theta}_{2}=\chi \cdot Q_{2}$.

These functions satisfy the following: $\tilde{\theta}_{i}\in C^{m}\left( 
\mathbb{R}^{n}\right) $ and $\left\vert \partial ^{\beta }\tilde{\theta}%
_{i}\right\vert \leq C\delta ^{-\left\vert \beta \right\vert }$ on $\mathbb{R%
}^{n}$ for $\left\vert \beta \right\vert \leq m$, $i=1,2$; $\tilde{\theta}%
_{1}\geq \frac{1}{10}$ on $\mathbb{R}^{n}$; $J_{x}\left( \tilde{\theta}%
_{i}\right) =Q_{i}$ for $i=1,2$; outside $B_{n}\left( x,c_{0}\delta \right) $
we have $\tilde{\theta}_{1}=1$ and $\tilde{\theta}_{2}=0$. Setting $\theta
_{i}=\tilde{\theta}_{i}\cdot \left( \tilde{\theta}_{1}^{2}+\tilde{\theta}%
_{2}^{2}\right) ^{-1/2}$ for $i=1$, $2$, we find that

\begin{itemize}
\item[\refstepcounter{equation}\text{(\theequation)}\label{inf19}] $\theta
_{i}\in C^{m}\left( \mathbb{R}^{n}\right) $ and $\left\vert \partial ^{\beta
}\theta _{i}\right\vert \leq C\delta ^{-\left\vert \beta \right\vert }$ on $%
\mathbb{R}^{n}$ for $\left\vert \beta \right\vert \leq m$, $i=1,2$;
\end{itemize}

\begin{itemize}
\item[\refstepcounter{equation}\text{(\theequation)}\label{inf20}] $\theta
_{1}^{2}+\theta _{2}^{2}= 1$ on $\mathbb{R}^{n}$;
\end{itemize}

\begin{itemize}
\item[\refstepcounter{equation}\text{(\theequation)}\label{inf21}] $%
J_{x}\left( \theta _{i}\right) =Q_{i}$ for $i=1,2$ (here we use (\ref{inf8}%
)); and
\end{itemize}

\begin{itemize}
\item[\refstepcounter{equation}\text{(\theequation)}\label{inf22}] outside $%
B_{n}\left( x,c_{0}\delta \right) $ we have $\theta _{1}=1$ and $\theta
_{2}=0$.
\end{itemize}

Now set

\begin{itemize}
\item[\refstepcounter{equation}\text{(\theequation)}\label{inf23}] $F=\theta
_{1}^{2}F_{1}+\theta _{2}^{2}F_{2}=F_{1}+\theta _{2}^{2}\left(
F_{2}-F_{1}\right) $ (see (\ref{inf20})).
\end{itemize}

Clearly $F\in C^{m}(\mathbb{R}^{n})$. By \eqref{inf14}, we have $%
J_{x}(F_{2}-F_{1})=P_{2}-P_{1}$; hence \eqref{inf5} yields the estimate%
\begin{equation*}
\left\vert \partial ^{\beta }\left( F_{2}-F_{1}\right) \left( x\right)
\right\vert \leq CM\delta ^{m-\left\vert \beta \right\vert }\text{ for }%
\left\vert \beta \right\vert \leq m-1\text{.}
\end{equation*}

Together with \eqref{inf12}, this tells us that%
\begin{equation*}
\left\vert \partial ^{\beta }\left( F_{2}-F_{1}\right) \right\vert \leq
CM\delta ^{m-\left\vert \beta \right\vert }\text{ on }B_{n}\left(
x,c_{0}\delta \right) \text{ for }\left\vert \beta \right\vert \leq m\text{.}
\end{equation*}

Recalling \eqref{inf19}, we deduce that 
\begin{equation*}
\left\vert \partial ^{\beta }\left( \theta _{2}^{2}\cdot \left(
F_{2}-F_{1}\right) \right) \right\vert \leq CM\delta ^{m-\left\vert \beta
\right\vert }\text{ on }B_{n}\left( x,c_{0}\delta \right) \text{ for }%
\left\vert \beta \right\vert \leq m\text{.}
\end{equation*}%
Together with \eqref{inf12} and \eqref{inf23}, this implies
that 
\begin{equation*}
\left\vert \partial ^{\beta }F\right\vert \leq CM\text{ on }B_{n}\left(
x,c_{0}\delta \right) \text{,}
\end{equation*}%
since $0<\delta \leq 1$ (see \eqref{inf3}). On the other hand, outside $%
B_{n}(x,c_{0}\delta )$ we have $F=F_{1}$ by \eqref{inf22}, \eqref{inf23};
hence $|\partial ^{\beta }F|\leq CM$ outside $B_{n}(x,c_{0}\delta )$ for $%
|\beta |\leq m$, by \eqref{inf12}. Thus, $|\partial ^{\beta }F|\leq CM$ on
all of $\mathbb{R}^{n}$ for $|\beta |\leq m$, i.e.,

\begin{itemize}
\item[\refstepcounter{equation}\text{(\theequation)}\label{inf24}] $%
\left\Vert F\right\Vert _{C^{m}\left( \mathbb{R}^{n}\right) }\leq CM$.
\end{itemize}

Also, from \eqref{inf13} and \eqref{inf23} we have

\begin{itemize}
\item[\refstepcounter{equation}\text{(\theequation)}\label{inf25}] $F\geq 0$
on $\mathbb{R}^{n}$;
\end{itemize}

and \eqref{inf9}, \eqref{inf14}, \eqref{inf21}, \eqref{inf23}
imply that

\begin{itemize}
\item[\refstepcounter{equation}\text{(\theequation)}\label{inf26}] $%
J_{x}\left( F\right) =Q_{1}\odot _{x}Q_{1}\odot _{x}P_{1}+Q_{2}\odot
_{x}Q_{2}\odot _{x}P_{2}=P$.
\end{itemize}

Since $F\in C^{m}\left( \mathbb{R}^{n}\right) $ satisfies \eqref{inf24}, %
\eqref{inf25}, \eqref{inf26}, we have

\begin{itemize}
\item[\refstepcounter{equation}\text{(\theequation)}\label{inf27}] $P\in
\Gamma _{\ast }\left( x,CM\right) $.
\end{itemize}

Moreover,

\begin{itemize}
\item[\refstepcounter{equation}\text{(\theequation)}\label{inf28}] $P\left(
x\right) =\left( Q_{1}\left( x\right) \right) ^{2}f\left( x\right) +\left(
Q_{2}\left( x\right) \right) ^{2}f\left( x\right) =f\left( x\right) $,
\end{itemize}
thanks to \eqref{inf8}, \eqref{inf9}, \eqref{inf11}. From \eqref{inf27}, %
\eqref{inf28} we conclude that $P \in \Gamma_f(x,CM)$, completing the proof
of Lemma \ref{lemma-inf1}.
\end{proof}

\begin{lemma}
\label{lemma-inf2} Let $\left( P^{x}\right) _{x\in E}$ be a Whitney field on
the finite set $E$, and let $M>0$. Suppose that

\begin{itemize}
\item[\refstepcounter{equation}\text{(\theequation)}\label{inf29}] $P^{x}\in
\Gamma _{\ast }\left( x,M\right) $ for each $x\in E$,
\end{itemize}

and that

\begin{itemize}
\item[\refstepcounter{equation}\text{(\theequation)}\label{inf30}] $%
\left\vert \partial ^{\beta }\left( P^{x}-P^{x^{\prime }}\right) \left(
x\right) \right\vert \leq M\left\vert x-x^{\prime }\right\vert
^{m-\left\vert \beta \right\vert }$ for $x,x^{\prime }\in E$ and $\left\vert
\beta \right\vert \leq m-1$.
\end{itemize}

Then there exists $F\in C^{m}(\mathbb{R}^{n})$ such that

\begin{itemize}
\item[\refstepcounter{equation}\text{(\theequation)}\label{inf31}] $%
\left\Vert F\right\Vert _{C^{m}\left( \mathbb{R}^{n}\right) }\leq CM$,
\end{itemize}

\begin{itemize}
\item[\refstepcounter{equation}\text{(\theequation)}\label{inf32}] $F\geq 0$
on $\mathbb{R}^{n}$, and
\end{itemize}

\begin{itemize}
\item[\refstepcounter{equation}\text{(\theequation)}\label{inf33}] $%
J_{x}\left( F\right) =P^{x}$ for all $x\in E$.
\end{itemize}
\end{lemma}

\begin{proof}
We modify slightly Whitney's proof \cite{whitney-1934} of the Whitney
extension theorem. We say that a dyadic cube $Q \subset \mathbb{R}^n$ is
``OK'' if $\#(E\cap 5Q) \leq 1$ and $\delta_Q \leq 1$. Then every small
enough $Q$ is OK (because $E$ is finite), and no $Q$ of sidelength $\delta_Q
>1$ is OK. Also, let $Q, Q^{\prime }$ be dyadic cubes with $5Q \subset
5Q^{\prime }$. If $Q^{\prime }$ is OK, then also $Q$ is OK. We define a 
\underline{Calder\'on-Zygmund} (or \underline{CZ}) cube to be an OK cube $Q$
such that no $Q^{\prime }$ that strictly contains $Q$ is OK. The above
remarks imply that the CZ cubes form a partition of $\mathbb{R}^n$; that the
sidelengths of the CZ cubes are bounded above by $1$ and below by some
positive number; and that the following condition holds.

\begin{itemize}
\item[\refstepcounter{equation}\text{(\theequation)}\label{inf34}] %
\textquotedblleft Good Geometry": If $Q,Q^{\prime }\in $ CZ and $\frac{65}{64%
}Q\cap \frac{65}{64}Q^{\prime }\not=\emptyset $, then $\frac{1}{2}\delta
_{Q}\leq \delta _{Q^{\prime }}\leq 2\delta _{Q}$.
\end{itemize}

We classify CZ cubes into three types as follows. 

$Q \in CZ$ is of

\begin{description}
\item[Type 1] if $E \cap 5Q \not= \emptyset$

\item[Type 2] if $E \cap 5Q =\emptyset$ and $\delta_Q <1$.

\item[Type 3] if $E \cap 5Q =\emptyset$ and $\delta_Q=1$.
\end{description}

\underline{Let $Q\in $ CZ be of Type 1.} Since $Q$ is OK, we have $\#(E\cap
5Q)\leq 1$. Hence $E\cap 5Q$ is a singleton, $E\cap 5Q=\left\{ x_{Q}\right\} 
$. Since $P^{x_{Q}}\in \Gamma _{\ast }\left( x_{Q},M\right) $, there exists $%
F_{Q}\in C^{m}\left( \mathbb{R}^{n}\right) $ such that

\begin{itemize}
\item[\refstepcounter{equation}\text{(\theequation)}\label{inf35}] $%
\left\Vert F_{Q}\right\Vert _{C^{m}\left( \mathbb{R}^{n}\right) }\leq M$, $%
F_{Q}\geq 0$ on $\mathbb{R}^{n}$, $J_{x_{Q}}\left( F_{Q}\right) =P^{x_{Q}}$.
\end{itemize}

We fix $F_Q$ as in \eqref{inf35}.

\underline{Let $Q\in $ CZ be of Type 2.} Then $\delta _{Q^{+}}\leq 1$ but $%
Q^{+}$ is not OK; hence $\#\left( E\cap 5Q^{+}\right) \geq 2$. We pick $%
x_{Q}\in E\cap 5Q^{+}$. Since $P^{x_{Q}}\in \Gamma _{\ast }\left(
x_{Q},M\right) $, there exists $F_{Q}\in C^{m}\left( \mathbb{R}^{n}\right) $
satisfying \eqref{inf35}. We fix such an $F_{Q}$.

\underline{Let $Q \in $ CZ be of Type 3.} Then we set $F_Q =0$. In place of %
\eqref{inf35}, we have the trivial results

\begin{itemize}
\item[\refstepcounter{equation}\text{(\theequation)}\label{inf36}] $%
\left\Vert F_{Q}\right\Vert _{C^{m}\left( \mathbb{R}^{n}\right) }=0$ and $%
F_{Q}\geq 0$ on $\mathbb{R}^{n}$.
\end{itemize}

Thus, we have defined $F_{Q}$ for all $Q\in $ CZ, and we have defined $%
x_{Q}\in E\cap 5Q^{+}$ for all $Q$ of Type 1 or Type 2. Note that

\begin{itemize}
\item[\refstepcounter{equation}\text{(\theequation)}\label{inf37}] $%
J_{x}\left( F_{Q}\right) =P^{x}$ for all $x\in E\cap 5Q$.
\end{itemize}

Indeed, if $Q$ is of Type 1, then \eqref{inf37} follows from \eqref{inf35}
since $E\cap 5Q=\{x_{Q}\}$. If $Q$ is of Type 2 or Type 3, then \eqref{inf37}
holds vacuously since $E\cap 5Q=\emptyset $. Now suppose $Q,Q^{\prime }\in $
CZ and $\frac{65}{64}Q\cap \frac{65}{64}Q^{\prime }\not=\emptyset $. We will
show that

\begin{itemize}
\item[\refstepcounter{equation}\text{(\theequation)}\label{inf38}] $%
\left\vert \partial ^{\beta }\left( F_{Q}-F_{Q^{\prime }}\right) \right\vert
\leq CM\delta _{Q}^{m-\left\vert \beta \right\vert }$ on $\frac{65}{64}Q\cap 
\frac{65}{64}Q^{\prime }$ for $\left\vert \beta \right\vert \leq m$.
\end{itemize}

To see this, suppose first that $Q$ or $Q^{\prime }$ is of Type 3. Then $%
\delta_Q$ or $\delta_{Q^{\prime }}$ is equal to $1$, hence $\delta_Q \geq 
\frac{1}{2}$ by \eqref{inf34}. Consequently, \eqref{inf38} asserts simply
that

\begin{itemize}
\item[\refstepcounter{equation}\text{(\theequation)}\label{inf39}] $%
\left\vert \partial ^{\beta }\left( F_{Q}-F_{Q^{\prime }}\right) \right\vert
\leq CM$ on $\frac{65}{64}Q\cap \frac{65}{64}Q^{\prime }$ for $\left\vert
\beta \right\vert \leq m$,
\end{itemize}
and \eqref{inf39} follows at once from \eqref{inf35}, \eqref{inf36}. Thus, %
\eqref{inf38} holds if $Q$ or $Q^{\prime }$ is of Type 3. Suppose that
neither $Q$ nor $Q^{\prime }$ is of Type 3. Then $x_{Q}\in E\cap 5Q^{+}$, $%
x_{Q^{\prime }}\in E\cap 5(Q^{\prime +})$, $\frac{65}{64}Q\cap \frac{65}{64}%
Q^{\prime }\not=\emptyset $, $\frac{1}{2}\delta _{Q}\leq \delta _{Q^{\prime
}}\leq 2\delta _{Q}$. Consequently,

\begin{itemize}
\item[\refstepcounter{equation}\text{(\theequation)}\label{inf40}] $%
\left\vert x_{Q}-x_{Q^{\prime }}\right\vert \leq C\delta _{Q}$, and
\end{itemize}

\begin{itemize}
\item[\refstepcounter{equation}\text{(\theequation)}\label{inf41}] $%
\left\vert x-x_{Q}\right\vert $, $\left\vert x-x_{Q^{\prime }}\right\vert
\leq C\delta _{Q}$ for all $x\in \frac{65}{64}Q\cap \frac{65}{64}Q^{\prime }$%
.
\end{itemize}

Applying \eqref{inf35} to $Q$ and to $Q^{\prime }$, we find that

\begin{itemize}
\item[\refstepcounter{equation}\text{(\theequation)}\label{inf42}] $%
\left\vert \partial ^{\beta }\left( F_{Q}-P^{x_{Q}}\right) \left( x\right)
\right\vert \leq CM\left\vert x-x_{Q}\right\vert ^{m-\left\vert \beta
\right\vert }\leq CM\delta _{Q}^{m-\left\vert \beta \right\vert }$, and
\end{itemize}

\begin{itemize}
\item[\refstepcounter{equation}\text{(\theequation)}\label{inf43}] $%
\left\vert \partial ^{\beta }\left( F_{Q^{\prime }}-P^{x_{Q^{\prime
}}}\right) \left( x\right) \right\vert \leq CM\left\vert x-x_{Q^{\prime
}}\right\vert ^{m-\left\vert \beta \right\vert }\leq CM\delta
_{Q}^{m-\left\vert \beta \right\vert }$,
\end{itemize}

for $x\in \frac{65}{64}Q\cap \frac{65}{64}Q^{\prime }$, $\left\vert \beta
\right\vert \leq m$. \newline
Also, \eqref{inf30}, \eqref{inf40}, \eqref{inf41} imply that

\begin{itemize}
\item[\refstepcounter{equation}\text{(\theequation)}\label{inf44}] $%
\left\vert \partial ^{\beta }\left( P^{x_{Q}}-P^{x_{Q^{\prime }}}\right)
\left( x\right) \right\vert \leq CM\delta _{Q}^{m-\left\vert \beta
\right\vert }$ for $x\in \frac{65}{64}Q\cap \frac{65}{64}Q^{\prime }$, $%
\left\vert \beta \right\vert \leq m$.
\end{itemize}

(Recall, $P^{x_{Q}}-P^{x_{Q^{\prime }}}$ is a polynomial of degree at most $%
m-1$.)

Estimates \eqref{inf42}, \eqref{inf43}, \eqref{inf44} together imply %
\eqref{inf38} in case neither $Q$ nor $Q^{\prime }$ is of Type 3. Thus, %
\eqref{inf38} holds in all cases.

Next, as in Whitney \cite{whitney-1934}, we introduce a partition of unity

\begin{itemize}
\item[\refstepcounter{equation}\text{(\theequation)}\label{inf45}] $%
1=\sum_{Q\in CZ}\theta _{Q}$ on $\mathbb{R}^{n}$,
\end{itemize}

where each $\theta_Q \in C^m(\mathbb{R}^n)$, and

\begin{itemize}
\item[\refstepcounter{equation}\text{(\theequation)}\label{inf46}] support $%
\theta _{Q}\subset \frac{65}{64}Q$, $\left\vert \partial ^{\beta }\theta
_{Q}\right\vert \leq C\delta _{Q}^{-\left\vert \beta \right\vert }$ for $%
\left\vert \beta \right\vert \leq m$, $\theta _{Q}\geq 0$ on $\mathbb{R}^{n}$%
.
\end{itemize}

We define

\begin{itemize}
\item[\refstepcounter{equation}\text{(\theequation)}\label{inf47}] $%
F=\sum_{Q\in CZ}\theta _{Q}F_{Q}$ on $\mathbb{R}^{n}$.
\end{itemize}

Thus, $F \in C^m_{loc}(\mathbb{R}^n)$ since CZ is a locally finite partition
of $\mathbb{R}^n$, and $F \geq 0$ on $\mathbb{R}^n$ since $\theta_Q \geq 0$
and $F_Q \geq 0$ for each $Q$. Let $\hat{x} \in \mathbb{R}^n$, and let $\hat{%
Q}$ be the one and only CZ cube containing $\hat{x}$. Then for $|\beta| \leq
m$, we have

\begin{itemize}
\item[\refstepcounter{equation}\text{(\theequation)}\label{inf48}] $\partial
^{\beta }F\left( \hat{x}\right) =\partial ^{\beta }F_{\hat{Q}}\left( \hat{x}%
\right) +\sum_{Q\in CZ}\partial ^{\beta }\left( \theta _{Q}\cdot \left(
F_{Q}-F_{\hat{Q}}\right) \right) \left( \hat{x}\right) $.
\end{itemize}

A given $Q\in $ CZ enters into the sum in \eqref{inf48} only if $\hat{x}\in 
\frac{65}{64}Q$; there are at most $C$ such cubes $Q$, thanks to %
\eqref{inf34}. Moreover, for each $Q\in $ CZ with $\hat{x}\in \frac{65}{64}Q$%
, we learn from \eqref{inf38} and \eqref{inf46} that 
\begin{equation*}
\left\vert \partial ^{\beta }\left( \theta _{Q}\cdot \left( F_{Q}-F_{\hat{Q}%
}\right) \right) \left( \hat{x}\right) \right\vert \leq CM\delta
_{Q}^{m-\left\vert \beta \right\vert }\leq CM\text{ for }\left\vert \beta
\right\vert \leq m\text{, since }\delta _{Q}\leq 1\text{.}
\end{equation*}%
Since also $\left\vert \partial ^{\beta }F_{\hat{Q}}\left( \hat{x}\right)
\right\vert \leq CM$ for $\left\vert \beta \right\vert \leq m$ by %
\eqref{inf35}, \eqref{inf36}, it now follows from \eqref{inf48} that $%
\left\vert \partial ^{\beta }F\left( \hat{x}\right) \right\vert \leq CM$ for
all $\left\vert \beta \right\vert \leq m$. Here, $\hat{x}\in \mathbb{R}^{n}$
is arbitrary. Thus, $F\in C^{m}\left( \mathbb{R}^{n}\right) $ and $%
||F||_{C^{m}\left( \mathbb{R}^{n}\right) }\leq CM$.

Next, let $x\in E$. For any $Q\in $ CZ such that $x\in \frac{65}{64}Q$, we
have $J_{x}(F_{Q})=P^{x}$, by \eqref{inf37}. Since support $\theta
_{Q}\subset \frac{65}{64}Q$ for each $Q\in $ CZ, it follows that $%
J_{x}(\theta _{Q}F_{Q})=J_{x}(\theta _{Q})\odot _{x}P^{x}$ for each $Q\in $
CZ, and consequently, 
\begin{equation*}
J_{x}(F)=\sum_{Q\in CZ}J_{x}\left( \theta _{Q}F_{Q}\right) =\left[
\sum_{Q\in CZ}J_{x}\left( \theta _{Q}\right) \right] \odot _{x}P^{x}=P^{x}%
\text{, by (\ref{inf45}).}
\end{equation*}%
Thus, $F\in C^{m}\left( \mathbb{R}^{n}\right) $, $\left\Vert F\right\Vert
_{C^{m}\left( \mathbb{R}^{n}\right) }\leq CM$, $F\geq 0$ on $\mathbb{R}^{n}$%
, and $J_{x}\left( F\right) =P^{x}$ for each $x\in E$.

The proof of Lemma \ref{lemma-inf2} is complete.
\end{proof}

\begin{theorem}[Finiteness Principle for Non-negative $C^m$ Interpolation]
\label{theorem-fp-for-nonnegative-interpolation}

There exist constants $k^\#$, $C$, depending only on $m$, $n$, such that the
following holds.

Let $E \subset \mathbb{R}^n$ be finite, and let $f:E \rightarrow [0, \infty)$%
. Let $M_0 >0$. Suppose that for each $S \subset E$ with $\#(S) \leq k^\#$,
there exists $\vec{P}^S = (P^x)_{x\in S} \in Wh(S)$ such that

\begin{itemize}
\item $P^x \in \Gamma_f(x,M_0)$ for each $x \in S$, and

\item $|\partial^\beta (P^x -P^y)(x)| \leq M_0 |x-y|^{m-|\beta|}$ for $x, y
\in S$, $|\beta| \leq m-1$.
\end{itemize}

Then there exists $F\in C^{m}(\mathbb{R}^{n})$ such that

\begin{itemize}
\item $\|F\|_{C^m(\mathbb{R}^n)} \leq CM_0$,

\item $F \geq 0$ on $\mathbb{R}^n$, and

\item $F = f$ on $E$.
\end{itemize}
\end{theorem}

\begin{proof}
Suppose first that $E \subset \frac{1}{2}Q_0$ for a cube $Q_0$ of sidelength 
$\delta_{Q_0} =1$. Pick any $x_0 \in E$. (If $E$ is empty, our theorem holds
trivially.)

Let $S \subset E$ with $\#(S) \leq k^\#$.

Our present hypotheses supply the Whitney field $\vec{P}^S$ required in the hypotheses of Theorem \ref{theorem-fp-for-wsf}.

Hence, recalling Lemma \ref{lemma-inf1} and applying Theorem \ref{theorem-fp-for-wsf}, we obtain

\begin{itemize}
\item[\refstepcounter{equation}\text{(\theequation)}\label{inf49}] $P^0 \in
\Gamma_f(x_0, CM_0)$
\end{itemize}

and

\begin{itemize}
\item[\refstepcounter{equation}\text{(\theequation)}\label{inf50}] $F^0 \in
C^m(Q_0)$
\end{itemize}

such that

\begin{itemize}
\item[\refstepcounter{equation}\text{(\theequation)}\label{inf51}] $J_x(F^0)
\in \Gamma_f(x,CM_0)$ for all $x \in E \cap Q_0 =E$
\end{itemize}

and

\begin{itemize}
\item[\refstepcounter{equation}\text{(\theequation)}\label{inf52}] $%
|\partial^\beta(P^0 -F^0)|\leq CM_0$ on $Q_0$, for $|\beta|\leq m$.
\end{itemize}

From \eqref{inf1}, \eqref{inf2}, \eqref{inf49}, we have $|\partial^\beta P^0
(x_0)| \leq CM_0$ for $|\beta| \leq m-1$.

Since $P^0$ is a polynomial of degree at most $m-1$, and since $x_0 \in E
\subset Q_0$ with $\delta_{Q_0} =1$, it follows that $|\partial^\beta P^0 |
\leq CM_0$ on $Q_0$ for $|\beta| \leq m$.

Together with \eqref{inf52}, this tells us that

\begin{itemize}
\item[\refstepcounter{equation}\text{(\theequation)}\label{inf53}] $%
|\partial^\beta F^0| \leq CM_0$ on $Q_0$ for $|\beta| \leq m$.
\end{itemize}

Note that $F^0$ needn't be nonnegative.

Set $P^x = J_x(F^0)$ for $x \in E$. Then

\begin{itemize}
\item[\refstepcounter{equation}\text{(\theequation)}\label{inf54}] $P^{x}\in
\Gamma _{f}\left( x,CM_{0}\right) $ for $x\in E$, and
\end{itemize}

\begin{itemize}
\item[\refstepcounter{equation}\text{(\theequation)}\label{inf55}] $%
\left\vert \partial ^{\beta }\left( P^{x}-P^{y}\right) \left( x\right)
\right\vert \leq CM_{0}\left\vert x-y\right\vert ^{m-\left\vert \beta
\right\vert }$ for $x,y\in E$, $\left\vert \beta \right\vert \leq m-1$.
\end{itemize}

By Lemma \ref{lemma-inf2}, there exists $F\in C^{m}\left( \mathbb{R}%
^{n}\right) $ such that

\begin{itemize}
\item[\refstepcounter{equation}\text{(\theequation)}\label{inf56}] $%
\left\Vert F\right\Vert _{C^{m}\left( \mathbb{R}^{n}\right) }\leq CM_0$,
\end{itemize}

\begin{itemize}
\item[\refstepcounter{equation}\text{(\theequation)}\label{inf57}] $F\geq 0$
on $\mathbb{R}^{n}$, and
\end{itemize}

\begin{itemize}
\item[\refstepcounter{equation}\text{(\theequation)}\label{inf58}] $%
J_{x}\left( F\right) =P^{x}$ for each $x\in E$.
\end{itemize}

From \eqref{inf54} and \eqref{inf2}, we have $P^x(x)=f(x)$ for each $x \in E$%
; hence, \eqref{inf58} implies that

\begin{itemize}
\item[\refstepcounter{equation}\text{(\theequation)}\label{inf59}] $F\left(
x\right) =f\left( x\right) $ for each $x\in E$.
\end{itemize}

Our results \eqref{inf56}, \eqref{inf57}, \eqref{inf59} are the conclusions
of our theorem. Thus, we have proven Theorem \ref%
{theorem-fp-for-nonnegative-interpolation} in the case in which $E \subset 
\frac{1}{2} Q_0$ with $\delta_{Q_0} =1$.

To pass to the general case (arbitrary finite $E \subset \mathbb{R}^n$), we
set up a partition of unity $1 =\sum_{\nu} \chi_\nu$ on $\mathbb{R}^n$,
where each $\chi_\nu \in C^m(\mathbb{R}^n)$ and $\chi_\nu \geq 0$ on $%
\mathbb{R}^n$, $\|\chi_\nu\|_{C^m(\mathbb{R}^n)} \leq C$, support $\chi_\nu
\subset \frac{1}{2} Q_\nu$, with $\delta_{Q_\nu} =1$, and with any given
point of $\mathbb{R}^n$ belonging to at most $C$ of the $Q_\nu$.

For each $\nu$, we apply the known special case of our theorem to the set $%
E_\nu = E \cap \frac{1}{2}Q_\nu$ and the function $f_\nu = f|_{E_\nu}$.
Thus, we obtain $F_\nu \in C^m(\mathbb{R}^n)$, with $\|F_\nu\|_{C^m(\mathbb{R%
}^n)} \leq CM_0$, $F_\nu \geq 0$ on $\mathbb{R}^n$, and $F_\nu = f$ on $E
\cap \frac{1}{2} Q_\nu$.

Setting $F = \sum _\nu \chi_\nu F_\nu \in C^m_{loc}(\mathbb{R}^n)$, we
verify easily that $F \in C^m (\mathbb{R}^n)$, $\|F\|_{C^m(\mathbb{R}^n)}
\leq CM_0$, $F \geq 0$ on $\mathbb{R}^n$, and $F = f $ on $E$.

This completes the proof of Theorem \ref%
{theorem-fp-for-nonnegative-interpolation}.
\end{proof}

\begin{remark}
Conversely, we make the following trivial observation: Let $E\subset \mathbb{%
R}^{n}$ be finite, let $f:E\rightarrow \lbrack 0,\infty )$, and let $M_{0}>0$%
. Suppose $F\in C^{m}(\mathbb{R}^{n})$ satisfies $\Vert F\Vert _{C^{m}(%
\mathbb{R}^{n})}\leq M_{0}$, $F\geq 0$ on $\mathbb{R}^{n}$, $F=f$ on $E$.
Then for each $x\in E$, we have

\begin{itemize}
\item $P^{x}=J_{x}(F)\in \Gamma _{f}(x,M_{0})$ by \eqref{inf1}, \eqref{inf2}%
; and

\item $|\partial ^{\beta }(P^{x}-P^{y})(x)|\leq CM_{0}|x-y|^{m-|\beta |}$
for $x,y\in E$, $|\beta |\leq m-1$.
\end{itemize}

Therefore, for any $S \subset E$, the Whitney field $\vec{P}^S = (P^x)_{x
\in S} \in Wh(S)$ satisfies

\begin{itemize}
\item $P^{x}\in \Gamma _{f}(x,CM_{0})$ for $x\in S$, and

\item $|\partial ^{\beta }(P^{x}-P^{y})(x)|\leq CM_{0}|x-y|^{m-|\beta |}$
for $x,y\in S$, $|\beta |\leq m-1$.
\end{itemize}
Note that Theorem \ref{Th4} (A) follows easily from Theorem \ref{theorem-fp-for-nonnegative-interpolation}.
\end{remark}

\section{Computable Convex Sets}

\label{ccs}

In this section, we discuss computational issues regarding the convex set

\begin{itemize}
\item[\refstepcounter{equation}\text{(\theequation)}\label{cs1}] $\Gamma
_{\ast }\left( x,M\right) =\left\{ J_{x}\left( F\right) :F\in C^{m}\left( 
\mathbb{R}^{n}\right) \text{, }\left\Vert F\right\Vert _{C^{m}\left( \mathbb{%
R}^{n}\right) }\leq M\text{, }F\geq 0\text{ on }\mathbb{R}^{n}\right\} .$
\end{itemize}

We write $c$, $C$, $C^{\prime }$, etc., to denote constants determined by $m$
and $n$. These symbols may denote different constants in different
occurrences.

We will define convex sets $\tilde{\Gamma}_{\ast}(x,M) \subset \mathcal{P}$,
prove that

\begin{itemize}
\item[\refstepcounter{equation}\text{(\theequation)}\label{cs2}] $\tilde{%
\Gamma}_{\ast }(x,cM)\subset \Gamma _{\ast }\left( x,M\right) \subset \tilde{%
\Gamma}_{\ast }(x,CM)$ for all $x\in \mathbb{R}^{n}$, $M>0$,
\end{itemize}

and explain how (in principle) one can compute $\tilde{\Gamma}_{\ast}(x,M)$.

We may then use

\begin{itemize}
\item[\refstepcounter{equation}\text{(\theequation)}\label{cs3}] $\tilde{%
\Gamma}_{f}\left( x,M\right) =\left\{ P\in \tilde{\Gamma}_{\ast
}(x,M):P\left( x\right) =f\left( x\right) \right\} $
\end{itemize}

in place of $\Gamma_f(x,M)$ in the statement of Theorem \ref%
{theorem-fp-for-nonnegative-interpolation}. (The assertion in terms of $%
\tilde{\Gamma}_f$ follows trivially from \eqref{cs2} and the original
assertion in terms of $\Gamma_f$.)

To achieve \eqref{cs2}, we will define

\begin{itemize}
\item[\refstepcounter{equation}\text{(\theequation)}\label{cs4}] $\tilde{%
\Gamma}_{\ast }(x,M)=\left\{ MP\left( \cdot +x\right) ):P\in \tilde{\Gamma}%
_{0}\right\} $, for a convex set $\tilde{\Gamma}_{0}$.
\end{itemize}

We will prove that

\begin{itemize}
\item[\refstepcounter{equation}\text{(\theequation)}\label{cs5}] $%
\Gamma_{\ast}(0,c) \subset \tilde{\Gamma}_0 \subset \Gamma_{\ast}(0,C)$.
\end{itemize}

Property \eqref{cs2} then follows at once from \eqref{cs1}, \eqref{cs4}, and %
\eqref{cs5}.

Thus, our task is to define a convex set $\tilde{\Gamma}_{0}$ satisfying %
\eqref{cs5}, and explain how (in principle) one can compute $\tilde{\Gamma}%
_{0}$.

Recall that $\mathcal{P}$ is the vector space of $(m-1)$-jets. We will work
in the space of $m$-jets. In this section, we let $\mathcal{P}^{+}$ denote
the vector space of real-valued polynomials of degree at most $m$ on $%
\mathbb{R}^{n}$, and we write $J_{x}^{+}(F)$ to denote the $m^{th}$-degree
Taylor polynomial of $F$ at $x$, i.e.,%
\begin{equation*}
J_{x}^{+}\left( F\right) \left( y\right) =\sum_{\left\vert \alpha
\right\vert \leq m}\frac{1}{\alpha !}\left( \partial ^{\alpha }F\left(
x\right) \right) \cdot \left( y-x\right) ^{\alpha }\text{.}
\end{equation*}

We define

\begin{itemize}
\item[\refstepcounter{equation}\text{(\theequation)}\label{cs6}] $\Gamma
_{0}^{+}=\left\{ 
\begin{array}{c}
P\in \mathcal{P}^{+}:\left\vert \partial ^{\beta }P\left( 0\right)
\right\vert \leq 1\text{ for }\left\vert \beta \right\vert \leq m\text{; }%
P\left( x\right) +\left\vert x\right\vert ^{m}\geq 0\text{ for all }x\in 
\mathbb{R}^{n}\text{;} \\ 
\text{and for every }\epsilon >0\text{, there exists }\delta >0\text{ such
that} \\ 
\text{ }P\left( x\right) +\epsilon \left\vert x\right\vert ^{m}\geq 0\text{
for }\left\vert x\right\vert \leq \delta \text{.}%
\end{array}%
\right\} $;
\end{itemize}

Later, we will discuss how $\Gamma^+_0$ may be computed in principle.

We next establish the following result.

\begin{lemma}
\label{lemma-ccs} For small enough $c$ and large enough $C$, the following
hold.

\begin{description}
\item[(A)] If $F \in C^m(\mathbb{R}^n)$, $\|F\|_{C^m(\mathbb{R}^n)} \leq c$, 
$F \geq 0$ on $\mathbb{R}^n$, then $J_0^+(F) \in \Gamma_0^+$.

\item[(B)] If $P \in \Gamma_0^+$, then there exists $F \in C^m(\mathbb{R}^n)$
such that $\|F\|_{C^m(\mathbb{R}^n)} \leq C$, $F \geq 0$ on $\mathbb{R}^n$,
and $J_0^+(F) = P$.
\end{description}
\end{lemma}

\begin{proof}
(A) follows trivially from Taylor's theorem. We prove (B).

Let $P\in \Gamma _{0}^{+}$ be given. We introduce cutoff functions $\varphi $%
, $\chi \in C^{m}\left( \mathbb{R}^{n}\right) $ with the following properties.

\begin{itemize}
\item[\refstepcounter{equation}\text{(\theequation)}\label{cs7}] $\left\Vert
\chi \right\Vert _{C^{m}\left( \mathbb{R}^{n}\right) }\leq C$, $\chi =1$ in
a neighborhood of $0$, $\chi =0$ outside $B_{n}\left( 0,1/2\right) $, and $%
0\leq \chi \leq 1$ on $\mathbb{R}^{n}$.
\end{itemize}

\begin{itemize}
\item[\refstepcounter{equation}\text{(\theequation)}\label{cs8}] $\left\Vert
\varphi \right\Vert _{C^{m}\left( \mathbb{R}^{n}\right) }\leq C$, $\varphi
=1 $ for $1/2\leq \left\vert x\right\vert \leq 2$, $\varphi \geq 0$ on $\mathbb{R}%
^{n}$,
\end{itemize}
and $\varphi \left( x\right)
=0$ unless $1/4<\left\vert x\right\vert <4$.

For $k \geq 0$, let

\begin{itemize}
\item[\refstepcounter{equation}\text{(\theequation)}\label{cs9}] $\varphi
_{k}\left( x\right) =\varphi \left( 2^{k}x\right) $ $\left( x\in \mathbb{R}%
^{n}\right) $.
\end{itemize}

Thus,

\begin{itemize}
\item[\refstepcounter{equation}\text{(\theequation)}\label{cs10}] $%
\left\Vert \varphi _{k}\right\Vert _{C^{m}\left( \mathbb{R}^{n}\right) }\leq
C2^{mk}$, $\varphi _{k}\geq 0$ on $\mathbb{R}^{n}$, $\varphi _{k}\left(
x\right) =1$ for $2^{-1-k}\leq \left\vert x\right\vert \leq 2^{1-k}$, $%
\varphi _{k}\left( x\right) =0$ unless $2^{-2-k}\leq \left\vert x\right\vert
\leq 2^{2-k}$.
\end{itemize}

Also, for $k \geq 0$, we define a real number $b_k$ as follows.

\begin{itemize}
\item[\refstepcounter{equation}\text{(\theequation)}\label{cs11}] $b_{k}=0$
if $P\left( x\right) \geq 0$ for $\left\vert x\right\vert \leq 2^{-k}$; $%
b_{k}=-\min \left\{ P\left( x\right) :\left\vert x\right\vert \leq
2^{-k}\right\} $ otherwise.
\end{itemize}

Since $P \in \Gamma_0^+$, the $b_k$ satisfy the following:

\begin{itemize}
\item[\refstepcounter{equation}\text{(\theequation)}\label{cs12}] $0\leq
b_{k}\leq 2^{-mk}$ for all $k\geq 0$.
\end{itemize}

\begin{itemize}
\item[\refstepcounter{equation}\text{(\theequation)}\label{cs13}] $%
b_{k}\cdot 2^{mk}\rightarrow 0$ as $k\rightarrow \infty $.
\end{itemize}

By definition of the $b_k$, we have also for each $k \geq 0$ that

\begin{itemize}
\item[\refstepcounter{equation}\text{(\theequation)}\label{cs14}] $P\left(
x\right) +b_{k}\geq 0$ for $\left\vert x\right\vert \leq 2^{-k}$.
\end{itemize}

We define a function $\tilde{F}$ on the closed unit ball $\overline{B_n(0,1)}$ by setting

\begin{itemize}
\item[\refstepcounter{equation}\text{(\theequation)}\label{cs15}] $\tilde{F}%
\left( x\right) =P\left( x\right) +\sum_{k=0}^{\infty }b_{k}\varphi
_{k}\left( x\right) $ for $x\in \overline{B_{n}\left( 0,1\right)} $.
\end{itemize}
(The sum contains
at most $C$ nonzero terms for any given $x$.)

We will check that

\begin{itemize}
\item[\refstepcounter{equation}\text{(\theequation)}\label{cs16}] $\tilde{F}%
\geq 0$ on $\overline{B_{n}\left( 0,1\right)} $.
\end{itemize}

Indeed, $\tilde{F}\left( 0\right) =P(0)\geq 0$ since each $\varphi _{k}(0)=0$
and $P\in \Gamma _{0}^{+}$. For $\hat{x}\in \overline{B_{n}(0,1)}\setminus \{0\}$ we
have $2^{-1-\hat{k}}\leq |\hat{x}|\leq 2^{-\hat{k}}$ for some $\hat{k}\geq 0$%
.

We then have $\varphi _{\hat{k}}(\hat{x})=1$ by \eqref{cs10}, hence $P(\hat{x%
})+b_{\hat{k}}\varphi _{\hat{k}}(\hat{x})\geq 0$ by \eqref{cs14}. Since also 
$b_{k}\varphi _{k}(\hat{x})\geq 0$ for all $k$, it follows that $$\tilde{F}(%
\hat{x})=\left[ P\left( \hat{x}\right) +b_{\hat{k}}\varphi _{\hat{k}}\left( 
\hat{x}\right) \right] +\sum_{k\not=\hat{k}}b_{k}\varphi _{k}\left( x\right)
\geq 0,$$ completing the proof of \eqref{cs16}.

Next, we check that

\begin{itemize}
\item[\refstepcounter{equation}\text{(\theequation)}\label{cs17}] $\tilde{F}%
\in C^{m}\left( \overline{B_{n}\left( 0,1\right)} \right) $, $\left\Vert \tilde{F}%
\right\Vert _{C^{m}\left( \overline{B_{n}\left( 0,1\right)} \right) }\leq C$, $%
J_{0}^{+}\left( \tilde{F}\right) =P$.
\end{itemize}

To see this, let

\begin{itemize}
\item[\refstepcounter{equation}\text{(\theequation)}\label{cs18}] $\tilde{F}%
_{K}=P+\sum_{k=0}^{K}b_{k}\varphi _{k}$ for $K\geq 0$.
\end{itemize}

Since $P\in \Gamma _{0}^{+}$, we have $\left\vert \partial ^{\beta }P\left(
0\right) \right\vert \leq 1$ for $\left\vert \beta \right\vert \leq m$, hence

\begin{itemize}
\item[\refstepcounter{equation}\text{(\theequation)}\label{cs19}] $%
\left\Vert P\right\Vert _{C^{m}\left(\overline{ B_{n}\left( 0,1\right)} \right) }\leq C$%
.
\end{itemize}

Also, \eqref{cs10} and \eqref{cs12} give%
\begin{equation*}
\left\Vert b_{k}\varphi _{k}\right\Vert _{C^{m}\left(\overline{ B_n\left( 0,1\right)}
\right) }\leq C\text{ for each }k\text{.}
\end{equation*}

Since any given $x \in \overline{ B_n(0,1)}$ belongs to at most $C$ of the supports of
the $\varphi_k$, it follows that

\begin{itemize}
\item[\refstepcounter{equation}\text{(\theequation)}\label{cs20}] $%
\left\Vert \sum_{k=0}^{K}b_{k}\varphi _{k}\right\Vert _{C^{m}\left(
\overline{B_{n}\left( 0,1\right)} \right) }\leq C$.
\end{itemize}

From \eqref{cs18}, \eqref{cs19}, \eqref{cs20}, we see that

\begin{itemize}
\item[\refstepcounter{equation}\text{(\theequation)}\label{cs21}] $\tilde{F}%
_{K}\in C^{m}\left( \overline{B_{n}\left( 0,1\right)} \right) $ and $\left\Vert \tilde{F%
}\right\Vert _{C^{m}\left( \overline{B_{n}\left( 0,1\right)} \right) }\leq C$.
\end{itemize}

Also, \eqref{cs10} and \eqref{cs18} tell us that

\begin{itemize}
\item[\refstepcounter{equation}\text{(\theequation)}\label{cs22}] $%
J_{0}^{+}\left( \tilde{F}_{K}\right) =P$ for each $K$.
\end{itemize}

Furthermore for $K_{1} < K_{2}$, \eqref{cs18} gives $\tilde{F}_{K_{2}}-%
\tilde{F}_{K_{1}}=\sum_{K_{1}<k\leq K_{2}}b_{k}\varphi _{k}$. Let $\epsilon
>0$. From \eqref{cs10} and \eqref{cs13} we see that 
\begin{equation*}
\max_{K_{1}<k\leq K_{2}}\left\Vert b_{k}\varphi _{k}\right\Vert
_{C^{m}\left( \overline{B_{n}\left( 0,1\right)} \right) }<\epsilon \text{ if }K_{1}%
\text{ is large enough.}
\end{equation*}%
Since any given point lies in support $\varphi _{k}$ for at most $C$
distinct $k$, it follows that%
\begin{equation*}
\left\Vert \sum_{K_{1}<k\leq K_{2}}b_{k}\varphi _{k}\right\Vert
_{C^{m}\left( \overline{B_{n}\left( 0,1\right)} \right) }\leq C\epsilon \text{ if }%
K_{2}>K_{1}\text{ and }K_{1}\text{ is large enough.}
\end{equation*}

Thus, $(\tilde{F}_{K})_{K\geq 0}$ is a Cauchy sequence in $C^{m}(\overline{B_{n}(0,1)})$%
. Consequently, $\tilde{F}_{K}\rightarrow \tilde{F}_{\infty }$ in $%
C^{m}(\overline{B_{n}(0,1)})$-norm for some $\tilde{F}_{\infty }\in C^{m}(\overline{B_{n}(0,1))}$.
From \eqref{cs21} and \eqref{cs22}, we have%
\begin{equation*}
\left\Vert \tilde{F}_{\infty }\right\Vert _{C^{m}\left( \overline{B_{n}\left(
0,1\right)} \right) }\leq C\text{ and }J_{0}^{+}\left( \tilde{F}_{\infty
}\right) =P\text{.}
\end{equation*}

On the other hand, comparing \eqref{cs15} to \eqref{cs18}, and recalling
that any given $x$ belongs to support $\theta_k$ for at most $C$ distinct $k$%
, we conclude that $\tilde{F}_K \rightarrow \tilde{F}$ pointwise as $K
\rightarrow \infty$.

Since also $\tilde{F}_K \rightarrow \tilde{F}_\infty$ pointwise as $K
\rightarrow \infty$, we have $\tilde{F}_\infty = \tilde{F}$.

Thus, $\tilde{F}\in C^{m}\left( \overline{B_{n}\left( 0,1\right)} \right) $, $%
\left\Vert \tilde{F}\right\Vert _{C^{m}\left(\overline{ B_{n}\left( 0,1\right)} \right)
}\leq C$, and $J_{0}^{+}\left( \tilde{F}\right) =P$, completing the proof of %
\eqref{cs17}.

Finally, we recall the cutoff function $\chi$ from \eqref{cs7}, and define $%
F=\chi \tilde{F}$ on $\mathbb{R}^n$.

From \eqref{cs16}, \eqref{cs17}, and the properties \eqref{cs7} of $\chi $,
we conclude that $F\in C^{m}\left( \mathbb{R}^{n}\right) $, $\left\Vert
F\right\Vert _{C^{m}\left( \mathbb{R}^{n}\right) }\leq C$, $F\geq 0$ on $%
\mathbb{R}^{n}$, and $J_{0}^{+}\left( F\right) =P$.

Thus, we have established (B).

The proof of Lemma \ref{lemma-ccs} is complete.
\end{proof}

Now let $\pi :\mathcal{P}^{+}\rightarrow \mathcal{P}$ denote the natural
projection from $m$-jets at $0$ to $\left( m-1\right) $-jets at $0$, namely, 
$\pi P=J_{0}\left( P\right) $ for $P\in \mathcal{P}^{+}$.

We then set $\tilde{\Gamma}_{0}=\pi \Gamma _{0}^{+}$.

From the above lemma, we learn the following.

\begin{description}
\item[$\left( A^{\prime }\right) $] Let $F\in C^{m}\left( \mathbb{R}%
^{n}\right) $ with $\left\Vert F\right\Vert _{C^{m}\left( \mathbb{R}%
^{n}\right) }\leq c$, $F\geq 0$ on $\mathbb{R}^{n}$. Then $J_{0}\left(
F\right) \in \tilde{\Gamma}_{0}$.

\item[$\left( B^{\prime }\right) $] Let $P\in \tilde{\Gamma}_{0}$. Then
there exists $F\in C^{m}\left( \mathbb{R}^{n}\right) $ such that $\left\Vert
F\right\Vert _{C^{m}\left( \mathbb{R}^{n}\right) }\leq C$, $F\geq 0$ on $%
\mathbb{R}^{n}$, and $J_{0}\left( F\right) =P$.
\end{description}

Recalling the definition (\ref{cs1}), we conclude from ($A^{\prime }$), ($%
B^{\prime }$) that $\Gamma _{\ast }\left( 0,c\right) \subset \tilde{\Gamma}%
_{0}\subset \Gamma _{\ast }\left( 0,C\right) $.

Thus, our $\tilde{\Gamma}_{0}$ satisfies the key condition (\ref{cs5}).

We discuss briefly how the convex set $\tilde{\Gamma}_{0}$ may be computed
in principle. Recall \cite{hormander-2007} that a semialgebraic set is a
subset of a vector space obtained by taking finitely many unions,
intersections, and complements of sets of the form $\left\{ P>0\right\} $
for polynomials $P$. Any subset of a vector space $V$ defined by $E=\left\{
x\in V:\Phi \left( x\right) \text{ is true}\right\} $, where $\Phi $ is a
formula of first-order predicate calculus (for the theory of real-closed
fields) is semialgebraic; moreover, there is an algorithm that accepts $\Phi 
$ as input and exhibits $E$ as a Boolean combination of sets of the form $%
\left\{ P>0\right\} $ for polynomials $P$. For any given $m$, $n$, we see,
by inspection of the definitions of $\Gamma _{0}^{+}$ and $\tilde{\Gamma}_{0}
$, that $\Gamma _{0}^{+}\subset \mathcal{P}^{+}$ is defined by a formula of
first-order predicate calculus; hence, the same holds for $\tilde{\Gamma}%
_{0}\subset \mathcal{P}$.

Therefore, in principle, we can compute $\tilde{\Gamma}_{0}$ as a Boolean
combination of sets of the form $\left\{ P\in \mathcal{P}:\Pi \left(
P\right) >0\right\} $, where $\Pi $ is a polynomial on $\mathcal{P}$.

In practice, we make no claim that we know how to compute $\tilde{\Gamma}%
_{0} $.

It would be interesting to give a more practical method to compute a convex
set satisfying (\ref{cs5}).

\section{Analogues for $C^{m-1,1}(\mathbb{R}^n)$}

\label{afc}

So far we have worked with functions in $C^{m}\left( \mathbb{R}^{n}\right) $%
. In this section we give the $C^{m-1,1}$-analogues of some our main results.

We note that a given continuous function $F:\mathbb{R}^{n}\rightarrow 
\mathbb{R}$ belongs to $C^{m-1,1}\left( \mathbb{R}^{n}\right) $ if and only
if its distribution derivatives $\partial ^{\beta }F$ belong to $L^{\infty
}\left( \mathbb{R}^{n}\right) $ for $\left\vert \beta \right\vert \leq m$.
We may take the norm on $C^{m-1,1}\left( \mathbb{R}^{n}\right) $ to be 
\begin{equation*}
\left\Vert F\right\Vert _{C^{m-1,1}\left( \mathbb{R}^{n}\right)
}=\max_{\left\vert \beta \right\vert \leq m}\text{ess.}\sup_{x\in \mathbb{R}%
^{n}}\left\vert \partial ^{\beta }F\left( x\right) \right\vert \text{,}
\end{equation*}%
whereas for $F\in C^{m}\left( \mathbb{R}^{n}\right) $ we have 
\begin{equation*}
\left\Vert F\right\Vert _{C^{m}\left( \mathbb{R}^{n}\right)
}=\max_{\left\vert \beta \right\vert \leq m}\sup_{x\in \mathbb{R}%
^{n}}\left\vert \partial ^{\beta }F\left( x\right) \right\vert \text{.}
\end{equation*}%
Moreover, the derivatives $\partial ^{\beta }F$ of $F\in C^{m-1,1}\left( 
\mathbb{R}^{n}\right) $ of order $\left\vert \beta \right\vert \leq m-1$ are
continuous. Also, Taylor's theorem holds in the form 
\begin{equation*}
\left\vert \partial ^{\beta }F\left( y\right) -\sum_{\left\vert \beta
\right\vert +\left\vert \gamma \right\vert \leq m-1}\frac{1}{\gamma !}\left[
\partial ^{\gamma +\beta }F\left( x\right) \right] \cdot \left( y-x\right)
^{\gamma }\right\vert \leq C\left\Vert F\right\Vert _{C^{m-1,1}\left( 
\mathbb{R}^{n}\right) }\cdot \left\vert y-x\right\vert ^{m-\left\vert \beta
\right\vert }
\end{equation*}%
for $x,y\in \mathbb{R}^{n}$.

Similar remarks apply to $C^{m-1,1}\left( Q\right) $ and $C^{m}\left(
Q\right) $ for cubes $Q\subset \mathbb{R}^{n}$.

Therefore, we may repeat the proofs of Lemmas \ref{lemma-inf1} and \ref%
{lemma-inf2} in Section \ref{inf}, to derive the following results.

\begin{lemma}
\label{lemma-inf1'}For $x\in \mathbb{R}^{n}$, $M>0$, let 
\begin{equation*}
\Gamma _{\ast }^{\prime }\left( x,M\right) =\left\{ 
\begin{array}{c}
P\in \mathcal{P}:\exists F\in C^{m-1,1}\left( \mathbb{R}^{n}\right) \text{
such that } \\ 
\left\Vert F\right\Vert _{C^{m-1,1}\left( \mathbb{R}^{n}\right) }\leq
M,F\geq 0\text{ on }\mathbb{R}^{n}\text{, }J_{x}\left( F\right) =P%
\end{array}%
\right\} \text{.}
\end{equation*}%
Let $f:E\rightarrow \lbrack 0,\infty )$, where $E\subset \mathbb{R}^{n}$ is
finite. For $x\in E$, $M>0$, let 
\begin{equation*}
\Gamma _{f}^{\prime }\left( x,M\right) =\left\{ P\in \Gamma _{\ast }^{\prime
}\left( x,M\right) :P\left( x\right) =f\left( x\right) \right\} \text{.}
\end{equation*}%
Then $\vec{\Gamma}_{f}^{\prime }:=\left( \Gamma _{f}^{\prime }\left(
x,M\right) \right) _{x\in E,M>0}$ is a $\left( C,1\right) $-convex shape
field, where $C$ depends only on $m$, $n$.
\end{lemma}

\begin{lemma}
\label{lemma-inf2'}Let $E$, $f$, $\Gamma _{\ast }^{\prime }\left( x,M\right) 
$ be as in Lemma \ref{lemma-inf1'}, and let $M>0$, $\vec{P}=\left(
P^{x}\right) _{x\in E}\in Wh\left( E\right) $. Suppose we have $P^{x}\in
\Gamma _{\ast }^{\prime }\left( x,M\right) $ for all $x \in E$, and $%
\left\vert \partial ^{\beta }\left( P^{x}-P^{y}\right) \left( x\right)
\right\vert \leq M\left\vert x-y\right\vert ^{m-\left\vert \beta \right\vert
}$ for $x,y\in E$, $\left\vert \beta \right\vert \leq m-1$. Then there
exists $F\in C^{m-1,1}\left( \mathbb{R}^{n}\right) $ such that $J_{x}\left(
F\right) =P^{x}$ for all $x\in E$, and $\left\Vert F\right\Vert
_{C^{m-1,1}\left( \mathbb{R}^{n}\right) }\leq CM$, where $C$ depends only on 
$m$, $n$.
\end{lemma}

Similarly, by making small changes in the proof of Theorem \ref%
{theorem-fp-for-nonnegative-interpolation}, we obtain the following result.

\begin{lemma}
\label{lemma-inf3'}There exist $k^{\#}$, $C$, depending only on $m$, $n$ for
which the following holds.

Let $E\subset \mathbb{R}^{n}$ be finite, let $f:E\rightarrow \lbrack
0,\infty )$, and let $M_{0}>0$. Suppose that for each $S\subset E$ with $%
\#\left( S\right) \leq k^{\#}$ there exists $\vec{P}^{S}=\left( P^{x}\right)
_{x\in S}\in Wh\left( S\right) $ such that $P^{x}\in \Gamma _{f}^{\prime
}\left( x,M_{0}\right) $ for all $x\in S$, and $\left\vert \partial ^{\beta
}\left( P^{x}-P^{y}\right) \right\vert \leq M_{0}\left\vert x-y\right\vert
^{m-\left\vert \beta \right\vert }$ for $x,y\in S$, $\left\vert \beta
\right\vert \leq m-1$.

Then there exists $F\in C^{m-1,1}\left( \mathbb{R}^{n}\right) $ such that $%
\left\Vert F\right\Vert _{C^{m-1,1}\left( \mathbb{R}^{n}\right) }\leq CM_{0}$%
, $F\geq 0$ on $\mathbb{R}^{n}$, and $F=f$ on $E$.
\end{lemma}

In adapting the proof of Theorem \ref%
{theorem-fp-for-nonnegative-interpolation}, we keep the function called $F^{0}$ in $C^{m}\left( \mathbb{R}^{n}\right) $, but we take the
functions called $F$ and $F_{\nu }$ to belong to $C^{m-1,1}\left( \mathbb{R}%
^{n}\right) $.

Now we can easily deduce the following result.

\begin{theorem}[Finiteness Principle for Non-negative $C^{m-1,1}$%
-Interpolation]
\label{theorem-fp-for-cm-1-interpolation}There exists constants $k^{\#}$, $C$%
, depending only on $m,n$ for which the following holds.

Let $f:E\rightarrow \lbrack 0,\infty )$, with $E \subset \mathbb{R}^n$
arbitrary (not necessarily finite). Let $M_0 >0$. Suppose that for each $S
\subset E$ with $\#(S) \leq k^\#$ there exists $\vec{P}=(P^x)_{x\in S}\in
Wh(S)$ such that

\begin{itemize}
\item $P^x \in \Gamma_f^{\prime }(x,M_0)$ for all $x \in S$,

\item $\left\vert \partial ^{\beta }\left( P^{x}-P^{y}\right) \left(
x\right) \right\vert \leq M_{0}\left\vert x-y\right\vert ^{m-\left\vert
\beta \right\vert }$ for $x,y\in S$, $\left\vert \beta \right\vert \leq m-1$.
\end{itemize}

Then there exists $F\in C^{m-1,1}\left( \mathbb{R}^{n}\right) $ such that

\begin{itemize}
\item $||F||_{C^{m-1,1}(\mathbb{R}^n)}\leq CM_0$,

\item $F \geq 0$, and

\item $F=f$ on $E$.
\end{itemize}
\end{theorem}

\begin{proof}
Suppose first that $E\subset Q$ for some cube $Q\subset \mathbb{R}^{n}$.
Then by Ascoli's theorem,%
\begin{equation*}
\left\{ F\in C^{m-1,1}\left( Q\right) :\left\Vert F\right\Vert
_{C^{m-1,1}\left( Q\right) }\leq CM_{0}\text{, }F\geq 0\text{ on }Q\right\}
\equiv X
\end{equation*}
is compact in the $C^{m-1}(Q)$-norm topology.

For each finite $E_0 \subset E$, Lemma \ref{lemma-inf3'} tells us that there
exists $F \in X$ such that $F = f$ on $E_0$.

Consequently, there exists $F \in X$ such that $F = f$ on $E$. That is,

\begin{itemize}
\item[\refstepcounter{equation}\text{(\theequation)}\label{afc1}] $F\in
C^{m-1,1}\left( Q\right) $, $\left\Vert F\right\Vert _{C^{m-1,1}\left(
Q\right) }\leq CM_{0}$, $F\geq 0$ on $Q$, $F=f$ on $E$.
\end{itemize}

We have achieved \eqref{afc1}, assuming that $E \subset Q$.

Now suppose $E \subset \mathbb{R}^n$ is arbitrary.

We introduce a partition of unity $1=\sum_{\nu }\theta _{\nu }$ on $\mathbb{R%
}^{n}$, with $\theta _{\nu }\geq 0$ on $\mathbb{R}^{n}$, $\theta _{\nu }\in
C^{m}\left( \mathbb{R}^{n}\right) $, $\left\Vert \theta _{\nu }\right\Vert
_{C^{m}\left( \mathbb{R}^{n}\right) }\leq C$, support $\theta _{\nu }\subset
Q_{\nu }$ for a cube $Q_{\nu }\subset \mathbb{R}^{n}$, with (say) $\delta
_{Q_{\nu }}=1$, and such that any given $x\in \mathbb{R}^{n}$ has a
neighborhood that intersects at most $C$ of the $Q_{\nu }$. (Here $C$
depends only on $m,n$.)

Applying our result \eqref{afc1} to $f|_{E\cap Q_{\nu }}:E\cap Q_{\nu
}\rightarrow \lbrack 0,\infty )$ for each $\nu $, we obtain functions $%
F_{\nu }\in C^{m-1,1}\left( Q_{\nu }\right) $ such that $\left\Vert F_{\nu
}\right\Vert _{C^{m-1,1}\left( Q_{\nu }\right) }\leq CM_{0}$, $F_{\nu }\geq
0 $ on $Q_{\nu }$, $F_{\nu }=f$ on $E\cap Q_{\nu }$.

(Here $C$ depends only on $m, n$.)

We define $F=\sum_{\nu }\theta _{\nu }F_{\nu }$ on $\mathbb{R}^{n}$. One
checks easily that $\left\Vert F\right\Vert _{C^{m-1,1}\left( \mathbb{R}%
^{n}\right) }\leq C^{\prime }M_{0}$ with $C^{\prime }$ determined by $m$, $n$%
; $F\geq 0$ on $\mathbb{R}^{n}$; and $F=f$ on $E$.

This completes the proof of Theorem \ref{theorem-fp-for-cm-1-interpolation}.
\end{proof}

Note that Theorem \ref{theorem-fp-for-cm-1-interpolation} easily implies Theorem \ref{Th4} (B).

As in the case of non-negative $C^{m}$-interpolation, we want to replace $%
\Gamma _{f}^{\prime }(x,M)$ by something easier to calculate. In the $%
C^{m-1,1}$-setting, it is enough to make the following observation.

Define $\tilde{\Gamma}_{0}^{\prime }=\{P\in \mathcal{P}:\left\vert \partial
^{\beta }P\left( 0\right) \right\vert \leq 1$ for $\left\vert \beta
\right\vert \leq m-1$ and $P\left( x\right) +\left\vert x\right\vert
^{m}\geq 0$ for all $x\in \mathbb{R}^{n}\}$.

Then

\begin{itemize}
\item[\refstepcounter{equation}\text{(\theequation)}\label{afc2}] $\Gamma
_{\ast }^{\prime }\left( 0,c\right) \subset \tilde{\Gamma}_{0}^{\prime
}\subset \tilde{\Gamma}_{\ast }^{\prime }\left( 0,C\right) $ with $c$, $C$
depending only on $m$, $n$.
\end{itemize}

Indeed, the first inclusion in \eqref{afc2} is immediate from the
definitions and Taylor's theorem. To prove the second inclusion, we let $P
\in \tilde{\Gamma}^{\prime }_0$ be given, and set $F(x)=\chi(x)(P(x)+|x|^m)$%
, where $\chi$ is a non-negative $C^m$ function with norm at most $C_{\ast}$
(depending only on $m$, $n$), satisfying $J_0(\chi)=1$ and support $\chi
\subset B_n(0,1)$.

We then have $F\in C^{m-1,1}(\mathbb{R}^{n})$, $\left\Vert F\right\Vert
_{C^{m-1,1}\left( \mathbb{R}^{n}\right) }\leq C$ (depending only on $m$, $n$%
), $F\geq 0$ on $\mathbb{R}^{n}$, $J_{0}\left( F\right) =P$. Hence, $P\in
\Gamma _{\ast }^{\prime }\left( 0,C\right) $, completing the proof of %
\eqref{afc2}.

This concludes our discussion of interpolation by nonnegative $C^{m-1,1}$
functions.

Next, we present the $C^{m-1,1}$ analogue of the Finiteness Principle for
vector-valued functions (Theorem \ref{theorem-basic-fininite-principle}).

\begin{theorem}[Finiteness Principle for Vector-Valued Functions in $%
C^{m-1,1}$]
\label{theorem-basic-finite-principle-for-cm-1} Fix $m$, $n$, $D\geq 1$.
Then there exist $k^{\#}$, $C$, determined by $m$, $n$, $D$, such that the
following holds.

Let $E \subset \mathbb{R}^n$ be arbitrary. For each $x \in E$, let $K(x)
\subset \mathbb{R}^n$ be closed and convex.

Assume that for any $S\subset E$ with $\#(S)\leq k^{\#}$, there exists $%
F^{S}\in C^{m-1,1}(\mathbb{R}^{n},\mathbb{R}^{D})$ such that $\left\Vert
F^{S}\right\Vert _{C^{m-1,1}\left( \mathbb{R}^{n},\mathbb{R}^{D}\right)
}\leq 1$ and $F^{S}\left( x\right) \in K\left( x\right) $ for all $x\in S$.

Then there exists $F\in C^{m-1,1}\left( \mathbb{R}^{n},\mathbb{R}^{D}\right) 
$ such that $\left\Vert F\right\Vert _{C^{m-1,1}\left( \mathbb{R}^{n},%
\mathbb{R}^{D}\right) }\leq C$ and $F\left( x\right) \in K\left( x\right) $
for all $x\in E$.
\end{theorem}

\begin{proof}
We write $c$, $C$, $C^{\prime }$, etc., to denote constants determined by $m$%
, $n$, $D$. These symbols may denote different constants in different
occurrences.

Suppose first that $E$ is finite.

Given $F^{S}$ as in the above hypothesis, set $\vec{P}=\left( P^{x}\right)
_{x\in S}$, with $P^{x}=J_{x}\left( F^{S}\right) \in \mathcal{P}^{D}$. Then $%
P^{x}\left( x\right) \in K\left( x\right) $ for $x\in S$, $\left\vert
\partial ^{\beta }P^{x}\left( x\right) \right\vert \leq C$ for $x\in S$, $%
\left\vert \beta \right\vert \leq m-1$, and $\left\vert \partial ^{\beta
}\left( P^{x}-P^{y}\right) \left( x\right) \right\vert \leq C\left\vert
x-y\right\vert ^{m-\left\vert \beta \right\vert }$ for $x,y\in S$, $%
\left\vert \beta \right\vert \leq m-1$.

By Whitney's extension theorem for finite sets, there exists

\begin{itemize}
\item $\tilde{F}^{S}\in C^{m}\left( \mathbb{R}^{n},\mathbb{R}^{D}\right) $
such that

\item $\left\Vert \tilde{F}^{S}\right\Vert _{C^{m}\left( \mathbb{R}^{n},%
\mathbb{R}^{D}\right) }\leq C$, and
\end{itemize}

$J_{x}\left( \tilde{F}^{S}\right) =P^{x}$ for $x\in S$; in particular,

\begin{itemize}
\item $\tilde{F}^{S}\left( x\right) =F^{S}\left( x\right) \in K\left(
x\right) $ for $x\in S$.
\end{itemize}

Thanks to the above bullet points, we have satisfied the hypotheses of the
Finiteness Principle for Vector-Valued Functions (Theorem \ref%
{theorem-basic-fininite-principle}). Hence, we obtain $F\in C^{m}\left( 
\mathbb{R}^{n},\mathbb{R}^{D}\right) \subset C^{m-1,1}\left( \mathbb{R}^{n},%
\mathbb{R}^{D}\right) $ such that $\left\Vert F\right\Vert _{C^{m-1,1}\left( 
\mathbb{R}^{n},\mathbb{R}^{D}\right) }\leq C\left\Vert F\right\Vert
_{C^{m}\left( \mathbb{R}^{n},\mathbb{R}^{D}\right) }\leq C^{\prime }$, and $%
F\left( x\right) \in K\left( x\right) $ for each $x\in E$. Thus, we have
proven Theorem \ref{theorem-basic-finite-principle-for-cm-1} in the case of
finite $E$.

Next, suppose $E$ is an arbitrary subset of a cube $Q\subset \mathbb{R}^{n}$%
. Then 
\begin{equation*}
X=\left\{ F\in C^{m-1,1}\left( Q,\mathbb{R}^{D}\right) :\left\Vert
F\right\Vert _{C^{m-1,1}\left( Q,\mathbb{R}^{D}\right) }\leq C\right\} 
\end{equation*}%
is compact in the topology of the $C^{m-1}(Q,\mathbb{R}^{D})$-norm, by
Ascoli's theorem.

For each $x\in E$, let 
\begin{equation*}
X(x)=\left\{ F\in X:F\left( x\right) \in K\left( x\right) \right\} \text{.}
\end{equation*}

Then each $X(x)$ is a closed subset of $X$, since $K(x)\subset \mathbb{R}%
^{n} $ is closed. Moreover, given finitely many points $x_{1},\cdots
,x_{N}\in E$, we have $X(x_{1})\cap \cdots \cap X(x_{N})\not=\emptyset $,
thanks to Theorem \ref{theorem-basic-finite-principle-for-cm-1} in the known
case of finite sets. 

Consequently, $\bigcap_{x\in E}X(x)\not=\emptyset $. Thus, there exists $%
F\in C^{m-1,1}\left( Q,\mathbb{R}^{D}\right) $ such that

\begin{itemize}
\item[\refstepcounter{equation}\text{(\theequation)}\label{afc3}] $%
\left\Vert F\right\Vert _{C^{m-1,1}\left( Q,\mathbb{R}^{D}\right) }\leq C$
and $F\left( x\right) \in K\left( x\right) $ for all $x\in E$.
\end{itemize}

We have achieved \eqref{afc3} under the assumption $E\subset Q$.

Finally, let $E\subset \mathbb{R}^{n}$ be arbitrary.

We introduce a partition of unity $1=\sum_{\nu }\theta _{\nu }$ on $\mathbb{R%
}^{n}$ with $\theta _{\nu }\in C^{m}\left( \mathbb{R}^{n}\right) $, $%
\left\Vert \theta \right\Vert _{C^{m}\left( \mathbb{R}^{n}\right) }\leq C$, $%
\theta _{\nu }\geq 0$ on $\mathbb{R}^{n}$, support $\theta _{\nu }\subset
Q_{\nu }$ where $Q_{\nu }$ is a cube of sidelength $1$, and any given $x\in 
\mathbb{R}^{n}$ has a neighborhood that meets at most $C$ of the $Q_{\nu }$.
For each $\nu $, the known case of Theorem \ref%
{theorem-basic-finite-principle-for-cm-1} yields a function $F_{\nu }\in
C^{m-1,1}\left( Q_{\nu },\mathbb{R}^{D}\right) $ with $\left\Vert F_{\nu
}\right\Vert _{C^{m-1,1}\left( Q_{\nu },\mathbb{R}^{D}\right) }\leq C$ and $%
F_{\nu }\left( x\right) \in K\left( x\right) $ for all $x\in E\cap Q_{\nu }$.

We set $F=\sum_{\nu }\theta _{\nu }F_{\nu }$ on $\mathbb{R}^{n}$. One checks
easily that $F\in C^{m-1,1}\left( \mathbb{R}^{n},\mathbb{R}^{D}\right) $, $%
\left\Vert F\right\Vert _{C^{m-1,1}\left( \mathbb{R}^{n},\mathbb{R}%
^{D}\right) }\leq C$, and $F\left( x\right) \in K\left( x\right) $ for all $%
x\in E$.

This completes the proof of Theorem \ref%
{theorem-basic-finite-principle-for-cm-1}.
\end{proof}

\bigskip

\bibliographystyle{plain}
\def\cprime{$'$} \def\cprime{$'$}

\vspace{1mm} 
\newpage

\begin{flushleft} 
Charles Fefferman \\
Affiliation: Department of Mathematics, Princeton University, Fine Hall Washington Road, Princeton, New Jersey, 08544, USA \\
Email: cf$\MVAt$math.princeton.edu  \\
\vspace{2mm} 
Arie Israel \\ 
Affiliation: The University of Texas at Austin, 
Department of Mathematics,
2515 Speedway Stop C1200, 
Austin, Texas, 78712-1202, USA \\ 
Email: arie$\MVAt$math.utexas.edu \\
\vspace{2mm} 
Garving K. Luli \\ 
Affiliation: Department of Mathematics, 
University of California, Davis, 
One Shields Ave, 
Davis, California, 95616, USA \\ 
Email: kluli@$\MVAt$math.ucdavis.edu 
\end{flushleft}

\end{document}